\documentclass{amsart} 

\usepackage{AKstyle}

\numberwithin{figure}{section}

\usepackage{tikz}
\usetikzlibrary{decorations.markings}
\usepackage[margin=1in]{geometry}

\numberwithin{equation}{section}

\definecolor{gold}{rgb}{1.0, 0.75, 0.0}
\usepackage{hyperref}

\newcommand{\SW}[1]{\hyperref[SW#1]{SW#1}}

\usepackage{caption}
\begin{document}

\title{
Arithmetic Polyhedra}
\author{Daniel Allcock}
\email{allcock@math.utexas.edu}
\address{Department of Mathematics, UT Austin, Austin, TX}

\author{Pat Devlin}
\email{pdevlin2@swarthmore.edu}
\address{Department of Mathematics, Swarthmore College, Swarthmore, PA}

\author{Anna Felikson}
\email{anna.felikson@durham.ac.uk}
\address{Department of Mathematics, Durham University, Durham, UK}

\author{Alex Kontorovich}
\email{alex.kontorovich@rutgers.edu}
\address{Department of Mathematics, Rutgers University, New Brunswick, NJ}

\author{Ian Whitehead}
\email{whiteh2@stolaf.edu}
\address{Department of Mathematics, Statistics, and Computer Science, St. Olaf College, Northfield, MN}

\begin{abstract} \noindent    The Koebe-Andreev-Thurston theorem assigns a 
    $3$-dimensional hyperbolic reflection group
    to each combinatorial polyhedron. A natural question is: which of them 
    are arithmetic? In 2016, Kontorovich-Nakamura conjectured that all arithmetic reflection groups 
    obtained in this way are commensurable to those obtained from the tetrahedron, square pyramid, or cuboctahedron. In this paper, we prove the conjecture. It is a consequence of the following result of independent interest: all arithmetic ideal, right-angled hyperbolic polyhedra are obtained by gluing together copies of one of three ``seed'' polyhedra. 
\end{abstract}

\maketitle

\tableofcontents

\begingroup
\renewcommand{\thesection}{1} 
\begin{figure}[h]
    \centering
\, \hfill \includegraphics[width=.3\textwidth]{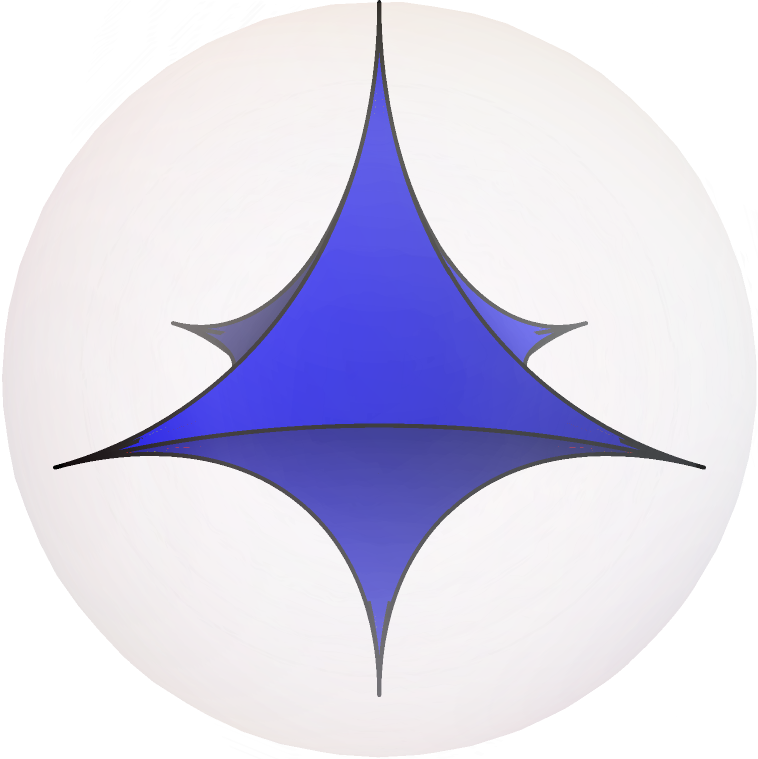} \hfill 
\includegraphics[width=.3\textwidth]{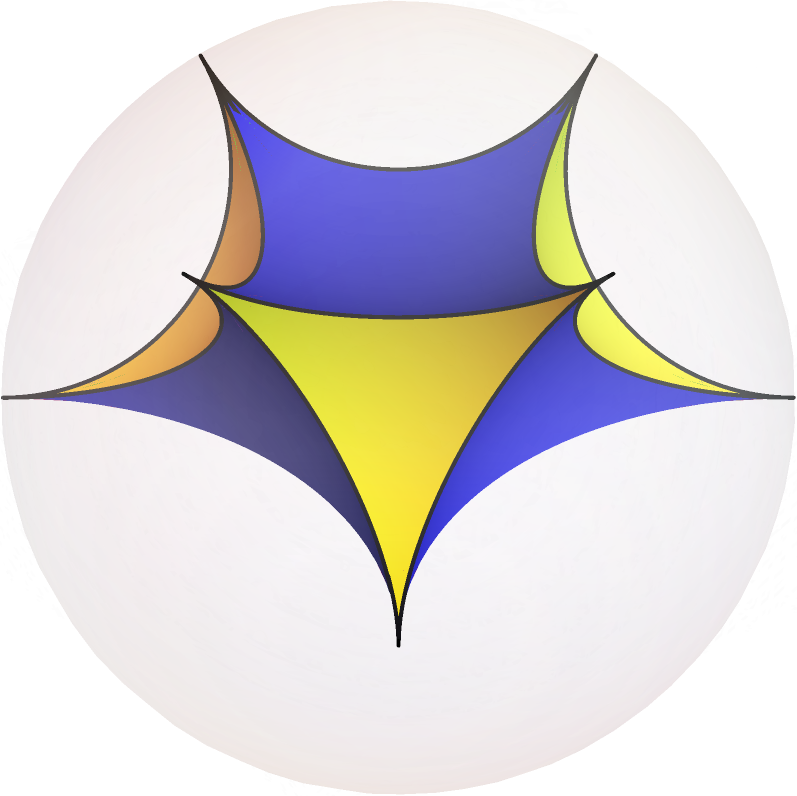}\hfill
\includegraphics[width=.3\textwidth]{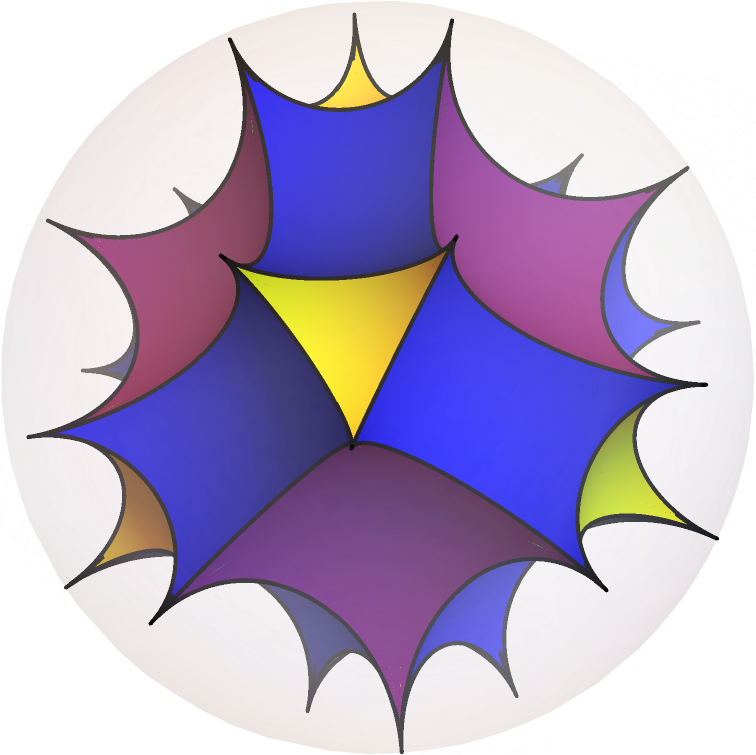} \hfill \,
\caption{The three commensurability classes of arithmetic, ideal, right-angled hyperbolic polyhedra. From left: octahedron, square antiprism, and rhombicuboctahedron.}
\label{fig:placeholder}
\end{figure}
\endgroup

\newpage

\section{Introduction}

\subsection{Motivating Question}\

By a \emph{combinatorial polyhedron}, we will mean a graph that is isomorphic to the $1$-skeleton of a convex 3-dimensional Euclidean polyhedron. The Koebe--Andreev--Thurston (KAT) theorem, reviewed below, assigns to each combinatorial polyhedron~$G$ a 
\emph{canonical} $3$-dimensional hyperbolic polyhedron~$\cP(G)$. This $\cP(G)$ has finite volume, with  all ideal vertices and all right dihedral angles. It is a fundamental
domain for the group $\Gamma(\cP(G))$ generated by the reflections across its
faces. The group $\Gamma(\cP(G))$ is a lattice in the isometry group $\Isom(\bH^3)$ of hyperbolic $3$-space, i.e., a discrete subgroup with finite-volume quotient.  In this paper, we answer the following natural question:

\begin{question} \label{QuestionVersion1}
    When is the lattice $\Gamma(\cP(G))$ arithmetic?     
\end{question}

\noindent
Vinberg's criterion for arithmeticity \cite{Vinberg1967} (see \secref{sec:Arithmeticity} below) can be applied to any chosen
$\cP(G)$, giving an algorithmic answer to the question that treats one case
at a time.  Our goal is different:
a global description of all $\cP(G)$ for which $\Gamma(\cP(G))$ is arithmetic.

\subsection{Koebe-Andreev-Thurston Polyhedra}\

We sketch the KAT construction. However, it is not really necessary for understanding our results.  
Indeed, the properties of $\cP(G)$ mentioned above show that an answer to \qref{QuestionVersion1}
follows from an answer to \qref{QuestionVersion2} below.  The latter
concerns only hyperbolic polyhedra, and does not involve the KAT construction.

By a \emph{geometrization} of a combinatorial polyhedron~$G$, we will mean a choice of polyhedron $\Pi$ in Euclidean space, with 1-skeleton isomorphic to $G$, that admits a ``midsphere,'' that is, a
sphere tangent to all edges.  The KAT theorem states
that every combinatorial polyhedron has a geometrization.  This has been proved at many
levels of generality; see, e.g., \cite{Koebe1936, Andreev1970a, Andreev1970c,
Thurstonbook, Schramm1991, Schramm1991b, HeSchramm1993, Rivin1994, Rivin1996,
ColindeVerdiere1991, ChowLuo2003, BrightwellScheinerman1993, Stephenson2005,
Ziegler2007, BobenkoSpringborn2004, RoederHubbardDunbar2007, Kapovich2001}. 
\figref{fig:KATpic} shows a geometrization $\Pi$ of the $1$-skeleton $G$ of a triangular prism (in green). Simultaneously, it shows a geometrization of its dual polyhedron $\widehat\Pi$ (the triangular bipyramid, in red). Their common midsphere is shown in yellow.

\setcounter{figure}{1}
\begin{figure}[ht]
        \centering
        \includegraphics[width=.48\textwidth]{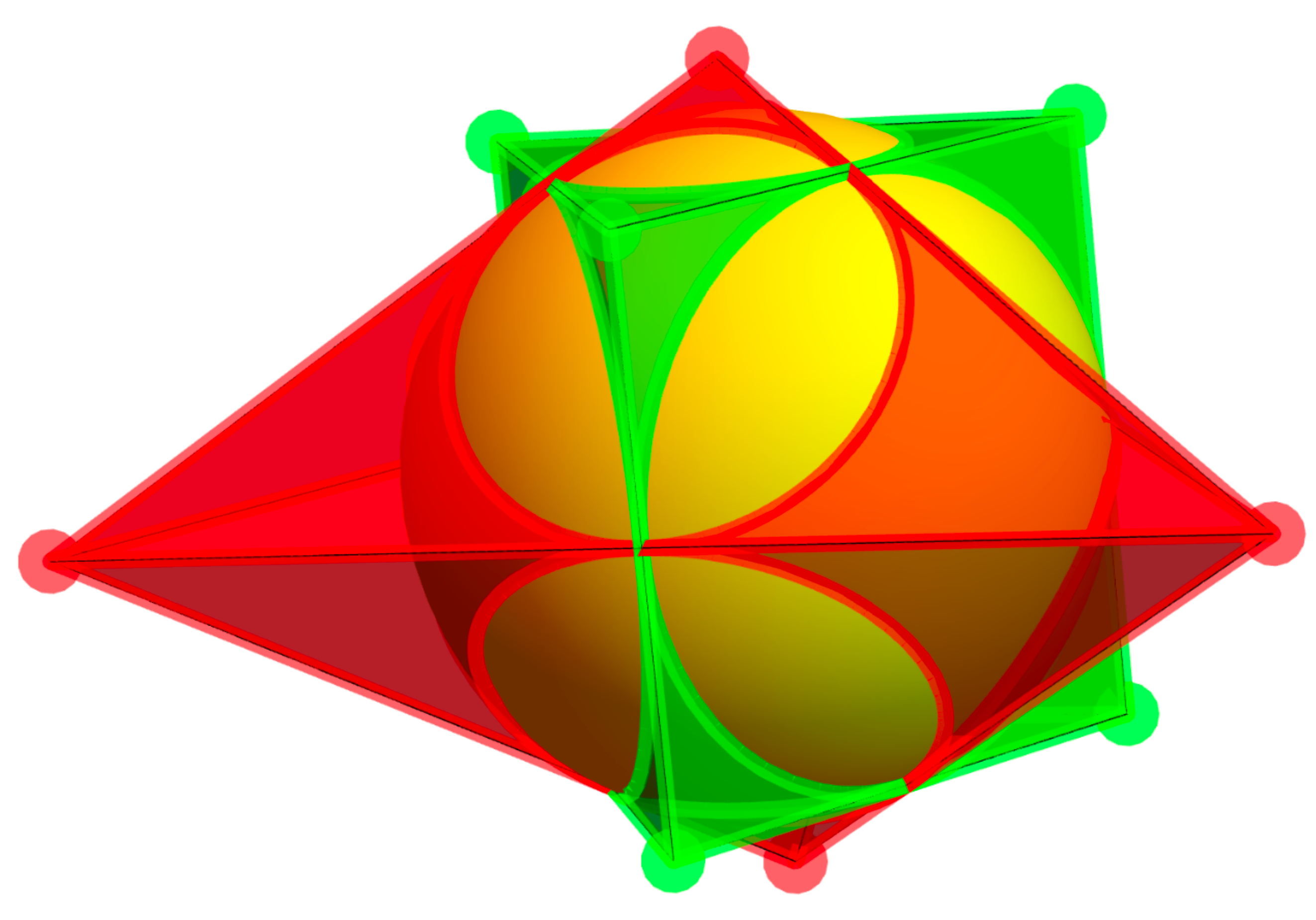} \hfill
        \includegraphics[width=.48\textwidth]{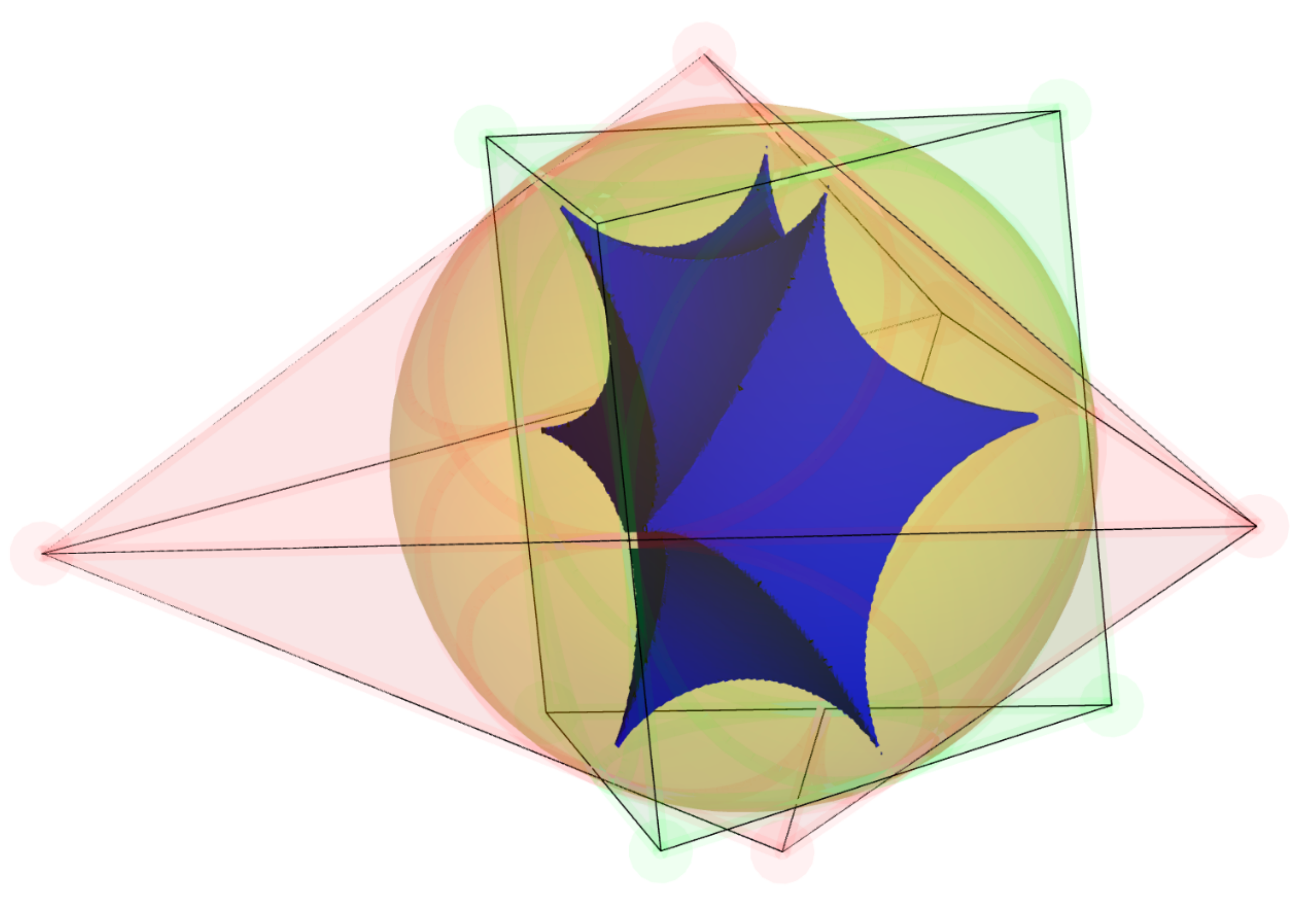}
        \caption{Geometrized polyhedron $\Pi$ and its dual $\widehat \Pi$; the corresponding hyperbolic polyhedron $\cP(\Pi)$ in the ball model of $\bH^3$}
    \label{fig:KATpic}
\end{figure}

We 
identify the midsphere with the boundary of
the ball model of 
hyperbolic space, and
define $\cP(\Pi)$ as the 
convex hull, in the hyperbolic ball, of the points
of tangency.  
The combinatorial polyhedron underlying $\cP(\Pi)$
is the \emph{rectification} of~$\Pi$, also called the ``maximal truncation,'' or, in the graph-theoretic context, the ``medial graph.'' See \figref{fig:rectification}.
The faces of $\cP(\Pi)$ correspond to the
faces and vertices of~$\Pi$, and its vertices correspond to the edges of~$\Pi$. For $\Pi$ the triangular prism, $\cP(\Pi)$ appears in \figref{fig:KATpic}. It has $3$ quadrilateral faces
(corresponding to $\Pi$'s rectangular faces) and $8$ triangular faces (two corresponding to $\Pi$'s triangular faces and six corresponding
to $\Pi$'s vertices).  

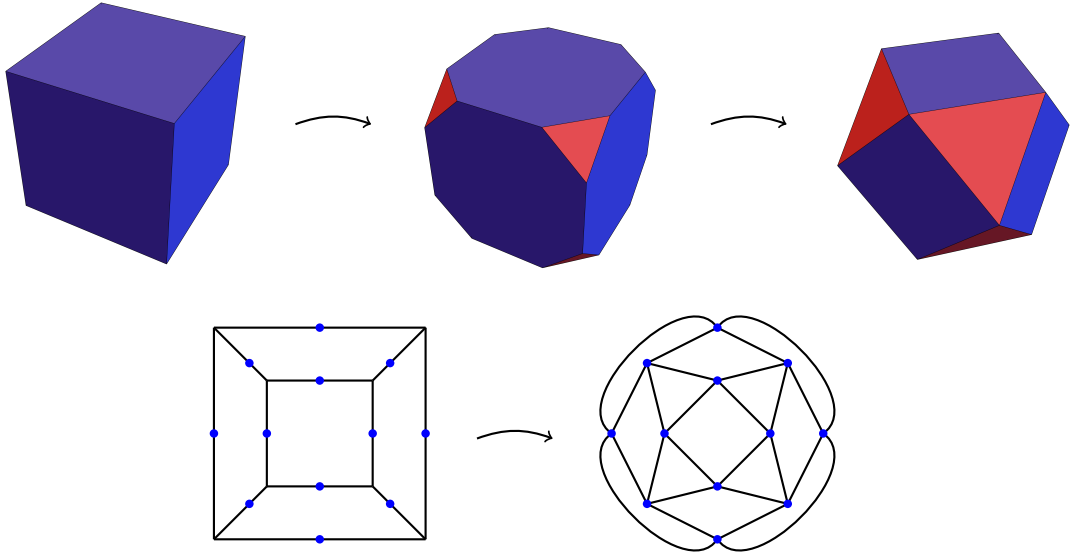
\begin{figure}[ht]
    \centering
    \, \hfill \hfill
    \includegraphics[width=.2\textwidth]{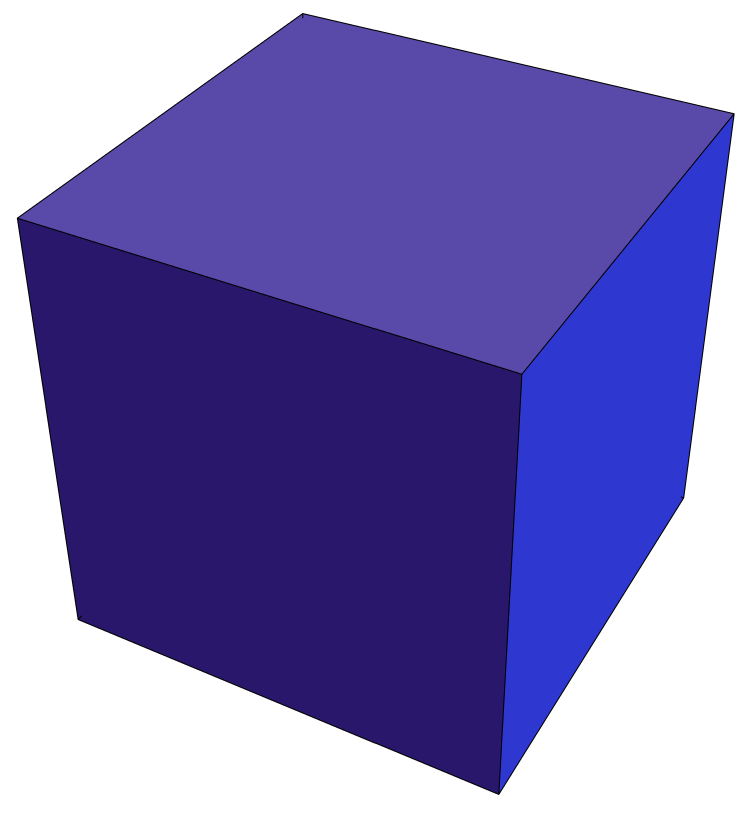} \hfill \begin{tikzpicture}[scale=1, baseline=(current bounding box.center)]
\draw[white] (0, -4) -- (0,0);
\draw[thick, ->] (0,0) to[bend left=20] (1,0);
\end{tikzpicture} \hfill \includegraphics[width=.2\textwidth]{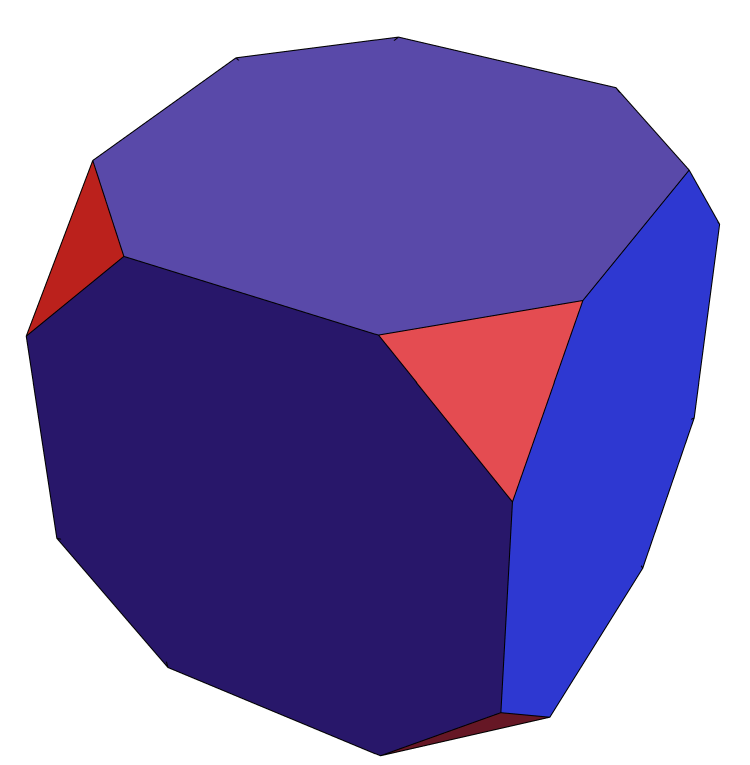} \hfill \begin{tikzpicture}[scale=1, baseline=(current bounding box.center)]
\draw[white] (0, -4) -- (0,0);
\draw[thick, ->] (0,0) to[bend left=20] (1,0);
\end{tikzpicture} \hfill \includegraphics[width=.2\textwidth]{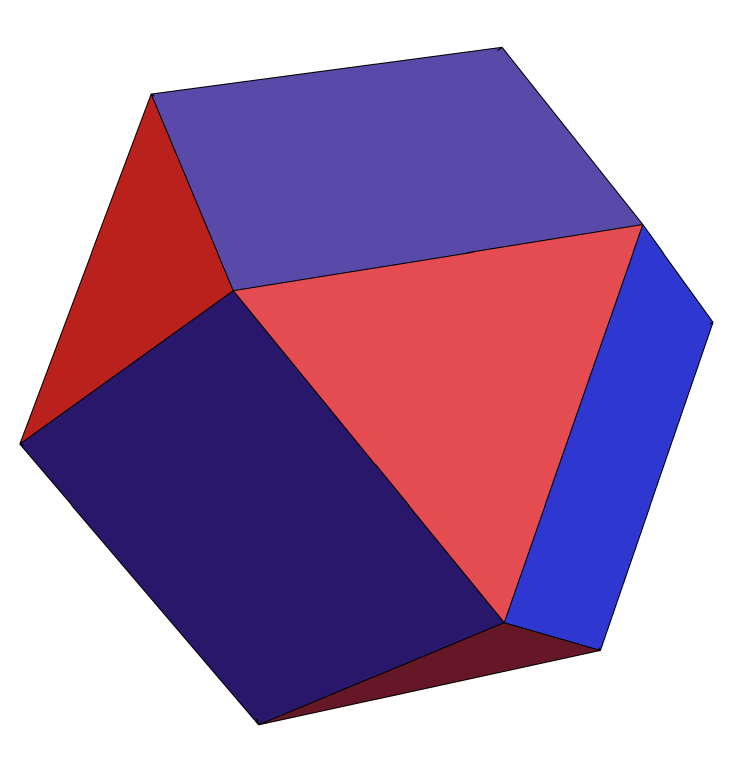} \hfill \hfill \,

\vspace{-2cm}
\, \hfill
\begin{tikzpicture}[scale=.7, baseline=(current bounding box.center)]
  \coordinate (A) at (1,1);
  \coordinate (B) at (1,-1);
  \coordinate (C) at (-1,-1);
  \coordinate (D) at (-1,1);
  \coordinate (E) at (2,2);
  \coordinate (F) at (2,-2);
  \coordinate (G) at (-2,-2);
  \coordinate (H) at (-2,2);

  \draw[thick] (E) -- (F);
  \draw[thick] (F) -- (G);
  \draw[thick] (G) -- (H);
  \draw[thick] (H) -- (E);
  
  \draw[thick] (A) -- (E);
  \draw[thick] (B) -- (F);
  \draw[thick] (C) -- (G);
  \draw[thick] (D) -- (H);
  
  \draw[thick] (A) -- (B);
  \draw[thick] (B) -- (C);
  \draw[thick] (C) -- (D);
  \draw[thick] (D) -- (A);

    \coordinate (MAB) at (1,0);        
  \coordinate (MBC) at (0, -1);        
  \coordinate (MCD) at (-1,0);        
  \coordinate (MDA) at (0,1);        
  \coordinate (MAE) at (1.33,1.33);  
  \coordinate (MBF) at (1.33,-1.33);  
  \coordinate (MCG) at (-1.33,-1.33);  
  \coordinate (MDH) at (-1.33,1.33);  
  \coordinate (MEF) at (2,0);    
  \coordinate (MFG) at (0,-2);    
  \coordinate (MGH) at (-2,0);    
  \coordinate (MHE) at (0,2); 

  \fill[blue] (MAB) circle (0.08);
  \fill[blue] (MBC) circle (0.08);
  \fill[blue] (MCD) circle (0.08);
  \fill[blue] (MDA) circle (0.08);
  \fill[blue] (MAE) circle (0.08);
  \fill[blue] (MBF) circle (0.08);
  \fill[blue] (MCG) circle (0.08);
  \fill[blue] (MDH) circle (0.08);
  \fill[blue] (MEF) circle (0.08);
  \fill[blue] (MFG) circle (0.08);
  \fill[blue] (MGH) circle (0.08);
  \fill[blue] (MHE) circle (0.08);
\end{tikzpicture}
\begin{tikzpicture}[scale=1, baseline=(current bounding box.center)]
\draw[white] (-.5, 0) -- (0,0);
\draw[thick, ->] (0,0) to[bend left=20] (1,0);
\end{tikzpicture}
\begin{tikzpicture}[scale=.7, baseline=(current bounding box.center)]
  \coordinate (A) at (1,1);
  \coordinate (B) at (1,-1);
  \coordinate (C) at (-1,-1);
  \coordinate (D) at (-1,1);
  \coordinate (E) at (2,2);
  \coordinate (F) at (2,-2);
  \coordinate (G) at (-2,-2);
  \coordinate (H) at (-2,2);

    \coordinate (MAB) at (1,0);        
  \coordinate (MBC) at (0, -1);        
  \coordinate (MCD) at (-1,0);        
  \coordinate (MDA) at (0,1);        
  \coordinate (MAE) at (1.33,1.33);  
  \coordinate (MBF) at (1.33,-1.33);  
  \coordinate (MCG) at (-1.33,-1.33);  
  \coordinate (MDH) at (-1.33,1.33);  
  \coordinate (MEF) at (2,0);    
  \coordinate (MFG) at (0,-2);    
  \coordinate (MGH) at (-2,0);    
  \coordinate (MHE) at (0,2); 

  \draw[thick] (MAB) -- (MDA);
  \draw[thick] (MAB) -- (MAE);
  \draw[thick] (MDA) -- (MAE);
  \draw[thick] (MAB) -- (MBC);
  \draw[thick] (MAB) -- (MBF);
  \draw[thick] (MBC) -- (MBF);
  \draw[thick] (MBC) -- (MCD);
  \draw[thick] (MBC) -- (MCG);
  \draw[thick] (MCD) -- (MCG);
  \draw[thick] (MDA) -- (MCD);
  \draw[thick] (MDA) -- (MDH);
  \draw[thick] (MCD) -- (MDH);
  \draw[thick] (MAE) -- (MEF);
  \draw[thick] (MAE) -- (MHE);
  \draw[thick] (MHE) to[bend left=100] (MEF);
  \draw[thick] (MBF) -- (MEF);
  \draw[thick] (MBF) -- (MFG);
  \draw[thick] (MEF) to[bend left=100] (MFG);
  \draw[thick] (MCG) -- (MFG);
  \draw[thick] (MCG) -- (MGH);
    \draw[thick] (MFG) to[bend left=100] (MGH);
    \draw[thick] (MDH) -- (MGH);  
    \draw[thick] (MDH) -- (MHE);  
    \draw[thick] (MGH) to[bend left=100] (MHE);

  \fill[blue] (MAB) circle (0.08);
  \fill[blue] (MBC) circle (0.08);
  \fill[blue] (MCD) circle (0.08);
  \fill[blue] (MDA) circle (0.08);
  \fill[blue] (MAE) circle (0.08);
  \fill[blue] (MBF) circle (0.08);
  \fill[blue] (MCG) circle (0.08);
  \fill[blue] (MDH) circle (0.08);
  \fill[blue] (MEF) circle (0.08);
  \fill[blue] (MFG) circle (0.08);
  \fill[blue] (MGH) circle (0.08);
  \fill[blue] (MHE) circle (0.08);

\end{tikzpicture}
\hfill \,

\caption{Above: Truncating a cube to its rectification, a cuboctahedron.\\
Below: The corresponding graph-theoretic rectification (medial graph). Midpoints of edges of the cube are marked on the left, and become vertices on the right.}
\label{fig:rectification}
\end{figure}

By construction, $\cP(\Pi)$ is \emph{ideal}, that is, it has all its vertices at the boundary of $\bH^3$. It is also \emph{right-angled}, meaning that every dihedral angle is~$\pi/2$. We see this in \figref{fig:KATpic}: the faces of $\cP(\Pi)$ lie in the hyperbolic planes bounded by the red and green circles.  These circles meet orthogonally whenever they meet transversally.  Because the Poincar\'e ball model is conformal, the planes bounded by two such circles meet each other at the same angle as the circles do, that is, orthogonally.

These properties also imply that $\cP(\Pi)$ is determined up to isometry by the combinatorial data of $G$. Because $\cP(\Pi)$ is right-angled, the Poincar\'e Polyhedron Theorem shows that its reflection group $\Gamma(\cP(\Pi))$ is a discrete subgroup of $\Isom(\bH^3)$, with $\cP(\Pi)$ as a fundamental domain. It also gives a presentation for $\Gamma(\cP(\Pi))$ that involves only the combinatorics of~$\cP(\Pi)$, hence depends only on~$G$.  (The proof
of Poincar\'e's theorem is quite simple for reflection groups; see \cite{delaHarpe}.)  Also, $\cP(\Pi)$ has finite volume because it is the convex hull of a finite number of ideal points. So $\Gamma(\cP(\Pi))$ is a lattice in $\Isom(\bH^3)$. Finally, the group structure of $\Gamma(\cP(\Pi))$ determines $\cP(\Pi)$ up to isometry by the Mostow rigidity theorem. This shows that $\cP(\Pi)$ depends only on~$G$ and is independent of its geometrization, justifying the notation~$\cP(G)$.

\subsection{Reformulation}\label{sec:Reformulation}\

Based on the fact that $\cP(G)$ is an ideal, right-angled hyperbolic polyhedron, we can see that
\qref{QuestionVersion1} is a special case of the following question, which we answer
in \thmref{Thm3CommensurabilityClasses} and \thmref{Thm:MainConverse}.

\begin{question} \label{QuestionVersion2}
    For which ideal, right-angled hyperbolic polyhedra $\cP$ is $\Gamma(\cP)$  an arithmetic lattice?
\end{question}

This question is of independent interest, in light of the distinguished role played by ideal, right-angled groups in a variety of subjects, including reflection groups, Coxeter groups, knot and link complements, and circle packings.  Also, it is only seemingly more general than \qref{QuestionVersion1}:

\begin{prop} \label{prop:Ian}
Every ideal,  right-angled hyperbolic polyhedron has the form $\cP(G)$ for some combinatorial polyhedron~$G$.
\end{prop}

\noindent
This fact, which we do not need in this  paper, is
closely related to results of Andreev in \cite{Andreev1970c}, and its proof
will appear in forthcoming work of Whitehead \cite{Whitehead2026}.

\subsection{Results}\label{sec:MainThms}\

All hyperbolic polyhedra appearing in this paper are Coxeter polyhedra,
meaning that their dihedral angles are integral submultiples of~$\pi$.  
See \secref{Notation} for background.
Just as for
right-angled polyhedra, the group $\Gamma(\cP)$
generated by reflections across the
faces of a finite-volume Coxeter polyhedron $\cP\subseteq \bH^3$ 
is a lattice
in $\Isom(\bH^3)$.  We will call $\cP$ \emph{arithmetic} if $\Gamma(\cP)$
is, and we will say that two Coxeter polyhedra are \emph{commensurable}
if their reflection groups are.
Our first main result was conjectured around 2016 by Kontorovich and Nakamura 
 (see \cite{KontorovichNakamura2019, Kontorovich2016IAS}):

\begin{theorem}
    \label{Thm3CommensurabilityClasses}
    The following ideal, right-angled polyhedra are arithmetic, and
    every ideal, right-angled arithmetic polyhedron is commensurable to exactly
    one of them:
    \begin{enumerate}
            \item
                the ideal, right-angled octahedron $\cO$,
            \item
                the ideal, right-angled square antiprism $\cA$, and
            \item
                the ideal, right-angled rhombicuboctahedron $\cR$.
    \end{enumerate}
\end{theorem}

The hyperbolic polyhedra $\cO$, $\cA$, and $\cR$, shown in \figref{fig:placeholder}, exist by applying the KAT construction to the tetrahedron, square pyramid, and cuboctahedron, respectively. We recall that the cuboctahedron means the common rectification of the cube and octahedron; see again \figref{fig:rectification}. The rhombicuboctahedron is defined as the common rectification of the cuboctahedron and its dual, the rhombic dodecahedron.  It has $26$ faces: $18$ are quadrilateral and $8$ are
triangular.

There are many ideal, right-angled polyhedra within each commensurability class. For example, one can glue an arithmetic ideal, right-angled polyhedron to its reflection across a face, and the result is a new arithmetic ideal, right-angled polyhedron. One can construct infinitely many arithmetic examples by iterating this process.   More generally, one can glue two ideal, right-angled polyhedra together along a common face, and the result is also ideal and right-angled, but might or might not be arithmetic.  Therefore, the best
description one could hope for, of a general 
ideal, right-angled arithmetic hyperbolic polyhedron~$\cQ$, is that it can be obtained by assembling uncomplicated pieces in an uncomplicated way.
\thmref{Thm:MainConverse} 
provides such a description.

\begin{figure}[ht] 
\includegraphics[width=.3\textwidth]{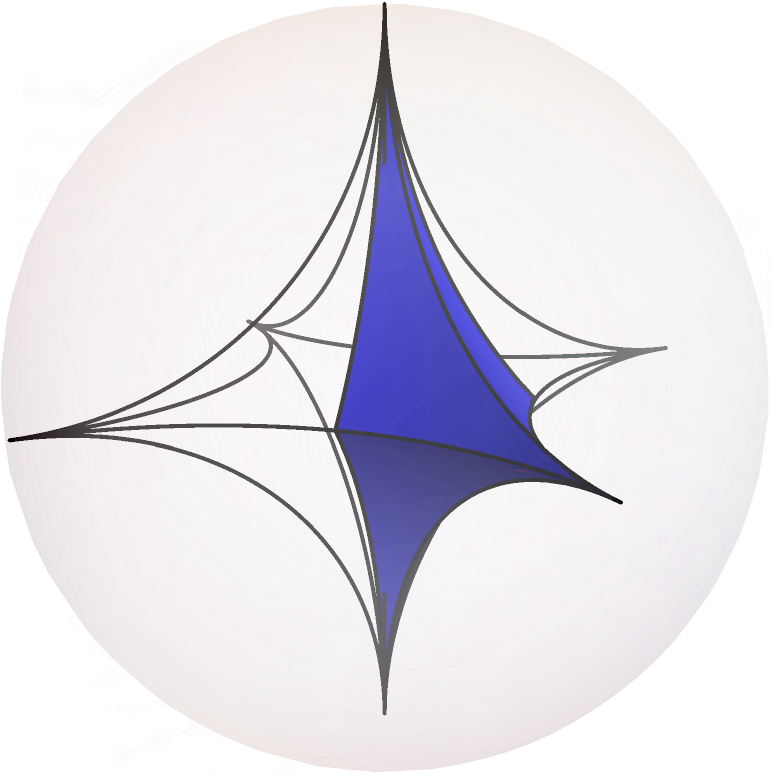} 
\caption{The right-angled triangular bipyramid $\frac14 \cO$, cut from the ideal, right-angled octahedron~$\cO$.}
\label{FigOctahedron''}
\end{figure}

\medskip
Before stating it,
we define one more hyperbolic polyhedron: the right-angled triangular bipyramid, $\frac14 \cO$. It is
right-angled, but not ideal: three of its vertices are ideal and
two are not. In fact,
$\cO$ can be cut into four copies of $\frac14 \cO$, explaining the notation.
Namely, 
choose a pair of opposite vertices of $\cO$, find the
two planes of symmetry of~$\cO$ that contain these vertices
but no others, and cut along them; see \figref{FigOctahedron''}.
The following theorem contains the results of 
Propositions \ref{Rhombicuboctahedron1}, \ref{Antiprism}, \ref{Rhombicuboctahedron2}
and~\ref{Bipyramid}.
These propositions themselves are more precise, giving in each case a 
\emph{canonical} tiling of~$\cQ$.

\begin{theorem}
    \label{Thm:MainConverse}
    \leavevmode
    \begin{enumerate}
        \item     Suppose  that $\cQ$ is an ideal, right-angled polyhedron commensurable
    with the ideal, right-angled octahedron~$\cO$.  Then $\cQ$ is tiled by~$\frac14 \cO$.  
\item 
Suppose that $\cQ$ is an ideal, right-angled polyhedron commensurable with the ideal, right-angled square antiprism ~$\cA$. Then $\cQ$ is tiled by $\cA$. 
\item 
Suppose that $\cQ$ is an ideal, right-angled polyhedron commensurable with the ideal, right-angled rhombicuboctahedron ~$\cR$. Then $\cQ$ is tiled by $\cR$.
    \end{enumerate}
\end{theorem}

See \secref{sec:refl} for the precise definition of tiling in these statements. We conjecture that in part (1), the tiling by $\frac14 \cO$ can be improved to a tiling by $\cO$, see Conjecture~\ref{Octahedron}.

\subsection{Application to Hyperbolic Link Complements}\

Reid showed that the figure eight knot complement is the only  hyperbolic knot complement whose fundamental group is arithmetic \cite{Reid1991}. The analogous problem is still open for link complements. Many link complements are closely related to quotients of  $\mathbb{H}^3$ by ideal, right-angled  reflection groups. For example, the $(2n)$-chain link complement is a hyperbolic manifold obtained by gluing four copies of the ideal, right-angled  hyperbolic $n$-gonal antiprism \cite[Section 6.8]{Thurstonbook}. The articles \cite{MeyerMillichapTrapp, NeumannReid, Kellerhals} give different proofs of the following theorem:
\begin{theorem} 
The complement of the $(2n)$-chain link is an arithmetic hyperbolic manifold if and only if $n=3$ or $4$. 
\end{theorem}

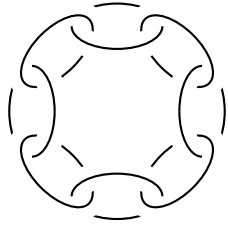
\begin{figure}[ht]
    \centering
\begin{tikzpicture}[scale=.75]
  
  \foreach \angle in {0, 90, 180, 270} {
    \draw[thick, rotate=\angle] (1.5,0) ellipse (0.4 and 0.8);
  }
  
  \foreach \angle in {45, 135, 225, 315} {
    \draw[white, line width=10pt, rotate=\angle] (1.9,0) 
      arc[start angle=0, end angle=90, x radius=0.4, y radius=0.8];
          \draw[white, line width=10pt, rotate=\angle] (1.9,0) 
      arc[start angle=0, end angle=-90, x radius=0.4, y radius=0.8];
  }
  
  \foreach \angle in {45, 135, 225, 315} {
    \draw[thick, rotate=\angle] (1.5,0) ellipse (0.4 and 0.8);
  }

  \foreach \angle in {0, 90, 180, 270} {
    \draw[white, line width=10pt, rotate=\angle] (1.1,0) 
      arc[start angle=180, end angle=270, x radius=0.4, y radius=0.8];
          \draw[white, line width=10pt, rotate=\angle] (1.1,0) 
      arc[start angle=180, end angle=90, x radius=0.4, y radius=0.8];
  }

  \foreach \angle in {0, 90, 180, 270} {
    \draw[thick, rotate=\angle] (1.1,0) 
      arc[start angle=180, end angle=270, x radius=0.4, y radius=0.8];
          \draw[thick, rotate=\angle] (1.1,0) 
      arc[start angle=180, end angle=90, x radius=0.4, y radius=0.8];
  }
  
\end{tikzpicture}
    \caption{$(2n)$-chain link for $n=4$.}
    \label{fig:chainlink}
\end{figure}

We will briefly sketch a proof of this theorem from our results. Because the fundamental group of the $(2n)$-chain link is commensurable with the reflection group of the ideal, right-angled  $n$-gonal hyperbolic antiprism, it suffices to show that this group is arithmetic if and only if $n=3$ or $4$. If $n=3$, the ideal, right-angled  $n$-gonal hyperbolic antiprism is the octahedron $\cO$ and if $n=4$ it is the square antiprism $\cA$. To deduce the ``only if'' part from \thmref{Thm3CommensurabilityClasses} and \thmref{Thm:MainConverse}, we can check that for $n>4$, the $n$-gonal antiprism $\cA_n$ is not tiled by the bipyramid $\frac14 \cO$, the antiprism $\cA$, or the rhombicuboctahedron $\cR$. Suppose that $\cA_n$ is tiled by one of these three. Note that $\frac14 \cO$, $\cA$, and $\cR$ have only triangular and quadrilateral faces. Thus the $n$-gonal face of $\cA_n$, for $n>4$, must be tiled by multiple faces of $\frac14 \cO$, $\cA$, or $\cR$. In particular, there must be some vertex $V$ on this face with two or more incident tiles along the face (three or more if the tile is $\frac14 \cO$). The vertex $V$ is an ideal vertex of $\cA_n$, with two pairs of parallel faces intersecting orthogonally. If one face contains two or more copies of $\cA$ or $\cR$ (or three or more copies of $\frac14 \cO$), then another face through $V$ must as well. See \corref{RectangleFact} below. This other face must be at least a quadrilateral. But besides the $n$-gonal face of $\cA_n$, the other faces through $V$ are all triangles, a contradiction. 

\subsection{Application to Superintegral Polyhedral Circle Packings}\

Kontorovich and Nakamura introduced polyhedral circle packings constructed via the Koebe-Andreev-Thurston theorem \cite{KontorovichNakamura2019}. The faces of the geometrized polyhedron $\Pi$ and its dual polyhedron $\widehat{\Pi}$ cut out a pair of dual circle configurations on the surface of the midsphere. The orbit of the circles associated to $\Pi$, under the group generated by reflections through the circles associated to $\widehat{\Pi}$, is an infinite, fractal circle packing. The orbit under the group generated by reflections through the circles associated to $\Pi$ and $\widehat{\Pi}$ is a larger fractal collection of circles, called a superpacking. The article \cite{KontorovichNakamura2019} asks the question: which polyhedra admit integral (super-)packings, that is (super-)packings where all the circles have integer curvatures? A polyhedron $\Pi$ is called (super-)integral if it admits an integral (super-)packing. \cite{KontorovichNakamura2019} establishes that if $\Pi$ is superintegral, then the reflection group $\Gamma$ is arithmetic. Thus our main theorems give a full classification of superintegral polyhedra: they arise from the commensurability classes of the tetrahedron (giving the Apollonian circle packing), the square pyramid (giving the Guettler-Mallows octahedral packing and the cubic packing), and the cuboctahedron/rhombic dodecahedron.

\subsection{Outline of the Proofs}\

Suppose $\cQ\subseteq \bH^3$ is an ideal right-angled arithmetic polyhedron with finite volume. The fact that $\Gamma(\cQ)$ is arithmetic, generated by reflections, and non-cocompact, forces it to preserve an isotropic symmetric bilinear form $\Phi$ over~$\Z$
with signature $(3,1)$; see \secref{sec:Arithmeticity}.
But only finitely many forms $\Phi$ (up to scale and isometry)
contain a reflective lattice in their orthogonal groups. These reflective lattices are partially ordered under inclusion, and R.~Scharlau found the maximal groups.  These are the groups $\Gamma(\cP)$ where $\cP$ varies over $49$ specific Coxeter polyhedra, called the Scharlau-Walhorn (SW) polyhedra because they were
published in \cite{ScharlauWalhorn1992}. So $\Gamma(\cQ)$ is a subgroup of  one of these~$\Gamma(\cP)$.  In particular, $\cQ$ is reflectively tiled by~$\cP$. See \secref{Notation} for more background, citations and definitions.

We examine the~$49$ cases.  For each~$\cP$, we consider the
ideal, right-angled polyhedra~$\cQ$ that $\cP$ tiles reflectively. The $\cQ$ that arise this way are exactly the answer to \qref{QuestionVersion2}.
For most~$\cP$, no such~$\cQ$ arise, by a variety of 
arguments concerning the shape of~$\cP$.
For example, Proposition~\ref{BadFacePair} says that
if $\cP$ has two disjoint compact faces, and all dihedral angles involving them have the form $\pi/2n$, then $\cP$ cannot reflectively
tile any ideal polyhedron. This and related criteria are developed in \secref{BadPairs}. In \secref{PotentialVertices} and \secref{PotentialFaces}, we give  more-sophisticated arguments to rule out $\cP$ based on its vertices and its faces. Then in \secref{Combinatorics}, we use combinatorial arguments to rule out the most stubborn cases. Together, these sections rule out~$45$ of the~$49$ Scharlau-Walhorn
polyhedra.  

The remaining four Scharlau-Walhorn polyhedra~$\cP$ do reflectively tile ideal, right-angled polyhedra~$\cQ$. These cases are studied in  \secref{Unfolding}. They lead to the examples in \thmref{Thm3CommensurabilityClasses}. 
(Two of these four polyhedra are commensurable,
which explains why that theorem has only three cases.)
Propositions \ref{Rhombicuboctahedron1}, \ref{Antiprism}, \ref{Rhombicuboctahedron2}
and~\ref{Bipyramid}, and hence 
\thmref{Thm:MainConverse}, follow
from a detailed study of how the $\Gamma(\cP)$-translates
of~$\cP$ fit together.

The four arithmetic Scharlau-Walhorn polyhedra are illustrated by the art piece in \figref{fig:IHES}, which was auctioned at the 2025 IHES Gala. The columns show (left to right): the geometrized polyhedron $\Pi$ and its dual against the midsphere, the intersection pattern of the polyhedra with the midsphere, the circle clusters and hyperbolic polyhedron $P$, and the tiling of $P$ by the Scharlau-Walhorn polyhedra.

\begin{figure}[ht]
   \centering
   \includegraphics[width=0.55\linewidth]{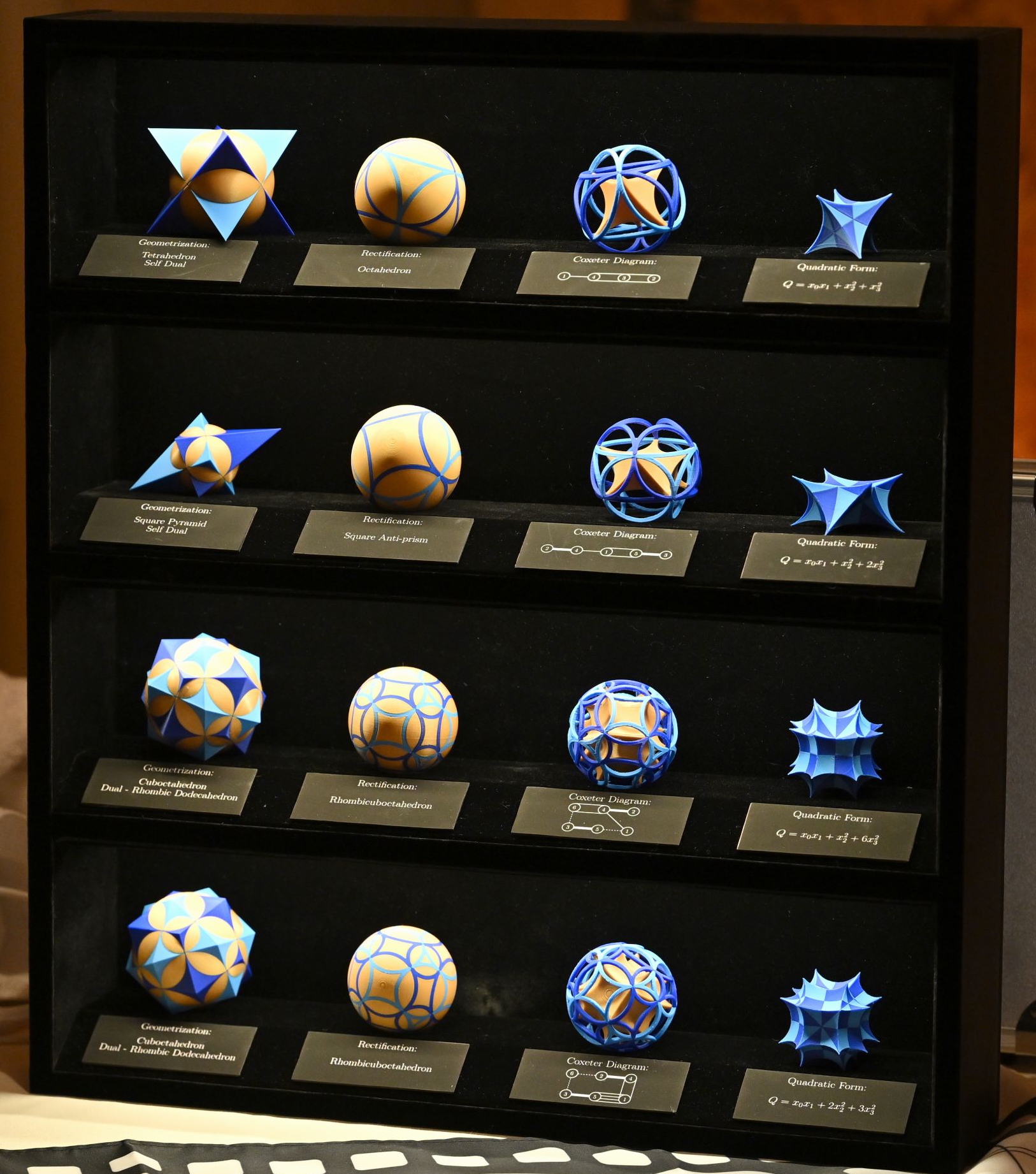}
   \caption{Artwork by Alex Kontorovich and Eric Vergo, at the 2025 IHES Gala.}
   \label{fig:IHES}
\end{figure}

\subsection*{Acknowledgements}\

We thank Arthur Baragar, Alice Mark, Allison N. Miller, Alan Reid, Phil Rehwinkel, Philip Yang, and David Yang for helpful conversations related to this project. We are grateful to the American Institute of Mathematics for hosting the workshop ``Arithmetic Reflection Groups and Crystallographic Packings,'' where this project was initiated. 
DA is partially supported by a Simons Foundation Fellowship.
AK is partially supported by NSF grant DMS-2302641, US-Israel BSF grant 2020119, and a Simons Foundation Fellowship.  This article contains no AI-generated ideas or text.  We did use Claude to help write Mathematica code to make some of the figures.

\section{Terminology, Notation, and Background} \label{Notation}

\subsection{Hyperbolic Coxeter Polyhedra, Reflection Groups, and Tilings}\label{sec:refl}\

For us, a hyperbolic polyhedron is always convex, and
will have finite-volume unless otherwise specified,
so each is an intersection of finitely many closed 
hyperbolic half-spaces. It can approach the boundary of $\bH^3$ only at a finite collection of points called cusps or ideal vertices. The polyhedron $\cP$ is called \emph{Coxeter} if each dihedral angle is equal to $\pi/n$ for some $n\in \N$. A Coxeter polyhedron $\cP$ determines a reflective lattice $\Gamma=\Gamma(\cP)$ generated by reflections across the
hyperplanes containing its faces. The lattice $\Gamma$ is contained in $\Isom(\bH^3)$, the isometry group of $\bH^3$.  The corresponding reflective tiling means the expression of 
$\bH^3$ as the union of the $\Gamma$-images of~$\cP$.
Thus the Koebe-Andreev-Thurston theorem produces, for each combinatorial polyhedron $G$, a hyperbolic polyhedron $\cP$, a reflective lattice $\Gamma$, and a reflective tiling of $\bH^3$.

A hyperbolic Coxeter polyhedron $\cP$ is determined up to isometry by the list of its faces and the dihedral angles between them. This determines the structure of $\Gamma$ as a Coxeter group, and thus determines the geometry of $\cP$ by the Mostow rigidity theorem. The dihedral angles are often encoded in a Coxeter diagram, or a Gram matrix (see \appref{ScharlauWalhornData} for definitions). We say that $\cP$ is \emph{ideal} if all its vertices are at the boundary of $\bH^3$ and \emph{right-angled} if all its dihedral angles are $\pi/2$. We also use these adjectives to describe the lattice $\Gamma$ and tiling of $\bH^3$ obtained from $\cP$. 
We say that polyhedra $\cP$ and $\cQ$ are \emph{commensurable} if the groups $\Gamma(\cP)$ and $\Gamma(\cQ)$ are commensurable (i.e. up to conjugation, they share a common finite-index subgroup). A key notion studied in this paper is the following.
\begin{definition}
We say that $\cQ$ is a \emph{$\cP$-polyhedron} if $\cQ$ is a convex hyperbolic polyhedron which is a finite union of copies of $\cP$ in the reflective tiling of $\bH^3$. In this case, we also say that $\cP$ \emph{reflectively tiles} $\cQ$.     
\end{definition}
In particular,  the reflection group generated by $\cQ$ is a finite-index subgroup of the group generated by $\cP$. More generally, we say that $\cP$ \emph{tiles} $\cQ$ if $\cQ$ is a union of isometric copies of $\cP$, with disjoint interiors. This is weaker than another common definition of the word, in which the faces of the tiles must match up. Situations where $\cP$ tiles $\cQ$ non-reflectively only appear in \secref{Unfolding}, and are clearly indicated in the text. In all these situations, there is a common polyhedron which reflectively tiles both $\cP$ and $\cQ$, ensuring their commensurability. Outside of $\secref{Unfolding}$, tilings are taken to be reflective. 

A two-dimensional polygon $P$ (hyperbolic, Euclidean, or spherical) can be described by a list of edges and angles between them. We will use the notation $(n_1, \ldots n_k)$ for a $k$-sided polygon with angles $\pi/n_1, \ldots,\pi/n_k$ in cyclic order around the perimeter. Hyperbolic polygons with more than 3 edges are not determined up to isometry by this data, but the structure of the reflection group and the combinatorics of the tiling are determined. 

\subsection{Arithmeticity} \label{sec:Arithmeticity}\

Let $k$ be a totally real number field, equipped with an identity embedding into $\R$, and let $\Phi$ be a quadratic form over the ring of integers, $\cO_k$, of $k$. We say that $\Phi$ is {\it hyperbolic} if: (1) its signature over $k$ is $(n,1)$ and (2) it is definite over any Galois conjugate $k^\sigma$. Then the orthogonal group $O_\Phi(\cO_k)$ preserving $\Phi$ over $\cO_k$ is well-known to be a lattice in $O_\Phi(\R)\cong O(n,1)=\Isom(\bH^n)$. A group $\G$ is called arithmetic (of simplest type) if it is commensurable to some such $O_\Phi(\cO_k)$. A theorem of Vinberg \cite{Vinberg1967} shows that if $\G$ is an arithmetic {\it reflective} lattice (that is, generated by reflections in hyperplanes), then $\G$ is of simplest type. The article \cite{Vinberg1967} also gives a very useful criterion for determining arithmeticity, in terms of dihedral angles. Godement's compactness criterion implies that if $\G$ is arithmetic and non-cocompact, then $\G$ has unipotents, in which case $k=\Q$ and $\Phi$ is isotropic over $\Q$, see \cite{BorelHC1962, VinbergGorbatsevichShvartsman1988}.

For a generic hyperbolic quadratic form $\Phi$, the reflections in its isometry group generate an infinite index subgroup; thus $\cO_\Phi$ cannot contain any reflective lattice $\Gamma$. Up to isometry and scaling, there are only finitely many quadratic forms $\Phi$ such that $\cO_\Phi$ contains a reflective lattice. This is one case of a larger theorem: there are only finitely many maximal arithmetic reflective hyperbolic lattices, over all dimensions and number fields! This theorem is due to the work of many people, e.g., \cite{Vinberg1981, LongMaclachlanReid2006, Agol2006, Nikulin2007, ABSW2008}; for a comprehensive survey and more background, see \cite{Belolipetsky2016}. In certain dimensions and over certain number fields, it is possible to give a complete classification. For isotropic forms over $\Q$ with signature $(3,1)$, Scharlau \cite{Scharlau} and Scharlau-Walhorn \cite{ScharlauWalhorn1992} gave a complete list of maximal-under-inclusion reflective lattices: there 
are 49. 

\begin{theorem}[Scharlau/Scharlau-Walhorn] \label{ScharlauWalhornTheorem}
    Let $\G$ be a non-cocompact arithmetic lattice acting on $\bH^3$, generated by reflections. 
    Then $\G$ is a subgroup of one of the 49 ``Scharlau-Walhorn'' lattices \hyperref[SW1]{SW1} through \hyperref[SW49]{SW49} listed in \appref{ScharlauWalhornData}.
\end{theorem}

These papers list the lattices, and give
considerable justification for why the list is complete, but 
do not claim a complete proof.  Scharlau has informed us that since then he has
completed a proof; and Allcock has found a different one (to appear).

This means that any arithmetic hyperbolic Coxeter polyhedron is reflectively tiled by one of the 49 Scharlau-Walhorn polyhedra. We are led to the following question:
\begin{question} \label{PolyhedronQuestion}
    Which hyperbolic Coxeter polyhedra $\cP$ admit an ideal, right-angled $\cP$-polyhedron?
\end{question}
The rest of the paper is devoted to answering this question for the 49 Scharlau-Walhorn polyhedra. We will rule out 45 of the 49 polyhedra, and show that the remaining four, \SW{2}, \SW{4}, \SW{12}, and \SW{13}, admit ideal, right-angled $\cP$-polyhedra. 

\section{Pairs of Bad Features} \label{BadPairs}

One relatively simple way to rule out certain polyhedra tiling ideal, right-angled ones is by inductive arguments using a pair of bad features---features that cannot exist in an ideal, right-angled  polyhedron. When we glue two polyhedra with pairs of bad features together, we may remove some of these features, but we introduce new ones. This section describes several such arguments.
The essential ingredient is the following proposition.

\begin{prop} \label{GluingPolyhedra}
Any $\cP$-polyhedron can be constructed via a sequence of steps in which two $\cP$-polyhedra are glued together along a single common face.
\end{prop}
\begin{proof}
Suppose that $\cQ$ is a $\cP$-polyhedron larger than $\cP$. Then $\cQ$ must contain at least two adjacent copies of $\cP$, glued along a face. If we extend this face to a full hyperplane, it separates $\cQ$ into two convex polyhedra glued along a single face, and each polyhedron is a finite union of copies of $\cP$. We can continue subdividing in this manner until we arrive at individual copies of $\cP$.
\end{proof}

Note that this proof requires the hypothesis that $\cP$ is a Coxeter polyhedron, i.e. that each dihedral angles has the form $\pi/n$. This is used in the step where the hyperplane through a face separates the $\Gamma$-translates of $\cP$. 

We say that a hyperbolic polyhedron has a pair of bad faces if it has two ultraparallel compact faces.

\begin{prop} \label{BadFacePair}
    Suppose that $\cP$ has a pair of bad faces. Suppose further that each of these faces meets its neighbors at angles of the form $\pi/(\hbox{even})$. Then every
    $\cP$-polyhedron has a pair of compact faces; in particular, no $\cP$-polyhedron is ideal. 
\end{prop}

\begin{proof}
By \propref{GluingPolyhedra}, it suffices to show that when two $\cP$-polyhedra $\cP_1$ and $\cP_2$ are glued together, the property of having a pair of bad faces is preserved. We will assume that each polyhedron has a distinguished pair of bad faces, tiled by copies of the bad faces from $\cP$. 

If $\cP_1$ and $\cP_2$ are glued along a face which is one of the distinguished pair of faces for both $\cP_1$ and $\cP_2$, then $\cP_1 \cup \cP_2$ still has two distinguished faces, one from $\cP_1$ and one from $\cP_2$. These faces are on opposite sides of the glued face, so they must be ultraparallel to each other.

If $\cP_1$ and $\cP_2$ are glued along a face which is not one of the distinguished pair of faces for one of the polyhedra, say $\cP_1$, then $\cP_1 \cup \cP_2$ still has two distinguished faces, both from $\cP_1$. These faces remain ultraparallel. It can happen that the gluing operation extends one or both of these faces. However, the condition that each bad face meets all its neighbors at angles of $\pi/(\hbox{even})$ ensures that the plane through this face is tiled by copies of the same face, so the extended face will remain compact. 
\end{proof}

\begin{example} \label{Example20}
The polyhedron \SW{20} has a pair of bad faces, numbered 1 and 7. These faces are compact, ultraparallel, and meet all their neighbors at angles of $\pi/2$ or $\pi/4$. Thus this polyhedron cannot tile an ideal polyhedron.
\end{example}

\begin{corollary}
The polyhedra \SW{20}, \SW{23}, \SW{24}, \SW{28}, \SW{30}, \SW{32}--\SW{37}, 
\SW{39}--\SW{41}, \SW{43}--\SW{46}, \SW{48}, and \SW{49}
cannot tile ideal polyhedra. \qed
\end{corollary}

In \appref{ArgumentData}, we list a pair of bad faces for each of these cases.

We say that $\cP$ has a pair of bad edges if it has two acute-angled edges, and 
there is no way to choose faces, one containing one of them and one containing
the other, that share an edge.

\begin{prop} \label{BadEdgePair}
Suppose that $\cP$ has a pair of bad edges. Then any $\cP$-polyhedron has a pair of bad edges; in particular, no $\cP$-polyhedron is right-angled. 
\end{prop}

\begin{proof} 
We will show that when two $\cP$-polyhedra $\cP_1$ and $\cP_2$ are glued together along a single face, the property of having a pair of bad edges is preserved. We assume that $\cP_1$ and $\cP_2$ each have a pair of distinguished bad edges. 

    First, suppose that the face of (say) $\cP_1$ involved in the gluing
     contains neither of the distinguished edges of $\cP_1$.  
     These edges extend to edges of $\cP_1 \cup \cP_2$, 
     with the same (acute) dihedral angles.  And each face 
     containing one of them remains
    parallel or ultraparallel to each face containing the other.

    The other case is that the gluing is along a face containing a distinguished edge of $\cP_1$ and a face containing a  distinguished edge of $\cP_2$. The other distinguished edge of $\cP_1$, and the other distinguished edge of $\cP_2$, are edges of $\cP_1\cup\cP_2$,
    with the same (acute) dihedral angles.  And the plane of the gluing separates the faces
    that contain one of these edges from the faces that contain the other.
\end{proof}

\begin{example} \label{Example25}
The polyhedron \SW{25} has a pair of bad edges. The edge between faces 1 and 4 has an angle of $\pi/3$, as does the edge through faces 9 and 10. Faces 1 and 4 are each ultraparallel to both 9 and 10. Thus this polyhedron cannot tile a right-angled polyhedron.
\end{example}

\begin{corollary}
The polyhedra \SW{25} and \SW{42} cannot tile right-angled polyhedra. \qed
\end{corollary}

We say that $\cP$ has a pair of bad vertices if it has two finite vertices, 
and there is no way to choose faces, one containing one of them and one containing
the other, that share an edge.
The proof of the next proposition is similar to that of \propref{BadEdgePair}, so we omit it. 

\begin{prop} \label{BadVertexPair}
Suppose that $\cP$ has a pair of bad vertices. Then any $\cP$-polyhedron has a pair of bad vertices; in particular, no $\cP$-polyhedron is ideal. \qed
\end{prop}

\begin{example} \label{Example47}
The polyhedron \SW{47} has multiple pairs of bad vertices. One pair is as follows: faces 1, 2 and 6 intersect at a finite vertex, as do faces 8, 9 and 12. Faces 1, 2 and 6 are each ultraparallel to all of 8, 9 and 12. Thus this polyhedron cannot tile an ideal polyhedron.
\end{example}

Polyhedron \SW{47} is in fact the only case ruled out by \propref{BadVertexPair} which was not already ruled out. 

We will use a variation of \propref{BadEdgePair} to rule out a few more cases. Suppose that an edge $E$ in the polyhedron $\cP$ is compact, and at each endpoint, $E$ is orthogonal to a plane in the reflective tiling of $\mathbb{H}^3$ by $\cP$. This property guarantees that the line through $E$ is tiled by $\Gamma$-images of $E$. So $E$ cannot be part of
an edge in an ideal $\cP$-polyhedron. We say that $\cP$ has a pair of bad compact 
edges if it has two compact edges with this property, and it is not possible to
choose faces of~$\cP$, one containing one of these edges and one containing the other,
that share an edge.
The next result has the same proof as \propref{BadEdgePair}.

\begin{prop} \label{BadCompactEdgePair}
    Suppose that $\cP$ has a pair of bad compact edges. Then any $\cP$-polyhedron has a pair of bad compact edges; in particular, no $\cP$-polyhedron is ideal. \qed
\end{prop}

\begin{remark} \label{BadCompactEdgePairRemark}
The same proof also shows, more generally, that if two $\cP$-polyhedra, each with a bad pair of compact edges, are glued along a face, then the resulting polyhedron still has a bad pair of compact edges. This holds even if $\cP$ itself does not have a bad pair.
\end{remark}

\begin{example} \label{Example14}
    In  polyhedron \SW{14}, the edge between faces 1 and 3 is compact. It intersects face 2 orthogonally at one endpoint, and intersects the image $\sigma_4(F_1)$ of face~$1$
    under the reflection $\sigma_4$ across face 4, orthogonally at the other endpoint. Similarly, the edge between face 5 and face 6 is compact, intersecting face 4 orthogonally at one endpoint and face $\sigma_2(F_6)$ orthogonally at the other endpoint. Face 1 is parallel to face 6 and ultraparallel to face 5; face 3 is parallel to face 5 and ultraparallel to face 6. Thus the edges $1$-$3$ and $5$-$6$ form a bad compact edge pair, so \SW{14} cannot tile an ideal polyhedron. 
\end{example}

\begin{corollary}
The  polyhedra \SW{14}, \SW{17}, \SW{22}, and \SW{26}
 cannot tile ideal polyhedra. \qed
\end{corollary}

We have now ruled out 27 of the 49 Scharlau-Walhorn polyhedra. It is possible to prove further statements along the lines of \propref{BadFacePair}, \propref{BadEdgePair}, and \propref{BadVertexPair}. For example, we could extend to pairs consisting of two different types of bad features. But we will instead move on to a detailed analysis of the potential vertices and faces of an ideal, right-angled  $\cP$-polyhedron. This will lead to more powerful criteria.

\section{Potential Vertices of Ideal, Right-Angled  Polyhedra: Euler Characteristic} \label{PotentialVertices}

In this section, we consider the vertices of a hypothetical ideal $\cP$-polyhedron $\cQ$.
If we visualize a vertex of $\cQ$ as lying at the point vertically at infinity in the upper half-space model of $\mathbb{H}^3$, and take a sufficiently high horizontal cross section $T$ of $\cQ$, we obtain a Euclidean polygon. Since the adjacent faces of $\cQ$ meet at right angles, $T$ must be a rectangle. It is tiled by cross-sections $S$ of $\cP$. The cross-sections $S$ are all congruent Euclidean polygons, with each angle equal to $\pi/n$ for some $n \in \mathbb{N}$. 
The only such polygons are rectangles, the $(2, 3, 6)$ right triangle, the $(2, 4, 4)$ right triangle, the $(3, 3, 3)$ equilateral triangle. Each one generates an affine reflection group and a tiling of the Euclidean plane. Then $T$ is a rectangle in this tiling. We now describe the minimal rectangle(s) in each case, and show that any rectangle is tiled by minimal rectangles.

\begin{prop} \label{IdealVertices}
If $S$ is the $(2, 3, 6)$ right triangle, there is a minimal rectangle in the associated tiling consisting of six copies of $S$ as shown below. If $S$ is the $(2, 4, 4)$ right triangle, there are two possible minimal rectangles in the associated tiling, one consisting of two copies of $S$ and one consisting of four copies. If $S$ is the $(3, 3, 3)$ equilateral triangle, then there are no rectangles in the associated tiling. If $S$ is a rectangle, then it is the minimal rectangle. Any rectangle in these tilings is tiled by copies of one minimal rectangle.
\end{prop}

\begin{figure}[ht] 
\centering
\begin{tikzpicture}[scale=1.2]
  \def\r{1} 
  \def\h{1.732} 
  
  \clip (0,0) rectangle (3,3.464);
  
  \foreach \i in {0,1,2} {
    \foreach \j in {0,1,2} {
      \pgfmathsetmacro{\x}{1.5*\i*\r}
      \pgfmathsetmacro{\y}{\j*\h + 0.5*\h*mod(\i,2)}
      
      \draw (\x + \r, \y) 
        -- (\x + 0.5*\r, \y + 0.866*\r)
        -- (\x - 0.5*\r, \y + 0.866*\r)
        -- (\x - \r, \y)
        -- (\x - 0.5*\r, \y - 0.866*\r)
        -- (\x + 0.5*\r, \y - 0.866*\r)
        -- cycle;
      
      \draw (\x + \r, \y) -- (\x - \r, \y);
      \draw (\x + 0.5*\r, \y + 0.866*\r) -- (\x - 0.5*\r, \y - 0.866*\r);
      \draw (\x - 0.5*\r, \y + 0.866*\r) -- (\x + 0.5*\r, \y - 0.866*\r);
      
      \draw (\x + 0.75*\r, \y + 0.433*\r) -- (\x - 0.75*\r, \y - 0.433*\r);
      \draw (\x + 0.75*\r, \y - 0.433*\r) -- (\x - 0.75*\r, \y + 0.433*\r);
      \draw (\x, \y + 0.866*\r) -- (\x, \y - 0.866*\r);
    }
  }
  
  
  \fill[red, opacity=0.2] (1.5, 1.732+0.866) -- (1.5-.75, 1.732-0.433) -- (1.5, 1.732-0.866) -- (1.5+.75, 1.732+0.433) -- cycle;
  
  
\end{tikzpicture}\qquad
\begin{tikzpicture}[scale=1.025]
  \def\s{1} 
  
  \foreach \i in {0,1,2,3,4} {
    \draw (\i*\s, 0) -- (\i*\s, 4*\s);
  }
  \foreach \j in {0,1,2,3,4} {
    \draw (0, \j*\s) -- (4*\s, \j*\s);
  }
  
  \foreach \i in {0,1,2,3} {
    \foreach \j in {0,1,2,3} {
      \draw (\i, \j) -- (\i+1, \j+1);
      \draw (\i, \j+1) -- (\i+1, \j);
    }
  }

  \fill[red, opacity=0.2] (1, 1) -- (2, 1) --(2, 2) --(1, 2) -- cycle;

    \fill[red, opacity=0.2] (3, 3) -- (2.5, 2.5) --(3,2) --(3.5, 2.5) -- cycle;

\end{tikzpicture}\qquad 
\begin{tikzpicture}[scale=2.35]
  \def\s{1} 
  \def\h{0.866} 

  \clip (1,\h) rectangle (3,3*\h);
  
  \foreach \i in {0,1,2,3,4,5,6} {
    \draw (0, \i*\h*.5) -- (5*\s, \i*\h*.5);
  }
  
  \foreach \i in {-1,0,1,2,3,4,5,6} {
    \draw (\i*\s/2, 0) -- (\i*\s/2 + 2.5*\s, 5*\h);
    \draw (\i*\s/2 + 2.5*\s, 0) -- (\i*\s/2, 5*\h);
  }
  
\end{tikzpicture}

\caption{From left: tiling by $(2, 3, 6)$ triangle with minimal rectangle; tiling by $(2, 4, 4)$ triangle with minimal rectangles; tiling by $(3, 3, 3)$ triangle.}
\label{MinRectangles} 
\end{figure}
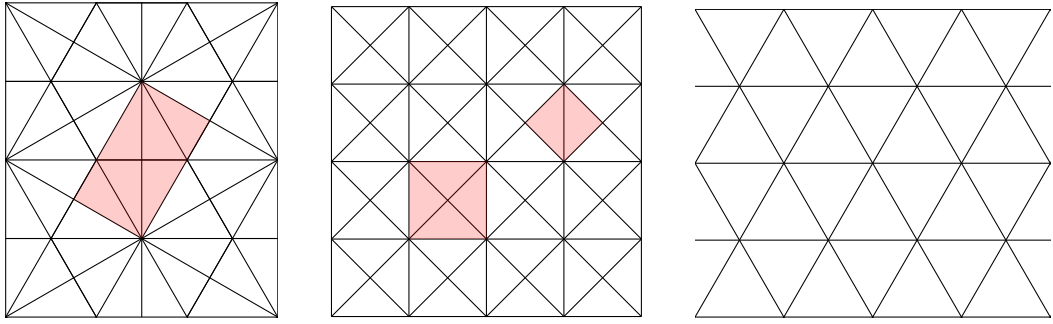

\begin{proof}
    The proof is by direct examination of the tilings shown in \figref{MinRectangles}. If we choose two orthogonal directions of edges in the tiling, we observe that the set of edges in these directions cuts out a rectangular grid with minimal rectangle as described.
\end{proof}

Thus, for example, if $\cP$ has an ideal vertex with faces forming a $(2, 3, 6)$ triangle, then $\cQ$ must have a multiple of 6 copies of $\cP$ incident to this vertex, in a rectangular configuration. Similarly, if $\cP$ has an ideal vertex with faces forming a $(2, 4, 4)$ triangle, then $\cQ$ must have a multiple of 2 copies of $\cP$ incident to this vertex, in a rectangular configuration. If $\cP$ has an ideal vertex with faces forming a $(3, 3, 3)$ triangle, then no right-angled $\cP$-polyhedron exists (but this does not apply to any of the Scharlau-Walhorn polyhedra). We also make an observation about the copies of $\cP$ incident to each face at an ideal vertex.

\begin{corollary} \label{RectangleFact}
If $F_1$ and $F_2$ are two faces of $\cQ$ which are parallel at an ideal vertex, then $F_1$ and $F_2$ contain the same number of copies of faces of $\cP$ incident to that vertex. \qed
\end{corollary}

While it may seem trivial that opposite sides of a rectangle are equal, this fact will be used in an important way in \secref{Combinatorics}.

Finally, the fact that four faces meet at each vertex of $\cQ$ is useful for Euler characteristic computations. Recall that the Euler characteristic $\chi$ of any genus 0 polyhedron (hyperbolic or otherwise), defined as the number of faces minus edges plus vertices, is 2. We can compute the Euler characteristic locally at a face $F$ as 
\begin{equation} \label{FaceLocalEC}
\chi_F = 1 - \frac{ \# \text{ edges of }F}{2} + \frac{\# \text{ vertices of }F}{4} = 1-\frac{n_F}{4}
\end{equation}
where $n_F$ is the number of edges (or vertices) of $F$. Since each edge is incident to two faces, and each vertex is incident to four faces, summing $\chi_F$ over all faces $F$ of $\cQ$ gives the Euler characteristic $\chi(\cQ)=2$.

Similarly, we can compute the Euler characteristic locally at a vertex $V$ as 
\begin{equation} \label{VertexLocalEC}
\chi_V = 1 - \frac{4}{2} + \sum_{\text{faces }F\text{ containing }V} \frac{1}{n_F} = -1 + \sum_F \frac{1}{n_F}  
\end{equation}
Again, since each vertex is incident to four edges, each edge has two vertices, and each face $F$ has $n_F$ vertices, summing $\chi_V$ over all vertices $V$ of $\cQ$ gives the Euler characteristic $\chi(\cQ)=2$.

\begin{prop} \label{AllQuadEuler}
Suppose that every face of $\cP$ is a hyperbolic polygon with four or more edges. Then no ideal, right-angled  $\cP$-polyhedron exists. 
\end{prop}

\begin{proof}
First we show that the property of all faces having four or more edges is preserved when two $\cP$-polyhedra $\cP_1$ and $\cP_2$ are glued along a face. In this situation, some faces of $\cP_1$ can be glued to faces of $\cP_2$ along a single edge. If a polygon with $n_1$ edges is glued to a polygon with $n_2$ edges, the resulting polygon has at least $n_1+n_2-4$ edges, since two edges are removed by the gluing and two more pairs of edges can be identified. Thus gluing two faces with four or more edges will produce a face with four or more edges. It now follows from \propref{GluingPolyhedra} that any $\cP$-polyhedron will have the same property: all faces will have four or more edges.

This implies that the local Euler characteristic $\chi_F$ computed at any face will be nonpositive, and so the sum of all the Euler characteristics cannot be 2. It follows that no ideal, right-angled  $\cP$-polyhedron exists. 
\end{proof}

\begin{example} \label{Example15}
    The polyhedron \SW{15} has six quadrilateral and two pentagonal faces. Thus by \propref{AllQuadEuler}, it cannot tile an ideal, right-angled  polyhedron. 
\end{example}

\begin{corollary} 
The  polyhedra \SW{15}, \SW{27}, \SW{31}, and \SW{38}
 cannot tile ideal, right-angled polyhedra. \qed
 \end{corollary}
In \appref{ArgumentData}, we describe the faces for each of these. 

We have now ruled out 31 of the Scharlau-Walhorn polyhedra. Subsequent sections will use Euler characteristic in conjunction with other criteria to be developed. 

\section{Potential Faces of Ideal, Right-Angled  Polyhedra} \label{PotentialFaces}

We can rule out ideal $\cP$-polyhedra by analyzing their potential faces. Any face of an ideal $\cP$-polyhedron is an ideal polygon, tiled by copies of faces of $\cP$. If $F$ is a hyperbolic Coxeter polygon, then an $F$-polygon is a convex finite-volume polygon in the tiling of $\mathbb{H}^2$ by reflections of $F$. The following question, a two-dimensional analogue of \qref{PolyhedronQuestion}, is relevant:

\begin{question} \label{IdealPolygonQuestion}
Which Coxeter polygons $F \subset \mathbb{H}^2$ admit an ideal $F$-polygon?
\end{question}

A complete answer to this question follows from work of Felikson.  We will call $F$ idealizable if there exists an ideal $F$-polygon. For example, a compact polygon is clearly not idealizable. The next theorem follows from \cite[Lemmas 12-15]{Felikson1998}:

\begin{thm} \label{IdealizablePolygons}
    An idealizable Coxeter polygon $F$ has one of the following forms: \begin{enumerate}
\item 
    \label{ItemIdealizableInfties}
    $(\infty, \ldots \infty)$ with $I\geq 3$ infinite vertices.
\item 
    \label{ItemIdealizableN}
$(n, \infty, \ldots \infty)$ with $n \in \Z_{\geq 2}$, and $I\geq 2$ infinite vertices.
\item 
    \label{ItemIdealizableN2}
    $(n, 2, \infty, \ldots \infty)$ with $n \in \Z_{\geq 2}$, and $I\geq 1$ ($I\geq2$ if $n=2$).
\item 
    \label{ItemIdealizable2N2}
    $(2, n, 2, \infty, \ldots \infty)$ with $n \in \Z_{\geq 2}$, and $I\geq 1$ infinite vertices.
    \qed
\end{enumerate} 
\end{thm}

\begin{figure}[ht] 
\centering
\, \hfill
\begin{tikzpicture}[scale=2]
	\draw (-1, 0) -- (1,0);
	\draw (0, 0) --  (.5, .866025);
	\draw (0, 0) --  (-.5, .866025);
	\draw (1,0) arc
    [
        start angle=-90,
        end angle=-210,
        x radius=0.57735,
        y radius =0.57735
    ] ;	
    \draw (-1,0) arc
    [
        start angle=-90,
        end angle=30,
        x radius=0.57735,
        y radius =0.57735
    ] ;
    \draw (.5, .866025) arc
    [
        start angle=-30,
        end angle=-150,
        x radius=0.57735,
        y radius =0.57735
    ] ;
    \draw[white] (-1,-1) -- (1, -1);
\end{tikzpicture}
\hfill
\begin{tikzpicture}[scale=2]
	\draw (-0.267949, 2) -- (1,2);
	\draw (-.5, 2-.866025) --  (0.133975, 2.232051);
	\draw (-.5, 2.866025) --  (0.133975, 2-0.232051);
	\draw (1,2) arc
    [
        start angle=-90,
        end angle=-150,
        x radius=1.73205,
        y radius =1.73205
    ] ;	
	\draw (1,2) arc
    [
        start angle=90,
        end angle=150,
        x radius=1.73205,
        y radius =1.73205
    ] ;	
    \draw (-.5, 2.866025) arc
    [
        start angle=30,
        end angle=-30,
        x radius=1.73205,
        y radius =1.73205
    ] ;
    
	\draw (-1, 0) -- (1,0);
	\draw (0,0) --  (0.133975, 0.232051);
	\draw (-.5, .866025) --  (0,0);
	\draw (1,0) arc
    [
        start angle=-90,
        end angle=-150,
        x radius=1.73205,
        y radius =1.73205
    ] ;	
    \draw (-.5, .866025) arc
    [
        start angle=30,
        end angle=-0,
        x radius=1.73205,
        y radius =1.73205
    ] ;
    
    \draw (-.5,0) arc
    [
        start angle=-30,
        end angle=30,
        x radius=0.866025,
        y radius =0.866025
    ] ;
    \draw (-.5,0) arc
    [
        start angle=30,
        end angle=42.5,
        x radius=0.866025,
        y radius =0.866025
    ] ;
    \draw (-1,0) arc
    [
        start angle=-90,
        end angle=30,
        x radius=0.57735,
        y radius =0.57735
    ] ;
\end{tikzpicture}
\hfill
\begin{tikzpicture}[scale=2]
	\draw (.5, 0.288675) -- (-.5, -0.288675);
	\draw (.5, -0.288675) -- (-.5, 0.288675);
	\draw (0, 0.57735) -- (0, -0.57735);
    
	\draw (1,0) arc
    [
        start angle=-90,
        end angle=-210,
        x radius=0.57735,
        y radius =0.57735
    ] ;	
    \draw (-1,0) arc
    [
        start angle=-90,
        end angle=30,
        x radius=0.57735,
        y radius =0.57735
    ] ;
    \draw (.5, .866025) arc
    [
        start angle=-30,
        end angle=-150,
        x radius=0.57735,
        y radius =0.57735
    ] ;
    	\draw (1,0) arc
    [
        start angle=90,
        end angle=210,
        x radius=0.57735,
        y radius =0.57735
    ] ;	
    \draw (-1,0) arc
    [
        start angle=90,
        end angle=-30,
        x radius=0.57735,
        y radius =0.57735
    ] ;
    \draw (.5, -.866025) arc
    [
        start angle=30,
        end angle=150,
        x radius=0.57735,
        y radius =0.57735
    ] ;
    \draw[white] (-1,-1) -- (1, -1);
\end{tikzpicture}
\hfill \,
\caption{From left: minimal ideal polygon tiled by $(3, \infty, \infty)$; minimal ideal polygons tiled by $(3, 2, \infty)$; minimal ideal polygon tiled by $(2,3,2, \infty)$.}
\label{MinIdealPolygons} 
\end{figure}
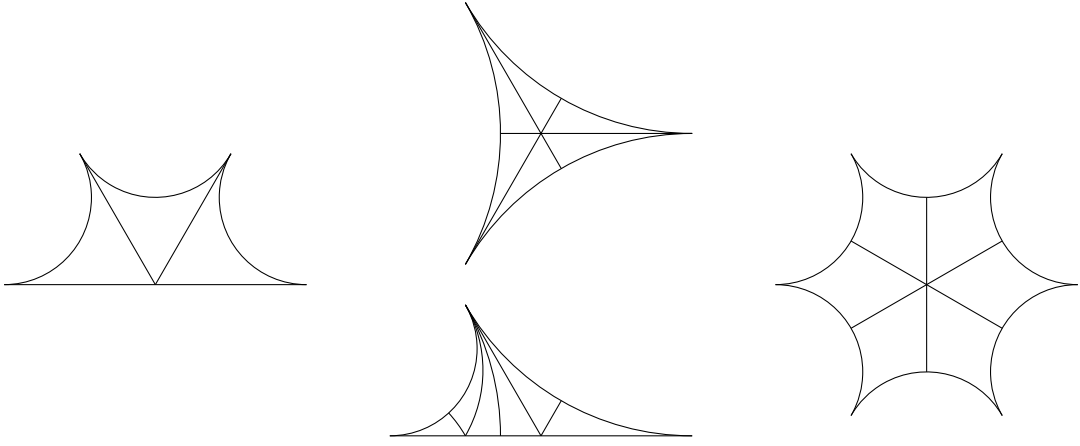

If $F$ is idealizable, then there are $1$ or~$2$ ``minimal'' ideal $F$-polygons, such
that
every ideal $F$-polygon is a nonoverlapping union of $\Gamma(F)$-images of
the minimal ones.
We define them as follows.  
In case \eqref{ItemIdealizableInfties}
there is only one minimal polygon: $F$ itself.  
In case \eqref{ItemIdealizableN} or \eqref{ItemIdealizableN2}, let 
$v$ be the first vertex of~$F$, while in case \eqref{ItemIdealizable2N2} let it be
the second one.  The angle of $F$ at~$v$ is $\pi/n$, and the description of 
the minimal polygons 
begins with the union $Z$ of the $2n$ copies of $F$ containing~$v$.  In
case \eqref{ItemIdealizable2N2}, there is only one minimal polygon: $Z$
itself.  In case \eqref{ItemIdealizableN}, the minimal polygons are 
got by cutting~$Z$ in half, along a mirror of $\Gamma(F)$.  Up
to the action of $\Gamma(F)$, this yields one minimal polygon
if $n$ is odd, or two if $n$ is even.  In case \eqref{ItemIdealizableN2}
with $n$ even, we use the same construction, but only half of the
possible cuttings yield an ideal polygon, and all these cuttings
are $\Gamma(F)$-equivalent.  We take the resulting
polygon as the unique minimal polygon.  The most complicated case
is \eqref{ItemIdealizableN2} with $n$ odd, when there  are two
minimal polygons, one of which is~$Z$.  To get the other, first cut~$Z$ along
a mirror through~$v$, which yields a polygon with all vertices ideal
except for a single right-angled corner.  Doubling this half of~$Z$, across the
edge of $Z$ that contains this corner, yields our second minimal polygon.  
Figure~\ref{MinIdealPolygons} illustrates the minimal ideal polygon(s) 
for $F$ the $(3, \infty, \infty)$ triangle, the $(3, 2, \infty)$ triangle, and 
the $(2, 3, 2, \infty)$ quadrilateral (cf \cite[Figure 19]{Felikson1998}).
The minimal polygons deserve their name:

\begin{prop}
    \label{MinimalPolygons}
    Suppose $F$ is an idealizable polygon, and $F'$ is an ideal $F$-polygon.
    Then there is a set of $\Gamma(F)$-translates of minimal polygons of~$F$,
    that have disjoint interiors and cover~$F'$.
\end{prop}

\begin{proof}
Case \eqref{ItemIdealizableInfties} is trivial.  In any other case,
suppose an edge~$E$ of~$F'$ contains a translate~$v'$ of~$v$.  
By working one's way along~$E$,
one can find the copies of~$F$ that
lie in~$F'$ and meet $E$ (in $\bH^2$).  Their union is one of minimal
polygons described above.  Removing it from~$F'$ leaves a set of disjoint ideal
$F$-polygons, which we can decompose by
induction.

This reduces us to the case that every translate of~$v$ lies in the interior of~$F'$.
The copy~$Z'$ of~$Z$ centered at such a point is an ideal polygon, and removing it from~$F'$ 
leaves some ideal polygons.  In case \eqref{ItemIdealizable2N2}, or in case \eqref{ItemIdealizableN2}
when $n$ is odd, $Z'$ is
one of our minimal polygons, and we use induction to decompose the remaining pieces.
In the other cases, we cut~$Z'$ in half, and again use induction on the remaining pieces.
%
\end{proof}

Now we apply this result in the three-dimensional setting. If faces $F_1$ and $F_2$ of $\cP$ meet along an edge at an angle of $\pi/n$ for $n$ odd, then each plane through that edge contains a $\Gamma$-translate of $F_1$ meeting a $\Gamma$-translate of $F_2$. If they meet along an edge at an angle of $\pi/n$ for $n$ even then each plane through that edge contains either two $\Gamma$-translates of $F_1$ or two $\Gamma$-translates of $F_2$. Thus we partition the faces of $\cP$ into equivalence classes, where two faces are equivalent if they are connected by a chain of faces which meet at angles of $\pi/n$ for $n$ odd. Two faces of $\cP$ have $\Gamma$-translates in the same plane if and only if they are connected by a chain of adjacent faces in that plane, which occurs if and only if they are in the same equivalence class. 

Each reflection in $\Gamma$ fixes a plane pointwise, and we call these planes the mirrors of $\Gamma$.  
Each mirror~$P$  is a union of $\Gamma$-translates of faces of~$\cP$, from a single
equivalence class. Consider the set of all mirrors which meet $P$ orthogonally, and let $\Gamma_P$ be the group generated by their reflections. This is a discrete hyperbolic reflection group acting on $P$, and we write $\Delta_P$ for one of its chambers.  
This is a 
(possibly $\infty$-sided) polygon, and a
fundamental domain for~$\Gamma_P$.
Now suppose $P$ cuts out a face of some ideal, right-angled $\cP$-polyhedron~$\cQ$.  
Then that face of~$\cQ$ is an ideal $\Delta_P$-polygon. The idea behind the next proposition 
is that this situation is impossible when $\Delta_P$ is not idealizable.
We call the equivalence class of faces of $\cP$, that corresponds to~$P$, 
idealizable if $\Delta_P$ is. We use the same language for each member of that equivalence class.
The following proposition is a generalization of
\propref{BadFacePair}.

\begin{prop} \label{BadFaceOrbit}
Let $\cP$ be a hyperbolic Coxeter polyhedron with set of faces $\Delta$. Let $\Delta' \subseteq \Delta$ be the set of non-idealizable faces of $\cP$. Let $\Gamma'$ be the group generated by reflections across the faces in $\Delta'$. Then any ideal, right-angled  $\cP$-polyhedron,
    that contains~$\cP$, contains the orbit of $\cP$ under $\Gamma'$. 
\end{prop}

\begin{proof}
Suppose $\cQ$ is such a polyhedron, and $F\in\Delta'$. Then $F$ cannot be part of an external face of $\cQ$, so $\cQ$ also contains the reflection of $\cP$ across $F$. 
We now show by induction
    that $\sigma_{i_1}\cdots \sigma_{i_k}(\cP) \subseteq \cQ$, where $\sigma_{i_1}, \ldots \sigma_{i_k}$ are reflections across faces of $\cP$ in $\Delta'$. Proceeding by induction on $k$, we may assume that $\sigma_{i_{k-1}}\cdots \sigma_{i_1}(\cP) \subseteq \cQ$. Then
$$\sigma_{i_1}\cdots \sigma_{i_k}(\cP) = \sigma_{i_1}\cdots \sigma_{i_{k-1}}\sigma_{i_k}\sigma_{i_{k-1}}\cdots\sigma_{i_1}(\sigma_{i_{1}}\cdots \sigma_{i_{k-1}}(\cP))$$
and $\sigma_{i_1}\cdots\sigma_{i_{k-1}}\sigma_{i_k}\sigma_{i_{k-1}}\cdots\sigma_{i_1}$ is a reflection across a face of $\sigma_{i_{1}}\cdots \sigma_{i_{k-1}}(\cP)$, corresponding to a face in $\Delta'$. Since this face cannot be part of an external face of $\cQ$, if $\sigma_{i_{1}}\cdots \sigma_{i_{k-1}}(\cP) \subseteq \cQ$, then $\sigma_{i_{1}}\cdots \sigma_{i_{k}}(\cP) \subseteq \cQ$. 
\end{proof}

\begin{corollary}
Suppose that the reflections across the
    non-idealizable faces of $\cP$ generate an infinite group. Then no finite-volume
    ideal, right-angled  $\cP$-polyhedron exists. \qed
\end{corollary}
This is immediate since the union of the $\Gamma'$-translates of $\cP$ would have infinite volume. 

\begin{example} \label{Example7}
    In polyhedron \SW{7}, $\{1, 4\}$ and $\{2,6\}$ are equivalence classes of faces. Each of faces 1, 2, 4, and 6 is a pentagon with one ideal vertex. Together, faces 1 and 4 tile a fundamental domain which is an $(\infty, 2, 2, \infty, 2, 2)$ hyperbolic hexagon; faces 2 and 6 do the same. This hexagon cannot tile an ideal polygon by \thmref{IdealizablePolygons}. The reflections across faces 1, 2, 4, 6 generate an infinite group because 1 is parallel to 6 and 2 is parallel to 4. Thus no ideal, right-angled  $\cP$-polyhedron exists.
\end{example}


\begin{example} \label{Example5}
In polyhedron \SW{5}, $\{1, 2, 4, 5\}$ is an equivalence class of faces. Faces 1, 5 are compact quadrilaterals while faces 2, 4 are quadrilaterals with a single ideal vertex. Together, these faces tile a fundamental domain which is an $(\infty, 2, 2, \infty, 2, 2)$ hyperbolic hexagon. This hexagon cannot tile an ideal polygon. The reflections across faces 1, 2, 4, 5 generate an infinite group, because face 2 is parallel to face 4. Thus no ideal, right-angled  $\cP$-polyhedron exists.
\end{example}

\begin{corollary}
The polyhedra \SW{5}, \SW{7}, \SW{9}, \SW{18}, \SW{21}, and \SW{29}
 cannot tile ideal, right-angled polyhedra. \qed
\end{corollary}
In \appref{ArgumentData}, we list the non-idealizable equivalence classes of faces for each of these.

The next two propositions show that
this argument is still useful when $\Gamma'$ is finite.  We write $\cP'$ for the 
union of the $\Gamma'$-images of $\cP$.

\begin{prop} \label{FiniteBadFaceOrbit1} Let $\Delta'$ be a 
    union of equivalence classes of
    faces of $\cP$, and  let $\Gamma'$ be the group generated by the reflections across them.
    Then $\cP'=\cup_{\gamma\in\Gamma'}\,\gamma(\cP)$  is 
    Coxeter polyhedron, possibly of
    infinite volume.  If $\Gamma'$ is finite, then $\cP'$ has finite volume and finitely many faces.
\end{prop}

\begin{proof}
The region $\cP'$ is a union of tiles given by translates of $\cP$. Each tile has the form $\sigma_{i_1}\cdots \sigma_{i_k}(\cP)$, where $F_{i_1}, \ldots F_{i_k}$ are faces in $\Delta'$. The sequence of tiles
\begin{equation*}
\cP, \  \sigma_{i_1}(\cP), \  \sigma_{i_1}\sigma_{i_2}(\cP), \ldots, \  \sigma_{i_1}\sigma_{i_2}\cdots \sigma_{i_k}(\cP)
\end{equation*} 
are adjacent along translates of faces in $\Delta'$. Conversely, the word $\sigma_{i_1}\sigma_{i_2}\cdots \sigma_{i_k}\in \Gamma'$ can be reconstructed from a sequence of tiles starting at $\cP$ and adjacent along translates of faces in $\Delta'$. So a given tile is in $\cP'$ if and only if such a sequence exists. 

Consider a face of the form $\gamma'(F)$ where $\gamma' \in \Gamma'$ and $F$ is a face of $\cP$ which is not in $\Delta'$. Let $P$ be the hyperbolic plane through $\gamma'(F)$. Because the faces not in $\Delta'$ meet faces in $\Delta'$ at angles of $\pi/2n$, $P$ must be tiled exclusively by translates of faces not in $\Delta'$. Since any tile in $\cP'$ is connected to $\cP$ via a sequence of tiles adjacent along translates of faces in $\Delta'$, this implies that we cannot have tiles on both sides of $P$. Thus $\gamma'(F)$ is an exterior face of $\cP'$, and $\cP'$ lies entirely on one side of the plane $P$ through the face. 

Now we will show that $\cP'$ is precisely the convex hyperbolic region bounded by all planes $P$ through faces $\gamma'(F)$ where $\gamma' \in \Gamma'$ and $F \notin \Delta'$. Suppose that $\gamma(\cP)$ for $\gamma\in \Gamma$ is an arbitrary tile, and suppose that for all the planes $P$, $\gamma(\cP)$ and $\cP$ lie on the same side of $P$. If we draw a hyperbolic line segment connecting a point in $\cP$ to a point in $\gamma(\cP)$, the first wall that it crosses must be in $\Delta'$, so that it enters a new tile in $\cP'$. But we may repeat the argument inductively to show that every wall this segment crosses is a $\Gamma'$-translate of a wall in $\Delta'$, and every tile it enters is in $\cP'$. We conclude that $\gamma(\cP) \in \cP'$. 

It remains to show that each angle between adjacent faces is $\pi/n$ for some $n\in \mathbb{N}$. Any edge of $\cP'$ is the translate of an edge of $\cP$. An edge of $\cP$ between two faces in $\Delta'$ becomes interior to $\cP'$. An edge between two faces not in $\Delta'$ becomes an edge between two exterior faces of $\cP'$, with the same angle. Finally, an edge between a face in $\Delta'$ and a face not in $\Delta'$ becomes an edge of $\cP'$ where two tiles are glued together along the $\Delta'$ face. Since the angle in $\cP$ between a $\Delta'$ face and a non-$\Delta'$ face is $\pi/2n$, the angle in $\cP'$ is $\pi/n$. Or if $n=1$, this edge in $\cP$ becomes interior to a face of $\cP'$.

Finally, if $\Gamma'$ is finite, then $\cP'$ is a finite union of tiles, so it has finite volume and finitely many faces (and conversely).
\end{proof}

\begin{prop} \label{FiniteBadFaceOrbit2}
    Let $\Delta'$ be a union of equivalence classes of faces of~$\cP$, all non-idealizable, and
    let $\Gamma'$ and $\cP'$ be as in \propref{FiniteBadFaceOrbit1}. If $\Gamma'$ is finite, then every ideal $\cP$-polyhedron is a $\cP'$-polyhedron. 
\end{prop}

\begin{proof}
Let $\cQ$ be an ideal $\cP$-polyhedron. Assuming that $\cP \subseteq \cQ$, as in \propref{BadFaceOrbit}, we have that $\cP' \subseteq \cQ$. If $\cQ$ is larger than this, then it must contain an additional copy of $\cP$ which is adjacent to $\cP'$ along a face $F$ of $\cP'$. Again by \propref{BadFaceOrbit}, $\cQ$ contains the orbit of this copy of $\cP$ under the group generated by its copies of the 
    faces in $\Delta'$. This orbit is the reflection of $\cP'$ across $F$. Repeating this process, we find that $\cQ$ is tiled by translates of $\cP'$, using the group generated by reflections across the faces of $\cP'$, so it is a $\cP'$-polyhedron.
\end{proof}

We conclude this section with six examples which use the tools introduced in this section alongside tools from previous sections. The first three combine \propref{FiniteBadFaceOrbit2} with propositions from \secref{BadPairs}.

\begin{example} \label{Example8}
    In the polyhedron $\cP={}$\SW{8}, faces 1 and 5 meet their neighbors in angles of $\pi/2n$, so each is its own equivalence class. These faces are compact and hence non-idealizable. Thus any $\cP$-polyhedron is in fact a $\cP'$-polyhedron, where $\Gamma'$ is the group generated by reflections $\sigma_1$ and $\sigma_5$, isomorphic to the Klein 4-group. 
    Among the faces of $\cP'$ are $F_4\cup\sigma_5(F_4)$ and
    its image under $\sigma_1$. Write $E_4$ for the edge along which they meet, 
    where they make an angle $\pi/3$ (twice the angle between $F_1$ and~$F_4$).
    Applying the diagram automorphism
    $1\leftrightarrow5$, $2\leftrightarrow6$,
    $3\leftrightarrow4$ of~$\cP$, we obtain an edge $E_3$ with dihedral angle~$\pi/3$,
    which lies in the faces $F_3\cup\sigma_1(F_3)$ and 
    $\sigma_5(F_3\cup\sigma_1(F_3))$ of~$\cP'$.
    Neither face containing $E_4$ meets either face
    containing $E_3$,  
    so $E_3$ and $E_4$ 
    form a bad edge pair. By \propref{BadEdgePair}, there is no 
    ideal, right-angled $\cP'$-polyhedron.  So there is no
    ideal, right-angled  $\cP$-polyhedron either.
\end{example}

\begin{example} \label{Example11}
    In the polyhedron $\cP={}$\SW{11}, there is one equivalence class of compact faces, namely \{1, 5\}. The group $\Gamma'$ generated by their reflections has six elements.  Among the faces of $\cP'$ are $F_2\cup\sigma_1(F_2)$ and its other two images under
    $\langle{\sigma_1,\sigma_5}\rangle$.  These three faces
    meet at a finite vertex~$V_{125}$, so named
    because it is 
    the vertex $F_1\cap F_2\cap F_5$ of~$\cP$.  Applying the diagram automorphism
    $1\leftrightarrow5$, $2\leftrightarrow4$, $3\leftrightarrow6$ of~$\cP$ gives another
    finite vertex~$V_{145}$ of~$\cP'$.  No face of~$\cP'$ containing
    $V_{125}$ meets any face of $\cP'$ containing $V_{145}$.  (The key to this is that
    faces $2$ and $4$ of~$\cP$ do not meet.)
    Therefore these vertices form a bad pair of vertices of~$\cP'$.
    By \propref{BadVertexPair}, 
    there is no ideal, right-angled $\cP'$-polyhedron.  
    So there is no
    ideal, right-angled  $\cP$-polyhedron either.
%
\end{example}

\begin{example} \label{Example19}
    In polyhedron $\cP={}$\SW{19}, there are two equivalence classes of 
    compact faces, namely \{1, 7\} and \{6\}. 
    The group $\Gamma'$ generated by their reflections  has 
    order twelve.   The union of the $\langle\sigma_1,\sigma_6,\sigma_7\rangle$-images
    of~$F_4$ equals the union of six faces of~$\cP'$.  One of these is
    $F_4\cup\sigma_7(F_4)$.  Two more are its images under
    $\langle\sigma_1,\sigma_7\rangle$, and we will see the others soon.  
    These three faces contain 
    the finite vertex $V_{147}=F_1\cap F_4\cap F_7$ of~$\cP$, which is also
    a vertex of~$\cP'$.  They are also ultraparallel to the plane~$P$
    containing $F_6$.  (For the face containing $F_4$ this is obvious; 
    for the others it follows by symmetry under $\langle\sigma_1,\sigma_7\rangle$,
    which fixes~$P$.)  It follows that $P$ separates the union of these three faces
    from the union of their $\sigma_6$-images.  Therefore $V_{147}$ and $\sigma_6(V_{147})$
    form a bad pair of vertices of~$\cP'$,
    and we can argue as in the previous example.  
\end{example}

The next example uses a more subtle combination of ideas from this section and ideas from \secref{BadPairs}.

\begin{example} \label{Example10}
    In polyhedron $\cP={}$\SW{10}, face $6$ is compact and meets its neighbors at angles of $\pi/2n$, so it is its own equivalence class. $\cQ=\cP \cup \sigma_6(\cP)$ is a Coxeter polyhedron with diagram 
\begin{center}
\begin{tikzpicture}
\coordinate (1) at (0,0);
\coordinate (2) at (0,1);
\coordinate (3) at (-2,-1);
\coordinate (4) at (1,0);
\coordinate (5) at (-1,0);
\coordinate (3') at (-2, 1);
\coordinate (2') at (0, -1);

\draw[thick] (4) -- (1);
\draw[thick, line width=2pt] (4) -- (2);
\draw[thick, line width=2pt] (4) -- (2');
\draw[dashed] (5) -- (1);
\draw[thick, line width=2pt] (5) -- (3);
\draw[thick, line width=2pt] (5) -- (3');
\draw[dashed] (3) -- (3');
\draw[dashed] (2) -- (3');
\draw[dashed] (2') -- (3);

\filldraw[fill=white] (1) circle (2pt) node[above] {\scriptsize 1};
\filldraw[fill=white] (2) circle (2pt) node[right] {\scriptsize 2};
\filldraw[fill=white] (2') circle (2pt) node[right] {\scriptsize$2'$};
\filldraw[fill=white] (3) circle (2pt) node[left] {\scriptsize 3};
\filldraw[fill=white] (3') circle (2pt) node[left] {\scriptsize$3'$};
\filldraw[fill=white] (4) circle (2pt) node[right] {\scriptsize 4};
\filldraw[fill=white] (5) circle (2pt) node[above] {\scriptsize 5};
\end{tikzpicture}
\end{center}
\noindent By \propref{FiniteBadFaceOrbit2}, every
    ideal, right-angled  $\cP$-polyhedron is also a $\cQ$-polyhedron.
    We claim that no ideal, right-angled $\cQ$-polyhedron~$\cR$ exists; suppose otherwise.
    Each face of~$\cQ$, other than $1$ and~$4$, forms its own equivalence class and meets
    all of its neighboring faces at right angles.  We repeatedly cut $\cR$ along
    planes of these types, as in \propref{GluingPolyhedra}, decomposing
    $\cR$ into right-angled $\cQ$-polyhedra (``blocks'') which are obtained by gluing
    copies of~$\cQ$ along faces of types $1$ and~$4$.
    Because blocks are right-angled, and
    the dihedral angle along the $1$-$4$ edge is~$\pi/3$, each block consists of
    either~$3$ or~$6$ copies of~$\cQ$, gathered around such an edge.  In particular,
    each block contains two copies of~$\cQ$ that are separated by a face of type~$4$.

    We claim that each block has
    a pair of bad compact edges, in the sense of \propref{BadCompactEdgePair}.  To see this, first
    observe
    that the $2$-$2'$ edge~$E$ of~$\cQ$ is compact and orthogonal to  faces $1$ and~$5$.
    So $E$ (or rather its extension to an edge of the block) is a candidate for one
    edge in a pair of bad compact edges.  We apply this observation in two copies
    of $\cQ$ in the block, separated by a face of type~$4$.  
    It is not possible to choose a face of the block that contains~$E$, and
    one that contains~$\sigma_4(E)$, that share an edge.  This is because
    the faces  of~$\cQ$ that contain~$E$ (namely $2$ and~$2'$) are parallel to face~$4$.
    Therefore $E$ and $\sigma_4(E)$ form a pair
    of bad compact edges of the block.  
    Therefore by \rmkref{BadCompactEdgePairRemark}, $\cR$ 
    has a pair of bad compact edges, contrary to the assumption that it is ideal. 
\end{example}

The last two examples in this section do not directly use the polyhedron $\cP'$, but they use the idea of minimal ideal polygons tiled by faces of $\cP$ alongside Euler characteristic arguments.

\begin{example} \label{Example16}
    In  polyhedron $\cP={}$\SW{16}, all angles are of the form $\pi/2n$ so each face forms its own equivalence class. Faces 1 and 5 are compact, so they cannot tile ideal polygons,
    so they cannot occur in the boundary of any $\cP$-polyhedron. 
    Faces 3 and 4 are quadrilaterals. Face 2 is a $(2, 6, \infty)$ hyperbolic triangle. By \thmref{IdealizablePolygons}, the minimal ideal polygon tiled by this triangle is a quadrilateral. Thus every ideal, right-angled  $\cP$-polyhedron lacks triangular faces. 
    By \propref{AllQuadEuler}, no such polyhedron can exist. 
\end{example}


\begin{example} \label{Example3}
    For polyhedron $\cP={}$\SW{3}, we claim that no ideal, right-angled $\cP$-polyhedron $\cQ$
    exists; suppose otherwise.
    The equivalence classes of faces of~$\cP$ are \{1, 4\}, \{2, 5\}, \{3\}, and \{6\}. 
    Writing $P$ for a plane of type~$\{1,4\}$, the fundamental polygon for $\Gamma_P$
    is a $(2, 3, 2, \infty)$ quadrilateral. By \thmref{IdealizablePolygons}, 
    the minimal ideal polygon tiled by this quadrilateral is a hexagon; see
    \figref{MinIdealPolygons}.  Therefore, 
     every face of~$\cQ$, that lies in a plane of type~$\{1,4\}$,
    has at least $6$ sides.  By the diagram automorphism of~$\cP$, the same holds
    for faces of~$\cQ$ that lie in planes of type~$\{2,5\}$.

    At the ideal vertex of~$\cP$, 
    faces 2 and 4 are parallel.  It follows that at every ideal vertex~$V$
    of~$\cQ$,
    there are two parallel faces of~$\cQ$, each of which lies in a plane
    of type $\{1,4\}$ or type $\{2,5\}$.  By the previous paragraph, these
    faces have at least six sides each.  The other faces incident to~$V$ have at least~$3$
    sides each.
    By Equation \eqref{VertexLocalEC}, the local Euler characteristic at~$V$ is
    nonpositive.  Summing over vertices shows that
    the total Euler characteristic is not $2$, which is 
    a contradiction.
\end{example}

We have now ruled out 43 of the 49 Scharlau-Walhorn polyhedra. Since four admit ideal, right-angled $\cP$-polyhedra, we have two to go. 

\section{Combinatorial Arguments} \label{Combinatorics}

In this section, we rule out the polyhedra $\cP={}$\SW{1}, \SW{6}. This is difficult because they are so simple (see \figref{HardPolyhedra}). We use all our results on potential vertices and faces of ideal, right-angled  $\cP$-polyhedra to reformulate  \qref{PolyhedronQuestion} combinatorially: the existence of such a $\cP$-polyhedron implies the existence of a
certain kind of combinatorial tiling of $S^2$. We prove these tilings do not exist, so
neither do the polyhedra.

\begin{figure}[ht] 
\centering
\, \hfill \includegraphics[width=.4\textwidth]{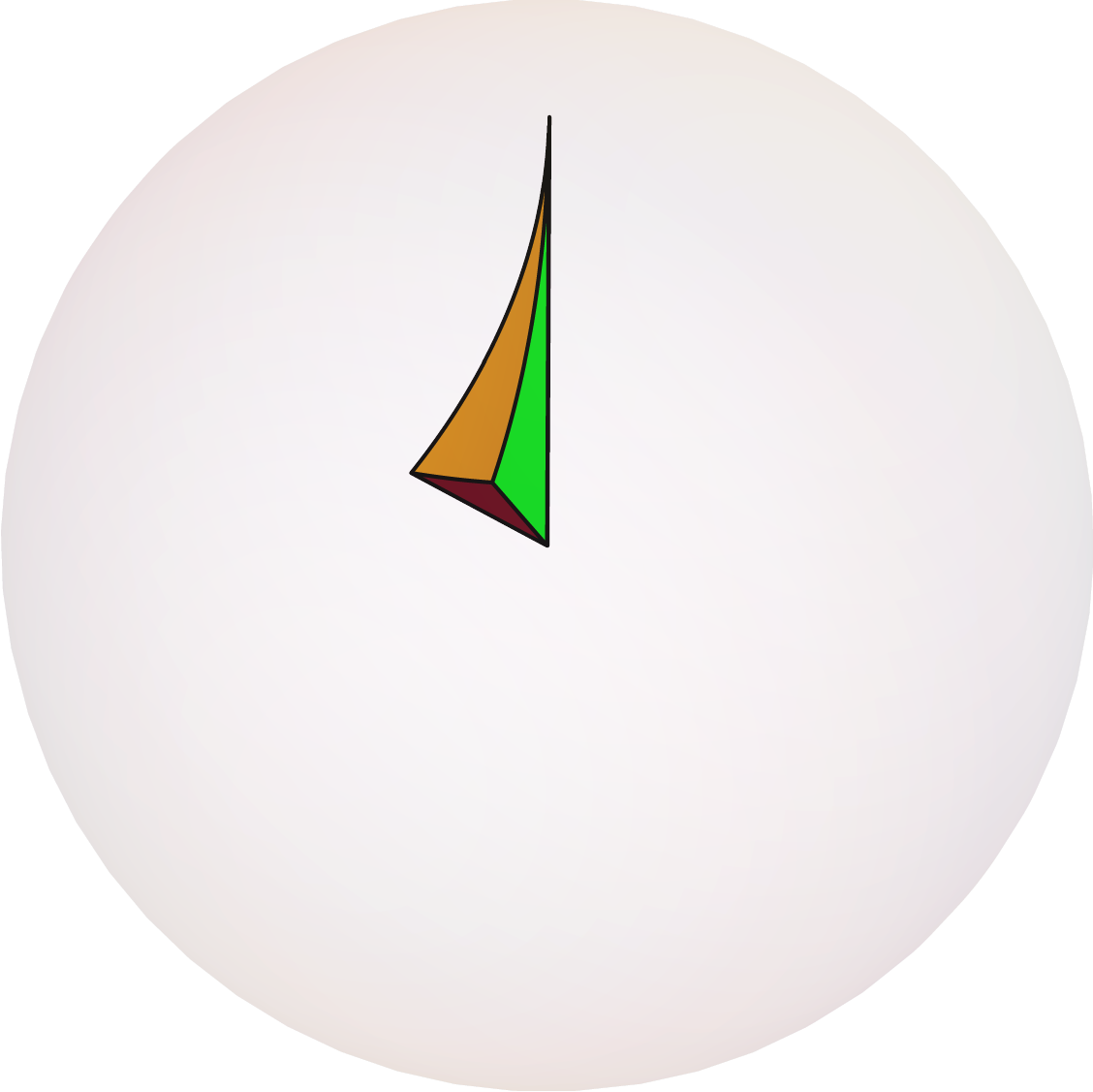} \hfill
\includegraphics[width=.4\textwidth]{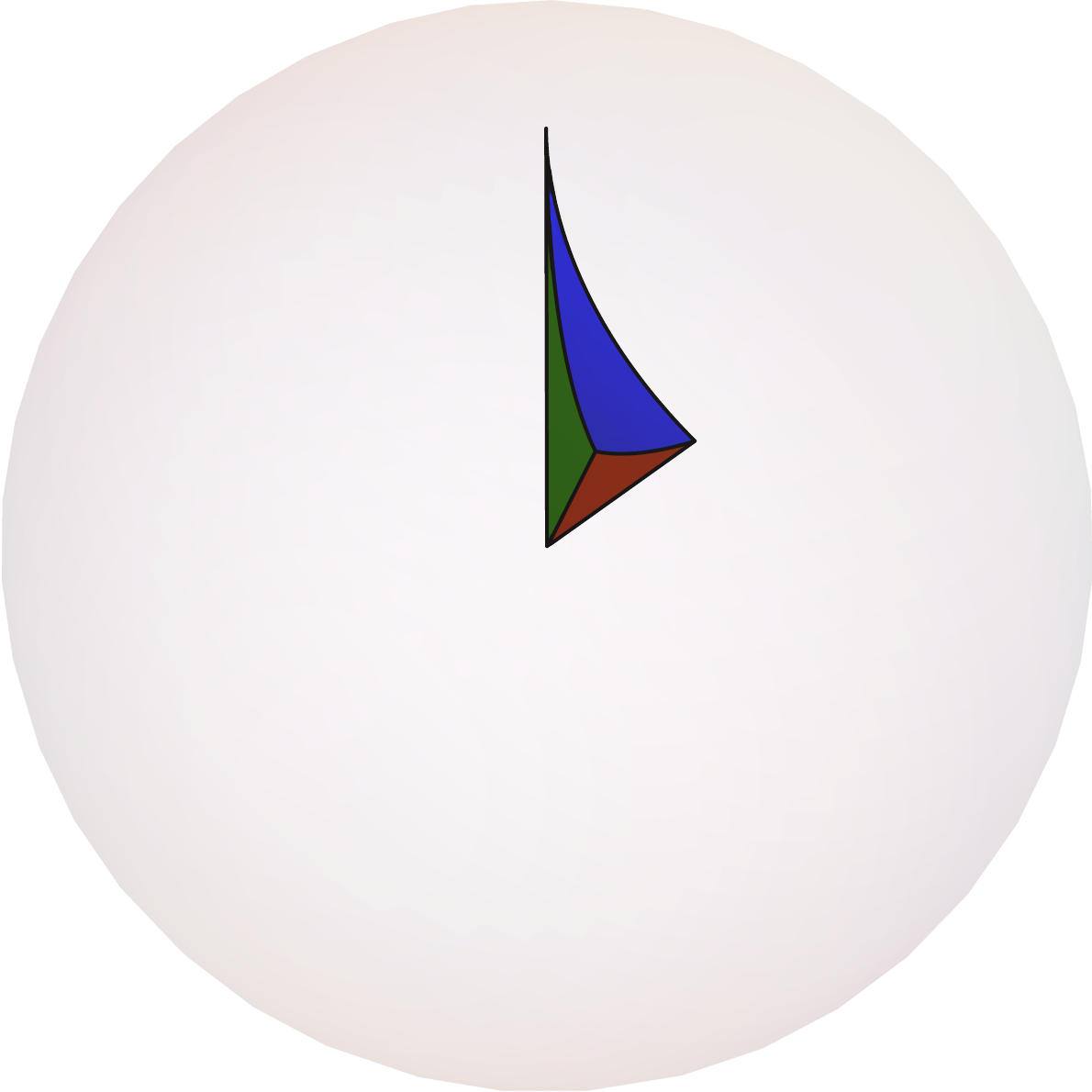} \hfill \,
\caption{From left to right: polyhedra \SW{1} and \SW{6}.}
\label{HardPolyhedra} 
\end{figure}

\begin{prop} \label{SW6Combinatorics}
    The polyhedron $\cP={}$\SW{6} does not admit an ideal, right-angled  $\cP$-polyhedron.
\end{prop}

\begin{proof}
Face 1 of $\cP$ is compact and meets its neighbors at angles of $\pi/2$ and $\pi/4$, so it is not idealizable. The other faces are. Thus by  \propref{FiniteBadFaceOrbit1} and \propref{FiniteBadFaceOrbit2}, we may immediately replace $\cP$ by $\cP'=\cP \cup \sigma_1(\cP)$. This is a Coxeter polyhedron with diagram

\begin{center}
\begin{tikzpicture} 
\coordinate (2) at (2,0);
\coordinate (3) at (3,0);
\coordinate (4) at (1,.5);
\coordinate (4') at (1,-.5);
\draw[thick, double distance = 2.5pt] (3) -- (2);
\draw[thick, double distance = .3pt] (3) -- (2);
\draw[thick] (4') -- (2);
\draw[thick] (4) -- (2);

\filldraw[fill=white] (2) circle (2pt) node[above] {\scriptsize 2};
\filldraw[fill=white] (3) circle (2pt) node[above] {\scriptsize 3};
\filldraw[fill=white] (4) circle (2pt) node[left] {\scriptsize 4};
\filldraw[fill=white] (4') circle (2pt) node[left] {\scriptsize$4'$};
\end{tikzpicture}
\end{center}

The equivalence classes of faces are $\{2, 4, 4'\}$ and $\{3\}$. The fundamental polygon for the $\{2, 4, 4'\}$ equivalence class is a $(2, 2, \infty, \infty)$ quadrilateral, while face $3$ is a $(2, \infty, \infty)$ triangle. These fundamental polygons are shown in the
    top row of \figref{SW6Faces}, colored blue and yellow.
Also shown are  the minimal ideal polygons they tile, from
    \propref{IdealizablePolygons}.  These are a quadrilateral and a triangle,
    each the union of  two copies of the fundamental polygon.

\begin{figure}[ht] 
\centering
\, \hfill
\begin{tikzpicture}[scale=2]    
    \clip (-.71,-.71) rectangle (.71,.71);
	\draw[thick] (0, .41421) -- (0, -.41421);
	\draw (0, 0) -- (0.707107,0.707107);
	\draw (0, 0) -- (0.707107, -0.707107);
    \fill[blue, opacity=0.2] circle (1);
    \fill[white] (1.414, 0) circle (1);
    \fill[white] (0, 1.414) circle (1);
    \fill[white] (0, -1.414) circle (1);
    \fill[white] (0, 1) -- (-1, 1) -- (-1, -1) -- (0, -1) -- cycle;
	\draw[thick] (0, .41421) -- (0, -.41421);
	\node at (.1, .25) {$4\phantom{'}$};
	\node at (.1, -.25) {$4'$};
	\node at (.25, 0) {2};
	\draw[thick] (0.707107,0.707107) arc
    [
        start angle=135,
        end angle=225,
        x radius=1,
        y radius =1
    ] ;	
	\draw[thick] (0.707107,0.707107) arc
    [
        start angle=-45,
        end angle=-90,
        x radius=1,
        y radius =1
    ] ;	
	\draw[thick] (0.707107,-0.707107) arc
    [
        start angle=45,
        end angle=90,
        x radius=1,
        y radius =1
    ] ;	
\end{tikzpicture}
\hfill
\begin{tikzpicture}[scale=2]
    \clip (-1,-.21) rectangle (1,1.2);
	\draw (0, 0) -- (0.707107, -0.707107);
    \fill[yellow, opacity=0.2] circle (1);
    \fill[white] (1,1) circle (1);
    \fill[white] (0, 1) -- (-1, 1) -- (-1, -1) -- (0, -1) -- cycle;
    \fill[white] (1, 0) -- (1, -1) -- (-1, -1) -- (-1, 0) -- cycle;
	\node at (.2, .2) {3};
	\draw[thick] (0, 0) -- (0, 1);
	\draw[thick] (0, 0) -- (1,0);
	\draw[thick] (1, 0) arc
    [
        start angle=-90,
        end angle=-180,
        x radius=1,
        y radius =1
    ] ;	
\end{tikzpicture}
\hfill \,

\,

\, \hfill
\begin{tikzpicture}[scale=2]
    \clip (-.71,-.71) rectangle (.71,.71);
	\draw[thick] (0, .41421) -- (0, -.41421);
    \fill[blue, opacity=0.2] circle (1);
    \fill[white] (1.414, 0) circle (1);
    \fill[white] (0, 1.414) circle (1);
    \fill[white] (0, -1.414) circle (1);
    \fill[white] (-1.414, 0) circle (1);
	\draw[thick] (0, .41421) -- (0, -.41421);
    \draw[thin, dashed] (-0.707107,-0.707107) -- (0.707107,0.707107);
    \draw[thin, dashed] (-0.707107,0.707107) -- (0.707107,-0.707107);
	\draw[thick] (0.707107,0.707107) arc
    [
        start angle=135,
        end angle=225,
        x radius=1,
        y radius =1
    ] ;	
	\draw[thick] (0.707107,0.707107) arc
    [
        start angle=-45,
        end angle=-135,
        x radius=1,
        y radius =1
    ] ;	
	\draw[thick] (0.707107,-0.707107) arc
    [
        start angle=45,
        end angle=135,
        x radius=1,
        y radius =1
    ] ;	
	\draw[thick] (-0.707107,0.707107) arc
    [
        start angle=45,
        end angle=-45,
        x radius=1,
        y radius =1
    ] ;	
\end{tikzpicture}
\hfill
\begin{tikzpicture}[scale=2]
    \clip (-1,-.21) rectangle (1,1.2);
	\draw (0, 0) -- (0.707107, -0.707107);
    \fill[yellow, opacity=0.2] circle (1);
    \fill[white] (1,1) circle (1);
    \fill[white] (-1,1) circle (1);
    \fill[white] (1, 0) -- (1, -1) -- (-1, -1) -- (-1, 0) -- cycle;
	\draw[thick] (0, 0) -- (0, 1);
	\draw[thick] (-1, 0) -- (1,0);
	\draw[thick] (1, 0) arc
    [
        start angle=-90,
        end angle=-180,
        x radius=1,
        y radius =1
    ] ;	
	\draw[thick] (-1, 0) arc
    [
        start angle=-90,
        end angle=0,
        x radius=1,
        y radius =1
    ] ;	
\end{tikzpicture}
\hfill \,

\caption{Above: fundamental polygons for the $\{2, 4, 4'\}$ and $\{3\}$ equivalence classes. Below: minimal ideal polygons---quadrilateral and triangle.}
\label{SW6Faces} 
\end{figure}
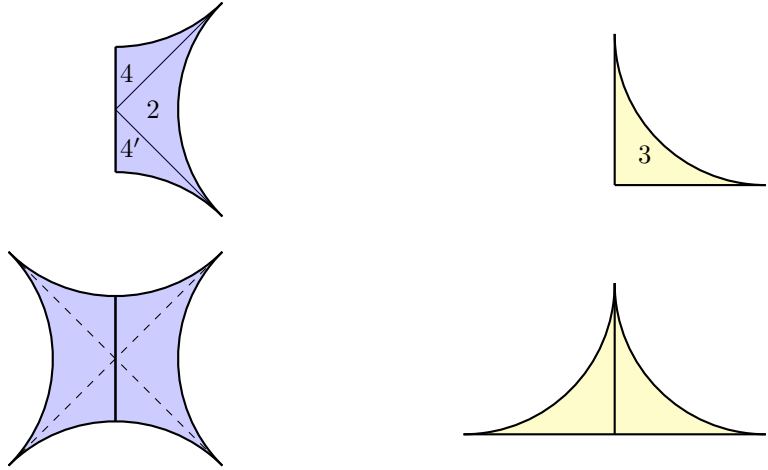

By \propref{IdealVertices}, each vertex of an ideal $\cP'$-polyhedron~$\cQ$
has four incident faces, 
forming a rectangle (the link of the vertex).   
At one ideal vertex of $\cP'$, faces 2, 3 and 4 form a $(2,3,6)$ triangle, 
and at the other, faces 2, 3 and $4'$ form another.  It follows that for each 
vertex of~$\cQ$, 
and one pair of opposite edges of its
rectangle, the corresponding faces of~$\cQ$
are tiled by copies of face~$3$.  Each of the other two
faces is tiled by copies of faces 2, 4 and~$4'$.   Because~$\cQ$ 
is ideal, each of its faces is tiled by copies of the minimal
polygons in the lower half of
\figref{SW6Faces}.
We annotate them as  shown in
\figref{SW6Tiles}.   
Namely, each ``short'' (unmarked) edge 
is a copy of the 2-3 edge of~$\cP'$.  
Each ``long'' (twice ticked) edge is the union of two copies
of the 3-4 edge (or two copies of the 3-$4'$ edge) of~$\cP'$.

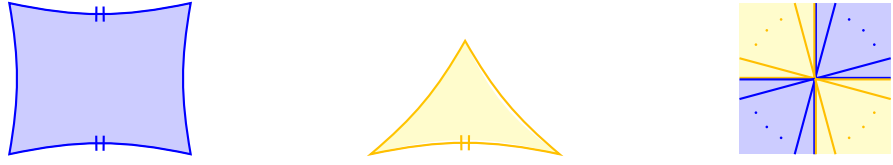
\begin{figure}[ht] 
\centering
\, \hfill
\begin{tikzpicture}
\coordinate (A) at (-1.2, 1);
\coordinate (B) at (-1.2, -1);
\coordinate (C) at (1.2, -1);
\coordinate (D) at (1.2, 1);
\draw[thick, blue] (A) to[bend left=10] (B) to[bend left=12] (C) to[bend left=10] (D) to[bend left=12] cycle;
\fill[opacity=0.2, blue] (A) to[bend left=10] (B) to[bend left=12] (C) to[bend left=10] (D) to[bend left=12] cycle;
\draw[blue, decoration={markings, mark=at position 0.48 with {\draw[blue, thick] (0,-0.1) -- (0,0.1);}}, 
    decoration={markings, mark=at position 0.52 with {\draw[blue, thick] (0,-0.1) -- (0,0.1);}},
    postaction={decorate}] (D) to[bend left=12] (A);
\draw[blue, decoration={markings, mark=at position 0.48 with {\draw[blue, thick] (0,-0.1) -- (0,0.1);}}, 
    decoration={markings, mark=at position 0.52 with {\draw[blue, thick] (0,-0.1) -- (0,0.1);}},
    postaction={decorate}] (B) to[bend left=12] (C);
\end{tikzpicture}
\hfill
\begin{tikzpicture}
\coordinate (A) at (-1.25, .3);
\coordinate (B) at (1.25, .3);
\coordinate (C) at (0, 1.8);
\draw[thick, gold] (A) to[bend left=12] (B) to[bend left=10] (C) to[bend left=10] cycle;
\fill[opacity=0.2, yellow] (A) to[bend left=15] (B) to[bend left=15] (C) to[bend left=10] cycle;
\draw[gold, decoration={markings, mark=at position 0.48 with {\draw[thick] (0,-0.1) -- (0,0.1);}}, 
    decoration={markings, mark=at position 0.52 with {\draw[thick] (0,-0.1) -- (0,0.1);}},
    postaction={decorate}] (A) to[bend left=12] (B);
\end{tikzpicture} \hfill 
\begin{tikzpicture}
\coordinate (C) at (0,0);
\coordinate (A0plus) at (1.0,.01);
\coordinate (A0minus) at (1.0,-.01);
\coordinate (A1) at (1,0.27);
\coordinate (A2) at (.9*0.87,.9*0.50);
\coordinate (A3) at (.9*0.71,.9*0.71);
\coordinate (A4) at (.9*0.50,.9*0.87);
\coordinate (A5) at (0.27,1);
\coordinate (A6plus) at (0.01,1.0);
\coordinate (A6minus) at (-0.01,1.0);
\coordinate (A7) at (-0.27,1);
\coordinate (A8) at (-.9*0.50,.9*0.87);
\coordinate (A9) at (-.9*0.71,.9*0.71);
\coordinate (A10) at (-.9*0.87,.9*0.50);
\coordinate (A11) at (-1,0.27);
\coordinate (A12plus) at (-1.0,0.01);
\coordinate (A12minus) at (-1.0,-0.01);
\coordinate (A13) at (-1,-0.27);
\coordinate (A14) at (-.9*0.87,-.9*0.50);
\coordinate (A15) at (-.9*0.71,-.9*0.71);
\coordinate (A16) at (-.9*0.50,-.9*0.87);
\coordinate (A17) at (-0.27,-1);
\coordinate (A18plus) at (0.01,-1.0);
\coordinate (A18minus) at (-0.01,-1.0);
\coordinate (A19) at (0.27,-1);
\coordinate (A20) at (.9*0.50,-.9*0.87);
\coordinate (A21) at (.9*0.71,-.9*0.71);
\coordinate (A22) at (.9*0.87,-.9*0.50);
\coordinate (A23) at (1,-0.27);
\fill[opacity=0.2, blue] (.01, .01) -- (1, .01) -- (1,1) -- (.01, 1) -- cycle;
\fill[opacity=0.2, yellow] (-.01, .01) -- (-1, .01) -- (-1,1) -- (-.01, 1) -- cycle;
\fill[opacity=0.2, blue] (-.01, -.01) -- (-1, -.01) -- (-1,-1) -- (-.01, -1) -- cycle;
\fill[opacity=0.2, yellow] (.01, -.01) -- (1, -.01) -- (1,-1) -- (.01, -1) -- cycle;
\draw[thick, blue] (0,.01) -- (A0plus);
\draw[thick, blue] (C) -- (A1);
\fill[blue] (A2) circle(.5pt);
\fill[blue] (A3) circle(.5pt);
\fill[blue] (A4) circle(.5pt);
\draw[thick, blue] (C) -- (A5);
\draw[thick, blue] (.01,0) -- (A6plus);
\draw[thick, gold] (-.01,0) -- (A6minus);
\draw[thick, gold] (C) -- (A7);
\fill[gold] (A8) circle(.5pt);
\fill[gold] (A9) circle(.5pt);
\fill[gold] (A10) circle(.5pt);
\draw[thick, gold] (C) -- (A11);
\draw[thick, gold] (0,.01) -- (A12plus);
\draw[thick, blue] (0,-.01) -- (A12minus);
\draw[thick, blue] (C) -- (A13);
\fill[blue] (A14) circle(.5pt);
\fill[blue] (A15) circle(.5pt);
\fill[blue] (A16) circle(.5pt);
\draw[thick, blue] (C) -- (A17);
\draw[thick, blue] (-.01,0) -- (A18minus);
\draw[thick, gold] (.01,0) -- (A18plus);
\draw[thick, gold] (C) -- (A19);
\fill[gold] (A20) circle(.5pt);
\fill[gold] (A21) circle(.5pt);
\fill[gold] (A22) circle(.5pt);
\draw[thick, gold] (C) -- (A23);
\draw[thick, gold] (0,-.01) -- (A0minus);
\end{tikzpicture} \hfill \,

\caption{From left: quadrilateral tile; triangle tile; pattern of tiles around a vertex.}
\label{SW6Tiles} 
\end{figure}

This tiling of $\partial\cQ\cong S^2$ obviously has the following properties:
\begin{enumerate}
\item Each yellow or blue face (union of contiguous yellow or blue tiles) is a topological
disc.
\item Long edges are glued to long edges and short edges to short edges.
\item At each vertex, the yellow and blue tiles meet as shown in \figref{SW6Tiles}, with a sector of yellow tiles, then a sector of blue tiles, then a sector of yellow tiles, then a sector of blue tiles around the vertex. 
\end{enumerate}
We will derive a contradiction.  
First, each vertex lies on at least two long edges.  (Choose a blue tile in each of the two blue sectors, and in each of these tiles choose the long edge that contains the vertex.)  So, one can start a path consisting of long edges at some vertex and continue it indefinitely. As the tiling is finite, there will be a circle consisting of these edges. A circle cuts the sphere into two discs. We may choose a circle of long edges which cuts out a minimal (by inclusion) disc~$D$. In particular, this disc has no internal long edges and no interior vertices. The short edges in~$D$ have their endpoints in~$\partial D$, and do not cross each other.  Therefore we may choose an ``innermost'' short edge: one of the two subdiscs, into which it divides~$D$, contains no other
short edge. By innermost-ness, this  subdisc consists of a single tile. This contradicts the fact that each tile has two short edges.
\end{proof}

\begin{prop} \label{SW1Combinatorics}
    The polyhedron $\cP={}$\SW{1} does not admit an ideal, right-angled  $\cP$-polyhedron.
\end{prop}

\begin{proof}
The equivalence classes of faces of $\cP$ are \{1, 2, 4\} and \{3\}. The fundamental polygon for the equivalence class \{1, 2, 4\} is a $(2, 6, \infty)$ triangle, while face 3 is a $(2, 3, \infty)$ triangle. These fundamental polygons are shown in \figref{SW1Faces}, colored blue and red.
We also illustrate the minimal ideal polygons tiled by each fundamental polygon, as given in \propref{IdealizablePolygons}. In the first case, the minimal ideal polygon is a quadrilateral tiled by six fundamental domains. In the second case, there are two choices of minimal ideal polygon, each one a triangle tiled by six fundamental domains. We call one of these the ``balanced'' triangle and one the ``slim'' triangle. 

\begin{figure}[ht] 
\centering
\, \hfill
\begin{tikzpicture}[scale=4]
	\clip (0,0) -- (1,0) -- (0.5, 0.288675) -- cycle;
	\fill[blue, opacity=0.2] (0,0) -- (1,0) -- (0.5, 0.288675) -- cycle;
	\fill[white] (1, 0.57735) circle (0.57735);
	\draw[ultra thick] (0,0) -- (1, 0);
	\draw[ultra thick] (0,0) -- (0.5, 0.288675);
	\draw[thick] (1, 0) arc
    [
        start angle=-90,
        end angle=-150,
        x radius=0.57735,
        y radius =0.57735
    ] ;	
	\draw (1,0) arc
    [
        start angle=-90,
        end angle=-114.75,
        x radius=1.73205,
        y radius =1.73205
    ] ;	
	\draw (0.267949, 0) arc
    [
        start angle=180,
        end angle=174.7,
        x radius=1.73205,
        y radius =1.73205
    ] ;	
	
	\node at (.2, .05) {1};
	\node at (.35, .06) {4};
	\node at (.47, .17) {2};
\end{tikzpicture}
\hfill
\begin{tikzpicture}[scale=4]
	\clip (0, 0) -- (1,0) --  (0.133975, 0.232051) -- cycle;
	\fill[red, opacity=0.2] (0, 0) -- (1,0) --  (0.133975, 0.232051) -- cycle;
	\fill[white] (1, 1.73205) circle (1.73205);
	\draw[ultra thick] (0, 0) -- (1,0);
	\draw[ultra thick] (0, 0) --  (0.133975, 0.232051);
	\draw[thick] (1,0) arc
    [
        start angle=-90,
        end angle=-120,
        x radius=1.73205,
        y radius =1.73205
    ] ;	
	\node at (.17, .09) {3};
\end{tikzpicture}
\hfill \,

\vspace{.25in}

\, \hfill
\begin{tikzpicture}[scale=2]
	\clip (-1,0) -- (-.5, -.433012) -- (1,0) -- (0.5, 0.866025) -- (-0.5, 0.866025) -- cycle;
	\fill[blue, opacity=0.2] (-1,0) -- (1,0) -- (0.5, 0.866025) -- (-0.5, 0.866025) -- cycle;
	\fill[white] (1, 0.57735) circle (0.57735);
	\fill[white] (-1, 0.57735) circle (0.57735);
	\fill[white] (0, 1.1547) circle (0.57735);
	\draw[thick] (-1,0) -- (1, 0);
	\draw[thick] (0,0) -- (0.5, 0.288675);
	\draw[thick] (0,0) -- (-0.5, 0.288675);
	\draw[thick] (0,0) -- (0.5, 0.866025);
	\draw[thick] (0,0) -- (-0.5, 0.866025);
	\draw[thick] (0,0) -- (0, 0.57735);
	\draw[thick] (0.5, 0.866025) arc
    [
        start angle=-30,
        end angle=-150,
        x radius=0.57735,
        y radius =0.57735
    ] ;	
	\draw[thick] (1, 0) arc
    [
        start angle=-90,
        end angle=-210,
        x radius=0.57735,
        y radius =0.57735
    ] ;	
	\draw[thin, dashed] (1,0) arc
    [
        start angle=-90,
        end angle=-150,
        x radius=1.73205,
        y radius =1.73205
    ] ;	
	\draw[thin, dashed] (0.267949, 0) arc
    [
        start angle=180,
        end angle=150,
        x radius=1.73205,
        y radius =1.73205
    ] ;	
	\draw[thick] (-1, 0) arc
    [
        start angle=-90,
        end angle=30,
        x radius=0.57735,
        y radius =0.57735
    ] ;	
	\draw[thin, dashed] (-1,0) arc
    [
        start angle=-90,
        end angle=-30,
        x radius=1.73205,
        y radius =1.73205
    ] ;	
	\draw[thin, dashed] (-0.267949, 0) arc
    [
        start angle=0,
        end angle=30,
        x radius=1.73205,
        y radius =1.73205
    ] ;	
\end{tikzpicture}
\hfill
\begin{tikzpicture}[scale=2]
	\clip (-.5, -.866025) -- (1,0) -- (-.5, .866025) -- cycle;
	\fill[red, opacity=0.2] (-.5, -.866025) -- (1,0) -- (-.5, .866025) -- cycle;
	\fill[white] (1, 1.73205) circle (1.73205);
	\fill[white] (1, -1.73205) circle (1.73205);
	\fill[white] (-0.267949-1.73205,0) circle (1.73205);
	\draw[thick] (-0.267949, 0) -- (1,0);
	\draw[thick] (-.5, -.866025) --  (0.133975, 0.232051);
	\draw[thick] (-.5, .866025) --  (0.133975, -0.232051);
	\draw[thick] (1,0) arc
    [
        start angle=-90,
        end angle=-150,
        x radius=1.73205,
        y radius =1.73205
    ] ;	
	\draw[thick] (1,0) arc
    [
        start angle=90,
        end angle=150,
        x radius=1.73205,
        y radius =1.73205
    ] ;	
    \draw[thick] (-.5, 0.866025) arc
    [
        start angle=30,
        end angle=-30,
        x radius=1.73205,
        y radius =1.73205
    ] ;
\end{tikzpicture}
\hfill
\begin{tikzpicture}[scale=2]
	\clip (-1, 0) -- (-.5, -.433012) -- (1,0) -- (-.5, .866025) -- cycle;
	\fill[red, opacity=0.2] (-1, 0) -- (1,0) -- (-.5, .866025) -- cycle;
	\fill[white] (1, 1.73205) circle (1.73205);
	\fill[white] (-1, 0.57735) circle (0.57735);
	\draw[thick] (-1, 0) -- (1,0);
	\draw[thick] (0,0) --  (0.133975, 0.232051);
	\draw[thick] (-.5, .866025) --  (0,0);
	\draw[thick] (1,0) arc
    [
        start angle=-90,
        end angle=-150,
        x radius=1.73205,
        y radius =1.73205
    ] ;	
    \draw[thick] (-.5, .866025) arc
    [
        start angle=30,
        end angle=-0,
        x radius=1.73205,
        y radius =1.73205
    ] ;
    
    \draw[thick] (-.5,0) arc
    [
        start angle=-30,
        end angle=30,
        x radius=0.866025,
        y radius =0.866025
    ] ;
    \draw[thick] (-.5,0) arc
    [
        start angle=30,
        end angle=42.5,
        x radius=0.866025,
        y radius =0.866025
    ] ;
    \draw[thick] (-1,0) arc
    [
        start angle=-90,
        end angle=30,
        x radius=0.57735,
        y radius =0.57735
    ] ;
\end{tikzpicture}
\hfill \,
\caption{Above: fundamental polygons for \{1, 2, 4\} and \{3\} equivalence classes. Below: minimal ideal polygons---quadrilateral, balanced triangle, and slim triangle.}
\label{SW1Faces} 
\end{figure}
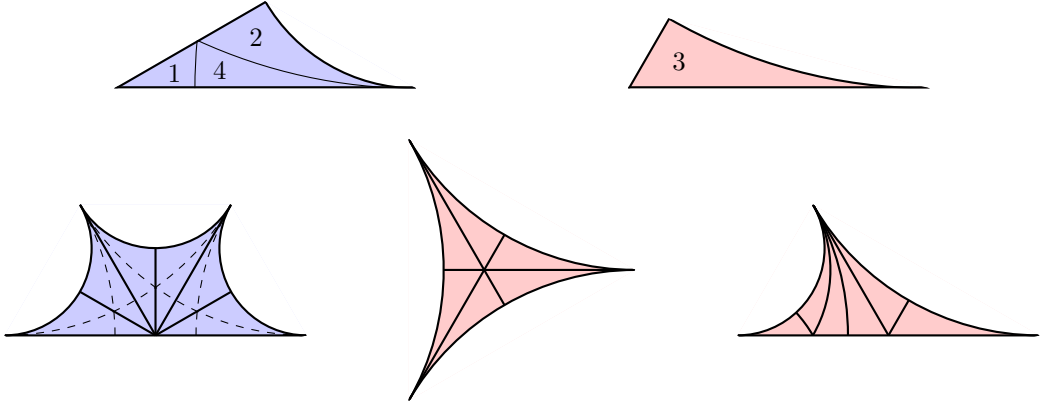

Suppose, for the sake of contradiction, that an ideal, right-angled
$\cP$-polyhedron~$\cQ$ exists.
At the ideal vertex of $\cP$, faces 2, 3, 4 form a $(2,3,6)$ triangle. Thus by \propref{IdealVertices}, every vertex of~$\cQ$ has four faces forming a rectangle, with one pair of opposite faces tiled by copies of faces 2 and 4, and the other pair tiled by copies of face 3. 
Moreover, by \corref{RectangleFact}, the number of faces of translates of $\cP$ incident to the vertex, on any one of the four faces, is equal to the number on the opposite face.

We can now reformulate our hyperbolic tiling problem as a purely combinatorial tiling problem. We illustrate the combinatorial tiles in \figref{SW1Tiles}, with some annotation. ``Short'' unmarked edges are copies of the original 2-3 edge, while ``long'' edges marked with two ticks are copies of the original 1-3 and 4-3 edges. We also label each vertex of a tile with the number of faces of translates of $\cP$ incident to that vertex.
Our statement requires a hypothesis on these vertex labels based on \corref{RectangleFact}. We claim that there is no tiling of the sphere by copies of these three tiles, with the following properties: 
\begin{enumerate}
\item Each red or blue face (collection of adjacent red or blue tiles) is a topological
disc, and these faces form a spherical polyhedron.
\item Long edges are glued to long edges and short edges are glued to short edges.
\item At each vertex, the red and blue tiles meet as shown in \figref{SW1Tiles}, with a sector of red tiles, then a sector of blue tiles, then a sector of red tiles, then a sector of blue tiles around the vertex. Moreover, the sum of the vertex labels is equal in the two red sectors and equal in the two blue sectors. 
\end{enumerate}

\begin{figure}[ht] 
\centering
\, \hfill \begin{tikzpicture}
    \coordinate (A3) at (0,0);
    \coordinate (B3) at (3.8,0);
    \coordinate (C3) at (2.8,1.732);
    \coordinate (D3) at (1,1.732);
    
    \draw[thick, blue] (A3) to[bend left=8] (B3) to[bend left=5] (C3) to[bend left=5] (D3) to[bend left=5] cycle;
    \fill[opacity=0.2, blue] (A3) to[bend left=8] (B3) to[bend left=5] (C3) to[bend left=5] (D3) to[bend left=5] cycle;
    
    \draw[blue, decoration={markings, mark=at position 0.49 with {\draw[blue, thick] (0,-0.1) -- (0,0.1);}}, 
    decoration={markings, mark=at position 0.51 with {\draw[blue, thick] (0,-0.1) -- (0,0.1);}},
    postaction={decorate}] (A3) to[bend left=8] (B3);
    
    \node at (1.1,1.5) {4};
    \node at (2.7,1.5) {4};
    \node at (0.4,0.25) {2};
    \node at (3.4,.25) {2};
\end{tikzpicture} \hfill
\begin{tikzpicture}
    \coordinate (A1) at (0,0);
    \coordinate (B1) at (2,0);
    \coordinate (C1) at (1,1.732);
    
    \draw[thick, red] (A1) to[bend left=5] (B1) to[bend left=5] (C1) to[bend left=5] cycle;
    \fill[opacity=0.2, red] (A1) to[bend left=5] (B1) to[bend left=5] (C1) to[bend left=5] cycle;

    \node at (0.4,0.25) {2};     
    \node at (1.6,0.25) {2};     
    \node at (1,1.3) {2};       
\end{tikzpicture} \hfill
\begin{tikzpicture}
    \coordinate (A2) at (0,0);
    \coordinate (B2) at (2,0);
    \coordinate (C2) at (-1,1.732);
    
    \draw[thick, red] (A2) to[bend left=5] (B2) to[bend left=8] (C2) to[bend left=5] cycle;
    \fill[opacity=0.2, red] (A2) to[bend left=5] (B2) to[bend left=8] (C2) to[bend left=5] cycle;
    
    \draw[red, decoration={markings, mark=at position 0.485 with {\draw[red, thick] (0,-0.1) -- (0,0.1);}}, 
    decoration={markings, mark=at position 0.515 with {\draw[red, thick] (0,-0.1) -- (0,0.1);}},
    postaction={decorate}] (B2) to[bend left=8] (C2);
    
    \node at (.1,0.2) {4};
    \node at (-0.5,1.2) {1};
    \node at (1.2,.2) {1};
\end{tikzpicture}  \hfill
\begin{tikzpicture}[scale=1]
\coordinate (C) at (0,0);
\coordinate (A0plus) at (1.0,.01);
\coordinate (A0minus) at (1.0,-.01);
\coordinate (A1) at (1,0.27);
\coordinate (A2) at (.9*0.87,.9*0.50);
\coordinate (A3) at (.9*0.71,.9*0.71);
\coordinate (A4) at (.9*0.50,.9*0.87);
\coordinate (A5) at (0.27,1);
\coordinate (A6plus) at (0.01,1.0);
\coordinate (A6minus) at (-0.01,1.0);
\coordinate (A7) at (-0.27,1);
\coordinate (A8) at (-.9*0.50,.9*0.87);
\coordinate (A9) at (-.9*0.71,.9*0.71);
\coordinate (A10) at (-.9*0.87,.9*0.50);
\coordinate (A11) at (-1,0.27);
\coordinate (A12plus) at (-1.0,0.01);
\coordinate (A12minus) at (-1.0,-0.01);
\coordinate (A13) at (-1,-0.27);
\coordinate (A14) at (-.9*0.87,-.9*0.50);
\coordinate (A15) at (-.9*0.71,-.9*0.71);
\coordinate (A16) at (-.9*0.50,-.9*0.87);
\coordinate (A17) at (-0.27,-1);
\coordinate (A18plus) at (0.01,-1.0);
\coordinate (A18minus) at (-0.01,-1.0);
\coordinate (A19) at (0.27,-1);
\coordinate (A20) at (.9*0.50,-.9*0.87);
\coordinate (A21) at (.9*0.71,-.9*0.71);
\coordinate (A22) at (.9*0.87,-.9*0.50);
\coordinate (A23) at (1,-0.27);
\fill[opacity=0.2, blue] (.01, .01) -- (1, .01) -- (1,1) -- (.01, 1) -- cycle;
\fill[opacity=0.2, red] (-.01, .01) -- (-1, .01) -- (-1,1) -- (-.01, 1) -- cycle;
\fill[opacity=0.2, blue] (-.01, -.01) -- (-1, -.01) -- (-1,-1) -- (-.01, -1) -- cycle;
\fill[opacity=0.2, red] (.01, -.01) -- (1, -.01) -- (1,-1) -- (.01, -1) -- cycle;
\draw[thick, blue] (0,.01) -- (A0plus);
\draw[thick, blue] (C) -- (A1);
\fill[blue] (A2) circle(.5pt);
\fill[blue] (A3) circle(.5pt);
\fill[blue] (A4) circle(.5pt);
\draw[thick, blue] (C) -- (A5);
\draw[thick, blue] (.01,0) -- (A6plus);
\draw[thick, red] (-.01,0) -- (A6minus);
\draw[thick, red] (C) -- (A7);
\fill[red] (A8) circle(.5pt);
\fill[red] (A9) circle(.5pt);
\fill[red] (A10) circle(.5pt);
\draw[thick, red] (C) -- (A11);
\draw[thick, red] (0,.01) -- (A12plus);
\draw[thick, blue] (0,-.01) -- (A12minus);
\draw[thick, blue] (C) -- (A13);
\fill[blue] (A14) circle(.5pt);
\fill[blue] (A15) circle(.5pt);
\fill[blue] (A16) circle(.5pt);
\draw[thick, blue] (C) -- (A17);
\draw[thick, blue] (-.01,0) -- (A18minus);
\draw[thick, red] (.01,0) -- (A18plus);
\draw[thick, red] (C) -- (A19);
\fill[red] (A20) circle(.5pt);
\fill[red] (A21) circle(.5pt);
\fill[red] (A22) circle(.5pt);
\draw[thick, red] (C) -- (A23);
\draw[thick, red] (0,-.01) -- (A0minus);
\end{tikzpicture} \hfill \,
\caption{From left: quadrilateral tile; balanced triangle tile; thin triangle tile; pattern of tiles around a vertex.}
\label{SW1Tiles} 
\end{figure}
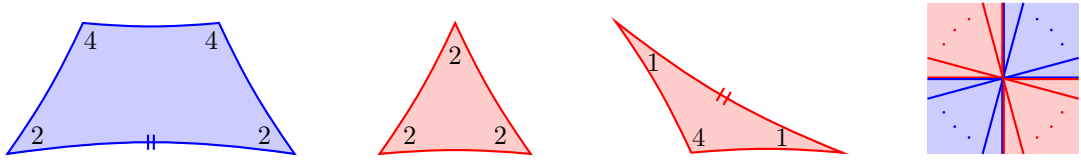

The main tool in the rest of the proof is a combinatorial analogue of \eqref{VertexLocalEC}. This can also be viewed as an angle deficit or discrete Gauss-Bonnet formula. 
\begin{lemma}\label{AngleDeficit}
    Suppose that a surface $S$ (possibly with boundary) admits a finite tiling by the tiles shown in \figref{SW1Tiles}
    . Define the angle deficit at a vertex $V$ as 
\begin{equation}
    AD_V=\left\lbrace \begin{array}{cc} 1 - \frac{1}{12}(\sum \text{angle labels at } V) & V \in S^{\circ} \\ \frac{1}{2} - \frac{1}{12}(\sum \text{angle labels at } V) & V \in \partial S \end{array}\right.
\end{equation} 
Then the sum of the angle deficits over all vertices $V$ is the Euler characteristic of $S$. 
\end{lemma}
\begin{proof}
Note that for each tile, the sum of the angle labels is equal to six times the number of edges, minus twelve. Thus, if we sum all the angle labels in the tiling, we obtain six times the number of boundary edges, plus twelve times the number of internal edges (because each internal edge is counted as part of two tiles), minus twelve times the number of faces. Hence, summing the angle deficit over all vertices gives 
\begin{equation*}
    \# \text{ int. vertices} + \frac{\# \text{ bdy. vertices}}{2} - \# \text{ int. edges} - \frac{\# \text{ bdy. edges}}{2}\ + \# \text{ faces}
\end{equation*}
Since the number of boundary vertices equals the number of boundary edges, this is simply the Euler characteristic $\chi(S)$.
\end{proof}

It follows from this lemma that a spherical tiling must contain at least one long edge where a red tile and a blue tile meet. If not, then at each vertex, the angle label in each red sector is at least 2, and the angle label in each blue sector is at least 4. Thus the angle deficit at each vertex is nonpositive, and the total angle deficit cannot equal the Euler characteristic of $2$. 

Moreover, at each vertex $V$ in the tiling, the number of incident long red-blue edges is even. We can see this as follows: in a red sector, if the total angle label is odd, then one edge of the sector is long and one is short. If the total angle label is even, then either both edges are long or both edges are short. Then by the third condition, since the two red sectors at a vertex have equal total angle labels, there are either zero, two, or four long red-blue edges.

So, one can start a path consisting of long red-blue edges at some vertex and continue it indefinitely. As the tiling is finite, there will be a cycle consisting of these edges. A cycle cuts the sphere into two discs. Let $\gamma$ be a cycle of long red-blue edges which cuts out a minimal (by inclusion) disc $D$. In particular, the interior of $D$ does not contain any long red-blue edges. 

We now consider the tiling of $D$. If a slim triangle has its long edge interior to $D$, then it meets another slim triangle along this edge; we can then replace this pair of slim triangles with a pair of balanced triangles, preserving the angle sums. Therefore we may assume that any slim triangle in $D$ has its long edge along the boundary of $D$. For the slim triangles at the boundary, we can split them into two triangles labeled $(1,2,3)$; this operation still preserves the angle deficit formula of \lemref{AngleDeficit}. These two operations are shown in \figref{Cuts}.

\begin{figure}[ht] 
\begin{tikzpicture}[scale=1.3]

    \coordinate (A1) at (0,0);
    \coordinate (B1) at (2,0);
    \coordinate (C1) at (1,1.732);
    \coordinate (D1) at (-1,1.732);
    
    \draw[thick, red] (A1) to[bend left=5] (B1) to[bend left=5] (C1) to[bend left=5] (D1) to[bend left=5] cycle;
    \draw[thick, red] (B1) to (D1);
    \fill[opacity=0.2, red] (A1) to[bend left=5] (B1) to[bend left=5] (C1) to[bend left=5] (D1) to[bend left=5] cycle;
    
    \draw[red, decoration={markings, mark=at position 0.485 with {\draw[red, thick] (0,-0.1) -- (0,0.1);}}, 
    decoration={markings, mark=at position 0.515 with {\draw[red, thick] (0,-0.1) -- (0,0.1);}},
    postaction={decorate}] (B1) to (D1);
    
    \node at (.1,0.2) {4};
    \node at (.9, 1.532) {4};
    \node at (-0.5,1.2) {1};
    \node at (1.2,.2) {1};
    \node at (1.5,.532) {1};
    \node at (-.2, 1.532) {1};

    \draw[->, thick] (1.7, .866) to[bend left=10] (2.3, .866);

\begin{scope}[xshift=3cm]
    \coordinate (A2) at (0,0);
    \coordinate (B2) at (2,0);
    \coordinate (C2) at (1,1.732);
    \coordinate (D2) at (-1,1.732);
    
    \draw[thick, red] (A2) to[bend left=5] (B2) to[bend left=5] (C2) to[bend left=5] (D2) to[bend left=5] cycle;
    \draw[thick, red] (A2) to (C2);
    \fill[opacity=0.2, red] (A2) to[bend left=5] (B2) to[bend left=5] (C2) to[bend left=5] (D2) to[bend left=5] cycle;
    
    \node at (.4,.25) {2};
    \node at (1.6,.25) {2};
    \node at (1,1.3) {2};
    \node at (0,.43) {2};
    \node at (.6,1.48) {2};
    \node at (-.6,1.48) {2};
\end{scope};

\begin{scope}[xshift=6.5cm]
    \coordinate (A3) at (0,0);
    \coordinate (B3) at (2,0);
    \coordinate (D3) at (-1,1.732);
    
    \draw[thick, red] (A3) to[bend left=5] (B3) to[bend left=5] (D3) to[bend left=5] cycle;
    \fill[opacity=0.2, red] (A3) to[bend left=5] (B3) to[bend left=5] (D3) to[bend left=5] cycle;
    
    \draw[red, decoration={markings, mark=at position 0.485 with {\draw[red, thick] (0,-0.1) -- (0,0.1);}}, 
    decoration={markings, mark=at position 0.515 with {\draw[red, thick] (0,-0.1) -- (0,0.1);}},
    postaction={decorate}] (B3) to[bend left=5] (D3);

    \node at (.1,0.2) {4};
    \node at (-0.5,1.2) {1};
    \node at (1.2,.2) {1};

    \draw[->, thick] (1.4, .866) to[bend left=10] (2.0, .866);
    
\end{scope};
\begin{scope}[xshift=9.5cm]
    \coordinate (A4) at (0,0);
    \coordinate (B4) at (2,0);
    \coordinate (D4) at (-1,1.732);
    \coordinate (E4) at (.9*.5, .9*.866);
    
    \draw[thick, red] (A4) to[bend left=5] (B4) to[bend left=5] (D4) to[bend left=5] cycle;
    \fill[opacity=0.2, red] (A4) to[bend left=5] (B4) to[bend left=5] (D4) to[bend left=5] cycle;
    \draw[thick, red] (A4) to (E4);
    
    \node at (.35,.2) {2};
    \node at (0,.38) {2};
    \node at (-0.5,1.2) {1};
    \node at (1.2,.2) {1};
    \node at (.15,.8) {3};
    \node at (.6,.55) {3};
\end{scope};
\end{tikzpicture}
\caption{Operations on slim triangles.}
\label{Cuts} 
\end{figure}
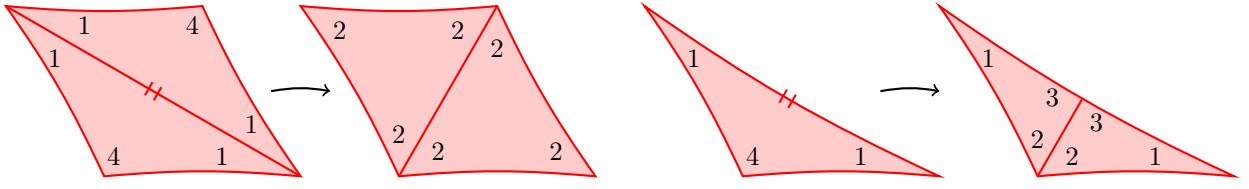

At any interior vertex of $D$, the angle deficit is bounded above by 0, since each red sector is labeled at least 2 and each blue sector is labeled at least 4. All positive contributions to the Euler characteristic must come from vertices along $\gamma$. 

Next, we consider edges in the interior of $D$ where two red tiles meet. The red-red edges form a graph embedded in $D$. At an interior vertex $V$ of $D$, each red tile must have an angle label of $2$. If $n$ red tiles meet $V$ in one sector, then the angle label is $2n$, and the opposite sector must also have $n$ red tiles. Thus both sectors have the same number ($n-1$) of red-red edges. By pairing up the opposite edges between the two sectors as shown in \figref{PairingUp}, we can decompose the red-red graph into paths which may intersect tangentially but do not cross. 

\begin{figure}[ht] 

\begin{tikzpicture} [scale=1.3]
    \draw[thick] (-1, 0) -- (1, 0);
    \draw[thick] (0, -1) -- (0, 1);
    \draw[thick, red] (0.38, 0.92) -- (-.38,-0.92);
    \draw[thick, red] (0.71, 0.71) -- (-0.71, -.71);
    \draw[thick, red] (0.92, .38) -- (-0.92,-0.38);
    \draw[->, thick] (1.2,0) to[bend left=10] (1.8, 0);

\begin{scope}[xshift=3cm]
    \draw[thick] (-1, 0) -- (1, 0);
    \draw[thick] (0, -1) -- (0, 1);
    \draw[thick, red] (0.38, 0.92) to[bend left=30] (-.92,-0.38);
    \draw[thick, red] (0.71, 0.71) -- (-0.71, -.71);
    \draw[thick, red] (0.92, .38) to[bend right=30] (-0.38, -.92);
\end{scope};

\begin{scope}[xshift=6.5cm, yshift=-.5cm]
    \draw[thick] (-1, 0) -- (1, 0);
    \draw[thick, red] (0, 0) to (-.71,.71);
    \draw[thick, red] (0, 0) -- (0.71, .71);
    \draw[thick, red] (0, 0) to (0, 1);
    \draw[->, thick] (1.2,.5) to[bend left=10] (1.8, .5);
\end{scope};
\begin{scope}[xshift=9.5cm, yshift=-.5cm]
    \draw[thick] (-1, 0) -- (1, 0);
    \draw[thick, red] (-.1, 0) to[bend right=30] (-.71,.71);
    \draw[thick, red] (0, 0) -- (0,1);
    \draw[thick, red] (.1, 0) to[bend left=30] (0.71, .71);
\end{scope};
\end{tikzpicture}

\caption{Separating red-red paths.}
\label{PairingUp} 
\end{figure}
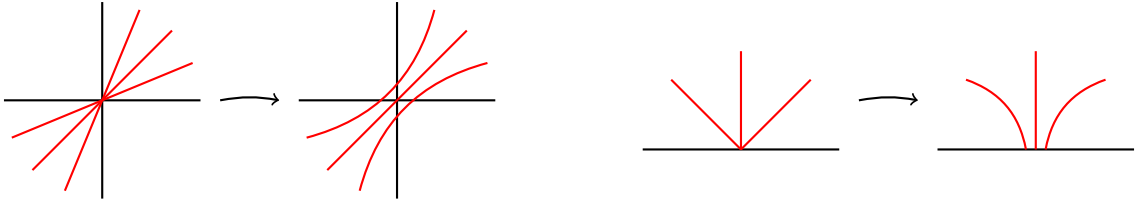

The red-red graph subdivides $D$ into disjoint regions. No such region $D'$  contains an interior red-red edge. Since slim triangles have been split into two pieces with a red-red edge, this further implies that $D'$ does not contain any full slim triangles. We will choose a particular subregion $D'$ to use below. We may choose $D'$ to be a topological disc. Moreover, using the separation of the red-red graph into disjoint paths, we may choose $D'$ whose boundary touches at most one of these paths. Let $\gamma'$ be the part of the boundary of $D'$ which consists of red-red edges. There are three cases to consider, as shown in \figref{SW1Disc}.

Case 1: $D$ does not contain any red-red edges. In this case, we take $D'=D$. 

Case 2: $\gamma'$ is a cycle. In this case the whole boundary of $D'$ consists of red-red edges. Moreover, we may assume that at every boundary vertex but one, $\gamma'$ crosses between opposite red sectors. The reason for the single exception is that $\gamma'$ may be a cycle within a longer path that intersects itself tangentially; in this case it does not cross between sectors at the self-intersection point.

Case 3: $\gamma'$ is a path connecting two distinct boundary vertices of $D$. In this case
the boundary of $D'$ is the union of two arcs: $\gamma'$, consisting of red-red edges, and the part of $\gamma$ that lies in $\partial D$, consisting of long red-blue edges. At non-endpoint vertices, $\gamma'$ crosses between opposite red sectors. The two endpoints where $\gamma'$ intersects $\gamma$ must be analyzed separately. 

\begin{figure}[ht] 
\centering
$\,$ \hfill
\begin{tikzpicture}[scale=1.5]
\coordinate (A0) at (0.873132,-0.0640563);
\coordinate (A1) at (0.540645,0.818804);
\coordinate (A2) at (-0.054978,0.828725);
\coordinate (A3) at (-0.820175,0.542033);
\coordinate (A4) at (-0.899071,-0.384612);
\coordinate (A5) at (-0.121119,-1.15943);
\coordinate (A6) at (0.777393,-0.717881);
\coordinate (A0big) at (1.05*0.873132,-1.05*0.0640563);
\coordinate (A1big) at (1.05*0.540645,1.05*0.818804);
\coordinate (A2big) at (-1.05*0.054978,1.05*0.828725);
\coordinate (A3big) at (-1.05*0.820175,1.05*0.542033);
\coordinate (A4big) at (-1.05*0.899071,-1.05*0.384612);
\coordinate (A5big) at (-1.05*0.121119,-1.05*1.15943);
\coordinate (A6big) at (1.05*0.777393,-1.05*0.717881);
\draw[thick, blue] (A0) to[bend left=10] (A1) to[bend left=10] (A2) to[bend left=10] (A3) to[bend left=10] (A4) to[bend left=10] (A5) to[bend left=10] (A6) to[bend left=10] cycle;
\fill[opacity=0.2, purple] (A0) to[bend left=10] (A1) to[bend left=10] (A2) to[bend left=10] (A3) to[bend left=10] (A4) to[bend left=10] (A5) to[bend left=10] (A6) to[bend left=10] cycle;
\draw[thick, red, decoration={markings, mark=at position 0.47 with {\draw[thick, black] (0,-0.1) -- (0,0.1);}}, decoration={markings, mark=at position 0.53 with {\draw[thick, black] (0,-0.1) -- (0,0.1);}}, postaction={decorate}, bend left=10] (A0big) to (A1big); 
\draw[thick, red, decoration={markings, mark=at position 0.47 with {\draw[thick, black] (0,-0.1) -- (0,0.1);}}, decoration={markings, mark=at position 0.53 with {\draw[thick, black] (0,-0.1) -- (0,0.1);}}, postaction={decorate}, bend left=10] (A1big) to (A2big); 
\draw[thick, red, decoration={markings, mark=at position 0.47 with {\draw[thick, black] (0,-0.1) -- (0,0.1);}}, decoration={markings, mark=at position 0.53 with {\draw[thick, black] (0,-0.1) -- (0,0.1);}}, postaction={decorate}, bend left=10] (A2big) to (A3big); 
\draw[thick, red, decoration={markings, mark=at position 0.47 with {\draw[thick, black] (0,-0.1) -- (0,0.1);}}, decoration={markings, mark=at position 0.53 with {\draw[thick, black] (0,-0.1) -- (0,0.1);}}, postaction={decorate}, bend left=10] (A3big) to (A4big); 
\draw[thick, red, decoration={markings, mark=at position 0.47 with {\draw[thick, black] (0,-0.1) -- (0,0.1);}}, decoration={markings, mark=at position 0.53 with {\draw[thick, black] (0,-0.1) -- (0,0.1);}}, postaction={decorate}, bend left=10] (A4big) to (A5big); 
\draw[thick, red, decoration={markings, mark=at position 0.47 with {\draw[thick, black] (0,-0.1) -- (0,0.1);}}, decoration={markings, mark=at position 0.53 with {\draw[thick, black] (0,-0.1) -- (0,0.1);}}, postaction={decorate}, bend left=10] (A5big) to (A6big); 
\draw[thick, red, decoration={markings, mark=at position 0.47 with {\draw[thick, black] (0,-0.1) -- (0,0.1);}}, decoration={markings, mark=at position 0.53 with {\draw[thick, black] (0,-0.1) -- (0,0.1);}}, postaction={decorate}, bend left=10] (A6big) to (A0big); 
\node at (0,0) {$D'=D$};
\end{tikzpicture} \hfill
\begin{tikzpicture}[scale=1.5]
\coordinate (A0) at (0.873132,-0.0640563);
\coordinate (A1) at (0.540645,0.818804);
\coordinate (A2) at (-0.054978,0.828725);
\coordinate (A3) at (-0.820175,0.542033);
\coordinate (A4) at (-0.899071,-0.384612);
\coordinate (A5) at (-0.121119,-1.15943);
\coordinate (A6) at (0.777393,-0.717881);
\coordinate (B0) at (0.673548,-0.0144488);
\coordinate (B1) at (0.524057,0.28765);
\coordinate (B2) at (0.0615068,0.37449);
\coordinate (B3) at (-0.113107,0.561535);
\coordinate (B4) at (-0.3924,0.324973);
\coordinate (B5) at (-0.501173,0.0464216);
\coordinate (B6) at (-0.531447,-0.356042);
\coordinate (B7) at (-0.159655,-0.529353);
\coordinate (B8) at (0.195534,-0.61514);
\coordinate (B9) at (0.418503,-0.30277);
\coordinate (B0big) at (1.06*0.673548,-1.06*0.0144488);
\coordinate (B1big) at (1.06*0.524057,1.06*0.28765);
\coordinate (B2big) at (1.06*0.0615068,1.06*0.37449);
\coordinate (B3big) at (-1.06*0.113107,1.06*0.561535);
\coordinate (B4big) at (-1.06*0.3924,1.06*0.324973);
\coordinate (B5big) at (-1.06*0.501173,1.06*0.0464216);
\coordinate (B6big) at (-1.06*0.531447,-1.06*0.356042);
\coordinate (B7big) at (-1.06*0.159655,-1.06*0.529353);
\coordinate (B8big) at (1.06*0.195534,-1.06*0.61514);
\coordinate (B9big) at (1.06*0.418503,-1.06*0.30277);
\draw[thick, decoration={markings, mark=at position 0.47 with {\draw[thick] (0,-0.1) -- (0,0.1);}}, decoration={markings, mark=at position 0.53 with {\draw[thick] (0,-0.1) -- (0,0.1);}}, postaction={decorate}, bend left=10] (A0) to (A1);
\draw[thick, decoration={markings, mark=at position 0.47 with {\draw[thick] (0,-0.1) -- (0,0.1);}}, decoration={markings, mark=at position 0.53 with {\draw[thick] (0,-0.1) -- (0,0.1);}}, postaction={decorate}, bend left=10] (A1) to (A2);
\draw[thick, decoration={markings, mark=at position 0.47 with {\draw[thick] (0,-0.1) -- (0,0.1);}}, decoration={markings, mark=at position 0.53 with {\draw[thick] (0,-0.1) -- (0,0.1);}}, postaction={decorate}, bend left=10] (A2) to (A3);
\draw[thick, decoration={markings, mark=at position 0.47 with {\draw[thick] (0,-0.1) -- (0,0.1);}}, decoration={markings, mark=at position 0.53 with {\draw[thick] (0,-0.1) -- (0,0.1);}}, postaction={decorate}, bend left=10] (A3) to (A4);
\draw[thick, decoration={markings, mark=at position 0.47 with {\draw[thick] (0,-0.1) -- (0,0.1);}}, decoration={markings, mark=at position 0.53 with {\draw[thick] (0,-0.1) -- (0,0.1);}}, postaction={decorate}, bend left=10] (A4) to (A5);
\draw[thick, decoration={markings, mark=at position 0.47 with {\draw[thick] (0,-0.1) -- (0,0.1);}}, decoration={markings, mark=at position 0.53 with {\draw[thick] (0,-0.1) -- (0,0.1);}}, postaction={decorate}, bend left=10] (A5) to (A6);
\draw[thick, decoration={markings, mark=at position 0.47 with {\draw[thick] (0,-0.1) -- (0,0.1);}}, decoration={markings, mark=at position 0.53 with {\draw[thick] (0,-0.1) -- (0,0.1);}}, postaction={decorate}, bend left=10] (A6) to (A0);
\draw[red, thick] (B0) to[bend left=10] (B1) to[bend left=10] (B2) to[bend left=10] (B3) to[bend left=10] (B4) to[bend left=10] (B5) to[bend left=10] (B6) to[bend left=10] (B7) to[bend left=10] (B8) to[bend left=10] (B9) to[bend left=10] cycle;
\fill[opacity=0.2, purple] (B0) to[bend left=10] (B1) to[bend left=10] (B2) to[bend left=10] (B3) to[bend left=10] (B4) to[bend left=10] (B5) to[bend left=10] (B6) to[bend left=10] (B7) to[bend left=10] (B8) to[bend left=10] (B9) to[bend left=10] cycle;
\draw[red, thick] (B0big) to[bend left=10] (B1big) to[bend left=10] (B2big) to[bend left=10] (B3big) to[bend left=10] (B4big) to[bend left=10] (B5big) to[bend left=10] (B6big) to[bend left=10] (B7big) to[bend left=10] (B8big) to[bend left=10] (B9big) to[bend left=10] cycle;
\node at (0,0) {$D'$};
\node at (1,-1) {$D$};
\end{tikzpicture} \hfill
\begin{tikzpicture}[scale=1.5]
\coordinate (A0) at (0.873132,-0.0640563);
\coordinate (A1) at (0.540645,0.818804);
\coordinate (A2) at (-0.054978,0.828725);
\coordinate (A3) at (-0.820175,0.542033);
\coordinate (A4) at (-0.899071,-0.384612);
\coordinate (A0big) at (1.05*0.873132+.02,-1.05*0.0640563-.05);
\coordinate (A1big) at (1.05*0.540645,1.05*0.818804);
\coordinate (A2big) at (-1.05*0.054978,1.05*0.828725);
\coordinate (A3big) at (-1.05*0.820175,1.05*0.542033);
\coordinate (A4big) at (-1.05*0.899071,-1.05*0.384612-.05);
\coordinate (A5big) at (-1.05*0.121119,-1.05*1.15943);
\coordinate (A6big) at (1.05*0.777393,-1.05*0.717881);
\coordinate (C5) at (-0.27,-0.05);
\coordinate (C6) at (0.02,-0.14);
\coordinate (C7) at (0.52,.10);
\coordinate (C5big) at (-0.27,-0.09);
\coordinate (C6big) at (0.02,-0.18);
\coordinate (C7big) at (0.52,.06);
\draw[thick, blue] (A0) to[bend left=10] (A1) to[bend left=10] (A2) to[bend left=10] (A3) to[bend left=10] (A4);
\fill[opacity=0.2, purple] (A0) to[bend left=10] (A1) to[bend left=10] (A2) to[bend left=10] (A3) to[bend left=10] (A4) to[bend left=10] (C5) to[bend left=10] (C6) to[bend left=10] (C7) to[bend left=10] cycle;
\draw[thick, red, decoration={markings, mark=at position 0.47 with {\draw[thick, black] (0,-0.1) -- (0,0.1);}}, decoration={markings, mark=at position 0.53 with {\draw[thick, black] (0,-0.1) -- (0,0.1);}}, postaction={decorate}, bend left=10] (A0big) to (A1big); 
\draw[thick, red, decoration={markings, mark=at position 0.47 with {\draw[thick, black] (0,-0.1) -- (0,0.1);}}, decoration={markings, mark=at position 0.53 with {\draw[thick, black] (0,-0.1) -- (0,0.1);}}, postaction={decorate}, bend left=10] (A1big) to (A2big); 
\draw[thick, red, decoration={markings, mark=at position 0.47 with {\draw[thick, black] (0,-0.1) -- (0,0.1);}}, decoration={markings, mark=at position 0.53 with {\draw[thick, black] (0,-0.1) -- (0,0.1);}}, postaction={decorate}, bend left=10] (A2big) to (A3big); 
\draw[thick, red, decoration={markings, mark=at position 0.47 with {\draw[thick, black] (0,-0.1) -- (0,0.1);}}, decoration={markings, mark=at position 0.53 with {\draw[thick, black] (0,-0.1) -- (0,0.1);}}, postaction={decorate}, bend left=10] (A3big) to (A4big); 
\draw[thick, decoration={markings, mark=at position 0.47 with {\draw[thick, black] (0,-0.1) -- (0,0.1);}}, decoration={markings, mark=at position 0.53 with {\draw[thick, black] (0,-0.1) -- (0,0.1);}}, postaction={decorate}, bend left=10] (A4big) to (A5big); 
\draw[thick, decoration={markings, mark=at position 0.47 with {\draw[thick, black] (0,-0.1) -- (0,0.1);}}, decoration={markings, mark=at position 0.53 with {\draw[thick, black] (0,-0.1) -- (0,0.1);}}, postaction={decorate}, bend left=10] (A5big) to (A6big); 
\draw[thick, decoration={markings, mark=at position 0.47 with {\draw[thick, black] (0,-0.1) -- (0,0.1);}}, decoration={markings, mark=at position 0.53 with {\draw[thick, black] (0,-0.1) -- (0,0.1);}}, postaction={decorate}, bend left=10] (A6big) to (A0big); 
\draw[thick, red] (A4) to[bend left=10] (C5) to[bend left=10] (C6) to[bend left=10] (C7) to[bend left=10] (A0);
\draw[thick, red] (A4big) to[bend left=10] (C5big) to[bend left=10] (C6big) to[bend left=10] (C7big) to[bend left=5] (A0big);
\node at (0,.3) {$D'$};
\node at (1,-1) {$D$};
\end{tikzpicture} \hfill $\,$

\caption{From left: case 1, case 2, case 3.}
\label{SW1Disc} 
\end{figure}
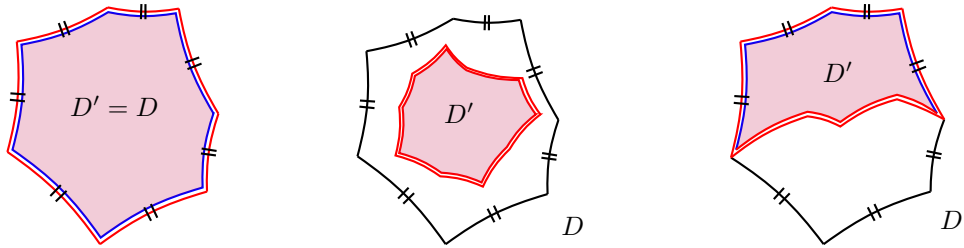

We now show that the total angle deficit in $D'$ cannot equal $1$. At any interior vertex, the angle deficit is bounded above by $0$. All positive contributions to the Euler characteristic must come from boundary vertices along $\gamma$ or $\gamma'$.

At a boundary vertex $V$ along $\gamma'$, where $\gamma'$ crosses between opposite red sectors, the angle deficit is bounded above by $-1/6$. At a boundary vertex $V$ along $\gamma$, where two full long edges meet, there must be a blue trapezoid in $D'$ along each edge (not a slim triangle). If there are any additional tiles in $D'$ touching $V$, then the angle deficit is bounded above by $0$. If not, then it is $1/6$. However, in the latter case, the positive angle deficit at $V$ is offset by a deficit less than or equal to $-1/3$ at the other vertex $U$ shared by the two tiles. This holds whether $U$ is internal or at the boundary. If $U$ plays this role for multiple pairs of boundary tiles, then its angle deficit will be less than or equal to $-1/3$ of the number of pairs. The bounds described here are illustrated in \figref{Bounds}:

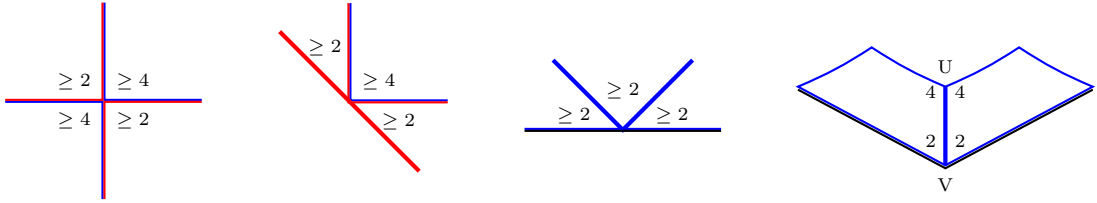
\begin{figure}[ht] 
\, \hfill
\begin{tikzpicture}[scale=1.3]
    \draw[thick, blue] (.01, 1) -- (.01, .01) -- (1, .01);
    \draw[thick, blue] (-.01, -1) -- (-.01, -.01) -- (-1, -.01);
    \draw[thick, red] (.01, -1) -- (.01, -.01) -- (1, -.01);
    \draw[thick, red] (-.01, 1) -- (-.01, .01) -- (-1, .01);
    \node[font=\scriptsize] at (.3, .2) {$\geq 4$};
    \node[font=\scriptsize] at (-.3, -.2) {$\geq 4$};
    \node[font=\scriptsize] at (-.3, .2) {$\geq 2$};
    \node[font=\scriptsize] at (.3, -.2) {$\geq 2$};
\end{tikzpicture} \hfill
\begin{tikzpicture}[scale=1.3]
    \draw[white] (0, -1) -- (0,0);
    \draw[thick, blue] (.01, 1) -- (.01, .01) -- (1, .01);
    \draw[thick, red] (-.01, 1) -- (-.01, -.01) -- (1, -.01);
    \draw[ultra thick, red] (.71, -.71) -- (-.71, .71);
    \node[font=\scriptsize] at (.3, .2) {$\geq 4$};
    \node[font=\scriptsize] at (.5, -.2) {$\geq 2$};
    \node[font=\scriptsize] at (-.25, .55) {$\geq 2$};
\end{tikzpicture} \hfill
\begin{tikzpicture}[scale=1.3]
    \draw[white] (0, -.71) -- (0,0);
    \draw[ultra thick, blue] (-.71, .71) -- (0, 0) -- (.71, .71);
    \draw[thick, blue] (-1, .01) -- (1, .01);
    \draw[thick] (-1, -.01) -- (1, -.01);
    \node[font=\scriptsize] at (-.5, .15) {$\geq 2$};
    \node[font=\scriptsize] at (.5, .15) {$\geq 2$};
    \node[font=\scriptsize] at (0, .4) {$\geq 2$};
\end{tikzpicture} \hfill
\begin{tikzpicture}[scale=1.3]
    \draw[thick, blue] (-1.5, .8) -- (0, 0) -- (1.5, .8) to[bend left=5] (.75, 1.2) to[bend left=5] (0, .8) to[bend left=5] (-.75, 1.2) to[bend left=5] cycle;
    \draw[ultra thick, blue] (0,0) -- (0, .8);
    \draw[thick] (-1.5, .77) -- (0, -.03) -- (1.5, .77);
    \node[font=\scriptsize] at (.15, .25) {$2$};
    \node[font=\scriptsize] at (-.15, .25) {$2$};
    \node[font=\scriptsize] at (.15, .75) {$4$};
    \node[font=\scriptsize] at (-.15, .75) {$4$};
    \node[font=\scriptsize] at (0, -.2) {V};
    \node[font=\scriptsize] at (0, 1) {U};
\end{tikzpicture} \hfill \,
\caption{From left: interior vertex; vertex along $\gamma'$; vertex along $\gamma$ with more than two tiles; vertex along $\gamma$ with two tiles.}
\label{Bounds} 
\end{figure}

Up to this point, the total angle deficit is bounded above by $0$, and the only vertices not yet accounted for are special ones depending on the case. In case 1, there are no special vertices to consider. In case 2, we only have one special vertex $V$, where the path $\gamma'$ does not cross between opposite red sectors. Since the angle deficit at $V$ is less than $1/3$, the total angle deficit is less than $1/3$. 

In case 3, the special vertices are at or adjacent to the endpoints of $\gamma'$. If an endpoint is not on a slim triangle, then this endpoint has an angle deficit bounded above by $1/6$. If an endpoint is on a slim triangle then it has an angle deficit of $1/4$, and its neighbor on half of the long edge has an angle deficit bounded above by $1/4$. These bounds are illustrated in \figref{BoundEndpoints}. If both endpoints of $\gamma'$ are on slim triangles, then the total angle deficit of the four special vertices could be $1$, but it would be offset by a deficit less than or equal to $-1/6$ at a vertex along $\gamma'$. Finally, if the endpoints of $\gamma'$ are neighbors, or share the same neighbor, this only reduces the total angle deficit. 

\begin{figure}[ht] 
\, \hfill
\begin{tikzpicture}[scale=1.5]
\draw[thick, red] (-.71, .71) -- (0,0) -- (.71, .71);
\draw[thick, black] (-.71, .68) -- (0,-.03) -- (.71, .68);
\node[font=\scriptsize] at (0, .3) {$\geq 2$}; 
\end{tikzpicture} \hfill
\begin{tikzpicture}[scale=1.5]
\draw[thick, red] (-.71, .71) -- (0,0);
\draw[thick, red] (0, 0) -- (0,1);
\draw[thick, blue] (.71, .71) -- (0, 0) -- (1,0);
\draw[thick, black] (-.71, .68) -- (0,-.03);
\draw[thick, black] (0,-.02) -- (1, -.02);
\node[font=\scriptsize] at (-.1, .3) {$2$}; 
\node[font=\scriptsize] at (.3, .1) {$2$}; 
\node[font=\scriptsize] at (.2, .45) {$\geq 0$}; 
\end{tikzpicture} \hfill
\begin{tikzpicture}[scale=1.5]
\draw[thick, red] (0, 1) -- (0,0) -- (1,0) to[bend left=5] (0, .5);
\draw[thick, blue] (1.71, .71) to[bend right=5] (1,0) -- (2, 0);
\draw[thick, black] (-.02, 1) -- (-.02,-.02) -- (2, -.02);
\node[font=\scriptsize] at (.1, .1) {$3$}; 
\node[font=\scriptsize] at (.55, .1) {$1$}; 
\node[font=\scriptsize] at (1.3, .1) {$2$}; 
\node[font=\scriptsize] at (.95, .25) {$\geq 0$}; 
\end{tikzpicture} \hfill \,
\caption{From left: case 2, special vertex; case 3, special vertex not on a slim triangle; case 3, special vertex on a slim triangle and neighbor.}
\label{BoundEndpoints}
\end{figure}
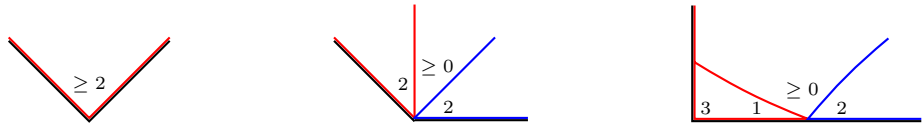

In all cases, we see that the total angle deficit in $D'$ is less than $1$, contradicting \lemref{AngleDeficit}. Thus no tiling of $D'$ exists, and no tiling of $S^2$ exists. We conclude that an ideal, right-angled  $\cP$-polyhedron cannot exist. 

\end{proof}

\section{Constructing ideal, right-angled  \texorpdfstring{$\cP$}{P}-Polyhedra} \label{Unfolding}

We have ruled out 45 of the 49 Scharlau-Walhorn polyhedra, leaving 
the cases
$\cP=$ \SW{2}, \SW{4}, \SW{12} and \SW{13}. 
These four polyhedra do admit right-angled ideal $\cP$-polyhedra. In this section, we describe all such $\cP$-polyhedra. In each case, the description is in terms
of a right-angled $\cP$-polyhedron that is easier
to understand than~$\cP$: the triangular bipyramid $\frac14\cO$, the square antiprism~$\cA$, 
and the rhombicuboctahedron~$\cR$ (twice).
Namely, every ideal right-angled $\cP$-polyhedron can be tiled (but not necessarily reflectively) by one of these.
In previous sections, we used 
``tiled by'' as shorthand for ``reflectively tiled by'', but here we will be explicit when we mean
the latter.  
The phrase ``is tiled by$\ldots$'' 
will mean ``is a union of isometric copies, with disjoint interiors, of$\ldots$''

\subsection{The case \SW{12}}
The polyhedron $\cP={}$\SW{12} is the simplest to understand:
cutting
an ideal rhombicuboctahedron along its planes of symmetry yields $48$ copies of~$\cP$.
(Recall that the rhombicuboctahedron is the rectification of the cuboctahedron, which is
the rectification of the cube.)
To see this, 
visualize $\cP$ with the 1-4-6 vertex at the center of the ball model of hyperbolic space.
The reflections across faces 1, 4, and 6
generate a finite reflection group of type $B_3$, which is the common
symmetry group of the regular octahedron, cuboctahedron, and rhombicuboctahedron. 
One can check that the union of the $B_3$-images of~$\cP$ is the
rhombicuboctahedron,
with the union of the images of face $2$, $3$, resp.\ $5$, being a $B_3$-orbit
of faces.
This proves the first part of the next proposition.

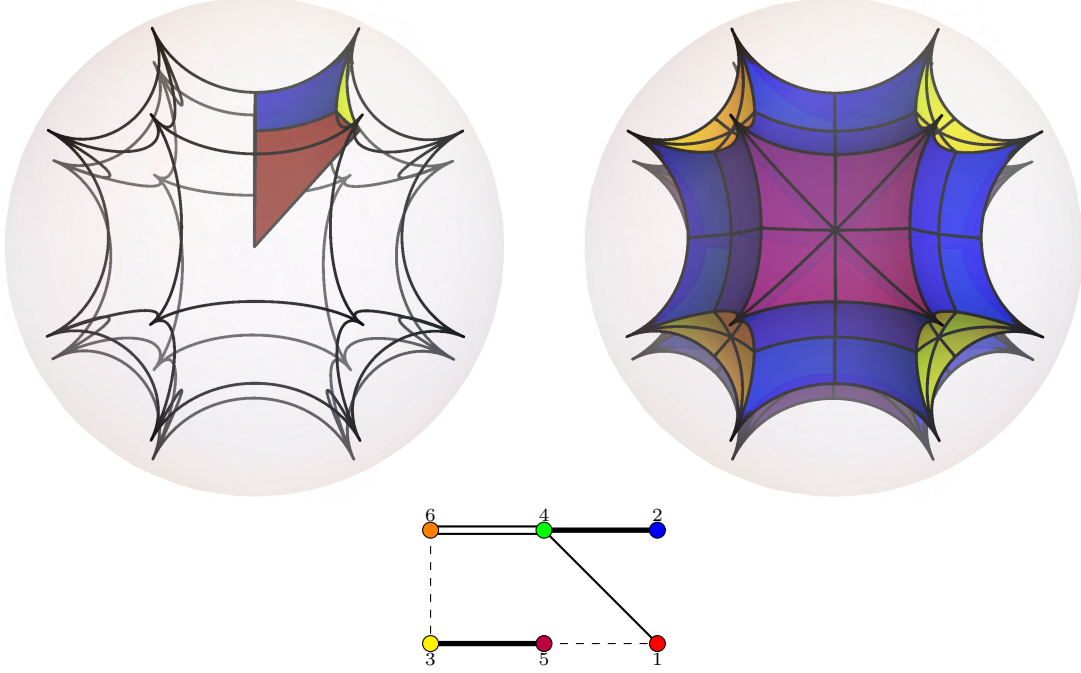
\begin{figure}[ht] 
\, \hfill \includegraphics[width=.4\textwidth]{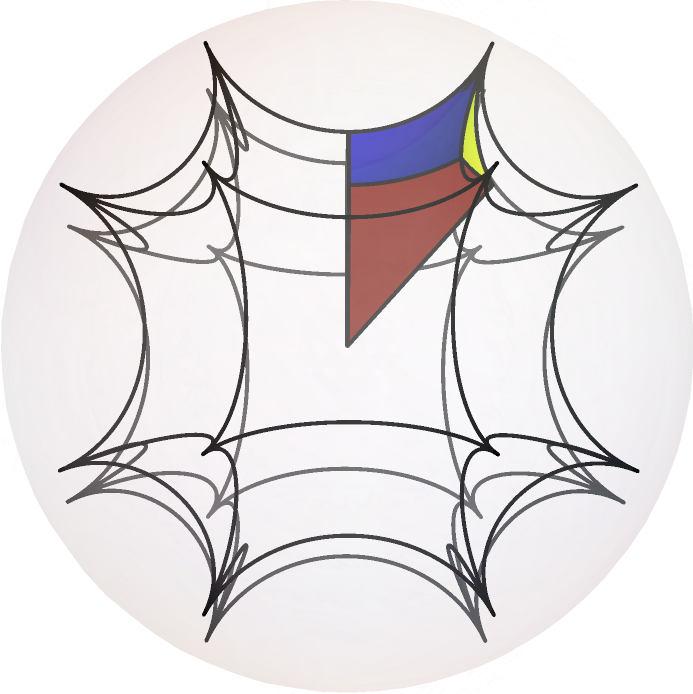} \hfill \includegraphics[width=.4\textwidth]{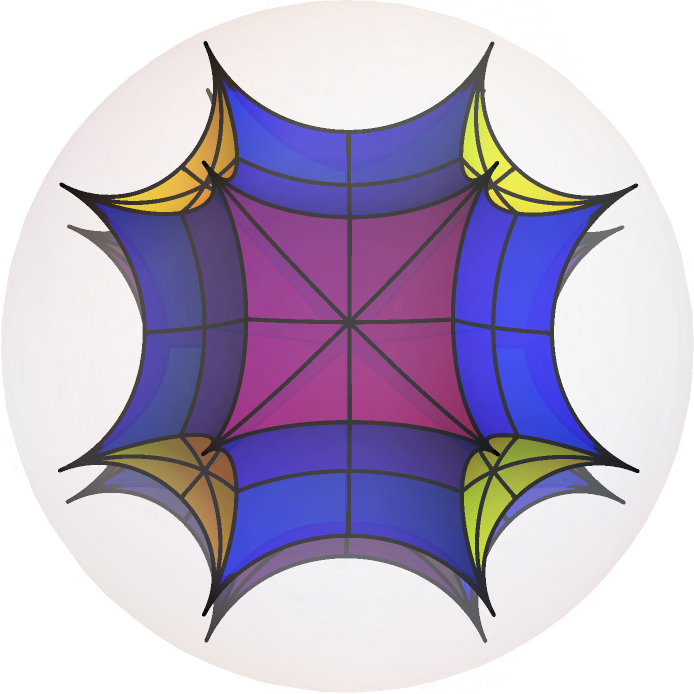} \hfill \,

\begin{tikzpicture}[scale=1.5]
\coordinate (2) at (1,1);
\coordinate (6) at (-1,1);
\coordinate (5) at (0,0);
\coordinate (4) at (0,1);
\coordinate (1) at (1,0);
\coordinate (3) at (-1,0);

\draw[thick] (4) -- (1);
\draw[thick, line width=2pt] (4) -- (2);
\draw[dashed] (5) -- (1);
\draw[thick, line width=2pt] (5) -- (3);
\draw[dashed] (6) -- (3);
\draw[thick, double distance = 2pt] (6) -- (4);

\filldraw[fill=red] (1) circle (2pt) node[below] {\scriptsize 1};
\filldraw[fill=blue] (2) circle (2pt) node[above] {\scriptsize 2};
\filldraw[fill=yellow] (3) circle (2pt) node[below] {\scriptsize 3};
\filldraw[fill=green] (4) circle (2pt) node[above] {\scriptsize 4};
\filldraw[fill=purple] (5) circle (2pt) node[below] {\scriptsize 5};
\filldraw[fill=orange] (6) circle (2pt) node[above] {\scriptsize 6};

\end{tikzpicture} 
\caption{From left: polyhedron \SW{12}; 48 copies of polyhedron  \SW{12} glued around the 1-4-6 vertex to form ideal, right-angled  rhombicuboctahedron; Coxeter diagram with face colors labeled.}
\label{FigRhombicuboctahedron1}
\end{figure}

\begin{prop} \label{Rhombicuboctahedron1}
Let $\cP$ be the polyhedron \SW{12}. 
    Let $\cR$ be the 
    union of the  $48$ images of $\cP$ that contain
    the 1-4-6 vertex.   This is an ideal, right-angled rhombicuboctahedron,
    and every  ideal, right-angled  $\cP$-polyhedron is also an $\cR$-polyhedron.  
\end{prop}

\begin{proof} 
Face 6 is compact and is its own equivalence class. Faces 1 and 4 form
    another equivalence class, whose fundamental domain is a pentagon with a single ideal vertex. \thmref{IdealizablePolygons} shows that these
    faces are not idealizable. They intersect at a finite vertex, and 
    their reflections generate the reflection group~$B_3$. 
By \propref{BadFaceOrbit}, every ideal, right-angled  $\cP$-polyhedron contains the images
    of $\cP$ under this group.  Their union
    is the ideal, right-angled  rhombicuboctahedron $\cR$. 
    By \propref{FiniteBadFaceOrbit2}, every ideal, right-angled  $\cP$-polyhedron is a $\cR$-polyhedron. 
\end{proof}

\subsection{The case $\cP={}$\SW{4}}
The next example, polyhedron \SW{4}, is similar but requires multiple steps. 
As in \propref{Rhombicuboctahedron1}, we consider a finite vertex whose incident faces generate a reflection group of type $B_3$: in this case, the 1-4-5 vertex. 
The union of the $B_3$-images of~$\cP$ is an ideal, right-angled
cuboctahedron~$\cC$. 
Unlike the previous case, $\cC$ is not 
the smallest possible such $\cP$-polyhedron.  One can choose a plane of symmetry of~$\cC$, uniquely up to conjugacy,
which cuts it into two ideal polyhedra.  Each is the union of $24$ copies
of~$\cP$.  Pick one and call it $\cA$; it is an
ideal, right-angled square antiprism.  

\begin{prop} \label{Antiprism}
Let $\cP$ be the Scharlau-Walhorn polyhedron \SW{4}.  Then 
    the ideal, right-angled square antiprism~$\cA$ is a $\cP$-polyhedron.
    Furthermore,  every ideal, right-angled
$\cP$-polyhedron~$\cQ$ has a canonical tiling by copies of 
    $\cA$ and the ideal cuboctahedron~$\cC$.  In particular,
    it is tiled (possibly non-canonically) by~$\cA$.
\end{prop}

The canonical tiling statement is more complicated than the 
non-canonical tiling statement,
but it gives a 
fairly explicit way to construct all 
possible~$\cQ$.

\begin{proof}
    By ``canonical'' we mean that we will define the  
    tiling  entirely in terms of the
reflective tiling of~$\cQ$ by~$\cP$.  Given this,
    one can cut each cuboctahedron into two antiprisms,
along any of six planes, to get a non-canonical tiling of~$\cQ$ by~$\cA$.
First we construct a tiling by $\cA$, $\cC$ and $\frac12\cA$, where the last means the
polyhedron obtained by cutting $\cA$ along any plane of symmetry. 
    (This is a right-angled $\cP$-polyhedron with exactly two finite vertices;
    these are the endpoints of a compact edge.)
    Then we show that the
$\frac12\cA$ tiles assemble themselves into copies of~$\cA$.

    Suppose $v\in\cQ$ is a $\Gamma(\cP)$-translate of the 1-4-5 vertex of~$\cP$.
    We now describe the copies of $\cC$, $\cA$ and $\frac12\cA$ that tile~$\cQ$.
    If $v$ lies in the interior of $\cQ$, then $\cQ$ contains 
    the~$48$ translates of~$\cP$ that contain~$v$; we take their union (a cuboctahedron)
    as one of our tiles.  One of the~$48$, and the cuboctahedron,
    appear on the left in Figure~\ref{FigAntiprismA}, with $v$ at the center.
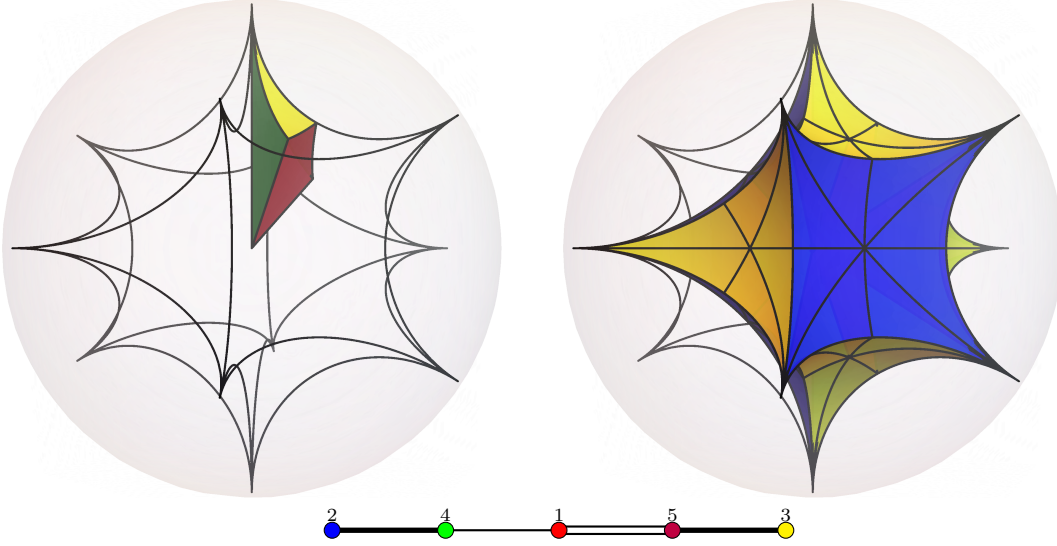
\begin{figure}[ht] 
\, \hfill
    \includegraphics[width=.4\textwidth]{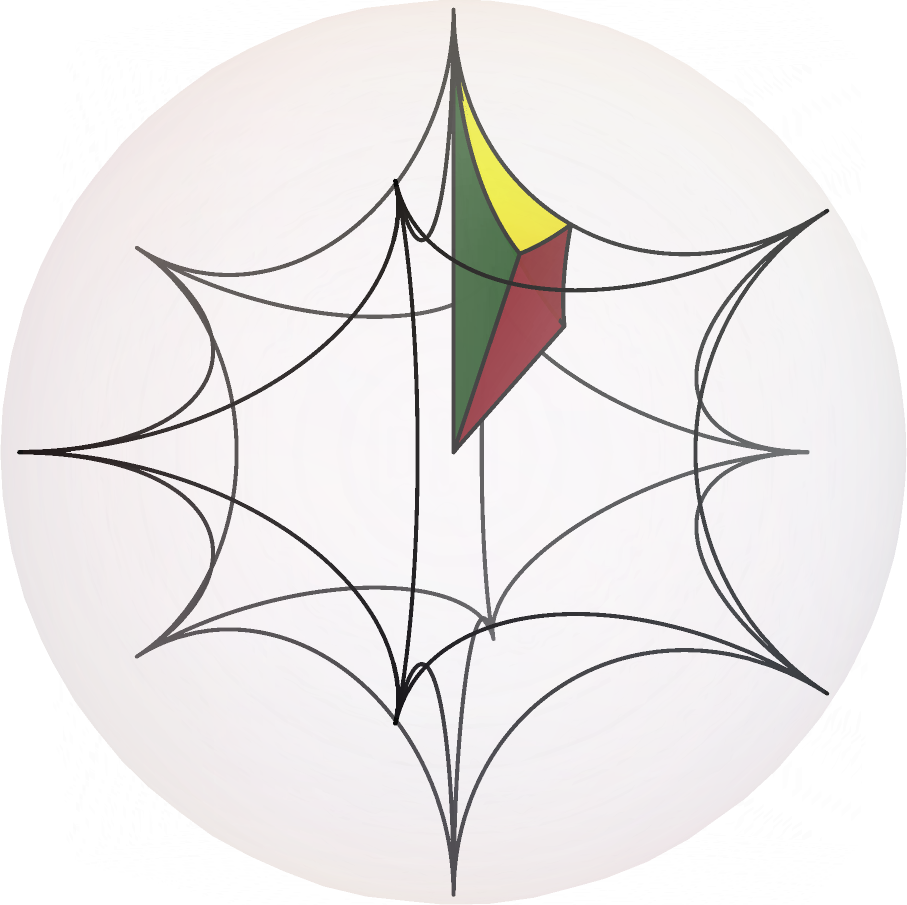} 
\hfill 
    \includegraphics[width=.4\textwidth]{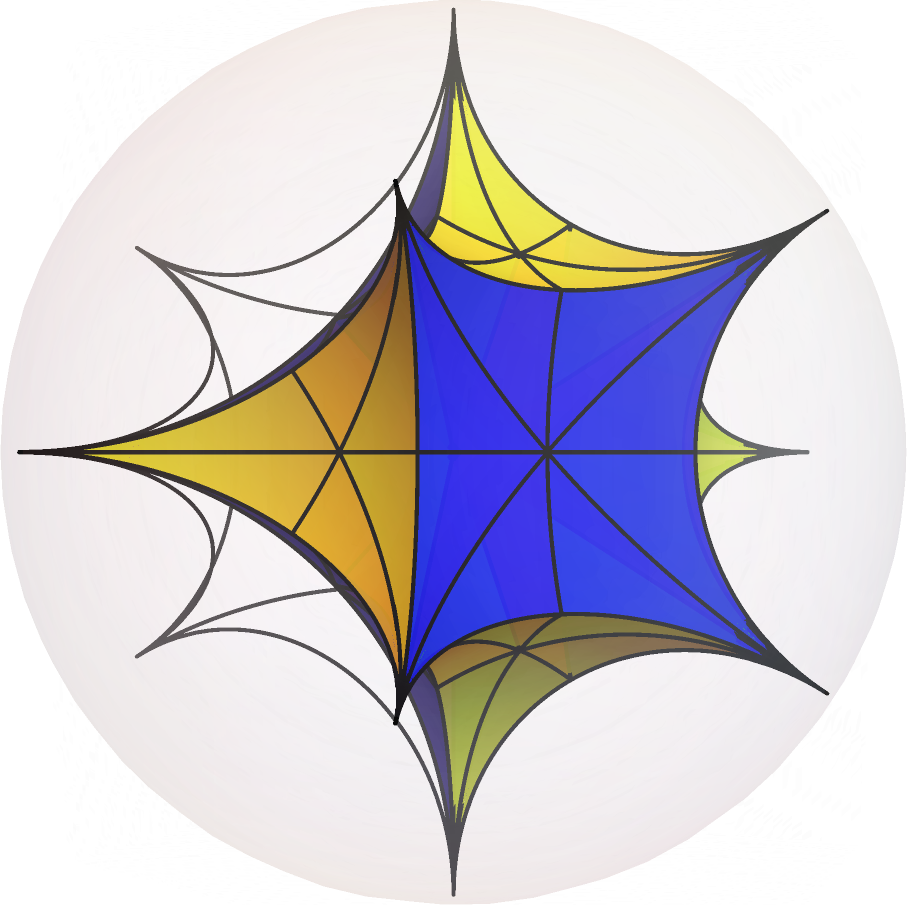} \hfill 
    \hfill \,

\begin{tikzpicture} [scale=1.5]
\coordinate (1) at (2,0);
\coordinate (2) at (0,0);
\coordinate (3) at (4,0);
\coordinate (4) at (1,0);
\coordinate (5) at (3,0);

\draw[thick] (4) -- (1);
\draw[thick, line width=2pt] (4) -- (2);
\draw[thick, double distance = 2pt] (5) -- (1);
\draw[thick, line width=2pt] (5) -- (3);

\filldraw[fill=red] (1) circle (2pt) node[above] {\scriptsize 1};
\filldraw[fill=blue] (2) circle (2pt) node[above] {\scriptsize 2};
\filldraw[fill=yellow] (3) circle (2pt) node[above] {\scriptsize 3};
\filldraw[fill=green] (4) circle (2pt) node[above] {\scriptsize 4};
\filldraw[fill=purple] (5) circle (2pt) node[above] {\scriptsize 5};
\end{tikzpicture}
\caption{From left: polyhedron \SW{4} with its 1-4-5 vertex~$v$ at the center
    of a cuboctahedron; $24$~translates of~\SW{4}, whose union is the
    ideal, right-angled square antiprism~$\cA$; Coxeter diagram with face colors labeled.}
    \label{FigAntiprismA}
\end{figure}
    Next, suppose $v$ lies in the interior of a face of~$\cQ$.
    There are two kinds of planes that contain~$v$:
    those tiled by copies of face~$5$ and those tiled by copies of faces $1$ and~$4$.
    The former cut the cuboctahedron along an ideal square, and the latter along
    a $(2,2,\infty,2,2,\infty)$ hexagon.
    If $v$ lies in the interior of a face of~$\cQ$ of type~$5$, then $\cQ$
    contains the~$24$ translates of~$\cP$ that contain~$v$ and lie on one side
    of that plane.  See the right side of Figure~\ref{FigAntiprismA}.
    Their union is a copy of~$\cA$, which we take as one of our tiles.
    We will return to the other case after studying the case
    that $v$ lies in an edge of~$\cQ$.  

    Suppose $v$ lies in an edge~$E$ of~$\cQ$.  
    In the reflective
    tiling of~$\bH^3$ by $\cP$, there are three kinds 
    of edges incident at~$v$, of types 
    1-4, 4-5 and 1-5. 
    Each 1-4 edge has dihedral angle $\pi/3$, so it cannot lie in
    an edge of any right-angled $\cP$-polyhedron. 
    Each 1-5 edge is compact, and orthogonal 
    to mirrors of~$\Gamma(\cP)$ at 
    its endpoints, so it cannot lie in an edge of any ideal $\cP$-polyhedron. 
    Therefore $E$ contains the 4-5 edge of a chamber containing~$v$.
    The left of Figure~\ref{FigAntiprismA} shows this chamber, with
    the 4-5 edge appearing vertical. 
    It follows that $E$ is the continuation of this edge, i.e., the
    vertical diameter
    of the cuboctahedron.  
    Note that $E$ lies in exactly two mirrors of~$\Gamma(\cP)$,
    and these cut the cuboctahedron into four sectors.  
    Also, $\cQ$ 
    is a convex union of chambers, and contains this
    chamber and diameter.  It follows that 
    $\cQ$ contains
    one of these sectors, pictured on the left in Figure~\ref{FigFoo}.
    We recognize it as
    $\frac12\cA$, because doubling it across its red/green face 
    yields the copy of~$\cA$ in Figure~\ref{FigAntiprismA}.  
    We take this copy of $\frac12\cA$ as 
    one of our tiles.

\begin{figure}[ht] 
\, \hfill 
    \includegraphics[width=.4\textwidth]{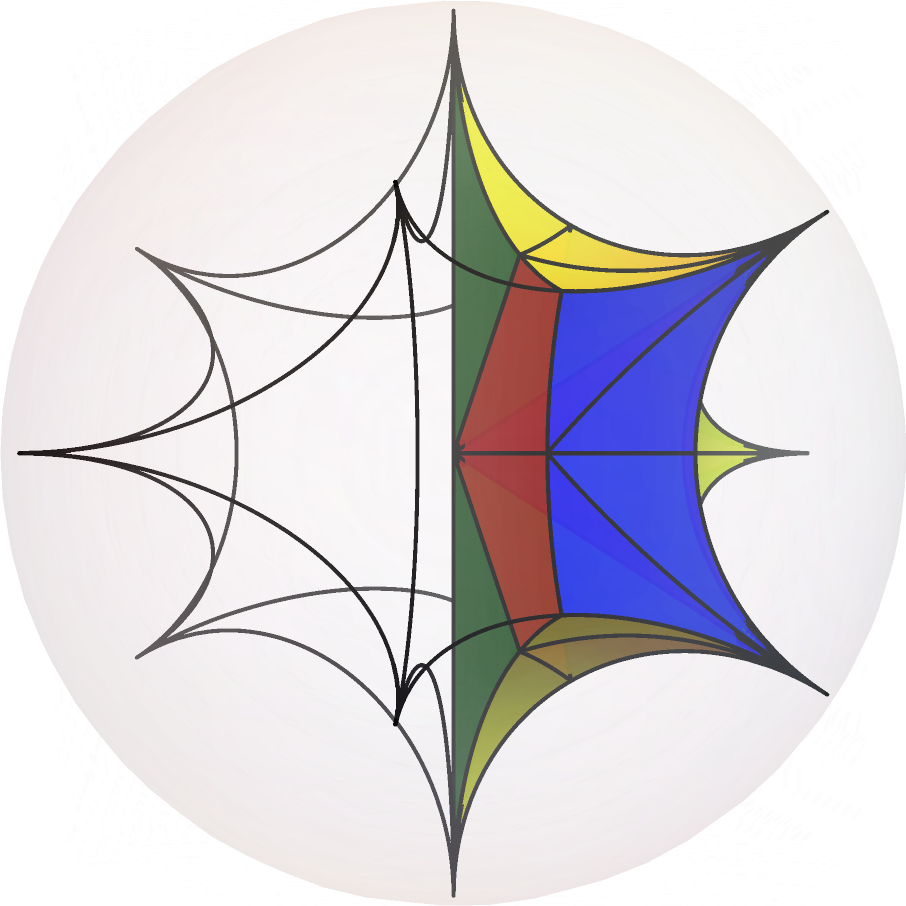} 
    \hfill 
    \includegraphics[width=.4\textwidth]{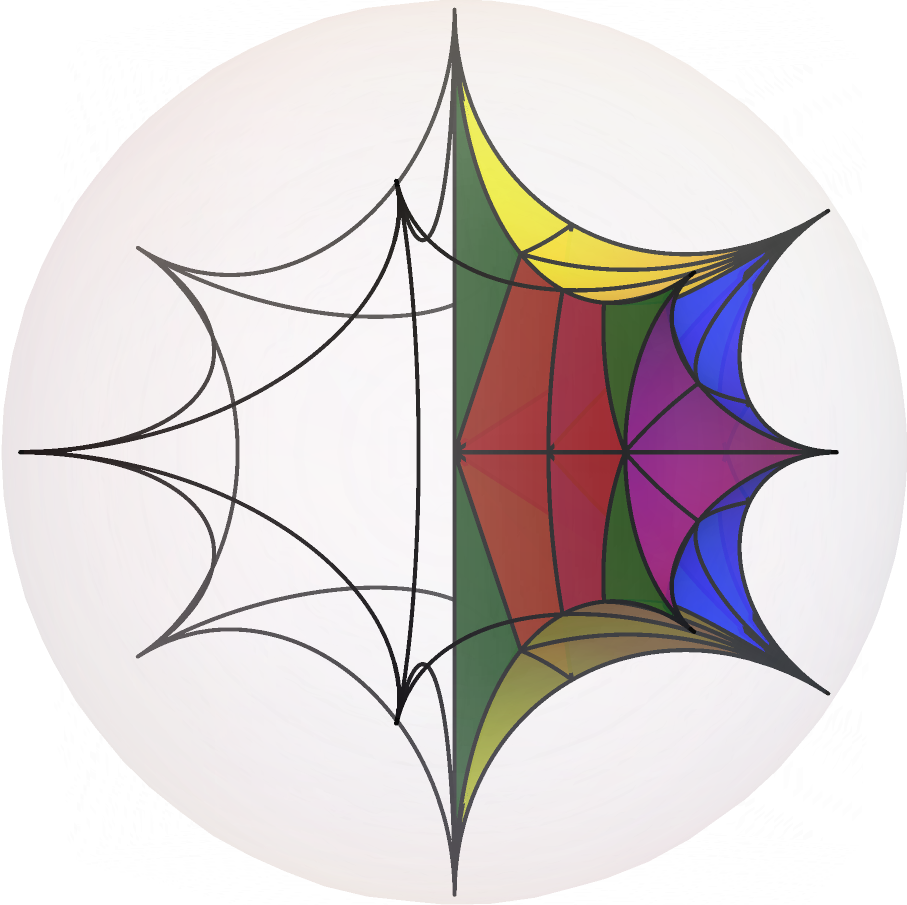} 
    \hfill \,
\caption{From left: the polyhedron~$\frac12\cA$; its double~$\cA'$ across its front
    (blue) face, which happens to be isometric to~$\cA$.}
\label{FigFoo}
\end{figure}

    Finally, suppose $v$ lies in the interior of a face of~$\cQ$ tiled
    by copies of faces $1$ and~$4$.  Then this face of~$\cQ$ meets the
    cuboctahedron in a $(2,2,\infty,2,2,\infty)$ hexagon.  Consider the
    line joining its ideal vertices.  We call it~$E$, for ease of comparison
    with the previous paragraph: by applying an element of $B_3$, we may
    suppose $E$ is the vertical diameter in Figure~\ref{FigFoo}.
    We saw that 
    there are two mirrors of~$\Gamma(\cP)$ that contain~$E$, and
    that they cut the cuboctahedron into four copies of $\frac12\cA$.
    Two of these lie in~$\cQ$.  After applying an element of $B_3$ that
    fixes~$E$, one of these appears on the left in Figure~\ref{FigFoo},
    with its green/red face in $\partial\cQ$.  The other is the
    reflection of this across the other face of~$\frac12\cA$
    that contains~$E$ (on the back).
    We include both these copies of $\frac12\cA$ among our tiles.

    We have now described all the tiles.  By construction, each chamber lying
    in~$\cQ$ lies in exactly one of them, according to how its 1-4-5 vertex
    lies in~$\cQ$.  This establishes that it is a tiling.  Finally, fix a tile
    $\frac12\cA$.  It has exactly one compact edge, and this lies in~$\partial\cQ$.
    On the left in Figure~\ref{FigFoo}, this edge is where blue meets red, 
    and the green/red face lies in~$\partial\cQ$.  
    Because $\cQ$ has no compact edges, $\cQ$ must extend across the blue
    face. The compact edge is also an edge in the tile on the other side,
    which must therefore be another copy of $\frac12\cA$.  Write $\cA'$ for
    their union, which appears on the right in Figure~\ref{FigFoo}.
    Surprisingly, $\cA'$ is also an ideal, right-angled square
    antiprism, hence isometric to~$\cA$.  In this manner, the $\frac12\cA$
    tiles in our tiling combine in pairs to form copies of~$\cA$.  
    This finishes the proof.
\end{proof}

\begin{remark}
    \label{RemNonReflectiveCuboctahedral}
    Although canonical, the tiling we have constructed may not be reflective.
    For example, consider the almost-hidden triangular yellow face of~$\frac12\cA$,
    on the right side of the first part of
    Figure~\ref{FigFoo}.  
    On the other side of this
    triangle
    lies a cuboctahedron, one of whose faces lies in the plane of
    the front blue face of~$\frac12\cA$.  Reflecting across this plane
    yields a second cuboctahedron, and extends $\frac12\cA$ to
    $\cA'$ as shown in the  second part of the figure.  The union
    of $\cA'$ and the two cuboctahedra is an ideal, right-angled $\cP$-polyhedron,
    and these are the tiles in its canonical tiling.  
    But each cuboctahedron contains only half of~$F$, so the faces of tiles
    don't even match up.

Also, in contrast to 
    \propref{Rhombicuboctahedron1}, there are
    right-angled  $\cP$-polyhedra that are not $\cA$-polyhedra.
    For example, one can
    glue two copies of $\cA$ together along a triangular face,
    in such a way that they are not reflections of each other, but their union
    is still a $\cP$-polyhedron.  
    The same construction applies with $\frac12\cA$ in place of~$\cA$.
    These complexities suggest that \propref{Rhombicuboctahedron1} is the best one can do.
\end{remark}

\subsection{The case \SW{13}}
Suppose a Coxeter polyhedron $\cP$ has a compact face which is 
orthogonal to its neighbors.  Then that plane slices every $\cP$-polyhedron in a
particularly simple way, allowing 
one to study such polyhedra by how they meet that plane.
This idea
 allows us to describe the ideal, right-angled 
\SW{13}-polyhedra: 

\begin{prop} \label{Rhombicuboctahedron2}
    Let $\cP$ be the polyhedron \SW{13}. Then the ideal, right-angled
    rhombicuboctahedron~$\cR$ (the common rectification of the cuboctahedron 
    and the rhombic dodecahedron) is a $\cP$-polyhedron.  Furthermore,
    every ideal, right-angled  $\cP$-polyhedron is tiled canonically by copies of~$\cR$.
\end{prop}

\begin{remark}
Both \SW{12} and \SW{13} tile~$\cR$ reflectively, 
so they are commensurable.  Here is a more direct way to see this:
gluing three copies of \SW{12} around its 1-4 edge, and gluing three copies of \SW{13} 
    around its 1-5 edge, yield isometric $\cP$-polyhedra.
\end{remark}

\begin{proof}[Proof of Proposition~~\ref{Rhombicuboctahedron2}]
    First suppose $\cP$ is any hyperbolic Coxeter polyhedron,
    with a compact face $F$ which 
    is orthogonal to its neighbors, and let $\cQ$ be an ideal
    $\cP$-polyhedron.  Since the plane of~$F$ is tiled by copies of~$F$,
    it  cannot contain a face of $\cQ$, so it meets the interior of~$\cQ$.
    Furthermore, 
    its intersection~$S$ with~$\cQ$
    is an $F$-polygon.  Because $F$ is orthogonal to its neighboring facets in~$\cP$,
    $S$ is orthogonal to every face of~$\cQ$ that it meets.  It follows
    that every edge of~$\cQ$ that meets $S$ does so orthogonally.
    In particular,
    $S$ is right-angled if~$\cQ$ is.

    Now we specialize to $\cP={}$\SW{13}, with $F$ being face 6.  The other faces that
    it meets are 1, 4 and 5, which it meets orthogonally.  Examining how they meet 
    each other shows that $F$ is a $(2,4,6)$ triangle, hence compact.  The
    previous paragraph shows:
    for every ideal, right-angled $\cP$-polyhedron~$\cQ$, its section~$S$ by any plane
    of type~$6$ is a right-angled polygon tiled by copies of the $(2,4,6)$ triangle. 
    Additionally:
\begin{itemize}
    \item Because face 1 is also compact and meets its neighbors at  angles of the form
        $\pi/(\hbox{even})$, it cannot 
    lie in a  face of~$\cQ$. Thus $S$ cannot have any edges cut out by hyperplanes of type~1. 
\item Because the 1-4 and 1-5 edges are compact and meet faces of~$\cP$
    orthogonally at their endpoints, they cannot lie in edges of~$\cQ$. 
        So no vertex of $S$ lies on any
        edge of type 1-4  or~1-5.
\end{itemize}
\figref{246tiling} shows the plane
    containing~$F$, tiled by $(2,4,6)$ 
    triangles. 
    The lines cut out by planes of type~1 are dashed---these cannot be edges of $S$, so
    $S$ is a union of the $(2,2,2,3)$-quadrilaterals obtained by ignoring the dashed lines.
    By a ``polygon in the tiling'', we mean a convex union of some of these 
    $(2,2,2,3)$-quadrilaterals.
    The lines cut out by planes of type~$4$ resp.~$5$
    are solid, and contain
    the short resp.\ long edges of the quadrilaterals.  
    We have colored red every vertex that lies on a dashed line; these are where
    the plane meets lines of types 1-4 and~1-5.  We showed above that $S$ has
    no such vertex: all vertices of~$S$ are blue.
    The shaded octagon
    is a candidate for~$S$; most of the proof amounts to showing that $S$ is
    tiled by such octagons.

\begin{figure}[ht]
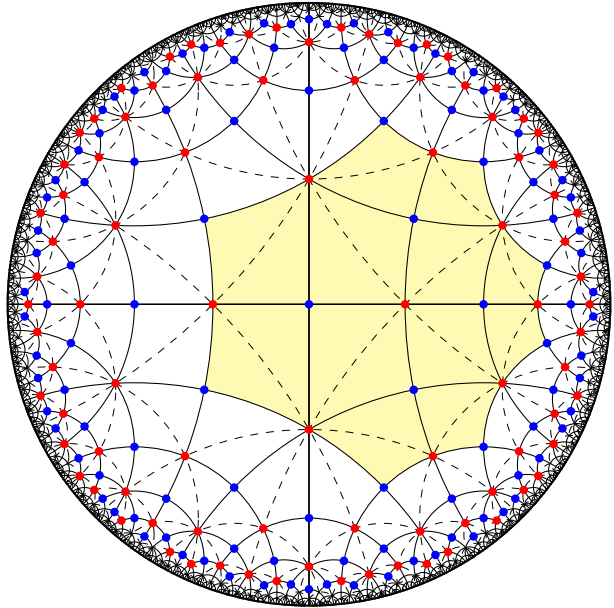
 
\include{HyperbolicTiling}

\caption{The tiling of a plane of type~$6$ by (2,4,6) triangles, with a right-angled octagon shaded yellow.}
\label{246tiling} 
\end{figure}

\begin{lemma}\label{LemNot1RedVertex}
    No polygon in the tiling 
    has exactly one vertex colored red.
\end{lemma}
\begin{proof}
    Suppose for a contradiction that such a polygon $P$ exists. If $P$ has a blue point
    in the interior of an edge, then cut~$P$ along the perpendicular line there. 
    This cuts $P$ into smaller convex polygons. 
    If the other end of the cut is colored blue, or if it is red and interior to
    an edge of~$P$,
    then one of the smaller 
    polygons has exactly one red vertex.  If the other end is the 
    unique red vertex of~$P$, then the same holds for both
    smaller polygons.
    Cutting repeatedly, we may suppose without loss that
    every blue point on $\partial P$ is a vertex of~$P$.
    Some picture-drawing leads to a contradiction.  Namely, 
    starting at the red vertex, the two blue points adjacent to it along the boundary 
    must also be vertices. If the red vertex has a $\pi/2$ or $\pi/3$ angle, 
    this forces the boundary to close up and form a quadrilateral with an additional red vertex.
    If the red vertex has a $2\pi/3$ angle, the boundary closes 
    up after two steps and forms a hexagon with an additional red vertex.
\end{proof}

    Suppose $P$ is a polygon in the tiling, with no red vertices,
    and $p\in\partial P$ is blue but not a vertex of~$P$.  
    If we cut
    along the 
    line orthogonal to $\partial P$ there, then the other end~$q$ of the
    cut cannot be red, by Lemma~\ref{LemNot1RedVertex}. So it is blue.  
    The cut line appears solid in Figure~\ref{246tiling}, and meets the
    interior of~$P$.  
    Since there are only two solid lines through each blue point,
    $q$ cannot be a vertex of~$P$.  Since the solid lines through~$q$
    are orthogonal, $\overline{pq}\perp\partial P$ at~$q$ as well as at~$p$.
    We will call $\overline{pq}$ a blue-blue segment across~$P$.  
    
\begin{lemma} \label{annagon} 
    Suppose $P$ is a polygon in the tiling, with no red vertices.
    Then cutting~$P$, along all blue-blue segments across it,
    decomposes~$P$ into right-angled octagons, each
    an image of the one shown in Figure~\ref{246tiling},
    under an element of the $(2,4,6)$ triangle group.
\end{lemma}

\begin{proof}
    Assume inductively that this is known for all polygons in the tiling,
    that lack red vertices and have
    less area than~$P$.  Suppose first that
    $P$ has no blue-blue segments across it.
    Then every edge of~$P$ has blue endpoints, with
    exactly one red point in between.  This forces $P$ to be
    the shaded octagon,
    up to an element of the triangle group.
    On the other hand, suppose 
    there exists a blue-blue segment $\overline{pq}$ across~$P$.
    Cut $P$ along it, into smaller polygons $P_1$ and~$P_2$ with no
    red vertices.
    We claim that cutting $P$ along all of its blue-blue segments
    yields the same pieces as cutting $P_1$ and~$P_2$ along all of their
    blue-blue segments.  Given this, the lemma follows by induction.

    To prove our claim, we must prove two things.  First, if
    $\overline{p'q'}$ is a blue-blue segment across~$P$, that
    crosses $\overline{pq}$, then its intersections with the $P_i$
    are blue-blue segments across them.  This amounts
    to the claim that $\overline{pq}\cap\overline{p'q'}$ is blue,
    which follows by applying \lemref{LemNot1RedVertex} to either one of the~$P_i$.
    Second, if $\overline{p'q'}$ is a blue-blue segment across (say) $P_1$,
    then it extends to a blue-blue
    segment across~$P$.  This holds because its extension across~$P_2$
    enters~$P_2$ at a blue point, so it also exits at a blue point, again by \lemref{LemNot1RedVertex}.
\end{proof}

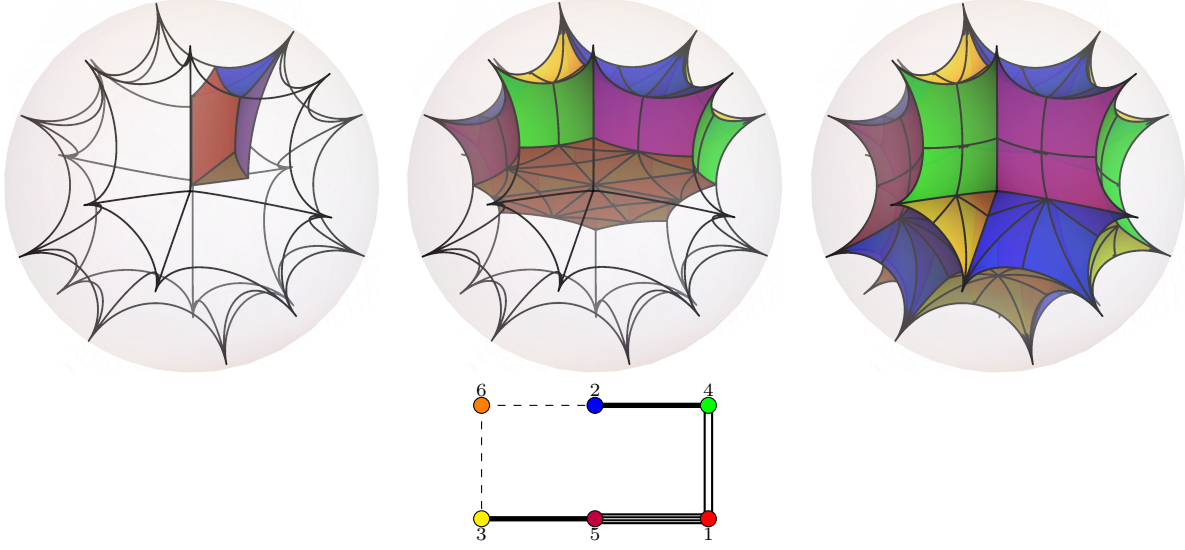
\begin{figure}[ht] 
\, \hfill \includegraphics[width=.3\textwidth]{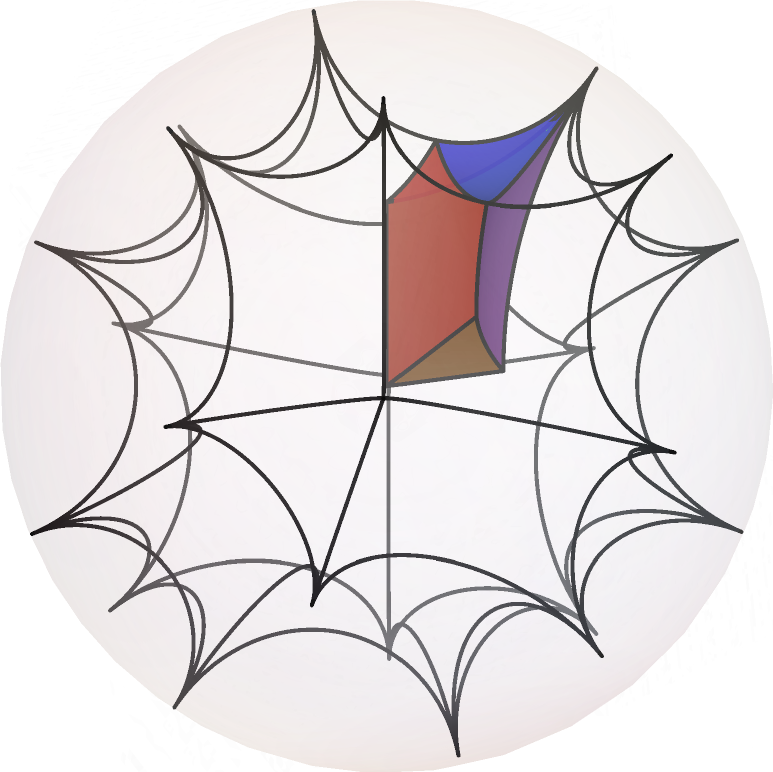} \hfill \includegraphics[width=.3\textwidth]{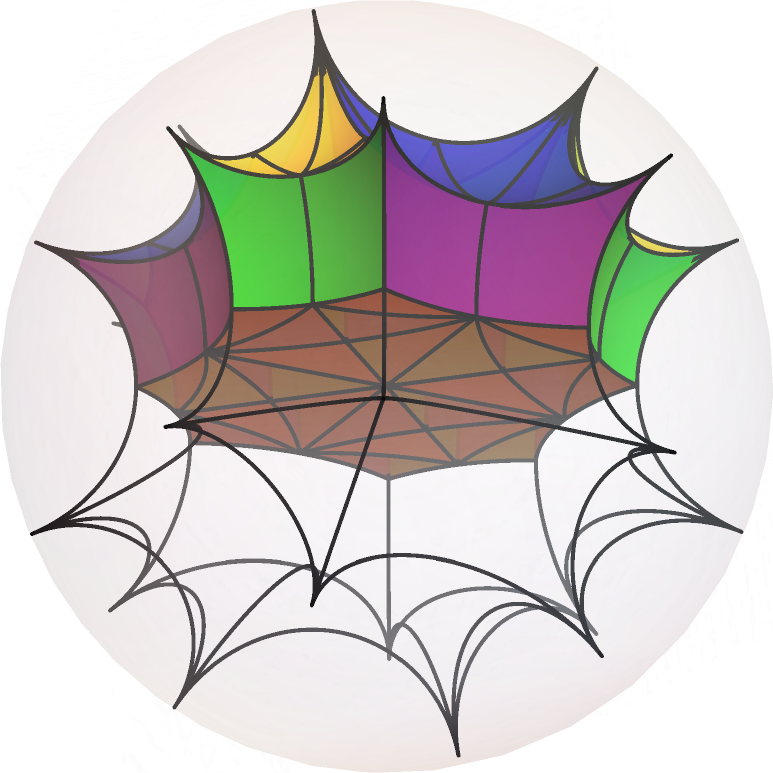} \hfill \includegraphics[width=.3\textwidth]{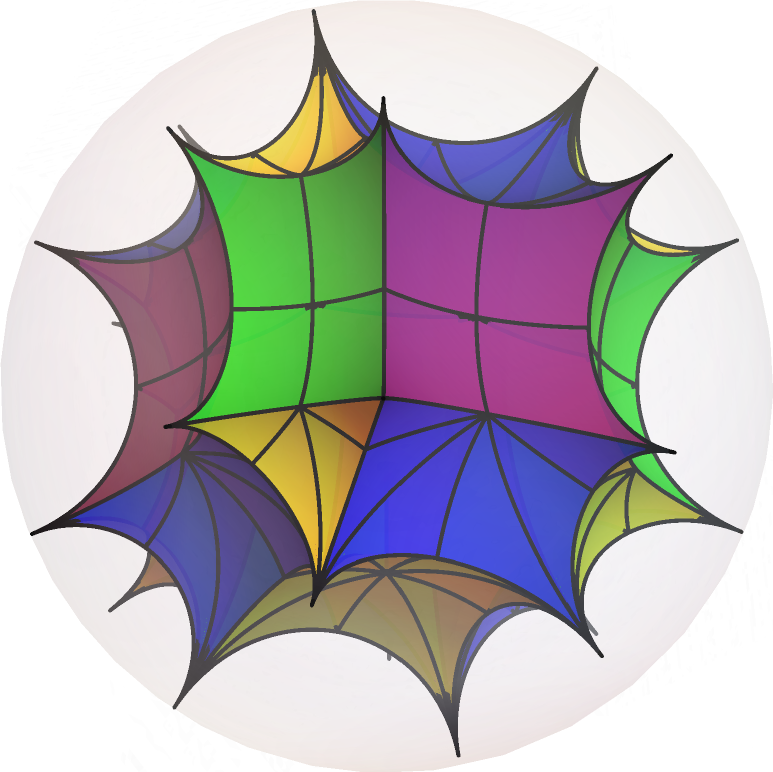} \hfill \,

\begin{tikzpicture} [scale=1.5]
\coordinate (4) at (1,1);
\coordinate (6) at (-1,1);
\coordinate (5) at (0,0);
\coordinate (2) at (0,1);
\coordinate (1) at (1,0);
\coordinate (3) at (-1,0);

\draw[thick, double distance = 2pt] (4) -- (1);
\draw[thick, line width=2pt] (4) -- (2);
\draw[thick, double distance = 2.5pt] (5) -- (1);
\draw[thick, double distance = .3pt] (5) -- (1);
\draw[thick, line width=2pt] (5) -- (3);
\draw[dashed] (6) -- (2);
\draw[dashed] (6) -- (3);

\filldraw[fill=red] (1) circle (2pt) node[below] {\scriptsize 1};
\filldraw[fill=blue] (2) circle (2pt) node[above] {\scriptsize 2};
\filldraw[fill=yellow] (3) circle (2pt) node[below] {\scriptsize 3};
\filldraw[fill=green] (4) circle (2pt) node[above] {\scriptsize 4};
\filldraw[fill=purple] (5) circle (2pt) node[below] {\scriptsize 5};
\filldraw[fill=orange] (6) circle (2pt) node[above] {\scriptsize 6};
\end{tikzpicture}
\caption{From left: polyhedron \SW{13}; the 24 copies of polyhedron \SW{13} along one side of the octagon; all 48 copies of polyhedron \SW{13} with a face in the octagon, forming the ideal, right-angled rhombicuboctahedron~$\cR$; Coxeter diagram with face colors labeled.}
\label{FigRhombicuboctahedron2}
\end{figure}

We return to the 
proof of \propref{Rhombicuboctahedron2}.
Consider a plane of type~$6$ that meets the interior of~$\cQ$,
and its intersection~$S$ with~$\cQ$.
\lemref{annagon} decomposes~$S$ canonically into octagons.  Each octagon is
the union of~$24$ of the $(2,4,6)$ triangles, each of which is a face of two
chambers of~$\Gamma(\cP)$.  These $48$ chambers lie in~$\cQ$, because $S$
meets the interior of~$\cQ$.  The $24$ chambers on one side of the
octagon appear in the middle of \figref{FigRhombicuboctahedron2},
and the full set of~$48$ appear
on the right.  The union~$U$ of the latter is isometric to~$\cR$.  
Each chamber of $\Gamma(\cP)$ in~$\cQ$ has its face of type~$6$
in exactly one of the type~$6$ planes that slice~$\cQ$, indeed in exactly one of 
the octagons into which  \lemref{annagon} decomposes that slice.   So it
lies in exactly one of the rhombicuboctahedra that we have 
constructed.  This establishes
the canonical tiling of~$\cQ$ by copies of~$\cR$.

\end{proof}

\begin{remark}
    Although canonical, this tiling may not be reflective, just
    as in Remark~\ref{RemNonReflectiveCuboctahedral}.
    For example, let $2\cR$ be the double of~$\cR$ 
    across one of its green faces in \figref{FigRhombicuboctahedron2}.
    Of the two adjacent yellow triangles, consider the upper one.  It extends
    to an ideal square face~$F$ of~$2\cR$,
    subdivided by its lines of symmetry.  So the union of $2\cR$,
    with the translate of~$\cR$ whose bottom face is~$F$, is a $\cP$-polyhedron.
    These three copies of~$\cR$ form the canonical
    decomposition, and they are not glued face-to-face.
\end{remark}

\subsection{The case \SW{2}}
Finally, let  $\cP$ be the polyhedron \SW{2}. 
There are two different types of ideal, right-angled $\cP$-polyhedra:

\begin{lemma} \label{All3No3}
Let $\cP$ be the polyhedron \SW{2} and $\cQ$ be an ideal, right-angled $\cP$-polyhedron. 
    Then either every face of~$\cQ$ is tiled by copies of face~$3$, or none is.
\end{lemma}

\begin{proof}

Polyhedron $\cP$ has three equivalence classes of faces: $\{1, 4\}$, $\{2\}$, and $\{ 3\}$. It has a unique ideal vertex, where faces 2, 3, and 4 meet. Restricting to a horospherical section, they form a $(2, 4, 4)$ Euclidean triangle with face 3 as the hypotenuse. Because $\cQ$ is right-angled, its horospherical cross-section at each cusp is a rectangle tiled by $(2,4,4)$ Euclidean triangles. By \propref{IdealVertices} (see also \figref{MinRectangles}, center), either all the external faces at this cusp are tiled by copies of face 3, or none of them are.   

Now suppose some face of~$\cQ$ is tiled by copies of face 3.  Then so is every face that shares an edge or an ideal vertex with it. Repeating this argument shows that every face of~$\cQ$ is tiled by copies of face 3.
\end{proof}

Similarly to polyhedra \SW{4} and \SW{12}, the reflections across
faces 1, 3 and~4 of $\cP={}$\SW{2} generate a reflection group of type~$B_3$. 
In the ball model of hyperbolic space, with the 1-3-4
vertex at the center, face 2 lies in the surface of an ideal, right-angled  
hyperbolic octahedron $\cO$.  This octahedron is
the union of the $B_3$-translates of~$\cP$.
Now choose two antipodal vertices of~$\cO$, 
and consider the two planes of symmetry of~$\cO$
that contain
them but no other vertices.  We define~$\frac12 \cO$ as polyhedron cut out by one of these planes of symmetry, and~$\frac14\cO$ as the polyhedron cut out by both planes of symmetry. Pictured on the right in \figref{FigOctahedron}, $\frac14\cO$
is the hyperbolic right-angled triangular bipyramid, 
with two finite vertices and three ideal vertices.
(Some may prefer \figref{FigOctahedron''}, which shows $\frac14\cO$ from another perspective.)

But we also have a different approach, suggested by \lemref{All3No3}. The reflections in faces 1, 2, and 4 generate a group of order 12. Taking the orbit of polyhedron \SW{12} under this group gives a polyhedron all of whose faces are tiled by copies of face 3. Surprisingly, this polyhedron is the same as $\frac14 \cO$, although it is tiled in a different way. See \figref{FigOctahedron}, center.

\begin{figure}[ht] 
\, \hfill \includegraphics[width=.3\textwidth]{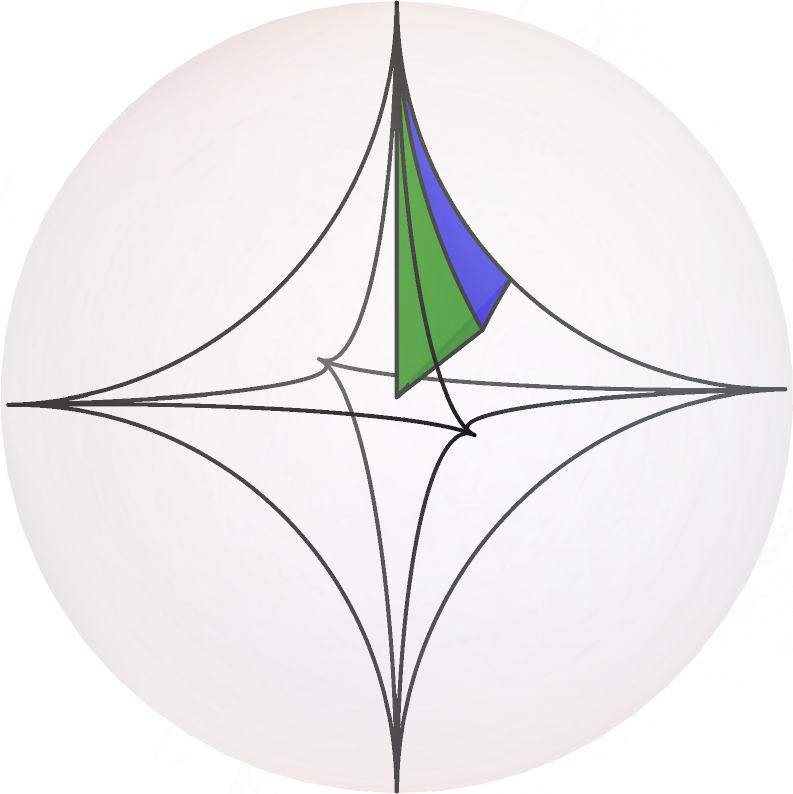} \hfill \includegraphics[width=.3\textwidth]{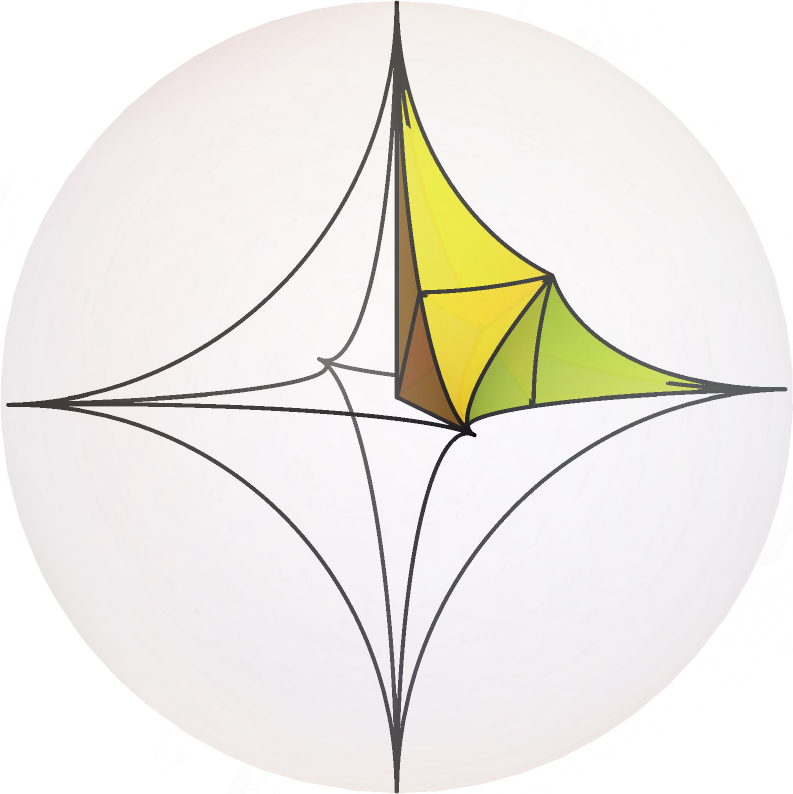} \hfill \includegraphics[width=.3\textwidth]{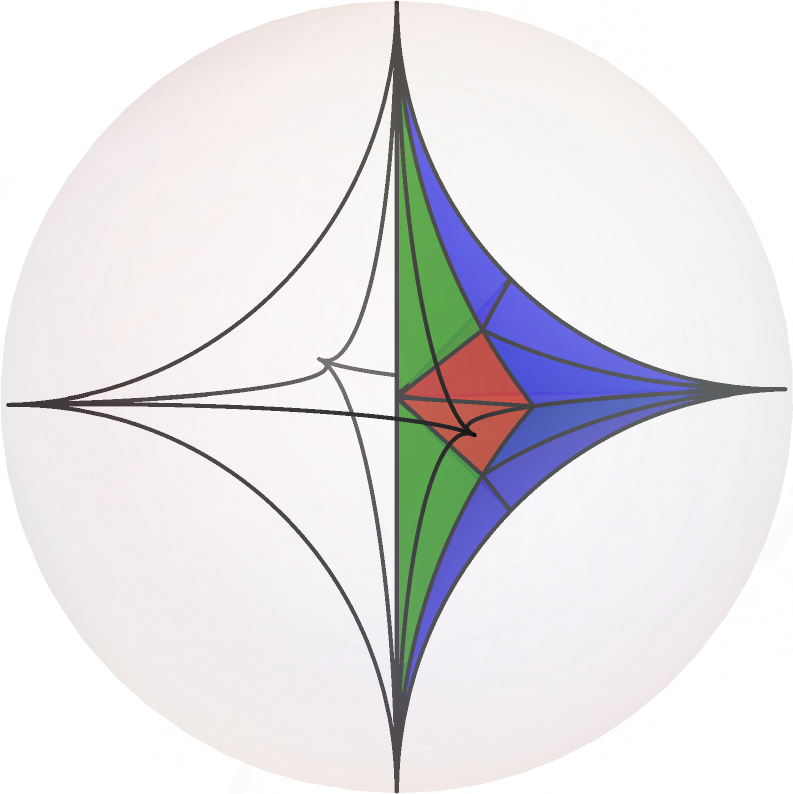} \hfill \,

\begin{tikzpicture}[scale=1.5]
\coordinate (1) at (0,0);
\coordinate (2) at (3,0);
\coordinate (3) at (2,0);
\coordinate (4) at (1,0);

\draw[thick, double distance = 2pt] (3) -- (2);
\draw[thick] (4) -- (1);
\draw[thick, double distance = 2pt] (4) -- (3);

\filldraw[fill=red] (1) circle (2pt) node[above] {\scriptsize 1};
\filldraw[fill=blue] (2) circle (2pt) node[above] {\scriptsize 2};
\filldraw[fill=yellow] (3) circle (2pt) node[above] {\scriptsize 3};
\filldraw[fill=green] (4) circle (2pt) node[above] {\scriptsize 4};
\end{tikzpicture} 
\caption{From left: polyhedron \SW{2}; 12 copies of polyhedron \SW{2} glued to form right-angled triangular bipyramid with all faces tiled by face 3; 12 copies of polyhedron \SW{2} glued to form right-angled triangular bipyramid with no faces tiled by face 3
    (see also \figref{FigOctahedron''}); Coxeter diagram with face colors labeled.}
\label{FigOctahedron}
\end{figure}
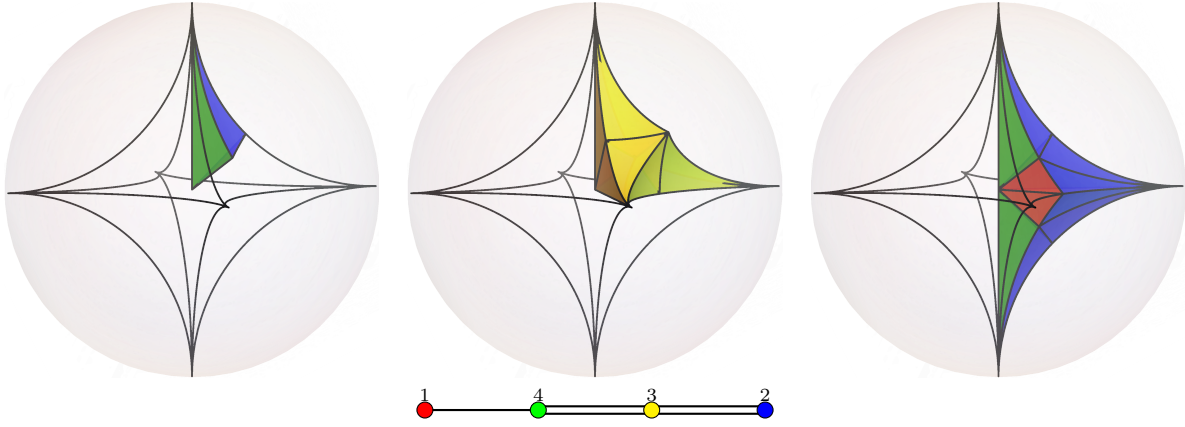

\begin{proposition} \label{Bipyramid}
Let $\cP$ be the polyhedron \SW{2}. Then the ideal, right-angled octahedron $\cO$ is a $\cP$-polyhedron, and every ideal, right-angled $\cP$-polyhedron $\cQ$ is tiled by~$\frac14 \cO$. Moreover:
\begin{enumerate}
    \item If $\cQ$ has all its faces tiled by copies of face 3, then $\cQ$ is tiled canonically and reflectively by~$\frac14 \cO$.
    \item If $\cQ$ has none of its faces tiled by copies of face 3, then $\cQ$ is tiled canonically by copies of~$\frac14 \cO$, $\frac12 \cO$, and $\cO$.    
\end{enumerate}
\end{proposition}

\begin{proof} 
Based on \lemref{All3No3}, we consider two possible cases: all faces of $\cQ$ are tiled by copies of face 3, or no faces of $\cQ$ tiled by copies of face 3. 

In the first case, whenever $\cQ$ contains a copy of $\cP$ it must also contain that copy's mirror images across its 
    faces 1, 2, and 4; otherwise, one of these faces would be external. Faces 1, 2, and 4 meet at a finite vertex and generate a reflection group of order 12. The orbit of $\cP$ under this group is $\frac14 \cO$, as shown in \figref{FigOctahedron}, center. Since, in the reflective tiling of $\bH^3$ by $\cP$, each tile contains one copy of the 1-2-4 vertex, this gives rise to a canonical tiling of $\cQ$ by copies of $\frac14 \cO$. Moreover, if two copies of $\cP$ are glued reflectively along a copy of face 3, taking the orbit of each copy under the group generated by its faces 1, 2, and 4 gives two copies of $\frac14 \cO$, also glued reflectively. So this tiling of~$\cQ$ by copies of $\frac14\cO$ is reflective.

    In the second case, we examine a $\Gamma(\cP)$-image $v$ of
    the (finite) 1-3-4 vertex of~$\cP$, that lies in~$\cQ$.  
    Because $\cQ$ has no finite vertices, $v$ lies in the interior of an edge of~$\cQ$, or of 
    a face of~$\cQ$, or of~$\cQ$ itself.  We claim that the union of the copies of $\cP$ in~$\cQ$, that contain~$v$, is a copy of $\frac14\cO$, or~$\frac12\cO$,
    or~$\cO$ respectively.  
    These unions, as $v$ varies,
    are the tiles in our promised canonical decomposition of~$\cQ$.
    If $v$ lies in the interior of~$\cQ$, then our claim is
    obvious.  
    Next suppose $v$ lies in the interior of a face~$F$ of~$\cQ$.
    Because $\cQ$ has no faces of type~$3$, $F$ has type $\{1,4\}$.
    The planes of this type, that contain~$v$, are those that contain
    exactly two vertices of the octahedron~$\cO$ centered at~$v$.  Therefore,
    by applying an element of $\Gamma(\cP)$, we may suppose without
    loss that $v$ is the 1-3-4 vertex of~$\cP$, that $\cQ$ contains $\cP$,
    and that $F$ lies in the green plane pictured in the first part of
    \figref{FigOctahedron}.  So $F$ contains the intersection of this plane with~$\cO$,
    which is a quadrilateral.
    Consider the convex hull of the union of the quadrilateral with~$\cP$.
    Because $\cQ$ is convex, it contains
    every $\Gamma(\cP)$-translate of $\cP$ whose interior meets this
    convex hull.  These translates are exactly the ones that lie in~$\cQ$
    and contain~$v$, and their  union  is $\frac12\cO$.
    This finishes the case that $v$ lies in the interior of a face of~$\cQ$.

    Finally, suppose $v$ lies in the interior of an edge~$E$ of~$\cQ$.
    We may
    suppose without loss that $v$ is the 1-3-4 vertex of~$\cP$, that $\cQ$ contains
    $\cP$, and that
    $E$ extends the 1-3, 1-4 or 3-4 edge of~$\cP$.  
    The 1-4 edge has an angle of $\pi/3$
    in~$\cP$, 
    so it cannot lie in~$E$. The 1-3 edge is right-angled, 
    but  cannot lie in~$E$ because $\cQ$ has no faces of
    type~3. 
    Therefore $E$ contains the 3-4 edge, which appears vertical
    in the first part of \figref{FigOctahedron}.  So $E$ is the vertical diameter
    of~$\cO$.  Also, 
    $\cQ$ has no faces of type~$3$, and therefore contains $\sigma_3(\cP)$.
    Consider the convex hull of $E$, $\cP$ and $\sigma_3(\cP)$.
    Since $\cQ$ is convex, it contains
    every  $\Gamma(\cP)$-translate of~$\cP$
    whose interior meets this convex hull.
    The union of these translates is the copy of $\frac14\cO$ shown in
    the third part of \figref{FigOctahedron}, finishing the proof.  
\end{proof}

We conclude with a formal statement of the conjecture 
mentioned after \thmref{Thm:MainConverse}: 

\begin{conj} \label{Octahedron}
Let $\cP$ be the polyhedron \SW{2}, and let $\cO$ be the ideal, right-angled  hyperbolic octahedron. Then any ideal, right-angled  $\cP$-polyhedron is tiled by $\cO$.
\end{conj}
By \propref{Bipyramid}, to prove this conjecture, it would suffice to show that every ideal polyhedron tiled by~$\frac14 \cO$ is in fact tiled by $\cO$.

\section{Further Directions}

We propose several possible conjectures and open questions raised by our work.

\begin{enumerate}
\item \conjref{Octahedron} is a natural next step from the results in \secref{Unfolding}, and would lead to a strengthening of our main theorems. It would establish that the unique minimal ideal, right-angled  hyperbolic polyhedra with arithmetic reflection groups are the octahedron, square antiprism, and rhombicuboctahedron.

\item In \secref{Unfolding}, for each of the polyhedra \SW{2}, \SW{4}, \SW{12}, \SW{13}, we exhibit a larger minimal $\cP$-polyhedron which tiles all ideal, right-angled  $\cP$-polyhedra. In the cases of \SW{4} and \SW{13}, these tilings may not be reflective. However, they are constrained in that the finer tiling by $\cP$ is reflective. It is a technical but interesting problem to give precise conditions on the possible tilings by the square antiprism and rhombicuboctahedron---which gluings of tiles are allowed? This question can also be stated at the level of Koebe-Andreev-Thurston polyhedra, whose possible gluings are investigated in \cite{KontorovichNakamura2019}.

\item The methods of this article can be extended to classify other arithmetic hyperbolic Coxeter polyhedra with special features. For example, one could attempt to classify arithmetic polyhedra which are right-angled but not necessarily ideal, or ideal but not necessarily right-angled. One could also relax the right-angled condition and consider polyhedra with all angles $\pi/2n$ for $n \in \N$. Finally, one could attempt to give a fully general answer to \qref{PolyhedronQuestion}, not restricted to the list of arithmetic Scharlau-Walhorn polyhedra. 

\item Even two-dimensional analogues of \qref{PolyhedronQuestion} are challenging. In \secref{PotentialFaces}, we use a classification, due to Felikson, of hyperbolic polygons which tile an ideal polygon \cite{Felikson1998}. It seems to be an open question to classify all compact hyperbolic polygons which tile a right-angled polygon. 

\item There may be further applications of our work to the problem of classifying arithmetic hyperbolic lattices commensurate with link complements. Our methods will apply in situations where the fundamental group of the link is a hyperbolic reflection group. This is not always the case, but it  holds for certain links. It would be interesting to examine further classes of examples, beyond the $(2n)$-chain link. 

\item On the level of polyhedral circle packings, it is an important open problem to classify the integral but not superintegral packings (see \cite{KontorovichNakamura2019}). There is a connection to quasi-arithmeticity but not arithmeticity here, so the results of \cite{ScharlauWalhorn1992} cannot be directly applied. 
\end{enumerate}

\newpage

\appendix

\section{Summary of Arguments for Scharlau-Walhorn Polyhedra} \label{ArgumentData}

The following table summarizes the results of \secref{BadPairs}-\secref{Combinatorics}, ruling out 45 of the 49 Scharlau-Walhorn polyhedra. All the data below was found by running Mathematica code on a list of Gram matrices for the Scharlau-Walhorn polyhedra. It can be checked by examination of the Coxeter diagrams and Gram matrices in \appref{ScharlauWalhornData}. Many of the Scharlau-Walhorn polyhedra can be ruled out in multiple ways by our results. For each polyhedron, we only list the first statement in the paper which rules it out. By the same token, if a polyhedron can be ruled out in multiple ways by a single statement, e.g. if it has multiple bad face pairs, we only list one. 

In the justifications for polyhedra ruled out by \propref{BadFaceOrbit}, we use the abbreviations (C) for a compact equivalence class, (P) for an equivalence class whose fundamental domain is a pentagon with one ideal vertex, (H) for an equivalence class whose fundamental domain is a hexagon with two ideal vertices, and (O) for an equivalence class whose fundamental domain is an octagon with two ideal vertices.

\renewcommand{\propref}[1]{\hyperref[#1]{Prop. \ref{#1}}}
\renewcommand{\exref}[1]{\hyperref[#1]{Ex. \ref{#1}}}
\renewcommand{\thmref}[1]{\hyperref[#1]{Thm. \ref{#1}}}

\begin{tabular}{c|c|c}
Polyhedron & Ruled Out By & Justification \\
\hline 
\SW{1} & \propref{SW1Combinatorics} & Combinatorial argument. \\
\SW{2} & N/A & Tiles ideal, right-angled  octahedron. See \propref{Bipyramid}. \\
\SW{3} & \secref{PotentialFaces} & Argument using \thmref{IdealizablePolygons} and Equation \eqref{VertexLocalEC}. See \exref{Example3}. \\
\SW{4} & N/A & Tiles ideal, right-angled  square antiprism. See \propref{Antiprism}. \\
\SW{5} & \propref{BadFaceOrbit} & Non-idealizable face class: $\{1,2,4,5\}$ (H). See \exref{Example5}. \\
\SW{6} & \propref{SW6Combinatorics} & Combinatorial argument.  \\
\SW{7} & \propref{BadFaceOrbit} & Non-idealizable face classes: $\{1,4\}$ (H), $\{2,6\}$ (H). See \exref{Example7}. \\
\SW{8} & \secref{PotentialFaces} & Argument using \propref{FiniteBadFaceOrbit2} and \propref{BadEdgePair}. See \exref{Example8}. \\
\SW{9} & \propref{BadFaceOrbit} & Non-idealizable face class: $\{1,2,4,5, 7\}$ (O). \\
\SW{10} & \secref{PotentialFaces} & Argument using \propref{FiniteBadFaceOrbit2} and \propref{BadCompactEdgePair}. See \exref{Example10}. \\
\SW{11} & \secref{PotentialFaces} & Argument using \propref{FiniteBadFaceOrbit2} and \propref{BadVertexPair}. See \exref{Example11}. \\
\SW{12} & N/A & Tiles ideal, right-angled  rhombicuboctahedron. See \propref{Rhombicuboctahedron1}. \\
\SW{13} & N/A &  Tiles ideal, right-angled  rhombicuboctahedron. See \propref{Rhombicuboctahedron2}. \\
\SW{14} & \propref{BadCompactEdgePair} & Bad compact edge pair: 1-3, 5-6. See \exref{Example14}. \\
\SW{15} & \propref{AllQuadEuler} & Six quadrilateral faces, 2 pentagonal faces.  See \exref{Example15}. \\
\SW{16} & \secref{PotentialFaces} & Argument using \thmref{IdealizablePolygons} and \propref{AllQuadEuler}. See \exref{Example16}. \\
\SW{17} & \propref{BadCompactEdgePair} & Bad compact edge pair: 2-9, 4-10. \\
\SW{18} & \propref{BadFaceOrbit} & Non-idealizable face classes: $\{6\}$ (C), $\{1,4\}$ (O). \\
\SW{19} & \secref{PotentialFaces} & Argument using \propref{FiniteBadFaceOrbit2} and \propref{BadVertexPair}. See \exref{Example19} \\
\SW{20} & \propref{BadFacePair} & Bad face pair: 1, 7. See \exref{Example20}. \\
\SW{21} & \propref{BadFaceOrbit} & Non-idealizable face classes: $\{6\}$ (C), $\{2\}$ (P), $\{8\}$ (P), $\{1,4,7\}$ (O). \\
\SW{22} & \propref{BadCompactEdgePair} & Bad compact edge pair: 2-6, 4-9. \\
\SW{23} & \propref{BadFacePair} & Bad face pair: 1, 9. \\
\SW{24} & \propref{BadFacePair} & Bad face pair: 1, 8. \\
\SW{25} & \propref{BadEdgePair} & Bad edge pair: 1-4, 9-10. See \exref{Example25}. \\
\SW{26} & \propref{BadCompactEdgePair} & Bad compact edge pair: 1-3, 5-8. \\
\SW{27} & \propref{AllQuadEuler} & Six quadrilateral faces. \\
\SW{28} & \propref{BadFacePair} & Bad face pair: 7, 10. \\
\SW{29} & \propref{BadFaceOrbit} & Non-idealizable face classes: $\{1,7\}$ (C), $\{6\}$ (C), $\{8\}$ (C), $\{4\}$ (P). \\
\SW{30} & \propref{BadFacePair} & Bad face pair: 1, 8. \\
\SW{31} & \propref{AllQuadEuler} & Six quadrilateral faces. \\
\SW{32} & \propref{BadFacePair} & Bad face pair: 5, 10. \\
\SW{33} & \propref{BadFacePair} & Bad face pair: 6, 10. \\
\SW{34} & \propref{BadFacePair} & Bad face pair: 6, 9. \\
\SW{35} & \propref{BadFacePair} & Bad face pair: 1, 8. \\
\SW{36} & \propref{BadFacePair} & Bad face pair: 5, 13. \\
\SW{37} & \propref{BadFacePair} & Bad face pair: 1, 10. \\
\SW{38} & \propref{AllQuadEuler} & Six quadrilateral faces. \\
\SW{39} & \propref{BadFacePair} & Bad face pair: 1, 7. \\
\SW{40} & \propref{BadFacePair} & Bad face pair: 6, 9. \\
\SW{41} & \propref{BadFacePair} & Bad face pair: 6, 15. \\
\SW{42} & \propref{BadEdgePair} & Bad edge pair: 2-3, 8-9. \\
\SW{43} & \propref{BadFacePair} & Bad face pair: 6, 13. \\
\SW{44} & \propref{BadFacePair} & Bad face pair: 1, 7. \\
\SW{45} & \propref{BadFacePair} & Bad face pair: 1, 10. \\
\SW{46} & \propref{BadFacePair} & Bad face pair: 1, 13. \\
\SW{47} & \propref{BadVertexPair} & Bad vertex pair: 1-2-6, 8-9-12. See \exref{Example47}. \\
\SW{48} & \propref{BadFacePair} & Bad face pair: 1, 11. \\
\SW{49} & \propref{BadFacePair} & Bad face pair: 1, 8. \\
\end{tabular}

\section{Scharlau-Walhorn Data} \label{ScharlauWalhornData}

\subsection{Coxeter Diagrams}

A \emph{Coxeter diagram} encodes a hyperbolic polyhedron via the data of its faces and dihedral angles. Each vertex in the diagram represents one face of the polyhedron. The numbering of faces below follows \cite{ScharlauWalhorn1992}. Vertices are not connected in the diagram if their associated faces meet at an angle of $\pi/2$. Two vertices are connected by an $n-$fold edge if their associated faces meet at an angle of $\pi/(n+2)$. A thick edge between two vertices means that the associated faces are parallel and meet at a single point at infinity. A dotted edge between two vertices means that the associated faces are ultraparallel and do not intersect, even at infinity. 

Diagrams are labeled SW$\#$ (Scharlau-Walhorn numbering), and a superscript $\star$ denotes a diagram that is difficult to see; these are hyperlinked to their Gram matrices, provided in \S\ref{subsec:Gram}. A diagram is blue if the polyhedron tiles an ideal, right-angled polyhedron.

\begin{center}
\phantomsection \label{SW1}\begin{tikzpicture}

\coordinate (1) at (3,0);
\coordinate (2) at (1,0);
\coordinate (3) at (0,0);
\coordinate (4) at (2,0);

\draw[thick, double distance = 2.5pt] (3) -- (2);
\draw[thick, double distance = .3pt] (3) -- (2);
\draw[thick] (4) -- (1);
\draw[thick] (4) -- (2);

\filldraw[fill=white] (1) circle (2pt) node[above] {\scriptsize 1};
\filldraw[fill=white] (2) circle (2pt) node[above] {\scriptsize 2};
\filldraw[fill=white] (3) circle (2pt) node[above] {\scriptsize 3};
\filldraw[fill=white] (4) circle (2pt) node[above] {\scriptsize 4};

\draw (3) node[above, xshift = 0.4cm, yshift = 0.7cm] {SW1:};
\end{tikzpicture}
\qquad 
\phantomsection \label{SW2}
{\color{blue}
\begin{tikzpicture} 
\coordinate (1) at (0,0);
\coordinate (2) at (3,0);
\coordinate (3) at (2,0);
\coordinate (4) at (1,0);

\draw[thick, double distance = 2pt] (3) -- (2);
\draw[thick] (4) -- (1);
\draw[thick, double distance = 2pt] (4) -- (3);

\filldraw[fill=white] (1) circle (2pt) node[above] {\scriptsize 1};
\filldraw[fill=white] (2) circle (2pt) node[above] {\scriptsize 2};
\filldraw[fill=white] (3) circle (2pt) node[above] {\scriptsize 3};
\filldraw[fill=white] (4) circle (2pt) node[above] {\scriptsize 4};

\draw (1) node[above, yshift = 0.7cm, xshift = 1cm] {SW2:};
\end{tikzpicture} 
}
\qquad 
 \phantomsection \label{SW3}\begin{tikzpicture}
\coordinate (1) at (-1,0);
\coordinate (2) at (1,0);
\coordinate (3) at (0,-1);
\coordinate (4) at (0,0);
\coordinate (5) at (1,-1);
\coordinate (6) at (-1,-1);

\draw[thick] (4) -- (1);
\draw[thick, line width=2pt] (4) -- (2);
\draw[thick] (5) -- (2);
\draw[dashed] (5) -- (3);
\draw[dashed] (6) -- (1);
\draw[thick, line width=2pt] (6) -- (3);

\filldraw[fill=white] (1) circle (2pt) node[above] {\scriptsize 1};
\filldraw[fill=white] (2) circle (2pt) node[above] {\scriptsize 2};
\filldraw[fill=white] (3) circle (2pt) node[below] {\scriptsize 3};
\filldraw[fill=white] (4) circle (2pt) node[above] {\scriptsize 4};
\filldraw[fill=white] (5) circle (2pt) node[below] {\scriptsize 5};
\filldraw[fill=white] (6) circle (2pt) node[below] {\scriptsize 6};

\draw (1) node[left, xshift = -0.1cm] {SW3:};
\end{tikzpicture}

\vspace{0.3cm}
{\color{blue}
 \phantomsection \label{SW4}\begin{tikzpicture}
\coordinate (1) at (2,0);
\coordinate (2) at (0,0);
\coordinate (3) at (4,0);
\coordinate (4) at (1,0);
\coordinate (5) at (3,0);

\draw[thick] (4) -- (1);
\draw[thick, line width=2pt] (4) -- (2);
\draw[thick, double distance = 2pt] (5) -- (1);
\draw[thick, line width=2pt] (5) -- (3);

\filldraw[fill=white] (1) circle (2pt) node[above] {\scriptsize 1};
\filldraw[fill=white] (2) circle (2pt) node[above] {\scriptsize 2};
\filldraw[fill=white] (3) circle (2pt) node[above] {\scriptsize 3};
\filldraw[fill=white] (4) circle (2pt) node[above] {\scriptsize 4};
\filldraw[fill=white] (5) circle (2pt) node[above] {\scriptsize 5};

\draw (2) node[above, xshift = 2cm, yshift = 0.7cm] {SW4:};
\end{tikzpicture}
}
\qquad 
 \phantomsection \label{SW5}\begin{tikzpicture}
\coordinate (1) at (-1,0);
\coordinate (2) at (1,0);
\coordinate (3) at (0,-1);
\coordinate (4) at (0,0);
\coordinate (5) at (1,-1);
\coordinate (6) at (-1,-1);

\draw[thick] (4) -- (1);
\draw[thick, line width=2pt] (4) -- (2);
\draw[thick] (5) -- (1);
\draw[thick] (5) -- (2);
\draw[dashed] (5) -- (3);
\draw[dashed] (6) -- (1);
\draw[thick, line width=2pt] (6) -- (3);

\filldraw[fill=white] (1) circle (2pt) node[above] {\scriptsize 1};
\filldraw[fill=white] (2) circle (2pt) node[above] {\scriptsize 2};
\filldraw[fill=white] (3) circle (2pt) node[below] {\scriptsize 3};
\filldraw[fill=white] (4) circle (2pt) node[above] {\scriptsize 4};
\filldraw[fill=white] (5) circle (2pt) node[below] {\scriptsize 5};
\filldraw[fill=white] (6) circle (2pt) node[below] {\scriptsize 6};

\draw (1) node[left, xshift = -0.1cm] {SW5:};
\end{tikzpicture}
\qquad 
 \phantomsection \label{SW6}\begin{tikzpicture}
\coordinate (1) at (0,0);
\coordinate (2) at (2,0);
\coordinate (3) at (3,0);
\coordinate (4) at (1,0);

\draw[thick, double distance = 2.5pt] (3) -- (2);
\draw[thick, double distance = .3pt] (3) -- (2);
\draw[thick, double distance = 2pt] (4) -- (1);
\draw[thick] (4) -- (2);

\filldraw[fill=white] (1) circle (2pt) node[above] {\scriptsize 1};
\filldraw[fill=white] (2) circle (2pt) node[above] {\scriptsize 2};
\filldraw[fill=white] (3) circle (2pt) node[above] {\scriptsize 3};
\filldraw[fill=white] (4) circle (2pt) node[above] {\scriptsize 4};

\draw (1) node[above, xshift = 0.4cm, yshift = 0.7cm] {SW6:};
\end{tikzpicture}

\vspace{0.5cm}
 \phantomsection \label{SW7}\begin{tikzpicture}
    \def\nodeorder{1,6,2,4,8,7,3,5}
   
  \foreach \i [count=\n] in \nodeorder {
    \coordinate (\i) at ({360/8 * (\n-1) + 360/8*.5}:.75);
    }

\draw[thick] (4) -- (1);
\draw[thick, line width=2pt] (4) -- (2);
\draw[dashed] (5) -- (1);
\draw[thick, line width=2pt] (5) -- (3);
\draw[thick, line width=2pt] (6) -- (1);
\draw[thick] (6) -- (2);
\draw[dashed] (6) -- (3);
\draw[dashed] (7) -- (2);
\draw[dashed] (7) -- (3);
\draw[dashed] (7) -- (5);
\draw[dashed] (8) -- (3);
\draw[dashed] (8) -- (4);
\draw[dashed] (8) -- (5);
\draw[thick, line width=2pt] (8) -- (7);


  \foreach \i [count=\n] in \nodeorder {
    \pgfmathsetmacro{\angle}{360/8 * (\n-1) + 360/8*.5}
    \filldraw[fill=white] (\i) circle (2pt) node[anchor=180+\angle] {\scriptsize \i};
    }

\draw (4) node[above, xshift = -0.4cm, yshift = 0.3cm] {SW7:};
\end{tikzpicture}
\qquad 
 \phantomsection \label{SW8}\begin{tikzpicture}
\coordinate (1) at (-1,0);
\coordinate (2) at (0,-1);
\coordinate (3) at (1,0);
\coordinate (4) at (0,0);
\coordinate (5) at (1,-1);
\coordinate (6) at (-1,-1);

\draw[thick, double distance = 2.5pt] (4) -- (1);
\draw[thick, double distance = .3pt] (4) -- (1);
\draw[thick, line width=2pt] (4) -- (3);
\draw[dashed] (5) -- (2);
\draw[thick, double distance = 2.5pt] (5) -- (3);
\draw[thick, double distance = .3pt] (5) -- (3);
\draw[dashed] (6) -- (1);
\draw[thick, line width=2pt] (6) -- (2);

\filldraw[fill=white] (1) circle (2pt) node[above] {\scriptsize 1};
\filldraw[fill=white] (2) circle (2pt) node[below] {\scriptsize 2};
\filldraw[fill=white] (3) circle (2pt) node[above] {\scriptsize 3};
\filldraw[fill=white] (4) circle (2pt) node[above] {\scriptsize 4};
\filldraw[fill=white] (5) circle (2pt) node[below] {\scriptsize 5};
\filldraw[fill=white] (6) circle (2pt) node[below] {\scriptsize 6};

\draw (1) node[above, xshift = 0.4cm, yshift=0.5cm] {SW8:};
\end{tikzpicture}
\qquad 
 \phantomsection \label{SW9}\begin{tikzpicture}
\def\nodeorder{1,4,5,2,7,3,6}
    
\newcount\nodecount
\nodecount=0
\foreach \i in \nodeorder {
  \global\advance\nodecount by 1
}
\pgfmathsetmacro{\nodecountvalue}{\the\nodecount}
\def\rotateby{-.25*360/\nodecountvalue}
   
\foreach \i [count=\n] in \nodeorder {
  \coordinate (\i) at ({360/\nodecountvalue * (\n-1) + \rotateby}:.75);
}

\draw[thick] (4) -- (1);
\draw[thick, line width=2pt] (4) -- (2);
\draw[thick] (5) -- (2);
\draw[dashed] (5) -- (3);
\draw[thick] (5) -- (4);
\draw[dashed] (6) -- (1);
\draw[thick, line width=2pt] (6) -- (3);
\draw[dashed] (6) -- (5);
\draw[dashed] (7) -- (1);
\draw[thick] (7) -- (2);
\draw[dashed] (7) -- (3);

\foreach \i [count=\n] in \nodeorder {
  \pgfmathsetmacro{\angle}{360/\nodecountvalue * (\n-1) + \rotateby}
  \filldraw[fill=white] (\i) circle (2pt) node[anchor=180+\angle] {\scriptsize \i};
}

\draw (2) node[above, yshift = 0.4cm, xshift= -.2cm] {SW9:};
\end{tikzpicture}
\qquad 
 \phantomsection \label{SW10}\begin{tikzpicture}
\coordinate (1) at (-1,1);
\coordinate (2) at (1,1);
\coordinate (3) at (0,0);
\coordinate (4) at (0,1);
\coordinate (5) at (-1,0);
\coordinate (6) at (1,0);

\draw[thick] (4) -- (1);
\draw[thick, line width=2pt] (4) -- (2);
\draw[dashed] (5) -- (1);
\draw[thick, line width=2pt] (5) -- (3);
\draw[thick, double distance = 2pt] (6) -- (2);
\draw[dashed] (6) -- (3);

\filldraw[fill=white] (1) circle (2pt) node[above] {\scriptsize 1};
\filldraw[fill=white] (2) circle (2pt) node[above] {\scriptsize 2};
\filldraw[fill=white] (3) circle (2pt) node[below] {\scriptsize 3};
\filldraw[fill=white] (4) circle (2pt) node[above] {\scriptsize 4};
\filldraw[fill=white] (5) circle (2pt) node[below] {\scriptsize 5};
\filldraw[fill=white] (6) circle (2pt) node[below] {\scriptsize 6};

\draw (1) node[above, xshift = 0.4cm, yshift=0.5cm] {SW10:};
\end{tikzpicture}

\vspace{0.5cm}
\phantomsection \label{SW11}\begin{tikzpicture}
 \coordinate (1) at (-1,0);
\coordinate (2) at (1,0);
\coordinate (3) at (0,-1);
\coordinate (4) at (0,0);
\coordinate (5) at (1,-1);
\coordinate (6) at (-1,-1);

\draw[thick, double distance = 2pt] (4) -- (1);
\draw[thick, line width=2pt] (4) -- (2);
\draw[thick] (5) -- (1);
\draw[thick, double distance = 2pt] (5) -- (2);
\draw[dashed] (5) -- (3);
\draw[dashed] (6) -- (1);
\draw[thick, line width=2pt] (6) -- (3);

\filldraw[fill=white] (1) circle (2pt) node[above] {\scriptsize 1};
\filldraw[fill=white] (2) circle (2pt) node[above] {\scriptsize 2};
\filldraw[fill=white] (3) circle (2pt) node[below] {\scriptsize 3};
\filldraw[fill=white] (4) circle (2pt) node[above] {\scriptsize 4};
\filldraw[fill=white] (5) circle (2pt) node[below] {\scriptsize 5};
\filldraw[fill=white] (6) circle (2pt) node[below] {\scriptsize 6};

\draw (1) node[left, xshift = -0.1cm] {SW11:};
\end{tikzpicture}
\qquad 
\phantomsection \label{SW12}
{\color{blue}
\begin{tikzpicture}
\coordinate (2) at (1,1);
\coordinate (6) at (-1,1);
\coordinate (5) at (0,0);
\coordinate (4) at (0,1);
\coordinate (1) at (1,0);
\coordinate (3) at (-1,0);

\draw[thick] (4) -- (1);
\draw[thick, line width=2pt] (4) -- (2);
\draw[dashed] (5) -- (1);
\draw[thick, line width=2pt] (5) -- (3);
\draw[dashed] (6) -- (3);
\draw[thick, double distance = 2pt] (6) -- (4);

\filldraw[fill=white] (1) circle (2pt) node[below] {\scriptsize 1};
\filldraw[fill=white] (2) circle (2pt) node[above] {\scriptsize 2};
\filldraw[fill=white] (3) circle (2pt) node[below] {\scriptsize 3};
\filldraw[fill=white] (4) circle (2pt) node[above] {\scriptsize 4};
\filldraw[fill=white] (5) circle (2pt) node[below] {\scriptsize 5};
\filldraw[fill=white] (6) circle (2pt) node[above] {\scriptsize 6};

\draw (6) node[left, xshift = -0.1cm] {SW12:};
\end{tikzpicture}
}
\qquad 
\phantomsection \label{SW13}
{\color{blue}
\begin{tikzpicture}
\coordinate (4) at (1,1);
\coordinate (6) at (-1,1);
\coordinate (5) at (0,0);
\coordinate (2) at (0,1);
\coordinate (1) at (1,0);
\coordinate (3) at (-1,0);

\draw[thick, double distance = 2pt] (4) -- (1);
\draw[thick, line width=2pt] (4) -- (2);
\draw[thick, double distance = 2.5pt] (5) -- (1);
\draw[thick, double distance = .3pt] (5) -- (1);
\draw[thick, line width=2pt] (5) -- (3);
\draw[dashed] (6) -- (2);
\draw[dashed] (6) -- (3);

\filldraw[fill=white] (1) circle (2pt) node[below] {\scriptsize 1};
\filldraw[fill=white] (2) circle (2pt) node[above] {\scriptsize 2};
\filldraw[fill=white] (3) circle (2pt) node[below] {\scriptsize 3};
\filldraw[fill=white] (4) circle (2pt) node[above] {\scriptsize 4};
\filldraw[fill=white] (5) circle (2pt) node[below] {\scriptsize 5};
\filldraw[fill=white] (6) circle (2pt) node[above] {\scriptsize 6};

\draw (6) node[left, xshift = -0.1cm] {SW13:};
\end{tikzpicture}
}

\vspace{0.4cm}
 \phantomsection \label{SW14}\begin{tikzpicture}

\def\nodeorder{1,4,2,6,3,7,8,5}
    
\newcount\nodecount
\nodecount=0
\foreach \i in \nodeorder {
  \global\advance\nodecount by 1
}
\pgfmathsetmacro{\nodecountvalue}{\the\nodecount}
\def\rotateby{.5*360/\nodecountvalue}
   
\foreach \i [count=\n] in \nodeorder {
  \coordinate (\i) at ({360/\nodecountvalue * (\n-1) + \rotateby}:.75);
}

\draw[thick, double distance = 2pt] (4) -- (1);
\draw[thick, line width=2pt] (4) -- (2);
\draw[dashed] (5) -- (1);
\draw[thick, line width=2pt] (5) -- (3);
\draw[thick, line width=2pt] (6) -- (1);
\draw[thick, double distance = 2pt] (6) -- (2);
\draw[dashed] (6) -- (3);
\draw[dashed] (7) -- (2);
\draw[dashed] (7) -- (3);
\draw[dashed] (7) -- (5);
\draw[dashed] (8) -- (3);
\draw[dashed] (8) -- (4);
\draw[dashed] (8) -- (5);
\draw[thick, line width=2pt] (8) -- (7);

\foreach \i [count=\n] in \nodeorder {
  \pgfmathsetmacro{\angle}{360/\nodecountvalue * (\n-1) + \rotateby}
  \filldraw[fill=white] (\i) circle (2pt) node[anchor=180+\angle] {\scriptsize \i};
}

\draw (2) node[left, xshift = -0.4cm, yshift=0.2cm] {SW14:};
\end{tikzpicture}
\qquad 
 \phantomsection \label{SW15}\begin{tikzpicture}

\def\nodeorder{8,1,6,7,3,2,5,4}
    
\newcount\nodecount
\nodecount=0
\foreach \i in \nodeorder {
  \global\advance\nodecount by 1
}
\pgfmathsetmacro{\nodecountvalue}{\the\nodecount}
\def\rotateby{.5*360/\nodecountvalue}
   
\foreach \i [count=\n] in \nodeorder {
  \coordinate (\i) at ({360/\nodecountvalue * (\n-1) + \rotateby}:.75);
}

\draw[dashed] (4) -- (1);
\draw[thick, line width=2pt] (4) -- (3);
\draw[dashed] (5) -- (1);
\draw[thick, line width=2pt] (5) -- (2);
\draw[thick] (6) -- (1);
\draw[dashed] (6) -- (2);
\draw[dashed] (6) -- (3);
\draw[dashed] (7) -- (3);
\draw[dashed] (7) -- (5);
\draw[thick] (7) -- (6);
\draw[thick] (8) -- (1);
\draw[dashed] (8) -- (2);
\draw[dashed] (8) -- (4);
\draw[dashed] (8) -- (7);

\foreach \i [count=\n] in \nodeorder {
  \pgfmathsetmacro{\angle}{360/\nodecountvalue * (\n-1) + \rotateby}
  \filldraw[fill=white] (\i) circle (2pt) node[anchor=180+\angle] {\scriptsize \i};
}

\draw (6) node[left, xshift = -0.4cm, yshift=0.2cm] {SW15:};
\end{tikzpicture}
\qquad 
 \phantomsection \label{SW16}\begin{tikzpicture}

\coordinate (1) at (2,0);
\coordinate (2) at (-1,0);
\coordinate (3) at (0,0);
\coordinate (4) at (1,0);
\coordinate (5) at (-2,0);

\draw[thick, double distance = 2pt] (3) -- (2);
\draw[thick, double distance = 2.5pt] (4) -- (1);
\draw[thick, double distance = .3pt] (4) -- (1);
\draw[thick, double distance = 2pt] (4) -- (3);
\draw[dashed] (5) -- (2);

\filldraw[fill=white] (1) circle (2pt) node[above] {\scriptsize 1};
\filldraw[fill=white] (2) circle (2pt) node[above] {\scriptsize 2};
\filldraw[fill=white] (3) circle (2pt) node[above] {\scriptsize 3};
\filldraw[fill=white] (4) circle (2pt) node[above] {\scriptsize 4};
\filldraw[fill=white] (5) circle (2pt) node[above] {\scriptsize 5};

\draw (5) node[above, xshift = 0.4cm, yshift = 0.5cm] {SW16:};

\end{tikzpicture}

\vspace{.4cm}
 \phantomsection \label{SW17}\begin{tikzpicture}

\def\nodeorder{1,4,7,5,2,8,3,9,10,6}
    
\newcount\nodecount
\nodecount=0
\foreach \i in \nodeorder {
  \global\advance\nodecount by 1
}
\pgfmathsetmacro{\nodecountvalue}{\the\nodecount}
\def\rotateby{0*360/\nodecountvalue}
   
\foreach \i [count=\n] in \nodeorder {
  \coordinate (\i) at ({360/\nodecountvalue * (\n-1) + \rotateby}:1.25);
}

\draw[thick] (4) -- (1);
\draw[thick, line width=2pt] (4) -- (2);
\draw[thick, double distance = 2.5pt] (5) -- (2);
\draw[thick, double distance = .3pt] (5) -- (2);
\draw[dashed] (5) -- (3);
\draw[dashed] (6) -- (1);
\draw[thick, line width=2pt] (6) -- (3);
\draw[dashed] (6) -- (5);
\draw[dashed] (7) -- (1);
\draw[dashed] (7) -- (3);
\draw[thick, double distance = 2.5pt] (7) -- (4);
\draw[thick, double distance = .3pt] (7) -- (4);
\draw[thick, line width=2pt] (7) -- (5);
\draw[dashed] (7) -- (6);
\draw[dashed] (8) -- (1);
\draw[thick] (8) -- (2);
\draw[dashed] (8) -- (3);
\draw[dashed] (8) -- (5);
\draw[dashed] (9) -- (1);
\draw[dashed] (9) -- (3);
\draw[dashed] (9) -- (4);
\draw[dashed] (9) -- (6);
\draw[dashed] (9) -- (8);
\draw[dashed] (10) -- (1);
\draw[dashed] (10) -- (2);
\draw[dashed] (10) -- (3);
\draw[dashed] (10) -- (6);
\draw[dashed] (10) -- (8);
\draw[thick, line width=2pt] (10) -- (9);

\foreach \i [count=\n] in \nodeorder {
  \pgfmathsetmacro{\angle}{360/\nodecountvalue * (\n-1) + \rotateby}
  \filldraw[fill=white] (\i) circle (2pt) node[anchor=180+\angle] {\scriptsize \i};
}

\draw (5) node[left, xshift = -0.5cm, yshift=0.2cm] {SW17:};
\end{tikzpicture}
\qquad
 \phantomsection \label{SW18}\begin{tikzpicture}

\def\nodeorder{7,1,6,4,2,8,9,3,5}
    
\newcount\nodecount
\nodecount=0
\foreach \i in \nodeorder {
  \global\advance\nodecount by 1
}
\pgfmathsetmacro{\nodecountvalue}{\the\nodecount}
\def\rotateby{.25*360/\nodecountvalue}
   
\foreach \i [count=\n] in \nodeorder {
  \coordinate (\i) at ({360/\nodecountvalue * (\n-1) + \rotateby}:1.25);
}

\draw[thick] (4) -- (1);
\draw[thick, line width=2pt] (4) -- (2);
\draw[dashed] (5) -- (1);
\draw[thick, line width=2pt] (5) -- (3);
\draw[thick, double distance = 2pt] (6) -- (1);
\draw[dashed] (6) -- (3);
\draw[thick, double distance = 2pt] (6) -- (4);
\draw[thick, line width=2pt] (7) -- (1);
\draw[dashed] (7) -- (2);
\draw[dashed] (7) -- (3);
\draw[dashed] (7) -- (5);
\draw[dashed] (8) -- (2);
\draw[dashed] (8) -- (3);
\draw[dashed] (8) -- (4);
\draw[dashed] (8) -- (5);
\draw[dashed] (9) -- (2);
\draw[dashed] (9) -- (3);
\draw[dashed] (9) -- (5);
\draw[dashed] (9) -- (6);
\draw[thick, line width=2pt] (9) -- (8);

\foreach \i [count=\n] in \nodeorder {
  \pgfmathsetmacro{\angle}{360/\nodecountvalue * (\n-1) + \rotateby}
  \filldraw[fill=white] (\i) circle (2pt) node[anchor=180+\angle] {\scriptsize \i};
}

\draw (4) node[left, xshift = -0.3cm, yshift=0.2cm] {SW18:};
\end{tikzpicture}
\qquad
\phantomsection \label{SW19}\begin{tikzpicture}

\def\nodeorder{3,5,7,1,4,2,6}
    
\newcount\nodecount
\nodecount=0
\foreach \i in \nodeorder {
  \global\advance\nodecount by 1
}
\pgfmathsetmacro{\nodecountvalue}{\the\nodecount}
\def\rotateby{-90+0*360/\nodecountvalue}
   
\foreach \i [count=\n] in \nodeorder {
  \coordinate (\i) at ({360/\nodecountvalue * (\n-1) + \rotateby}:1);
}

\draw[thick, double distance = 2pt] (4) -- (1);
\draw[thick, line width=2pt] (4) -- (2);
\draw[dashed] (5) -- (1);
\draw[thick, line width=2pt] (5) -- (3);
\draw[dashed] (6) -- (3);
\draw[dashed] (6) -- (4);
\draw[thick] (7) -- (1);
\draw[dashed] (7) -- (2);
\draw[dashed] (7) -- (3);

\foreach \i [count=\n] in \nodeorder {
  \pgfmathsetmacro{\angle}{360/\nodecountvalue * (\n-1) + \rotateby}
  \filldraw[fill=white] (\i) circle (2pt) node[anchor=180+\angle] {\scriptsize \i};
}

\draw (2) node[left, xshift = -0.3cm,yshift=.5cm] {SW19:};
\end{tikzpicture}

\phantomsection \label{SW20}\begin{tikzpicture}

\def\nodeorder{1,4,5,2,7,3,6}
    
\newcount\nodecount
\nodecount=0
\foreach \i in \nodeorder {
  \global\advance\nodecount by 1
}
\pgfmathsetmacro{\nodecountvalue}{\the\nodecount}
\def\rotateby{-.25*360/\nodecountvalue}
   
\foreach \i [count=\n] in \nodeorder {
  \coordinate (\i) at ({360/\nodecountvalue * (\n-1) + \rotateby}:1);
}

\draw[thick, double distance = 2pt] (4) -- (1);
\draw[thick, line width=2pt] (4) -- (2);
\draw[thick] (5) -- (2);
\draw[dashed] (5) -- (3);
\draw[thick] (5) -- (4);
\draw[dashed] (6) -- (1);
\draw[thick, line width=2pt] (6) -- (3);
\draw[dashed] (6) -- (5);
\draw[dashed] (7) -- (1);
\draw[thick, double distance = 2pt] (7) -- (2);
\draw[dashed] (7) -- (3);

\foreach \i [count=\n] in \nodeorder {
  \pgfmathsetmacro{\angle}{360/\nodecountvalue * (\n-1) + \rotateby}
  \filldraw[fill=white] (\i) circle (2pt) node[anchor=180+\angle] {\scriptsize \i};
}

\draw (2) node[left, xshift = -0.3cm, yshift=0.4cm] {SW20:};
\end{tikzpicture}
\qquad
\phantomsection \label{SW21}\begin{tikzpicture}

\def\nodeorder{1,9,10,4,2,6,8,7,3,5}
    
\newcount\nodecount
\nodecount=0
\foreach \i in \nodeorder {
  \global\advance\nodecount by 1
}
\pgfmathsetmacro{\nodecountvalue}{\the\nodecount}
\def\rotateby{-90+0*360/\nodecountvalue}
   
\foreach \i [count=\n] in \nodeorder {
  \coordinate (\i) at ({360/\nodecountvalue * (\n-1) + \rotateby}:1.25);
}

\draw[thick] (4) -- (1);
\draw[thick, line width=2pt] (4) -- (2);
\draw[dashed] (5) -- (1);
\draw[thick, line width=2pt] (5) -- (3);
\draw[dashed] (6) -- (1);
\draw[thick, double distance = 2pt] (6) -- (2);
\draw[dashed] (6) -- (3);
\draw[thick] (7) -- (1);
\draw[dashed] (7) -- (3);
\draw[dashed] (7) -- (4);
\draw[dashed] (7) -- (5);
\draw[dashed] (8) -- (2);
\draw[dashed] (8) -- (3);
\draw[dashed] (8) -- (5);
\draw[thick, double distance = 2pt] (8) -- (6);
\draw[thick, line width=2pt] (8) -- (7);
\draw[dashed] (9) -- (1);
\draw[dashed] (9) -- (2);
\draw[dashed] (9) -- (3);
\draw[dashed] (9) -- (4);
\draw[dashed] (9) -- (5);
\draw[dashed] (10) -- (2);
\draw[dashed] (10) -- (3);
\draw[dashed] (10) -- (4);
\draw[dashed] (10) -- (5);
\draw[dashed] (10) -- (6);
\draw[thick, line width=2pt] (10) -- (9);

\foreach \i [count=\n] in \nodeorder {
  \pgfmathsetmacro{\angle}{360/\nodecountvalue * (\n-1) + \rotateby}
  \filldraw[fill=white] (\i) circle (2pt) node[anchor=180+\angle] {\scriptsize \i};
}

\draw (8) node[left, xshift = -0.5cm] {SW21:};
\end{tikzpicture}
\qquad 
\phantomsection \label{SW22}\begin{tikzpicture}

\def\nodeorder{5,3,1,6,9,2,4,8,7}
    
\newcount\nodecount
\nodecount=0
\foreach \i in \nodeorder {
  \global\advance\nodecount by 1
}
\pgfmathsetmacro{\nodecountvalue}{\the\nodecount}
\def\rotateby{180+.25*360/\nodecountvalue}
   
\foreach \i [count=\n] in \nodeorder {
  \coordinate (\i) at ({360/\nodecountvalue * (\n-1) + \rotateby}:1.25);
}

\draw[thick] (4) -- (1);
\draw[thick, line width=2pt] (4) -- (2);
\draw[dashed] (5) -- (1);
\draw[thick, line width=2pt] (5) -- (3);
\draw[dashed] (6) -- (3);
\draw[dashed] (6) -- (4);
\draw[dashed] (6) -- (5);
\draw[dashed] (7) -- (2);
\draw[dashed] (7) -- (3);
\draw[dashed] (7) -- (5);
\draw[dashed] (8) -- (1);
\draw[dashed] (8) -- (3);
\draw[thick, double distance = 2pt] (8) -- (4);
\draw[thick, line width=2pt] (8) -- (7);
\draw[dashed] (9) -- (1);
\draw[dashed] (9) -- (2);
\draw[dashed] (9) -- (3);
\draw[dashed] (9) -- (5);
\draw[thick, line width=2pt] (9) -- (6);

\foreach \i [count=\n] in \nodeorder {
  \pgfmathsetmacro{\angle}{360/\nodecountvalue * (\n-1) + \rotateby}
  \filldraw[fill=white] (\i) circle (2pt) node[anchor=180+\angle] {\scriptsize \i};
}

\draw (7) node[left, xshift = -0.4cm] {SW22:};
\end{tikzpicture}

%
%
%
\phantomsection \label{SW23}\begin{tikzpicture}

\def\nodeorder{1,4,2,9,5,10,3,7,6,8}
    
\newcount\nodecount
\nodecount=0
\foreach \i in \nodeorder {
  \global\advance\nodecount by 1
}
\pgfmathsetmacro{\nodecountvalue}{\the\nodecount}
\def\rotateby{1*360/\nodecountvalue}
   
\foreach \i [count=\n] in \nodeorder {
  \coordinate (\i) at ({360/\nodecountvalue * (\n-1) + \rotateby}:1.25);
}

\draw[thick, double distance = 2pt] (4) -- (1);
\draw[thick, line width=2pt] (4) -- (2);
\draw[dashed] (5) -- (2);
\draw[dashed] (5) -- (3);
\draw[dashed] (6) -- (3);
\draw[dashed] (6) -- (4);
\draw[dashed] (7) -- (1);
\draw[thick, line width=2pt] (7) -- (3);
\draw[dashed] (7) -- (5);
\draw[dashed] (7) -- (6);
\draw[dashed] (8) -- (1);
\draw[dashed] (8) -- (3);
\draw[dashed] (8) -- (4);
\draw[thick, line width=2pt] (8) -- (5);
\draw[dashed] (8) -- (7);
\draw[dashed] (9) -- (1);
\draw[thick, double distance = 2pt] (9) -- (2);
\draw[dashed] (9) -- (3);
\draw[dashed] (9) -- (5);
\draw[dashed] (9) -- (6);
\draw[dashed] (10) -- (1);
\draw[dashed] (10) -- (2);
\draw[dashed] (10) -- (3);
\draw[thick, line width=2pt] (10) -- (6);
\draw[dashed] (10) -- (7);

\foreach \i [count=\n] in \nodeorder {
  \pgfmathsetmacro{\angle}{360/\nodecountvalue * (\n-1) + \rotateby}
  \filldraw[fill=white] (\i) circle (2pt) node[anchor=180+\angle] {\scriptsize \i};
}

\draw (9) node[left, xshift = -0.4cm] {SW23:};
\end{tikzpicture}
\qquad 
%
%
\phantomsection \label{SW24}\begin{tikzpicture}

\def\nodeorder{4,1,7,5,8,2,6,9,10,3}
    
\newcount\nodecount
\nodecount=0
\foreach \i in \nodeorder {
  \global\advance\nodecount by 1
}
\pgfmathsetmacro{\nodecountvalue}{\the\nodecount}
\def\rotateby{0*360/\nodecountvalue}
   
\foreach \i [count=\n] in \nodeorder {
  \coordinate (\i) at ({360/\nodecountvalue * (\n-1) + \rotateby}:1.25);
}

\draw[thick, double distance = 2.5pt] (4) -- (1);
\draw[thick, double distance = .3pt] (4) -- (1);
\draw[thick, line width=2pt] (4) -- (2);
\draw[dashed] (5) -- (2);
\draw[dashed] (5) -- (3);
\draw[dashed] (6) -- (1);
\draw[thick, line width=2pt] (6) -- (3);
\draw[dashed] (6) -- (5);
\draw[thick, double distance = 2.5pt] (7) -- (1);
\draw[thick, double distance = .3pt] (7) -- (1);
\draw[dashed] (7) -- (3);
\draw[dashed] (7) -- (4);
\draw[thick, line width=2pt] (7) -- (5);
\draw[dashed] (7) -- (6);
\draw[dashed] (8) -- (1);
\draw[thick, double distance = 2.5pt] (8) -- (2);
\draw[thick, double distance = .3pt] (8) -- (2);
\draw[dashed] (8) -- (3);
\draw[thick, double distance = 2.5pt] (8) -- (5);
\draw[thick, double distance = .3pt] (8) -- (5);
\draw[dashed] (9) -- (2);
\draw[dashed] (9) -- (3);
\draw[dashed] (9) -- (4);
\draw[dashed] (9) -- (6);
\draw[dashed] (9) -- (8);
\draw[dashed] (10) -- (1);
\draw[dashed] (10) -- (2);
\draw[dashed] (10) -- (3);
\draw[dashed] (10) -- (4);
\draw[dashed] (10) -- (6);
\draw[thick, line width=2pt] (10) -- (9);

\foreach \i [count=\n] in \nodeorder {
  \pgfmathsetmacro{\angle}{360/\nodecountvalue * (\n-1) + \rotateby}
  \filldraw[fill=white] (\i) circle (2pt) node[anchor=180+\angle] {\scriptsize \i};
}

\draw (8) node[left, xshift = -0.4cm] {SW24:};
\end{tikzpicture}
\qquad 
%
%
\phantomsection \label{SW25}\begin{tikzpicture}

  \def\nodeorder{1,4,6,2,7,9,10,11,13,12,8,3,5}
    
\newcount\nodecount
\nodecount=0
\foreach \i in \nodeorder {
  \global\advance\nodecount by 1
}
\pgfmathsetmacro{\nodecountvalue}{\the\nodecount}
\def\rotateby{.25*360/\nodecountvalue}
   
\foreach \i [count=\n] in \nodeorder {
  \coordinate (\i) at ({360/\nodecountvalue * (\n-1) + \rotateby}:2.);
}

\draw[thick] (4) -- (1);
\draw[thick, line width=2pt] (4) -- (2);
\draw[dashed] (5) -- (1);
\draw[thick, line width=2pt] (5) -- (3);
\draw[thick, double distance = 2pt] (6) -- (2);
\draw[dashed] (6) -- (3);
\draw[thick, double distance = 2pt] (6) -- (4);
\draw[dashed] (6) -- (5);
\draw[dashed] (7) -- (1);
\draw[thick, double distance = 2pt] (7) -- (2);
\draw[dashed] (7) -- (3);
\draw[thick, line width=2pt] (7) -- (6);
\draw[dashed] (8) -- (1);
\draw[dashed] (8) -- (2);
\draw[dashed] (8) -- (3);
\draw[dashed] (8) -- (5);
\draw[dashed] (9) -- (1);
\draw[thick, line width=2pt] (9) -- (2);
\draw[dashed] (9) -- (3);
\draw[dashed] (9) -- (4);
\draw[dashed] (9) -- (5);
\draw[thick, double distance = 2pt] (9) -- (7);
\draw[dashed] (9) -- (8);
\draw[dashed] (10) -- (1);
\draw[dashed] (10) -- (3);
\draw[dashed] (10) -- (4);
\draw[dashed] (10) -- (5);
\draw[dashed] (10) -- (6);
\draw[dashed] (10) -- (8);
\draw[thick] (10) -- (9);
\draw[dashed] (11) -- (1);
\draw[dashed] (11) -- (3);
\draw[dashed] (11) -- (4);
\draw[dashed] (11) -- (5);
\draw[dashed] (11) -- (7);
\draw[dashed] (11) -- (8);
\draw[dashed] (11) -- (10);
\draw[dashed] (12) -- (1);
\draw[dashed] (12) -- (2);
\draw[dashed] (12) -- (3);
\draw[dashed] (12) -- (4);
\draw[dashed] (12) -- (5);
\draw[thick, line width=2pt] (12) -- (8);
\draw[dashed] (12) -- (10);
\draw[dashed] (12) -- (11);
\draw[dashed] (13) -- (1);
\draw[dashed] (13) -- (3);
\draw[dashed] (13) -- (4);
\draw[dashed] (13) -- (5);
\draw[dashed] (13) -- (6);
\draw[dashed] (13) -- (7);
\draw[dashed] (13) -- (8);
\draw[thick, line width=2pt] (13) -- (11);
\draw[dashed] (13) -- (12);

\foreach \i [count=\n] in \nodeorder {
  \pgfmathsetmacro{\angle}{360/\nodecountvalue * (\n-1) + \rotateby}
  \filldraw[fill=white] (\i) circle (2pt) node[anchor=180+\angle] {\scriptsize \i};
}

\draw (9) node[left, xshift = -0.4cm] {SW25:};
\end{tikzpicture}

%
%
%
\phantomsection \label{SW26}\begin{tikzpicture}

  \def\nodeorder{8,1,4,2,7,6,5,3}
    
\newcount\nodecount
\nodecount=0
\foreach \i in \nodeorder {
  \global\advance\nodecount by 1
}
\pgfmathsetmacro{\nodecountvalue}{\the\nodecount}
\def\rotateby{.5*360/\nodecountvalue}
   
\foreach \i [count=\n] in \nodeorder {
  \coordinate (\i) at ({360/\nodecountvalue * (\n-1) + \rotateby}:1.25);
}

\draw[thick, double distance = 2.5pt] (4) -- (1);
\draw[thick, double distance = .3pt] (4) -- (1);
\draw[thick, line width=2pt] (4) -- (2);
\draw[dashed] (5) -- (1);
\draw[thick, line width=2pt] (5) -- (3);
\draw[dashed] (6) -- (3);
\draw[dashed] (6) -- (4);
\draw[dashed] (7) -- (2);
\draw[dashed] (7) -- (3);
\draw[dashed] (7) -- (5);
\draw[thick, line width=2pt] (7) -- (6);
\draw[thick, line width=2pt] (8) -- (1);
\draw[dashed] (8) -- (2);
\draw[dashed] (8) -- (3);

\foreach \i [count=\n] in \nodeorder {
  \pgfmathsetmacro{\angle}{360/\nodecountvalue * (\n-1) + \rotateby}
  \filldraw[fill=white] (\i) circle (2pt) node[anchor=180+\angle] {\scriptsize \i};
}

\draw (4) node[left, xshift = -0.4cm] {SW26:};
\end{tikzpicture}
\qquad 
\phantomsection \label{SW27}\begin{tikzpicture}

\coordinate (1) at (-1,1);
\coordinate (2) at (0,0);
\coordinate (3) at (1,0);
\coordinate (4) at (0,1);
\coordinate (5) at (-1,0);
\coordinate (6) at (1,1);

\draw[thick, double distance = 2.5pt] (3) -- (2);
\draw[thick, double distance = .3pt] (3) -- (2);
\draw[dashed] (4) -- (1);
\draw[thick] (4) -- (2);
\draw[dashed] (5) -- (2);
\draw[dashed] (6) -- (3);
\draw[thick, double distance = 2.5pt] (6) -- (4);
\draw[thick, double distance = .3pt] (6) -- (4);

\filldraw[fill=white] (1) circle (2pt) node[above] {\scriptsize 1};
\filldraw[fill=white] (2) circle (2pt) node[below] {\scriptsize 2};
\filldraw[fill=white] (3) circle (2pt) node[below] {\scriptsize 3};
\filldraw[fill=white] (4) circle (2pt) node[above] {\scriptsize 4};
\filldraw[fill=white] (5) circle (2pt) node[below] {\scriptsize 5};
\filldraw[fill=white] (6) circle (2pt) node[above] {\scriptsize 6};

\draw (1) node[above, xshift = 0.3cm, yshift = 0.4cm] {SW27:};
\end{tikzpicture}
\qquad 
%
%
\phantomsection \label{SW28}\begin{tikzpicture}

  \def\nodeorder{10,8,1,6,4,2,7,5,3,9,11}
    
\newcount\nodecount
\nodecount=0
\foreach \i in \nodeorder {
  \global\advance\nodecount by 1
}
\pgfmathsetmacro{\nodecountvalue}{\the\nodecount}
\def\rotateby{-.25*360/\nodecountvalue}
   
\foreach \i [count=\n] in \nodeorder {
  \coordinate (\i) at ({360/\nodecountvalue * (\n-1) + \rotateby}:1.25);
}

\draw[thick] (4) -- (1);
\draw[thick, line width=2pt] (4) -- (2);
\draw[dashed] (5) -- (1);
\draw[thick, line width=2pt] (5) -- (3);
\draw[thick, double distance = 2.5pt] (6) -- (1);
\draw[thick, double distance = .3pt] (6) -- (1);
\draw[dashed] (6) -- (3);
\draw[thick, double distance = 2.5pt] (6) -- (4);
\draw[thick, double distance = .3pt] (6) -- (4);
\draw[dashed] (6) -- (5);
\draw[dashed] (7) -- (1);
\draw[thick, double distance = 2pt] (7) -- (2);
\draw[dashed] (7) -- (3);
\draw[thick, line width=2pt] (8) -- (1);
\draw[dashed] (8) -- (2);
\draw[dashed] (8) -- (3);
\draw[dashed] (8) -- (5);
\draw[dashed] (8) -- (7);
\draw[dashed] (9) -- (2);
\draw[dashed] (9) -- (3);
\draw[dashed] (9) -- (5);
\draw[dashed] (9) -- (6);
\draw[dashed] (9) -- (7);
\draw[dashed] (10) -- (2);
\draw[dashed] (10) -- (3);
\draw[dashed] (10) -- (4);
\draw[dashed] (10) -- (5);
\draw[dashed] (10) -- (7);
\draw[thick, double distance = 2pt] (10) -- (8);
\draw[dashed] (10) -- (9);
\draw[dashed] (11) -- (2);
\draw[dashed] (11) -- (3);
\draw[dashed] (11) -- (4);
\draw[dashed] (11) -- (5);
\draw[dashed] (11) -- (6);
\draw[dashed] (11) -- (7);
\draw[thick, line width=2pt] (11) -- (9);

\foreach \i [count=\n] in \nodeorder {
  \pgfmathsetmacro{\angle}{360/\nodecountvalue * (\n-1) + \rotateby}
  \filldraw[fill=white] (\i) circle (2pt) node[anchor=180+\angle] {\scriptsize \i};
}

\draw (4) node[left, xshift = -0.4cm] {SW28:};
\end{tikzpicture}

\vspace{0.4cm}
%
%
%
\phantomsection \label{SW29}\begin{tikzpicture}

  \def\nodeorder{7,1,4,2,6,8,5,3}
    
\newcount\nodecount
\nodecount=0
\foreach \i in \nodeorder {
  \global\advance\nodecount by 1
}
\pgfmathsetmacro{\nodecountvalue}{\the\nodecount}
\def\rotateby{-.5*360/\nodecountvalue}
   
\foreach \i [count=\n] in \nodeorder {
  \coordinate (\i) at ({360/\nodecountvalue * (\n-1) + \rotateby}:1.25);
}

\draw[thick, double distance = 2.5pt] (4) -- (1);
\draw[thick, double distance = .3pt] (4) -- (1);
\draw[thick, line width=2pt] (4) -- (2);
\draw[dashed] (5) -- (1);
\draw[thick, line width=2pt] (5) -- (3);
\draw[thick, double distance = 2pt] (6) -- (1);
\draw[dashed] (6) -- (2);
\draw[dashed] (6) -- (3);
\draw[thick] (7) -- (1);
\draw[dashed] (7) -- (3);
\draw[dashed] (7) -- (4);
\draw[dashed] (8) -- (2);
\draw[dashed] (8) -- (3);
\draw[dashed] (8) -- (5);
\draw[dashed] (8) -- (7);

\foreach \i [count=\n] in \nodeorder {
  \pgfmathsetmacro{\angle}{360/\nodecountvalue * (\n-1) + \rotateby}
  \filldraw[fill=white] (\i) circle (2pt) node[anchor=180+\angle] {\scriptsize \i};
}

\draw (2) node[left, xshift = -0.3cm] {SW29:};
\end{tikzpicture}
\qquad 
%
%
\phantomsection \label{SW30}\begin{tikzpicture}

  \def\nodeorder{5,7,10,2,11,3,1,8,4,9,6}
    
\newcount\nodecount
\nodecount=0
\foreach \i in \nodeorder {
  \global\advance\nodecount by 1
}
\pgfmathsetmacro{\nodecountvalue}{\the\nodecount}
\def\rotateby{1.75*360/\nodecountvalue}
   
\foreach \i [count=\n] in \nodeorder {
  \coordinate (\i) at ({360/\nodecountvalue * (\n-1) + \rotateby}:1.25);
}

\draw[dashed] (4) -- (1);
\draw[thick, line width=2pt] (4) -- (3);
\draw[dashed] (5) -- (3);
\draw[dashed] (6) -- (1);
\draw[thick, line width=2pt] (6) -- (2);
\draw[dashed] (6) -- (5);
\draw[dashed] (7) -- (2);
\draw[dashed] (7) -- (3);
\draw[dashed] (7) -- (4);
\draw[thick, double distance = 2pt] (7) -- (5);
\draw[dashed] (7) -- (6);
\draw[dashed] (8) -- (1);
\draw[dashed] (8) -- (2);
\draw[dashed] (8) -- (3);
\draw[dashed] (8) -- (5);
\draw[dashed] (9) -- (2);
\draw[dashed] (9) -- (4);
\draw[dashed] (9) -- (5);
\draw[dashed] (9) -- (6);
\draw[dashed] (9) -- (8);
\draw[dashed] (10) -- (1);
\draw[dashed] (10) -- (2);
\draw[dashed] (10) -- (4);
\draw[dashed] (10) -- (5);
\draw[thick, double distance = 2pt] (10) -- (7);
\draw[dashed] (11) -- (1);
\draw[dashed] (11) -- (2);
\draw[dashed] (11) -- (3);
\draw[dashed] (11) -- (6);
\draw[dashed] (11) -- (9);
\draw[dashed] (11) -- (10);

\foreach \i [count=\n] in \nodeorder {
  \pgfmathsetmacro{\angle}{360/\nodecountvalue * (\n-1) + \rotateby}
  \filldraw[fill=white] (\i) circle (2pt) node[anchor=180+\angle] {\scriptsize \i};
}

\draw (10) node[left, xshift = -0.4cm] {SW30:};
\end{tikzpicture}
\qquad 
%
%
%
\phantomsection \label{SW31}\begin{tikzpicture}

\coordinate (1) at (1,0);
\coordinate (2) at (0,1);
\coordinate (3) at (-1,1);
\coordinate (4) at (1,1);
\coordinate (5) at (-1,0);
\coordinate (6) at (0,0);

\draw[thick, double distance = 2pt] (3) -- (2);
\draw[dashed] (4) -- (1);
\draw[thick, double distance = 2pt] (4) -- (2);
\draw[dashed] (5) -- (2);
\draw[thick, double distance = 2pt] (6) -- (1);
\draw[dashed] (6) -- (3);
\draw[thick, double distance = 2pt] (6) -- (5);

\filldraw[fill=white] (1) circle (2pt) node[below] {\scriptsize 1};
\filldraw[fill=white] (2) circle (2pt) node[above] {\scriptsize 2};
\filldraw[fill=white] (3) circle (2pt) node[above] {\scriptsize 3};
\filldraw[fill=white] (4) circle (2pt) node[above] {\scriptsize 4};
\filldraw[fill=white] (5) circle (2pt) node[below] {\scriptsize 5};
\filldraw[fill=white] (6) circle (2pt) node[below] {\scriptsize 6};

\draw (3) node[above, xshift = 0.4cm, yshift = 0.4cm] {SW31:};
\end{tikzpicture}

%
%
%
\phantomsection \label{SW32}\begin{tikzpicture}

  \def\nodeorder{8,4,1,9,2,11,6,3,5,10,7}
    
\newcount\nodecount
\nodecount=0
\foreach \i in \nodeorder {
  \global\advance\nodecount by 1
}
\pgfmathsetmacro{\nodecountvalue}{\the\nodecount}
\def\rotateby{1.75*360/\nodecountvalue}
   
\foreach \i [count=\n] in \nodeorder {
  \coordinate (\i) at ({360/\nodecountvalue * (\n-1) + \rotateby}:1.25);
}

\draw[thick] (4) -- (1);
\draw[thick, line width=2pt] (4) -- (2);
\draw[dashed] (5) -- (3);
\draw[dashed] (5) -- (4);
\draw[dashed] (6) -- (1);
\draw[thick, line width=2pt] (6) -- (3);
\draw[dashed] (6) -- (5);
\draw[dashed] (7) -- (2);
\draw[dashed] (7) -- (3);
\draw[dashed] (7) -- (6);
\draw[dashed] (8) -- (1);
\draw[dashed] (8) -- (3);
\draw[thick, double distance = 2.5pt] (8) -- (4);
\draw[thick, double distance = .3pt] (8) -- (4);
\draw[dashed] (8) -- (6);
\draw[dashed] (8) -- (7);
\draw[dashed] (9) -- (1);
\draw[dashed] (9) -- (2);
\draw[dashed] (9) -- (3);
\draw[dashed] (9) -- (6);
\draw[dashed] (10) -- (1);
\draw[dashed] (10) -- (3);
\draw[thick, double distance = 2pt] (10) -- (4);
\draw[dashed] (10) -- (5);
\draw[dashed] (10) -- (7);
\draw[dashed] (10) -- (9);
\draw[dashed] (11) -- (1);
\draw[dashed] (11) -- (2);
\draw[dashed] (11) -- (3);
\draw[dashed] (11) -- (5);
\draw[dashed] (11) -- (6);
\draw[dashed] (11) -- (7);
\draw[dashed] (11) -- (9);

\foreach \i [count=\n] in \nodeorder {
  \pgfmathsetmacro{\angle}{360/\nodecountvalue * (\n-1) + \rotateby}
  \filldraw[fill=white] (\i) circle (2pt) node[anchor=180+\angle] {\scriptsize \i};
}

\draw (1) node[left, xshift = -0.4cm] {SW32:};
\end{tikzpicture}
\qquad 
%
%
\phantomsection \label{SW33}\begin{tikzpicture}

  \def\nodeorder{11,1,4,2,12,10,9,3,5,6,7,13,8,14}
    
\newcount\nodecount
\nodecount=0
\foreach \i in \nodeorder {
  \global\advance\nodecount by 1
}
\pgfmathsetmacro{\nodecountvalue}{\the\nodecount}
\def\rotateby{.5*360/\nodecountvalue}
   
\foreach \i [count=\n] in \nodeorder {
  \coordinate (\i) at ({360/\nodecountvalue * (\n-1) + \rotateby}:1.75);
}

\draw[thick, double distance = 2pt] (4) -- (1);
\draw[thick, line width=2pt] (4) -- (2);
\draw[dashed] (5) -- (1);
\draw[thick, line width=2pt] (5) -- (3);
\draw[dashed] (6) -- (2);
\draw[dashed] (6) -- (3);
\draw[dashed] (6) -- (5);
\draw[dashed] (7) -- (1);
\draw[dashed] (7) -- (2);
\draw[dashed] (7) -- (3);
\draw[dashed] (7) -- (5);
\draw[dashed] (8) -- (3);
\draw[dashed] (8) -- (4);
\draw[dashed] (8) -- (5);
\draw[dashed] (8) -- (6);
\draw[dashed] (8) -- (7);
\draw[dashed] (9) -- (2);
\draw[dashed] (9) -- (3);
\draw[dashed] (9) -- (4);
\draw[dashed] (9) -- (5);
\draw[dashed] (10) -- (1);
\draw[dashed] (10) -- (3);
\draw[dashed] (10) -- (4);
\draw[dashed] (10) -- (6);
\draw[dashed] (10) -- (7);
\draw[dashed] (10) -- (8);
\draw[dashed] (10) -- (9);
\draw[dashed] (11) -- (1);
\draw[dashed] (11) -- (2);
\draw[dashed] (11) -- (3);
\draw[dashed] (11) -- (6);
\draw[dashed] (11) -- (8);
\draw[dashed] (11) -- (9);
\draw[dashed] (12) -- (1);
\draw[thick, line width=2pt] (12) -- (2);
\draw[dashed] (12) -- (3);
\draw[dashed] (12) -- (4);
\draw[dashed] (12) -- (5);
\draw[dashed] (12) -- (6);
\draw[thick, double distance = 2pt] (12) -- (10);
\draw[dashed] (12) -- (11);
\draw[dashed] (13) -- (1);
\draw[dashed] (13) -- (2);
\draw[dashed] (13) -- (3);
\draw[dashed] (13) -- (4);
\draw[dashed] (13) -- (5);
\draw[dashed] (13) -- (6);
\draw[dashed] (13) -- (8);
\draw[dashed] (13) -- (9);
\draw[dashed] (14) -- (1);
\draw[dashed] (14) -- (3);
\draw[dashed] (14) -- (4);
\draw[dashed] (14) -- (5);
\draw[dashed] (14) -- (6);
\draw[dashed] (14) -- (7);
\draw[thick, line width=2pt] (14) -- (8);
\draw[dashed] (14) -- (9);
\draw[dashed] (14) -- (11);
\draw[dashed] (14) -- (13);

\foreach \i [count=\n] in \nodeorder {
  \pgfmathsetmacro{\angle}{360/\nodecountvalue * (\n-1) + \rotateby}
  \filldraw[fill=white] (\i) circle (2pt) node[anchor=180+\angle] {\scriptsize \i};
}

\draw (12) node[left, xshift = -0.4cm, yshift=0.3cm] {SW33:};
\end{tikzpicture}
\qquad 
%
%
\phantomsection \label{SW34}\begin{tikzpicture}

\def\nodeorder{6,1,4,7,10,12,9,3,5,11,13,8,2}
    
\newcount\nodecount
\nodecount=0
\foreach \i in \nodeorder {
  \global\advance\nodecount by 1
}
\pgfmathsetmacro{\nodecountvalue}{\the\nodecount}
\def\rotateby{1.25*360/\nodecountvalue}
   
\foreach \i [count=\n] in \nodeorder {
  \coordinate (\i) at ({360/\nodecountvalue * (\n-1) + \rotateby}:1.75);
}

\draw[thick, double distance = 2.5pt] (4) -- (1);
\draw[thick, double distance = .3pt] (4) -- (1);
\draw[thick, line width=2pt] (4) -- (2);
\draw[dashed] (5) -- (1);
\draw[thick, line width=2pt] (5) -- (3);
\draw[thick, double distance = 2pt] (6) -- (1);
\draw[dashed] (6) -- (3);
\draw[dashed] (6) -- (4);
\draw[thick] (7) -- (1);
\draw[dashed] (7) -- (2);
\draw[dashed] (7) -- (3);
\draw[thick, double distance = 2.5pt] (7) -- (4);
\draw[thick, double distance = .3pt] (7) -- (4);
\draw[dashed] (7) -- (5);
\draw[dashed] (8) -- (1);
\draw[dashed] (8) -- (2);
\draw[dashed] (8) -- (3);
\draw[dashed] (9) -- (2);
\draw[dashed] (9) -- (3);
\draw[dashed] (9) -- (5);
\draw[dashed] (9) -- (6);
\draw[dashed] (9) -- (7);
\draw[dashed] (9) -- (8);
\draw[dashed] (10) -- (2);
\draw[dashed] (10) -- (3);
\draw[dashed] (10) -- (4);
\draw[dashed] (10) -- (5);
\draw[dashed] (10) -- (6);
\draw[thick, double distance = 2pt] (10) -- (7);
\draw[dashed] (10) -- (8);
\draw[dashed] (11) -- (1);
\draw[dashed] (11) -- (2);
\draw[dashed] (11) -- (3);
\draw[dashed] (11) -- (5);
\draw[dashed] (11) -- (6);
\draw[dashed] (11) -- (7);
\draw[dashed] (11) -- (8);
\draw[dashed] (12) -- (1);
\draw[dashed] (12) -- (2);
\draw[dashed] (12) -- (3);
\draw[thick, line width=2pt] (12) -- (4);
\draw[dashed] (12) -- (5);
\draw[dashed] (12) -- (6);
\draw[dashed] (12) -- (8);
\draw[dashed] (12) -- (9);
\draw[dashed] (13) -- (1);
\draw[dashed] (13) -- (2);
\draw[dashed] (13) -- (3);
\draw[dashed] (13) -- (5);
\draw[dashed] (13) -- (6);
\draw[dashed] (13) -- (8);
\draw[dashed] (13) -- (9);
\draw[dashed] (13) -- (10);
\draw[thick, line width=2pt] (13) -- (11);

\foreach \i [count=\n] in \nodeorder {
  \pgfmathsetmacro{\angle}{360/\nodecountvalue * (\n-1) + \rotateby}
  \filldraw[fill=white] (\i) circle (2pt) node[anchor=180+\angle] {\scriptsize \i};
}

\draw (7) node[left, xshift = -0.4cm, yshift=.2cm] {SW34:};
\end{tikzpicture}

%
%
%
\phantomsection \label{SW35}\begin{tikzpicture}

\def\nodeorder{9,1,14,11,4,2,7,12,13,10,5,3,8,6}
        
\newcount\nodecount
\nodecount=0
\foreach \i in \nodeorder {
  \global\advance\nodecount by 1
}
\pgfmathsetmacro{\nodecountvalue}{\the\nodecount}
\def\rotateby{4*360/\nodecountvalue}
   
\foreach \i [count=\n] in \nodeorder {
  \coordinate (\i) at ({360/\nodecountvalue * (\n-1) + \rotateby}:2.25);
}

\draw[dashed] (4) -- (1);
\draw[thick, line width=2pt] (4) -- (2);
\draw[dashed] (5) -- (1);
\draw[thick, line width=2pt] (5) -- (3);
\draw[dashed] (6) -- (3);
\draw[dashed] (6) -- (4);
\draw[dashed] (7) -- (2);
\draw[dashed] (7) -- (5);
\draw[dashed] (7) -- (6);
\draw[dashed] (8) -- (1);
\draw[dashed] (8) -- (2);
\draw[dashed] (8) -- (3);
\draw[thick, double distance = 2pt] (8) -- (6);
\draw[dashed] (8) -- (7);
\draw[thick, double distance = 2pt] (9) -- (1);
\draw[dashed] (9) -- (2);
\draw[dashed] (9) -- (3);
\draw[dashed] (9) -- (5);
\draw[dashed] (9) -- (6);
\draw[dashed] (10) -- (2);
\draw[dashed] (10) -- (3);
\draw[dashed] (10) -- (4);
\draw[dashed] (10) -- (5);
\draw[dashed] (10) -- (7);
\draw[dashed] (10) -- (8);
\draw[dashed] (10) -- (9);
\draw[dashed] (11) -- (2);
\draw[dashed] (11) -- (3);
\draw[dashed] (11) -- (4);
\draw[dashed] (11) -- (5);
\draw[dashed] (11) -- (6);
\draw[dashed] (11) -- (7);
\draw[dashed] (11) -- (8);
\draw[dashed] (12) -- (1);
\draw[dashed] (12) -- (2);
\draw[dashed] (12) -- (3);
\draw[dashed] (12) -- (4);
\draw[dashed] (12) -- (5);
\draw[dashed] (12) -- (7);
\draw[dashed] (12) -- (9);
\draw[dashed] (12) -- (11);
\draw[dashed] (13) -- (1);
\draw[dashed] (13) -- (2);
\draw[dashed] (13) -- (3);
\draw[dashed] (13) -- (4);
\draw[dashed] (13) -- (5);
\draw[dashed] (13) -- (6);
\draw[dashed] (13) -- (7);
\draw[thick, line width=2pt] (13) -- (10);
\draw[dashed] (13) -- (12);
\draw[dashed] (14) -- (1);
\draw[dashed] (14) -- (2);
\draw[dashed] (14) -- (3);
\draw[dashed] (14) -- (4);
\draw[dashed] (14) -- (5);
\draw[dashed] (14) -- (6);
\draw[dashed] (14) -- (7);
\draw[dashed] (14) -- (9);
\draw[thick, line width=2pt] (14) -- (11);

\foreach \i [count=\n] in \nodeorder {
  \pgfmathsetmacro{\angle}{360/\nodecountvalue * (\n-1) + \rotateby}
  \filldraw[fill=white] (\i) circle (2pt) node[anchor=180+\angle] {\scriptsize \i};
}

\draw (1) node[left, xshift = -0.4cm] {SW35:};
\end{tikzpicture}
\qquad 
%
%
\phantomsection \label{SW36}\begin{tikzpicture}

    \def\nodeorder{2,9,10,14,6,3,5,12,7,13,15,11,1,4,8}
        
\newcount\nodecount
\nodecount=0
\foreach \i in \nodeorder {
  \global\advance\nodecount by 1
}
\pgfmathsetmacro{\nodecountvalue}{\the\nodecount}
\def\rotateby{90+0*360/\nodecountvalue}
   
\foreach \i [count=\n] in \nodeorder {
  \coordinate (\i) at ({360/\nodecountvalue * (\n-1) + \rotateby}:2.5);
}

\draw[thick] (4) -- (1);
\draw[thick, line width=2pt] (4) -- (2);
\draw[dashed] (5) -- (2);
\draw[dashed] (5) -- (3);
\draw[dashed] (6) -- (1);
\draw[thick, line width=2pt] (6) -- (3);
\draw[dashed] (6) -- (5);
\draw[dashed] (7) -- (3);
\draw[dashed] (7) -- (4);
\draw[dashed] (7) -- (6);
\draw[dashed] (8) -- (1);
\draw[thick, double distance = 2pt] (8) -- (2);
\draw[dashed] (8) -- (3);
\draw[thick, double distance = 2pt] (8) -- (4);
\draw[dashed] (8) -- (6);
\draw[dashed] (9) -- (1);
\draw[thick, double distance = 2pt] (9) -- (2);
\draw[dashed] (9) -- (3);
\draw[dashed] (9) -- (5);
\draw[dashed] (9) -- (7);
\draw[dashed] (9) -- (8);
\draw[dashed] (10) -- (1);
\draw[thick, line width=2pt] (10) -- (2);
\draw[dashed] (10) -- (3);
\draw[thick] (10) -- (4);
\draw[dashed] (10) -- (5);
\draw[dashed] (10) -- (6);
\draw[dashed] (10) -- (7);
\draw[thick, double distance = 2pt] (10) -- (9);
\draw[dashed] (11) -- (1);
\draw[dashed] (11) -- (3);
\draw[dashed] (11) -- (4);
\draw[dashed] (11) -- (5);
\draw[dashed] (11) -- (6);
\draw[dashed] (11) -- (7);
\draw[dashed] (11) -- (9);
\draw[dashed] (12) -- (1);
\draw[dashed] (12) -- (3);
\draw[dashed] (12) -- (4);
\draw[dashed] (12) -- (5);
\draw[dashed] (12) -- (6);
\draw[dashed] (12) -- (7);
\draw[dashed] (12) -- (8);
\draw[dashed] (12) -- (10);
\draw[dashed] (12) -- (11);
\draw[dashed] (13) -- (1);
\draw[dashed] (13) -- (2);
\draw[dashed] (13) -- (3);
\draw[dashed] (13) -- (4);
\draw[dashed] (13) -- (5);
\draw[dashed] (13) -- (6);
\draw[dashed] (13) -- (7);
\draw[dashed] (13) -- (8);
\draw[dashed] (13) -- (11);
\draw[dashed] (14) -- (1);
\draw[dashed] (14) -- (3);
\draw[dashed] (14) -- (4);
\draw[dashed] (14) -- (5);
\draw[dashed] (14) -- (6);
\draw[dashed] (14) -- (7);
\draw[dashed] (14) -- (8);
\draw[dashed] (14) -- (9);
\draw[thick] (14) -- (10);
\draw[dashed] (14) -- (11);
\draw[dashed] (15) -- (1);
\draw[dashed] (15) -- (3);
\draw[dashed] (15) -- (4);
\draw[dashed] (15) -- (5);
\draw[dashed] (15) -- (6);
\draw[dashed] (15) -- (7);
\draw[dashed] (15) -- (8);
\draw[dashed] (15) -- (9);
\draw[thick, line width=2pt] (15) -- (11);
\draw[dashed] (15) -- (12);
\draw[dashed] (15) -- (13);

\foreach \i [count=\n] in \nodeorder {
  \pgfmathsetmacro{\angle}{360/\nodecountvalue * (\n-1) + \rotateby}
  \filldraw[fill=white] (\i) circle (2pt) node[anchor=180+\angle] {\scriptsize \i};
}

\draw (9) node[left, xshift = -0.4cm,yshift=.3cm] {SW36:};
\end{tikzpicture}

%
%
%
\phantomsection \label{SW37}\begin{tikzpicture}

    \def\nodeorder{4,9,10,8,6,5,3,2,11,1,7}
        
\newcount\nodecount
\nodecount=0
\foreach \i in \nodeorder {
  \global\advance\nodecount by 1
}
\pgfmathsetmacro{\nodecountvalue}{\the\nodecount}
\def\rotateby{90+0*360/\nodecountvalue}
   
\foreach \i [count=\n] in \nodeorder {
  \coordinate (\i) at ({360/\nodecountvalue * (\n-1) + \rotateby}:1.5);
}

\draw[dashed] (4) -- (2);
\draw[dashed] (4) -- (3);
\draw[dashed] (5) -- (1);
\draw[thick, line width=2pt] (5) -- (3);
\draw[dashed] (5) -- (4);
\draw[dashed] (6) -- (1);
\draw[thick, line width=2pt] (6) -- (2);
\draw[dashed] (6) -- (4);
\draw[thick, double distance = 2pt] (7) -- (1);
\draw[dashed] (7) -- (3);
\draw[thick] (7) -- (4);
\draw[dashed] (7) -- (6);
\draw[dashed] (8) -- (2);
\draw[dashed] (8) -- (5);
\draw[dashed] (8) -- (6);
\draw[dashed] (8) -- (7);
\draw[dashed] (9) -- (1);
\draw[dashed] (9) -- (2);
\draw[thick] (9) -- (4);
\draw[dashed] (9) -- (5);
\draw[dashed] (9) -- (7);
\draw[dashed] (10) -- (1);
\draw[dashed] (10) -- (2);
\draw[dashed] (10) -- (3);
\draw[dashed] (10) -- (7);
\draw[dashed] (10) -- (8);
\draw[thick, double distance = 2pt] (10) -- (9);
\draw[dashed] (11) -- (1);
\draw[dashed] (11) -- (2);
\draw[dashed] (11) -- (3);
\draw[dashed] (11) -- (6);
\draw[dashed] (11) -- (8);
\draw[dashed] (11) -- (9);

\foreach \i [count=\n] in \nodeorder {
  \pgfmathsetmacro{\angle}{360/\nodecountvalue * (\n-1) + \rotateby}
  \filldraw[fill=white] (\i) circle (2pt) node[anchor=180+\angle] {\scriptsize \i};
}

\draw (9) node[left, xshift = -0.4cm] {SW37:};
\end{tikzpicture}
\qquad 
%
%
%
\phantomsection \label{SW38}\begin{tikzpicture}

\coordinate (1) at (1,0);
\coordinate (3) at (0,1);
\coordinate (2) at (-1,1);
\coordinate (4) at (1,1);
\coordinate (5) at (-1,0);
\coordinate (6) at (0,0);

\draw[thick, double distance = 2.5pt] (3) -- (2);
\draw[thick, double distance = .3pt] (3) -- (2);
\draw[dashed] (4) -- (1);
\draw[thick] (4) -- (3);
\draw[thick] (5) -- (1);
\draw[dashed] (5) -- (3);
\draw[thick, double distance = 2.5pt] (6) -- (1);
\draw[thick, double distance = .3pt] (6) -- (1);
\draw[dashed] (6) -- (2);

\filldraw[fill=white] (1) circle (2pt) node[below] {\scriptsize 1};
\filldraw[fill=white] (2) circle (2pt) node[above] {\scriptsize 2};
\filldraw[fill=white] (3) circle (2pt) node[above] {\scriptsize 3};
\filldraw[fill=white] (4) circle (2pt) node[above] {\scriptsize 4};
\filldraw[fill=white] (5) circle (2pt) node[below] {\scriptsize 5};
\filldraw[fill=white] (6) circle (2pt) node[below] {\scriptsize 6};

\draw (2) node[above, xshift = 0.4cm, yshift=0.4cm] {SW38:};
\end{tikzpicture}
 \qquad 
%
%
\phantomsection \label{SW39}\begin{tikzpicture}

\def\nodeorder{
        7, 14, 12, 16,  3,5, 13,10, 11,   18,9, 6,15, 17, 2, 4, 1, 8}
        
\newcount\nodecount
\nodecount=0
\foreach \i in \nodeorder {
  \global\advance\nodecount by 1
}
\pgfmathsetmacro{\nodecountvalue}{\the\nodecount}
\def\rotateby{5*360/\nodecountvalue}
   
\foreach \i [count=\n] in \nodeorder {
  \coordinate (\i) at ({360/\nodecountvalue * (\n-1) + \rotateby}:2.5);
}

\draw[thick, double distance = 2.5pt] (4) -- (1);
\draw[thick, double distance = .3pt] (4) -- (1);
\draw[thick, line width=2pt] (4) -- (2);
\draw[dashed] (5) -- (1);
\draw[thick, line width=2pt] (5) -- (3);
\draw[dashed] (6) -- (3);
\draw[dashed] (6) -- (4);
\draw[dashed] (6) -- (5);
\draw[dashed] (7) -- (1);
\draw[dashed] (7) -- (3);
\draw[dashed] (7) -- (4);
\draw[thick, double distance = 2pt] (8) -- (1);
\draw[dashed] (8) -- (2);
\draw[dashed] (8) -- (3);
\draw[dashed] (8) -- (5);
\draw[dashed] (8) -- (6);
\draw[dashed] (8) -- (7);
\draw[dashed] (9) -- (2);
\draw[dashed] (9) -- (3);
\draw[dashed] (9) -- (4);
\draw[dashed] (9) -- (5);
\draw[dashed] (9) -- (7);
\draw[dashed] (10) -- (1);
\draw[dashed] (10) -- (2);
\draw[dashed] (10) -- (3);
\draw[dashed] (10) -- (6);
\draw[dashed] (10) -- (8);
\draw[dashed] (10) -- (9);
\draw[dashed] (11) -- (2);
\draw[dashed] (11) -- (3);
\draw[dashed] (11) -- (5);
\draw[dashed] (11) -- (6);
\draw[dashed] (11) -- (7);
\draw[dashed] (11) -- (9);
\draw[dashed] (11) -- (10);
\draw[dashed] (12) -- (1);
\draw[dashed] (12) -- (2);
\draw[dashed] (12) -- (3);
\draw[dashed] (12) -- (4);
\draw[dashed] (12) -- (5);
\draw[dashed] (12) -- (8);
\draw[dashed] (12) -- (10);
\draw[dashed] (12) -- (11);
\draw[dashed] (13) -- (1);
\draw[dashed] (13) -- (2);
\draw[dashed] (13) -- (3);
\draw[dashed] (13) -- (5);
\draw[dashed] (13) -- (6);
\draw[dashed] (13) -- (7);
\draw[dashed] (13) -- (9);
\draw[dashed] (13) -- (11);
\draw[dashed] (13) -- (12);
\draw[dashed] (14) -- (1);
\draw[dashed] (14) -- (2);
\draw[dashed] (14) -- (3);
\draw[dashed] (14) -- (4);
\draw[dashed] (14) -- (5);
\draw[dashed] (14) -- (6);
\draw[thick, double distance = 2pt] (14) -- (7);
\draw[dashed] (14) -- (8);
\draw[dashed] (14) -- (9);
\draw[dashed] (14) -- (11);
\draw[thick, double distance = 2.5pt] (14) -- (12);
\draw[thick, double distance = .3pt] (14) -- (12);
\draw[dashed] (15) -- (1);
\draw[dashed] (15) -- (2);
\draw[dashed] (15) -- (3);
\draw[dashed] (15) -- (4);
\draw[dashed] (15) -- (5);
\draw[dashed] (15) -- (6);
\draw[dashed] (15) -- (7);
\draw[dashed] (15) -- (10);
\draw[dashed] (15) -- (11);
\draw[dashed] (15) -- (13);
\draw[dashed] (15) -- (14);
\draw[dashed] (16) -- (1);
\draw[dashed] (16) -- (2);
\draw[dashed] (16) -- (3);
\draw[dashed] (16) -- (4);
\draw[dashed] (16) -- (5);
\draw[dashed] (16) -- (6);
\draw[dashed] (16) -- (7);
\draw[dashed] (16) -- (9);
\draw[dashed] (16) -- (10);
\draw[dashed] (16) -- (11);
\draw[thick, line width=2pt] (16) -- (12);
\draw[dashed] (17) -- (1);
\draw[dashed] (17) -- (2);
\draw[dashed] (17) -- (3);
\draw[dashed] (17) -- (4);
\draw[dashed] (17) -- (5);
\draw[dashed] (17) -- (6);
\draw[dashed] (17) -- (7);
\draw[dashed] (17) -- (8);
\draw[dashed] (17) -- (9);
\draw[dashed] (17) -- (10);
\draw[dashed] (17) -- (11);
\draw[dashed] (17) -- (13);
\draw[thick, line width=2pt] (17) -- (15);
\draw[dashed] (18) -- (1);
\draw[dashed] (18) -- (2);
\draw[dashed] (18) -- (3);
\draw[dashed] (18) -- (4);
\draw[dashed] (18) -- (5);
\draw[dashed] (18) -- (6);
\draw[dashed] (18) -- (8);
\draw[dashed] (18) -- (9);
\draw[dashed] (18) -- (10);
\draw[dashed] (18) -- (11);
\draw[dashed] (18) -- (13);
\draw[dashed] (18) -- (15);
\draw[dashed] (18) -- (16);
\draw[dashed] (18) -- (17);

\foreach \i [count=\n] in \nodeorder {
  \pgfmathsetmacro{\angle}{360/\nodecountvalue * (\n-1) + \rotateby}
  \filldraw[fill=white] (\i) circle (2pt) node[anchor=180+\angle] {\scriptsize \i};
}

\draw (12) node[left, xshift = -0.5cm] {\hyperref[subsub:39]{SW39$^\star$}:};
\end{tikzpicture}

\vspace{0.4cm}
%
%
%
\phantomsection \label{SW40}\begin{tikzpicture}

\def\nodeorder{2,6,8,1,7,15,12,17,16, 13, 9,18,10,11,14,3,5,4}
        
\newcount\nodecount
\nodecount=0
\foreach \i in \nodeorder {
  \global\advance\nodecount by 1
}
\pgfmathsetmacro{\nodecountvalue}{\the\nodecount}
\def\rotateby{1.5*360/\nodecountvalue}
   
\foreach \i [count=\n] in \nodeorder {
  \coordinate (\i) at ({360/\nodecountvalue * (\n-1) + \rotateby}:2.5);
}

\draw[dashed] (4) -- (1);
\draw[thick, line width=2pt] (4) -- (2);
\draw[dashed] (5) -- (1);
\draw[thick, line width=2pt] (5) -- (3);
\draw[dashed] (6) -- (2);
\draw[dashed] (6) -- (5);
\draw[thick] (7) -- (1);
\draw[dashed] (7) -- (3);
\draw[dashed] (7) -- (4);
\draw[dashed] (7) -- (6);
\draw[thick] (8) -- (1);
\draw[dashed] (8) -- (2);
\draw[dashed] (8) -- (3);
\draw[dashed] (8) -- (4);
\draw[dashed] (8) -- (5);
\draw[thick, double distance = 2pt] (8) -- (6);
\draw[dashed] (8) -- (7);
\draw[dashed] (9) -- (1);
\draw[dashed] (9) -- (2);
\draw[dashed] (9) -- (3);
\draw[dashed] (9) -- (6);
\draw[dashed] (9) -- (7);
\draw[dashed] (10) -- (1);
\draw[dashed] (10) -- (2);
\draw[dashed] (10) -- (3);
\draw[dashed] (10) -- (4);
\draw[dashed] (10) -- (6);
\draw[dashed] (11) -- (1);
\draw[dashed] (11) -- (2);
\draw[dashed] (11) -- (3);
\draw[dashed] (11) -- (5);
\draw[dashed] (11) -- (7);
\draw[dashed] (11) -- (10);
\draw[dashed] (12) -- (1);
\draw[dashed] (12) -- (2);
\draw[dashed] (12) -- (3);
\draw[dashed] (12) -- (4);
\draw[dashed] (12) -- (5);
\draw[dashed] (12) -- (6);
\draw[dashed] (12) -- (7);
\draw[dashed] (12) -- (9);
\draw[dashed] (12) -- (11);
\draw[dashed] (13) -- (2);
\draw[dashed] (13) -- (3);
\draw[dashed] (13) -- (4);
\draw[dashed] (13) -- (5);
\draw[dashed] (13) -- (6);
\draw[dashed] (13) -- (7);
\draw[dashed] (13) -- (9);
\draw[dashed] (13) -- (10);
\draw[dashed] (13) -- (11);
\draw[dashed] (14) -- (1);
\draw[dashed] (14) -- (2);
\draw[dashed] (14) -- (3);
\draw[dashed] (14) -- (4);
\draw[dashed] (14) -- (5);
\draw[dashed] (14) -- (6);
\draw[dashed] (14) -- (8);
\draw[dashed] (14) -- (9);
\draw[dashed] (14) -- (11);
\draw[dashed] (14) -- (13);
\draw[dashed] (15) -- (2);
\draw[dashed] (15) -- (3);
\draw[dashed] (15) -- (4);
\draw[dashed] (15) -- (5);
\draw[dashed] (15) -- (6);
\draw[thick, double distance = 2pt] (15) -- (7);
\draw[dashed] (15) -- (8);
\draw[dashed] (15) -- (9);
\draw[dashed] (15) -- (10);
\draw[dashed] (15) -- (11);
\draw[dashed] (15) -- (12);
\draw[dashed] (15) -- (13);
\draw[dashed] (15) -- (14);
\draw[dashed] (16) -- (1);
\draw[dashed] (16) -- (2);
\draw[dashed] (16) -- (3);
\draw[dashed] (16) -- (4);
\draw[dashed] (16) -- (5);
\draw[dashed] (16) -- (6);
\draw[dashed] (16) -- (7);
\draw[dashed] (16) -- (8);
\draw[dashed] (16) -- (9);
\draw[dashed] (16) -- (10);
\draw[dashed] (16) -- (11);
\draw[thick, line width=2pt] (16) -- (13);
\draw[dashed] (17) -- (2);
\draw[dashed] (17) -- (3);
\draw[dashed] (17) -- (4);
\draw[dashed] (17) -- (5);
\draw[dashed] (17) -- (6);
\draw[dashed] (17) -- (7);
\draw[dashed] (17) -- (8);
\draw[dashed] (17) -- (9);
\draw[dashed] (17) -- (10);
\draw[dashed] (17) -- (11);
\draw[thick, line width=2pt] (17) -- (12);
\draw[dashed] (17) -- (14);
\draw[dashed] (18) -- (1);
\draw[dashed] (18) -- (2);
\draw[dashed] (18) -- (3);
\draw[dashed] (18) -- (4);
\draw[dashed] (18) -- (5);
\draw[dashed] (18) -- (6);
\draw[dashed] (18) -- (8);
\draw[dashed] (18) -- (9);
\draw[dashed] (18) -- (10);
\draw[dashed] (18) -- (11);
\draw[dashed] (18) -- (12);
\draw[dashed] (18) -- (13);
\draw[dashed] (18) -- (17);

\foreach \i [count=\n] in \nodeorder {
  \pgfmathsetmacro{\angle}{360/\nodecountvalue * (\n-1) + \rotateby}
  \filldraw[fill=white] (\i) circle (2pt) node[anchor=180+\angle] {\scriptsize \i};
}

\draw (15) node[left, xshift = -0.5cm] {\hyperref[subsub:40]{SW40$^\star$:}};
\end{tikzpicture}
\qquad 
%
%
\phantomsection \label{SW41}\begin{tikzpicture}

    \def\nodeorder{2,12,10,11,9,14,8,6,15,7,3,5,13,1,4}
        
\newcount\nodecount
\nodecount=0
\foreach \i in \nodeorder {
  \global\advance\nodecount by 1
}
\pgfmathsetmacro{\nodecountvalue}{\the\nodecount}
\def\rotateby{90+0*360/\nodecountvalue}
   
\foreach \i [count=\n] in \nodeorder {
  \coordinate (\i) at ({360/\nodecountvalue * (\n-1) + \rotateby}:2.5);
}

\draw[thick, double distance = 2.5pt] (4) -- (1);
\draw[thick, double distance = .3pt] (4) -- (1);
\draw[thick, line width=2pt] (4) -- (2);
\draw[dashed] (5) -- (1);
\draw[thick, line width=2pt] (5) -- (3);
\draw[dashed] (6) -- (2);
\draw[dashed] (6) -- (3);
\draw[dashed] (6) -- (5);
\draw[dashed] (7) -- (2);
\draw[dashed] (7) -- (3);
\draw[dashed] (7) -- (4);
\draw[dashed] (7) -- (5);
\draw[dashed] (8) -- (1);
\draw[dashed] (8) -- (2);
\draw[dashed] (8) -- (3);
\draw[dashed] (8) -- (6);
\draw[thick, line width=2pt] (8) -- (7);
\draw[dashed] (9) -- (3);
\draw[dashed] (9) -- (4);
\draw[dashed] (9) -- (5);
\draw[dashed] (9) -- (6);
\draw[dashed] (9) -- (8);
\draw[dashed] (10) -- (1);
\draw[dashed] (10) -- (3);
\draw[dashed] (10) -- (4);
\draw[dashed] (10) -- (6);
\draw[dashed] (10) -- (7);
\draw[dashed] (10) -- (9);
\draw[dashed] (11) -- (1);
\draw[dashed] (11) -- (2);
\draw[dashed] (11) -- (3);
\draw[dashed] (11) -- (5);
\draw[dashed] (11) -- (9);
\draw[dashed] (11) -- (10);
\draw[dashed] (12) -- (1);
\draw[thick, line width=2pt] (12) -- (2);
\draw[dashed] (12) -- (3);
\draw[dashed] (12) -- (4);
\draw[dashed] (12) -- (5);
\draw[dashed] (12) -- (6);
\draw[dashed] (12) -- (8);
\draw[thick, double distance = 2.5pt] (12) -- (10);
\draw[thick, double distance = .3pt] (12) -- (10);
\draw[dashed] (12) -- (11);
\draw[dashed] (13) -- (1);
\draw[dashed] (13) -- (2);
\draw[dashed] (13) -- (3);
\draw[dashed] (13) -- (4);
\draw[dashed] (13) -- (5);
\draw[dashed] (13) -- (6);
\draw[dashed] (13) -- (9);
\draw[dashed] (13) -- (10);
\draw[thick, line width=2pt] (13) -- (11);
\draw[dashed] (14) -- (1);
\draw[dashed] (14) -- (3);
\draw[dashed] (14) -- (4);
\draw[dashed] (14) -- (5);
\draw[dashed] (14) -- (6);
\draw[dashed] (14) -- (7);
\draw[dashed] (14) -- (8);
\draw[thick, line width=2pt] (14) -- (9);
\draw[dashed] (14) -- (11);
\draw[dashed] (14) -- (13);
\draw[dashed] (15) -- (1);
\draw[dashed] (15) -- (2);
\draw[dashed] (15) -- (3);
\draw[dashed] (15) -- (4);
\draw[dashed] (15) -- (5);
\draw[dashed] (15) -- (6);
\draw[dashed] (15) -- (7);
\draw[dashed] (15) -- (9);
\draw[dashed] (15) -- (11);
\draw[dashed] (15) -- (14);

\foreach \i [count=\n] in \nodeorder {
  \pgfmathsetmacro{\angle}{360/\nodecountvalue * (\n-1) + \rotateby}
  \filldraw[fill=white] (\i) circle (2pt) node[anchor=180+\angle] {\scriptsize \i};
}

\draw (10) node[left, xshift = -0.5cm] {SW41:};
\end{tikzpicture}

%
%
%
\phantomsection \label{SW42}\begin{tikzpicture}

\def\nodeorder{6,9,8,7,2,3,4,1,5}
        
\newcount\nodecount
\nodecount=0
\foreach \i in \nodeorder {
  \global\advance\nodecount by 1
}
\pgfmathsetmacro{\nodecountvalue}{\the\nodecount}
\def\rotateby{-.75*360/\nodecountvalue}
   
\foreach \i [count=\n] in \nodeorder {
  \coordinate (\i) at ({360/\nodecountvalue * (\n-1) + \rotateby}:1.3);
}

\draw[thick, double distance = 2pt] (3) -- (2);
\draw[dashed] (4) -- (1);
\draw[thick, double distance = 2pt] (4) -- (3);
\draw[thick] (5) -- (1);
\draw[dashed] (5) -- (3);
\draw[dashed] (6) -- (2);
\draw[dashed] (6) -- (4);
\draw[dashed] (6) -- (5);
\draw[dashed] (7) -- (1);
\draw[dashed] (7) -- (2);
\draw[dashed] (7) -- (5);
\draw[dashed] (8) -- (2);
\draw[dashed] (8) -- (3);
\draw[dashed] (8) -- (4);
\draw[dashed] (8) -- (7);
\draw[dashed] (9) -- (1);
\draw[dashed] (9) -- (2);
\draw[dashed] (9) -- (3);
\draw[thick, double distance = 2pt] (9) -- (6);
\draw[thick, double distance = 2pt] (9) -- (8);

\foreach \i [count=\n] in \nodeorder {
  \pgfmathsetmacro{\angle}{360/\nodecountvalue * (\n-1) + \rotateby}
  \filldraw[fill=white] (\i) circle (2pt) node[anchor=180+\angle] {\scriptsize \i};
}

\draw (2) node[left, xshift = -0.4cm] {SW42:};
\end{tikzpicture}
\qquad 
%
%
\phantomsection \label{SW43}\begin{tikzpicture}

\def\nodeorder{11,8,4,2,6,9,12,10,7,1,3,5,13}
        
\newcount\nodecount
\nodecount=0
\foreach \i in \nodeorder {
  \global\advance\nodecount by 1
}
\pgfmathsetmacro{\nodecountvalue}{\the\nodecount}
\def\rotateby{-4.75*360/\nodecountvalue}
   
\foreach \i [count=\n] in \nodeorder {
  \coordinate (\i) at ({360/\nodecountvalue * (\n-1) + \rotateby}:1.75);
}

\draw[dashed] (4) -- (1);
\draw[thick, line width=2pt] (4) -- (2);
\draw[dashed] (5) -- (1);
\draw[thick, line width=2pt] (5) -- (3);
\draw[dashed] (6) -- (2);
\draw[dashed] (6) -- (5);
\draw[thick, double distance = 2pt] (7) -- (1);
\draw[dashed] (7) -- (3);
\draw[dashed] (7) -- (4);
\draw[dashed] (7) -- (6);
\draw[dashed] (8) -- (2);
\draw[dashed] (8) -- (3);
\draw[thick] (8) -- (4);
\draw[dashed] (8) -- (5);
\draw[dashed] (8) -- (7);
\draw[dashed] (9) -- (2);
\draw[dashed] (9) -- (3);
\draw[dashed] (9) -- (4);
\draw[dashed] (9) -- (5);
\draw[dashed] (9) -- (6);
\draw[dashed] (10) -- (1);
\draw[dashed] (10) -- (2);
\draw[dashed] (10) -- (3);
\draw[dashed] (10) -- (6);
\draw[thick, double distance = 2pt] (10) -- (7);
\draw[dashed] (10) -- (8);
\draw[dashed] (10) -- (9);
\draw[dashed] (11) -- (1);
\draw[dashed] (11) -- (2);
\draw[dashed] (11) -- (3);
\draw[dashed] (11) -- (4);
\draw[dashed] (11) -- (5);
\draw[dashed] (11) -- (6);
\draw[thick, line width=2pt] (11) -- (8);
\draw[dashed] (12) -- (1);
\draw[dashed] (12) -- (2);
\draw[dashed] (12) -- (3);
\draw[dashed] (12) -- (4);
\draw[dashed] (12) -- (5);
\draw[dashed] (12) -- (6);
\draw[dashed] (12) -- (7);
\draw[thick, line width=2pt] (12) -- (9);
\draw[dashed] (13) -- (1);
\draw[dashed] (13) -- (2);
\draw[dashed] (13) -- (3);
\draw[dashed] (13) -- (5);
\draw[dashed] (13) -- (6);
\draw[dashed] (13) -- (7);
\draw[dashed] (13) -- (9);
\draw[dashed] (13) -- (11);

\foreach \i [count=\n] in \nodeorder {
  \pgfmathsetmacro{\angle}{360/\nodecountvalue * (\n-1) + \rotateby}
  \filldraw[fill=white] (\i) circle (2pt) node[anchor=180+\angle] {\scriptsize \i};
}

\draw (1) node[left, xshift = -0.5cm, yshift=0.2cm] {SW43:};
\end{tikzpicture}
\qquad 
%
%
\phantomsection \label{SW44}\begin{tikzpicture}

\def\nodeorder{4,2,3,7,5,1,6}
        
\newcount\nodecount
\nodecount=0
\foreach \i in \nodeorder {
  \global\advance\nodecount by 1
}
\pgfmathsetmacro{\nodecountvalue}{\the\nodecount}
\def\rotateby{.75*360/\nodecountvalue}
   
\foreach \i [count=\n] in \nodeorder {
  \coordinate (\i) at ({360/\nodecountvalue * (\n-1) + \rotateby}:1.3);
}

\draw[thick, double distance = 2.5pt] (3) -- (2);
\draw[thick, double distance = .3pt] (3) -- (2);
\draw[dashed] (4) -- (1);
\draw[thick] (4) -- (2);
\draw[thick, double distance = 2pt] (5) -- (1);
\draw[dashed] (5) -- (2);
\draw[dashed] (6) -- (3);
\draw[dashed] (6) -- (4);
\draw[dashed] (7) -- (1);
\draw[dashed] (7) -- (3);

\foreach \i [count=\n] in \nodeorder {
  \pgfmathsetmacro{\angle}{360/\nodecountvalue * (\n-1) + \rotateby}
  \filldraw[fill=white] (\i) circle (2pt) node[anchor=180+\angle] {\scriptsize \i};
}

\draw (2) node[left, xshift = -0.8cm, yshift = 0.2cm] {SW44:};
\end{tikzpicture}

%
%
%
\phantomsection \label{SW45}\begin{tikzpicture}

\def\nodeorder{5,17,8,18,1,4,6,14,12,10,9,7,3,2,13,15,19,11,16,20}
        
\newcount\nodecount
\nodecount=0
\foreach \i in \nodeorder {
  \global\advance\nodecount by 1
}
\pgfmathsetmacro{\nodecountvalue}{\the\nodecount}
\def\rotateby{1.5*360/\nodecountvalue}
   
\foreach \i [count=\n] in \nodeorder {
  \coordinate (\i) at ({360/\nodecountvalue * (\n-1) + \rotateby}:2.5);
}

\draw[thick, double distance = 2.5pt] (4) -- (1);
\draw[thick, double distance = .3pt] (4) -- (1);
\draw[thick, line width=2pt] (4) -- (2);
\draw[dashed] (5) -- (3);
\draw[dashed] (5) -- (4);
\draw[dashed] (6) -- (2);
\draw[dashed] (6) -- (3);
\draw[thick, double distance = 2pt] (6) -- (4);
\draw[dashed] (7) -- (1);
\draw[thick, line width=2pt] (7) -- (3);
\draw[dashed] (7) -- (5);
\draw[dashed] (7) -- (6);
\draw[dashed] (8) -- (1);
\draw[dashed] (8) -- (3);
\draw[dashed] (8) -- (4);
\draw[dashed] (8) -- (6);
\draw[dashed] (8) -- (7);
\draw[dashed] (9) -- (2);
\draw[dashed] (9) -- (3);
\draw[dashed] (9) -- (5);
\draw[dashed] (9) -- (7);
\draw[dashed] (9) -- (8);
\draw[dashed] (10) -- (1);
\draw[dashed] (10) -- (2);
\draw[dashed] (10) -- (3);
\draw[dashed] (10) -- (5);
\draw[dashed] (10) -- (7);
\draw[dashed] (10) -- (8);
\draw[dashed] (10) -- (9);
\draw[dashed] (11) -- (1);
\draw[dashed] (11) -- (3);
\draw[dashed] (11) -- (4);
\draw[dashed] (11) -- (5);
\draw[dashed] (11) -- (6);
\draw[dashed] (11) -- (9);
\draw[dashed] (11) -- (10);
\draw[dashed] (12) -- (1);
\draw[dashed] (12) -- (2);
\draw[dashed] (12) -- (3);
\draw[dashed] (12) -- (4);
\draw[dashed] (12) -- (7);
\draw[dashed] (12) -- (9);
\draw[dashed] (12) -- (11);
\draw[dashed] (13) -- (1);
\draw[dashed] (13) -- (2);
\draw[dashed] (13) -- (3);
\draw[dashed] (13) -- (5);
\draw[dashed] (13) -- (6);
\draw[dashed] (13) -- (7);
\draw[thick, line width=2pt] (13) -- (8);
\draw[dashed] (13) -- (9);
\draw[dashed] (13) -- (11);
\draw[dashed] (13) -- (12);
\draw[dashed] (14) -- (1);
\draw[dashed] (14) -- (2);
\draw[dashed] (14) -- (3);
\draw[dashed] (14) -- (5);
\draw[dashed] (14) -- (6);
\draw[dashed] (14) -- (8);
\draw[dashed] (14) -- (9);
\draw[dashed] (14) -- (10);
\draw[dashed] (14) -- (12);
\draw[dashed] (14) -- (13);
\draw[dashed] (15) -- (1);
\draw[dashed] (15) -- (2);
\draw[dashed] (15) -- (3);
\draw[dashed] (15) -- (4);
\draw[dashed] (15) -- (5);
\draw[dashed] (15) -- (6);
\draw[dashed] (15) -- (7);
\draw[dashed] (15) -- (9);
\draw[dashed] (15) -- (11);
\draw[dashed] (15) -- (14);
\draw[dashed] (16) -- (1);
\draw[dashed] (16) -- (2);
\draw[dashed] (16) -- (3);
\draw[dashed] (16) -- (5);
\draw[dashed] (16) -- (6);
\draw[dashed] (16) -- (7);
\draw[dashed] (16) -- (8);
\draw[dashed] (16) -- (9);
\draw[dashed] (16) -- (10);
\draw[dashed] (16) -- (11);
\draw[dashed] (16) -- (12);
\draw[dashed] (16) -- (15);
\draw[dashed] (17) -- (1);
\draw[dashed] (17) -- (2);
\draw[dashed] (17) -- (3);
\draw[dashed] (17) -- (4);
\draw[dashed] (17) -- (5);
\draw[dashed] (17) -- (6);
\draw[dashed] (17) -- (7);
\draw[thick, double distance = 2pt] (17) -- (8);
\draw[dashed] (17) -- (9);
\draw[dashed] (17) -- (10);
\draw[dashed] (17) -- (12);
\draw[dashed] (17) -- (13);
\draw[dashed] (17) -- (15);
\draw[dashed] (18) -- (1);
\draw[dashed] (18) -- (2);
\draw[dashed] (18) -- (3);
\draw[dashed] (18) -- (4);
\draw[dashed] (18) -- (5);
\draw[dashed] (18) -- (6);
\draw[dashed] (18) -- (7);
\draw[thick, double distance = 2.5pt] (18) -- (8);
\draw[thick, double distance = .3pt] (18) -- (8);
\draw[dashed] (18) -- (9);
\draw[dashed] (18) -- (10);
\draw[dashed] (18) -- (11);
\draw[dashed] (18) -- (12);
\draw[dashed] (18) -- (14);
\draw[dashed] (18) -- (15);
\draw[dashed] (19) -- (1);
\draw[dashed] (19) -- (2);
\draw[dashed] (19) -- (3);
\draw[dashed] (19) -- (4);
\draw[dashed] (19) -- (5);
\draw[dashed] (19) -- (6);
\draw[dashed] (19) -- (7);
\draw[dashed] (19) -- (9);
\draw[dashed] (19) -- (10);
\draw[dashed] (19) -- (11);
\draw[dashed] (19) -- (12);
\draw[dashed] (19) -- (14);
\draw[thick, line width=2pt] (19) -- (15);
\draw[dashed] (19) -- (16);
\draw[dashed] (19) -- (17);
\draw[dashed] (20) -- (1);
\draw[dashed] (20) -- (2);
\draw[dashed] (20) -- (3);
\draw[dashed] (20) -- (4);
\draw[dashed] (20) -- (5);
\draw[dashed] (20) -- (6);
\draw[dashed] (20) -- (7);
\draw[dashed] (20) -- (9);
\draw[dashed] (20) -- (10);
\draw[dashed] (20) -- (11);
\draw[dashed] (20) -- (12);
\draw[dashed] (20) -- (13);
\draw[dashed] (20) -- (14);
\draw[dashed] (20) -- (15);
\draw[dashed] (20) -- (16);
\draw[dashed] (20) -- (19);

\foreach \i [count=\n] in \nodeorder {
  \pgfmathsetmacro{\angle}{360/\nodecountvalue * (\n-1) + \rotateby}
  \filldraw[fill=white] (\i) circle (2pt) node[anchor=180+\angle] {\scriptsize \i};
}

\draw (6) node[left, xshift = -0.5cm] {\hyperref[subsub:45]{SW45$^\star$:}};
\end{tikzpicture}
\qquad 
%
%
\phantomsection \label{SW46}\begin{tikzpicture}

\def\nodeorder{5,2,13,8,9,6,12,4,17,18,10,7,1,19,20,15,14,3,16,11}
        
\newcount\nodecount
\nodecount=0
\foreach \i in \nodeorder {
  \global\advance\nodecount by 1
}
\pgfmathsetmacro{\nodecountvalue}{\the\nodecount}
\def\rotateby{-6*360/\nodecountvalue}
   
\foreach \i [count=\n] in \nodeorder {
  \coordinate (\i) at ({360/\nodecountvalue * (\n-1) + \rotateby}:2.5);
}

\draw[dashed] (4) -- (1);
\draw[thick, line width=2pt] (4) -- (2);
\draw[dashed] (5) -- (2);
\draw[dashed] (6) -- (1);
\draw[thick, line width=2pt] (6) -- (3);
\draw[dashed] (6) -- (5);
\draw[thick, double distance = 2.5pt] (7) -- (1);
\draw[thick, double distance = .3pt] (7) -- (1);
\draw[dashed] (7) -- (2);
\draw[dashed] (7) -- (3);
\draw[thick, double distance = 2pt] (7) -- (4);
\draw[dashed] (7) -- (5);
\draw[dashed] (7) -- (6);
\draw[dashed] (8) -- (2);
\draw[dashed] (8) -- (3);
\draw[dashed] (8) -- (4);
\draw[dashed] (8) -- (6);
\draw[dashed] (9) -- (3);
\draw[dashed] (9) -- (4);
\draw[dashed] (9) -- (5);
\draw[dashed] (9) -- (6);
\draw[dashed] (9) -- (7);
\draw[dashed] (9) -- (8);
\draw[dashed] (10) -- (1);
\draw[dashed] (10) -- (3);
\draw[dashed] (10) -- (4);
\draw[dashed] (10) -- (5);
\draw[thick, double distance = 2.5pt] (10) -- (7);
\draw[thick, double distance = .3pt] (10) -- (7);
\draw[dashed] (10) -- (8);
\draw[dashed] (10) -- (9);
\draw[dashed] (11) -- (1);
\draw[dashed] (11) -- (2);
\draw[dashed] (11) -- (3);
\draw[dashed] (11) -- (6);
\draw[dashed] (11) -- (9);
\draw[dashed] (11) -- (10);
\draw[dashed] (12) -- (2);
\draw[dashed] (12) -- (3);
\draw[dashed] (12) -- (4);
\draw[dashed] (12) -- (5);
\draw[dashed] (12) -- (6);
\draw[dashed] (12) -- (8);
\draw[dashed] (12) -- (10);
\draw[dashed] (12) -- (11);
\draw[dashed] (13) -- (1);
\draw[dashed] (13) -- (2);
\draw[dashed] (13) -- (3);
\draw[dashed] (13) -- (5);
\draw[dashed] (13) -- (8);
\draw[dashed] (13) -- (9);
\draw[dashed] (13) -- (10);
\draw[dashed] (13) -- (11);
\draw[dashed] (13) -- (12);
\draw[dashed] (14) -- (1);
\draw[dashed] (14) -- (2);
\draw[dashed] (14) -- (3);
\draw[dashed] (14) -- (4);
\draw[dashed] (14) -- (5);
\draw[dashed] (14) -- (8);
\draw[dashed] (14) -- (9);
\draw[dashed] (14) -- (11);
\draw[dashed] (14) -- (12);
\draw[dashed] (15) -- (1);
\draw[thick, line width=2pt] (15) -- (2);
\draw[dashed] (15) -- (3);
\draw[dashed] (15) -- (4);
\draw[dashed] (15) -- (5);
\draw[dashed] (15) -- (6);
\draw[thick, double distance = 2pt] (15) -- (7);
\draw[dashed] (15) -- (8);
\draw[dashed] (15) -- (10);
\draw[dashed] (15) -- (11);
\draw[dashed] (15) -- (12);
\draw[dashed] (15) -- (13);
\draw[dashed] (15) -- (14);
\draw[dashed] (16) -- (1);
\draw[dashed] (16) -- (3);
\draw[dashed] (16) -- (4);
\draw[dashed] (16) -- (5);
\draw[dashed] (16) -- (6);
\draw[dashed] (16) -- (7);
\draw[dashed] (16) -- (8);
\draw[thick, line width=2pt] (16) -- (9);
\draw[dashed] (16) -- (11);
\draw[dashed] (16) -- (12);
\draw[dashed] (16) -- (13);
\draw[dashed] (16) -- (14);
\draw[dashed] (17) -- (1);
\draw[dashed] (17) -- (2);
\draw[dashed] (17) -- (3);
\draw[dashed] (17) -- (4);
\draw[dashed] (17) -- (5);
\draw[dashed] (17) -- (6);
\draw[dashed] (17) -- (8);
\draw[dashed] (17) -- (10);
\draw[dashed] (17) -- (11);
\draw[dashed] (17) -- (13);
\draw[dashed] (17) -- (14);
\draw[dashed] (17) -- (16);
\draw[dashed] (18) -- (1);
\draw[dashed] (18) -- (2);
\draw[dashed] (18) -- (3);
\draw[dashed] (18) -- (4);
\draw[dashed] (18) -- (5);
\draw[dashed] (18) -- (6);
\draw[dashed] (18) -- (8);
\draw[dashed] (18) -- (9);
\draw[dashed] (18) -- (11);
\draw[dashed] (18) -- (12);
\draw[dashed] (18) -- (13);
\draw[dashed] (18) -- (14);
\draw[dashed] (18) -- (15);
\draw[dashed] (18) -- (16);
\draw[dashed] (18) -- (17);
\draw[dashed] (19) -- (1);
\draw[dashed] (19) -- (2);
\draw[dashed] (19) -- (3);
\draw[dashed] (19) -- (4);
\draw[dashed] (19) -- (5);
\draw[dashed] (19) -- (6);
\draw[dashed] (19) -- (7);
\draw[dashed] (19) -- (8);
\draw[dashed] (19) -- (9);
\draw[dashed] (19) -- (11);
\draw[dashed] (19) -- (12);
\draw[dashed] (19) -- (13);
\draw[dashed] (19) -- (14);
\draw[dashed] (19) -- (17);
\draw[dashed] (20) -- (1);
\draw[dashed] (20) -- (2);
\draw[dashed] (20) -- (3);
\draw[dashed] (20) -- (4);
\draw[dashed] (20) -- (5);
\draw[dashed] (20) -- (6);
\draw[dashed] (20) -- (8);
\draw[dashed] (20) -- (9);
\draw[dashed] (20) -- (10);
\draw[dashed] (20) -- (11);
\draw[dashed] (20) -- (12);
\draw[dashed] (20) -- (13);
\draw[dashed] (20) -- (14);
\draw[dashed] (20) -- (16);
\draw[dashed] (20) -- (17);

\foreach \i [count=\n] in \nodeorder {
  \pgfmathsetmacro{\angle}{360/\nodecountvalue * (\n-1) + \rotateby}
  \filldraw[fill=white] (\i) circle (2pt) node[anchor=180+\angle] {\scriptsize \i};
}

\draw (19) node[left, xshift = -0.5cm, yshift=.2cm] {\hyperref[subsub:46]{SW46$^\star$:}};
\end{tikzpicture}

%
%
\phantomsection \label{SW47}\begin{tikzpicture}

\def\nodeorder{5,6,1,12,7,8,2,3,4,10,9,11}
        
\newcount\nodecount
\nodecount=0
\foreach \i in \nodeorder {
  \global\advance\nodecount by 1
}
\pgfmathsetmacro{\nodecountvalue}{\the\nodecount}
\def\rotateby{.5*360/\nodecountvalue}
   
\foreach \i [count=\n] in \nodeorder {
  \coordinate (\i) at ({360/\nodecountvalue * (\n-1) + \rotateby}:2);
}

\draw[thick, double distance = 2.5pt] (3) -- (2);
\draw[thick, double distance = .3pt] (3) -- (2);
\draw[dashed] (4) -- (1);
\draw[thick] (4) -- (3);
\draw[dashed] (5) -- (2);
\draw[dashed] (5) -- (4);
\draw[thick] (6) -- (1);
\draw[dashed] (6) -- (3);
\draw[dashed] (6) -- (4);
\draw[thick, double distance = 2.5pt] (6) -- (5);
\draw[thick, double distance = .3pt] (6) -- (5);
\draw[dashed] (7) -- (1);
\draw[dashed] (7) -- (3);
\draw[dashed] (7) -- (5);
\draw[dashed] (8) -- (1);
\draw[dashed] (8) -- (2);
\draw[dashed] (8) -- (5);
\draw[dashed] (8) -- (6);
\draw[thick, double distance = 2.5pt] (8) -- (7);
\draw[thick, double distance = .3pt] (8) -- (7);
\draw[dashed] (9) -- (1);
\draw[dashed] (9) -- (2);
\draw[dashed] (9) -- (4);
\draw[dashed] (9) -- (6);
\draw[dashed] (9) -- (7);
\draw[dashed] (10) -- (1);
\draw[dashed] (10) -- (2);
\draw[dashed] (10) -- (3);
\draw[dashed] (10) -- (4);
\draw[dashed] (10) -- (7);
\draw[dashed] (10) -- (8);
\draw[thick] (10) -- (9);
\draw[dashed] (11) -- (1);
\draw[dashed] (11) -- (2);
\draw[dashed] (11) -- (3);
\draw[dashed] (11) -- (4);
\draw[dashed] (11) -- (5);
\draw[dashed] (11) -- (8);
\draw[thick, double distance = 2.5pt] (11) -- (9);
\draw[thick, double distance = .3pt] (11) -- (9);
\draw[dashed] (12) -- (1);
\draw[dashed] (12) -- (2);
\draw[dashed] (12) -- (3);
\draw[dashed] (12) -- (4);
\draw[dashed] (12) -- (5);
\draw[dashed] (12) -- (6);
\draw[thick] (12) -- (7);
\draw[dashed] (12) -- (10);

\foreach \i [count=\n] in \nodeorder {
  \pgfmathsetmacro{\angle}{360/\nodecountvalue * (\n-1) + \rotateby}
  \filldraw[fill=white] (\i) circle (2pt) node[anchor=180+\angle] {\scriptsize \i};
}

\draw (7) node[left, xshift = -0.5cm, yshift=.2cm] {SW47:};
\end{tikzpicture}
\qquad
%
%
%
%
%
\phantomsection \label{SW48}\begin{tikzpicture}

\def\nodeorder{9,4,3,2,11,6,10,8,7,5,12,1}
        
\newcount\nodecount
\nodecount=0
\foreach \i in \nodeorder {
  \global\advance\nodecount by 1
}
\pgfmathsetmacro{\nodecountvalue}{\the\nodecount}
\def\rotateby{1.5*360/\nodecountvalue}
   
\foreach \i [count=\n] in \nodeorder {
  \coordinate (\i) at ({360/\nodecountvalue * (\n-1) + \rotateby}:2);
}

\draw[thick, double distance = 2.5pt] (3) -- (2);
\draw[thick, double distance = .3pt] (3) -- (2);
\draw[dashed] (4) -- (1);
\draw[thick] (4) -- (3);
\draw[dashed] (5) -- (2);
\draw[dashed] (5) -- (4);
\draw[dashed] (6) -- (3);
\draw[dashed] (6) -- (4);
\draw[dashed] (7) -- (1);
\draw[dashed] (7) -- (3);
\draw[thick, double distance = 2.5pt] (7) -- (5);
\draw[thick, double distance = .3pt] (7) -- (5);
\draw[dashed] (8) -- (1);
\draw[dashed] (8) -- (2);
\draw[dashed] (8) -- (4);
\draw[dashed] (8) -- (6);
\draw[thick] (8) -- (7);
\draw[dashed] (9) -- (1);
\draw[dashed] (9) -- (2);
\draw[thick, double distance = 2.5pt] (9) -- (4);
\draw[thick, double distance = .3pt] (9) -- (4);
\draw[dashed] (9) -- (5);
\draw[dashed] (9) -- (6);
\draw[dashed] (9) -- (7);
\draw[dashed] (10) -- (1);
\draw[dashed] (10) -- (2);
\draw[dashed] (10) -- (3);
\draw[dashed] (10) -- (5);
\draw[dashed] (10) -- (6);
\draw[thick, double distance = 2.5pt] (10) -- (8);
\draw[thick, double distance = .3pt] (10) -- (8);
\draw[dashed] (10) -- (9);
\draw[dashed] (11) -- (1);
\draw[dashed] (11) -- (2);
\draw[dashed] (11) -- (3);
\draw[dashed] (11) -- (5);
\draw[dashed] (11) -- (6);
\draw[dashed] (11) -- (7);
\draw[dashed] (11) -- (8);
\draw[dashed] (12) -- (1);
\draw[dashed] (12) -- (2);
\draw[dashed] (12) -- (3);
\draw[dashed] (12) -- (4);
\draw[dashed] (12) -- (5);
\draw[dashed] (12) -- (6);
\draw[dashed] (12) -- (7);

\foreach \i [count=\n] in \nodeorder {
  \pgfmathsetmacro{\angle}{360/\nodecountvalue * (\n-1) + \rotateby}
  \filldraw[fill=white] (\i) circle (2pt) node[anchor=180+\angle] {\scriptsize \i};
}

\draw (2) node[left, xshift = -0.5cm, yshift=.2cm] {SW48:};
\end{tikzpicture}

%
%
%
%
%
\phantomsection \label{SW49}\begin{tikzpicture}

\def\nodeorder{17,6,12,10,9,16,7,14,3,2,4,18,1,5,15,11,8,13}
        
\newcount\nodecount
\nodecount=0
\foreach \i in \nodeorder {
  \global\advance\nodecount by 1
}
\pgfmathsetmacro{\nodecountvalue}{\the\nodecount}
\def\rotateby{-1*360/\nodecountvalue}
   
\foreach \i [count=\n] in \nodeorder {
  \coordinate (\i) at ({360/\nodecountvalue * (\n-1) + \rotateby}:3);
}

\draw[thick, double distance = 2pt] (3) -- (2);
\draw[dashed] (4) -- (1);
\draw[thick, double distance = 2pt] (4) -- (2);
\draw[dashed] (5) -- (3);
\draw[dashed] (5) -- (4);
\draw[dashed] (6) -- (1);
\draw[dashed] (6) -- (3);
\draw[dashed] (6) -- (4);
\draw[dashed] (7) -- (2);
\draw[dashed] (7) -- (4);
\draw[dashed] (7) -- (5);
\draw[dashed] (7) -- (6);
\draw[dashed] (8) -- (1);
\draw[dashed] (8) -- (2);
\draw[dashed] (8) -- (5);
\draw[dashed] (8) -- (6);
\draw[dashed] (8) -- (7);
\draw[dashed] (9) -- (2);
\draw[dashed] (9) -- (3);
\draw[dashed] (9) -- (4);
\draw[dashed] (9) -- (6);
\draw[dashed] (9) -- (8);
\draw[dashed] (10) -- (1);
\draw[dashed] (10) -- (2);
\draw[dashed] (10) -- (3);
\draw[dashed] (10) -- (4);
\draw[dashed] (10) -- (7);
\draw[dashed] (10) -- (8);
\draw[thick, double distance = 2pt] (10) -- (9);
\draw[dashed] (11) -- (1);
\draw[dashed] (11) -- (3);
\draw[dashed] (11) -- (5);
\draw[dashed] (11) -- (7);
\draw[dashed] (11) -- (8);
\draw[dashed] (11) -- (9);
\draw[dashed] (11) -- (10);
\draw[dashed] (12) -- (1);
\draw[dashed] (12) -- (2);
\draw[dashed] (12) -- (3);
\draw[dashed] (12) -- (5);
\draw[thick, double distance = 2pt] (12) -- (6);
\draw[dashed] (12) -- (7);
\draw[dashed] (12) -- (9);
\draw[thick, double distance = 2pt] (12) -- (10);
\draw[dashed] (12) -- (11);
\draw[dashed] (13) -- (1);
\draw[dashed] (13) -- (2);
\draw[dashed] (13) -- (3);
\draw[dashed] (13) -- (4);
\draw[dashed] (13) -- (5);
\draw[dashed] (13) -- (6);
\draw[dashed] (13) -- (7);
\draw[dashed] (13) -- (9);
\draw[dashed] (13) -- (11);
\draw[dashed] (14) -- (1);
\draw[dashed] (14) -- (2);
\draw[thick, double distance = 2pt] (14) -- (3);
\draw[dashed] (14) -- (4);
\draw[dashed] (14) -- (5);
\draw[dashed] (14) -- (6);
\draw[thick, double distance = 2pt] (14) -- (7);
\draw[dashed] (14) -- (9);
\draw[dashed] (14) -- (10);
\draw[dashed] (14) -- (11);
\draw[dashed] (14) -- (12);
\draw[dashed] (15) -- (1);
\draw[dashed] (15) -- (2);
\draw[dashed] (15) -- (3);
\draw[dashed] (15) -- (5);
\draw[dashed] (15) -- (7);
\draw[dashed] (15) -- (8);
\draw[dashed] (15) -- (9);
\draw[dashed] (15) -- (10);
\draw[dashed] (15) -- (13);
\draw[dashed] (15) -- (14);
\draw[dashed] (16) -- (1);
\draw[dashed] (16) -- (2);
\draw[dashed] (16) -- (3);
\draw[dashed] (16) -- (4);
\draw[dashed] (16) -- (5);
\draw[dashed] (16) -- (6);
\draw[dashed] (16) -- (7);
\draw[dashed] (16) -- (8);
\draw[thick, double distance = 2pt] (16) -- (9);
\draw[dashed] (16) -- (11);
\draw[dashed] (16) -- (12);
\draw[dashed] (16) -- (15);
\draw[dashed] (17) -- (1);
\draw[dashed] (17) -- (2);
\draw[dashed] (17) -- (3);
\draw[dashed] (17) -- (4);
\draw[dashed] (17) -- (5);
\draw[dashed] (17) -- (6);
\draw[dashed] (17) -- (8);
\draw[dashed] (17) -- (10);
\draw[dashed] (17) -- (11);
\draw[dashed] (17) -- (12);
\draw[dashed] (17) -- (13);
\draw[dashed] (17) -- (14);
\draw[dashed] (17) -- (15);
\draw[dashed] (18) -- (1);
\draw[dashed] (18) -- (2);
\draw[dashed] (18) -- (3);
\draw[dashed] (18) -- (4);
\draw[dashed] (18) -- (5);
\draw[dashed] (18) -- (6);
\draw[dashed] (18) -- (8);
\draw[dashed] (18) -- (9);
\draw[dashed] (18) -- (10);
\draw[dashed] (18) -- (11);
\draw[dashed] (18) -- (12);
\draw[dashed] (18) -- (13);
\draw[dashed] (18) -- (15);

\foreach \i [count=\n] in \nodeorder {
  \pgfmathsetmacro{\angle}{360/\nodecountvalue * (\n-1) + \rotateby}
  \filldraw[fill=white] (\i) circle (2pt) node[anchor=180+\angle] {\scriptsize \i};
}

\draw (3) node[left, xshift = -0.5cm, yshift=.2cm] {\hyperref[subsub:49]{SW49$^\star$:}};
\end{tikzpicture}

\end{center}

\subsection{Gram Matrices}\label{subsec:Gram}\

We include Gram matrices for the Sharlau-Walhorn polyhedra 
whose Coxeter diagrams are difficult to 
parse. The Gram matrix is another encoding of the dihedral angles. In many sources, the entries of the Gram matrix are inner products of the normal vectors to the faces. We use a different convention which is still sufficient to describe the polyhedron. An entry of $n$ in the $(i, j)$ position indicates that face $i$ intersects face $j$ in an angle of $\pi/n$. Entries of $2$ are omitted. An entry of $\infty$ means that faces $i$ and $j$ are parallel, and an entry of $u$ means that they are ultraparallel. 

\subsubsection*{\emph{SW39:}}\label{subsub:39}
$$
\begin{array}{ccccc|ccccc|ccccc|ccc}
 \bullet &   &   & 6 & u &   & u & 4 &   & u &   & u & u & u & u & u & u &
   u \\
   & \bullet &   & \infty &   &   &   & u & u & u & u & u & u &
   u & u & u & u & u \\
   &   & \bullet &   & \infty & u & u & u & u & u & u & u &
   u & u & u & u & u & u \\
 6 & \infty &   & \bullet &   & u & u &   & u &   &   & u &   & u &
   u & u & u & u \\
 u &   & \infty &   & \bullet & u &   & u & u &   & u & u &
   u & u & u & u & u & u \\
   \hline
   &   & u & u & u & \bullet &   & u &   & u & u &   & u & u & u & u
   & u & u \\
 u &   & u & u &   &   & \bullet & u & u &   & u &   & u & 4 & u & u &
   u &   \\
 4 & u & u &   & u & u & u & \bullet &   & u &   & u &   & u &   &   & u &
   u \\
   & u & u & u & u &   & u &   & \bullet & u & u &   & u & u &   & u
   & u & u \\
 u & u & u &   &   & u &   & u & u & \bullet & u & u &   &   & u & u
   & u & u \\
   \hline
   & u & u &   & u & u & u &   & u & u & \bullet & u & u & u &
   u & u & u & u \\
 u & u & u & u & u &   &   & u &   & u & u & \bullet & u & 6 &   &
   \infty &   &   \\
 u & u & u &   & u & u & u &   & u &   & u & u & \bullet &   & u &  
   & u & u \\
 u & u & u & u & u & u & 4 & u & u &   & u & 6 &   & \bullet & u &  
   &   &   \\
 u & u & u & u & u & u & u &   &   & u & u &   & u & u & \bullet &   & \infty & u \\
 \hline
 u & u & u & u & u & u & u &   & u & u & u & \infty &   &   &   &
   \bullet &   & u \\
 u & u & u & u & u & u & u & u & u & u & u &   & u &   & \infty &   & \bullet & u \\
 u & u & u & u & u & u &   & u & u & u & u &   & u &   & u & u & u &
   \bullet \\
\end{array}
$$

\subsubsection*{\emph{SW40:}}\label{subsub:40}
$$
\begin{array}{ccccc|ccccc|ccccc|ccc}
 \bullet &   &   & u & u &   & 3 & 3 & u & u & u & u &   & u &   & u &   & u
   \\
   & \bullet &   & \infty &   & u &   & u & u & u & u & u &
   u & u & u & u & u & u \\
   &   & \bullet &   & \infty &   & u & u & u & u & u & u &
   u & u & u & u & u & u \\
 u & \infty &   & \bullet &   &   & u & u &   & u &   & u & u &
   u & u & u & u & u \\
 u &   & \infty &   & \bullet & u &   & u &   &   & u & u & u &
   u & u & u & u & u \\
   \hline
   & u &   &   & u & \bullet & u & 4 & u & u &   & u & u & u & u & u
   & u & u \\
 3 &   & u & u &   & u & \bullet & u & u &   & u & u & u &   & 4 & u &
   u &   \\
 3 & u & u & u & u & 4 & u & \bullet &   &   &   &   &   & u & u & u & u &
   u \\
 u & u & u &   &   & u & u &   & \bullet &   &   & u & u & u & u & u
   & u & u \\
 u & u & u & u &   & u &   &   &   & \bullet & u &   & u &   & u & u &
   u & u \\
   \hline
 u & u & u &   & u &   & u &   &   & u & \bullet & u & u & u & u &
   u & u & u \\
 u & u & u & u & u & u & u &   & u &   & u & \bullet &   &   & u &  
   & \infty & u \\
   & u & u & u & u & u & u &   & u & u & u &   & \bullet & u &
   u & \infty &   & u \\
 u & u & u & u & u & u &   & u & u &   & u &   & u & \bullet &
   u &   & u &   \\
   & u & u & u & u & u & 4 & u & u & u & u & u & u & u & \bullet &   &   &   \\
   \hline
 u & u & u & u & u & u & u & u & u & u & u &   & \infty &   & 
    & \bullet &   &   \\
   & u & u & u & u & u & u & u & u & u & u & \infty &   & u & 
    &   & \bullet & u \\
 u & u & u & u & u & u &   & u & u & u & u & u & u &   &   &   & u &
   \bullet \\
\end{array}
$$

\subsubsection*{\emph{SW45:}}\label{subsub:45}
$$
\begin{array}{ccccc|ccccc|ccccc|ccccc}
 \bullet &   &   & 6 &   &   & u & u &   & u & u & u & u & u & u & u &
   u & u & u & u \\
   & \bullet &   & \infty &   & u &   &   & u & u &   & u & u & u &
   u & u & u & u & u & u \\
   &   & \bullet &   & u & u & \infty & u & u & u & u & u &
   u & u & u & u & u & u & u & u \\
 6 & \infty &   & \bullet & u & 4 &   & u &   &   & u & u &   &   & u &  
   & u & u & u & u \\
   &   & u & u & \bullet &   & u &   & u & u & u &   & u & u & u & u
   & u & u & u & u \\ 
   \hline
   & u & u & 4 &   & \bullet & u & u &   &   & u &   & u & u & u & u &
   u & u & u & u \\
 u &   & \infty &   & u & u & \bullet & u & u & u &   & u &
   u &   & u & u & u & u & u & u \\
 u &   & u & u &   & u & u & \bullet & u & u &   &   & \infty & u &   & u & 4 & 6 &   &   \\
   & u & u &   & u &   & u & u & \bullet & u & u & u & u & u &
   u & u & u & u & u & u \\
 u & u & u &   & u &   & u & u & u & \bullet & u &   &   & u &   & u
   & u & u & u & u \\ \hline
 u &   & u & u & u & u &   &   & u & u & \bullet & u & u &   & u &
   u &   & u & u & u \\
 u & u & u & u &   &   & u &   & u &   & u & \bullet & u & u &   & u
   & u & u & u & u \\
 u & u & u &   & u & u & u & \infty & u &   & u & u & \bullet & u &   &   & u &   &   & u \\
 u & u & u &   & u & u &   & u & u & u &   & u & u & \bullet &
   u &   &   & u & u & u \\
 u & u & u & u & u & u & u &   & u &   & u &   &   & u & \bullet &
   u & u & u & \infty & u \\ \hline
 u & u & u &   & u & u & u & u & u & u & u & u &   &   & u & \bullet &   &   & u & u \\
 u & u & u & u & u & u & u & 4 & u & u &   & u & u &   & u &   & \bullet &   & u &   \\
 u & u & u & u & u & u & u & 6 & u & u & u & u &   & u & u &   &   &
   \bullet &   &   \\
 u & u & u & u & u & u & u &   & u & u & u & u &   & u & \infty & u & u &   & \bullet & u \\
 u & u & u & u & u & u & u &   & u & u & u & u & u & u & u & u &
     &   & u & \bullet \\
\end{array}
$$

\subsubsection*{\emph{SW46:}}\label{subsub:46}
$$
\begin{array}{ccccc|ccccc|ccccc|ccccc}
\bullet &  &  & u &  & u & 6 &  &  & u & u &  & u & u & u & u & u &u & u & u\\ & \bullet &  & \infty & u &  & u & u &  &  & u & u & u & u &\infty &  & u & u & u & u\\ &  & \bullet &  &  & \infty & u & u & u & u & u & u & u & u & u &u & u & u & u & u\\u & \infty &  & \bullet &  &  & 4 & u & u & u &  & u &  & u & u & u& u & u & u & u\\ & u &  &  & \bullet & u & u &  & u & u &  & u & u & u & u & u & u &u & u & u \\
\hline
u &  & \infty &  & u & \bullet & u & u & u &  & u & u &  &  & u & u& u & u & u & u\\6 & u & u & 4 & u & u & \bullet &  & u & 6 &  &  &  &  & 4 & u &  & & u & \\ & u & u & u &  & u &  & \bullet & u & u &  & u & u & u & u & u & u &u & u & u\\ &  & u & u & u & u & u & u & \bullet & u & u &  & u & u &  &\infty &  & u & u & u\\u &  & u & u & u &  & 6 & u & u & \bullet & u & u & u &  & u &  & u & &  & u \\
\hline
u & u & u &  &  & u &  &  & u & u & \bullet & u & u & u & u & u & u &u & u & u\\ & u & u & u & u & u &  & u &  & u & u & \bullet & u & u & u & u &  &u & u & u\\u & u & u &  & u &  &  & u & u & u & u & u & \bullet &  & u & u & u &u & u & u\\u & u & u & u & u &  &  & u & u &  & u & u &  & \bullet & u & u & u &u & u & u\\u & \infty & u & u & u & u & 4 & u &  & u & u & u & u & u & \bullet&  &  & u &  &  \\
\hline
u &  & u & u & u & u & u & u & \infty &  & u & u & u & u &  &\bullet & u & u &  & u\\u & u & u & u & u & u &  & u &  & u & u &  & u & u &  & u & \bullet &u & u & u\\u & u & u & u & u & u &  & u & u &  & u & u & u & u & u & u & u &\bullet &  & \\u & u & u & u & u & u & u & u & u &  & u & u & u & u &  &  & u &  &\bullet & \\u & u & u & u & u & u &  & u & u & u & u & u & u & u &  & u & u &  & & \bullet 
\end{array}
$$

\subsubsection*{\emph{SW49:}}\label{subsub:49}
$$
\begin{array}{ccccc|ccccc|ccccc|cccc}
\bullet &  &  & u &  & u &  & u &  & u & u & u & u & u & u & u & u & u \\
 & \bullet & 4 & 4 &  &  & u & u & u & u &  & u & u & u & u & u & u & u \\
 & 4 & \bullet &  & u & u &  &  & u & u & u & u & u & 4 & u & u & u & u \\
u & 4 &  & \bullet & u & u & u &  & u & u &  &  & u & u &  & u & u & u \\
 &  & u & u & \bullet &  & u & u &  &  & u & u & u & u & u & u & u & u \\ \hline
u &  & u & u &  & \bullet & u & u & u &  &  & 4 & u & u &  & u & u & u \\
 & u &  & u & u & u & \bullet & u &  & u & u & u & u & 4 & u & u &  & \ \\
u & u &  &  & u & u & u & \bullet & u & u & u &  &  &  & u & u & u & u \\
 & u & u & u &  & u &  & u & \bullet & 4 & u & u & u & u & u & 4 &  & u \\
u & u & u & u &  &  & u & u & 4 & \bullet & u & 4 &  & u & u &  & u & u \\ \hline
u &  & u &  & u &  & u & u & u & u & \bullet & u & u & u &  & u & u & u \\
u & u & u &  & u & 4 & u &  & u & 4 & u & \bullet &  & u &  & u & u & u \\
u & u & u & u & u & u & u &  & u &  & u &  & \bullet &  & u &  & u & u \\
u & u & 4 & u & u & u & 4 &  & u & u & u & u &  & \bullet & u &  & u &  \\
u & u & u &  & u &  & u & u & u & u &  &  & u & u & \bullet & u & u & u \\ \hline
u & u & u & u & u & u & u & u & 4 &  & u & u &  &  & u & \bullet &  &  \\
u & u & u & u & u & u &  & u &  & u & u & u & u & u & u &  & \bullet &  \\
u & u & u & u & u & u &  & u & u & u & u & u & u &  & u &  &  & \bullet 
\end{array}
$$

\bibliographystyle{alpha}

\bibliography{AKbibliog}

@inproceedings{delaHarpe,
    AUTHOR= {de la Harpe, Pierre},
    TITLE= {An invitation to {C}oxeter groups},
    BOOKTITLE={Group Theory from a Geometrical Viewpoint},
    YEAR={1991},
    editor={\'E. Ghys and A. Haefliger and A. Verjowsky},
    pages={193--253},
    address={Trieste},
    publisher={World Scientific},
}

@article {Felikson1998,
    AUTHOR = {Felikson, A. A.},
     TITLE = {Coxeter decompositions of hyperbolic polygons},
   JOURNAL = {European J. Combin.},
  FJOURNAL = {European Journal of Combinatorics},
    VOLUME = {19},
      YEAR = {1998},
    NUMBER = {7},
     PAGES = {801--817},
      ISSN = {0195-6698,1095-9971},
   MRCLASS = {52B05 (20F55)},
  MRNUMBER = {1649962},
MRREVIEWER = {Luis\ Paris},
       DOI = {10.1006/eujc.1998.0238},
       URL = {https://doi.org/10.1006/eujc.1998.0238},
}

@article {Reid1991,
    AUTHOR = {Reid, Alan W.},
     TITLE = {Arithmeticity of knot complements},
   JOURNAL = {J. London Math. Soc. (2)},
  FJOURNAL = {Journal of the London Mathematical Society. Second Series},
    VOLUME = {43},
      YEAR = {1991},
    NUMBER = {1},
     PAGES = {171--184},
      ISSN = {0024-6107,1469-7750},
   MRCLASS = {57M25 (22E40)},
  MRNUMBER = {1099096},
MRREVIEWER = {Colin\ C.\ Adams},
       DOI = {10.1112/jlms/s2-43.1.171},
       URL = {https://doi.org/10.1112/jlms/s2-43.1.171},
}

@article {Kellerhals,
    AUTHOR = {Kellerhals, Ruth},
     TITLE = {A polyhedral approach to the arithmetic and geometry of
              hyperbolic link complements},
   JOURNAL = {J. Knot Theory Ramifications},
  FJOURNAL = {Journal of Knot Theory and its Ramifications},
    VOLUME = {32},
      YEAR = {2023},
    NUMBER = {7},
     PAGES = {Paper No. 2350052, 24},
      ISSN = {0218-2165,1793-6527},
   MRCLASS = {57K10 (51F15 51M25)},
  MRNUMBER = {4626583},
       DOI = {10.1142/S0218216523500529},
       URL = {https://doi.org/10.1142/S0218216523500529},
}

@article {MeyerMillichapTrapp,
    AUTHOR = {Meyer, Jeffrey S. and Millichap, Christian and Trapp, Rolland},
     TITLE = {Arithmeticity and hidden symmetries of fully augmented pretzel
              link complements},
   JOURNAL = {New York J. Math.},
  FJOURNAL = {New York Journal of Mathematics},
    VOLUME = {26},
      YEAR = {2020},
     PAGES = {149--183},
      ISSN = {1076-9803},
   MRCLASS = {57K10 (57K32)},
  MRNUMBER = {4063955},
MRREVIEWER = {Michelle\ Chu},
}

@incollection {NeumannReid,
    AUTHOR = {Neumann, Walter D. and Reid, Alan W.},
     TITLE = {Arithmetic of hyperbolic manifolds},
 BOOKTITLE = {Topology '90 ({C}olumbus, {OH}, 1990)},
    SERIES = {Ohio State Univ. Math. Res. Inst. Publ.},
    VOLUME = {1},
     PAGES = {273--310},
 PUBLISHER = {de Gruyter, Berlin},
      YEAR = {1992},
      ISBN = {3-11-012598-6},
   MRCLASS = {57M50 (11F75 22E40)},
  MRNUMBER = {1184416},
MRREVIEWER = {Michael\ Kapovich},
}

@misc {Kontorovich2016IAS,
    AUTHOR = {Kontorovich, Alex},
     TITLE = {Thin Groups and Applications},
      YEAR = {2016},
NOTE = {Lecture at IAS, \url{https://www.ias.edu/video/analysisbeyond/2016/0521-AlexKontorovich}},
}

@misc {Whitehead2026,
    AUTHOR = {Whitehead, Ian},
     TITLE = {A Note on Right-Angled Ideal Hyperbolic Polyhedra},
      YEAR = {2026},
NOTE = {In preparation},
}

@incollection {VinbergGorbatsevichShvartsman1988,
    AUTHOR = {Vinberg, \`E.\ B. and Gorbatsevich, V. V. and Shvartsman, O.
              V.},
     TITLE = {Discrete subgroups of {L}ie groups},
 BOOKTITLE = {Current problems in mathematics. {F}undamental directions,
              {V}ol.\ 21 ({R}ussian)},
    SERIES = {Itogi Nauki i Tekhniki},
     PAGES = {5--120, 215},
 PUBLISHER = {Akad. Nauk SSSR, Vsesoyuz. Inst. Nauchn. i Tekhn. Inform.,
              Moscow},
      YEAR = {1988},
   MRCLASS = {22E40},
  MRNUMBER = {968445},
MRREVIEWER = {G.\ A.\ So\u ifer},
}

@article {BorelHC1962,
    AUTHOR = {Borel, Armand and Harish-Chandra},
     TITLE = {Arithmetic subgroups of algebraic groups},
   JOURNAL = {Ann. of Math. (2)},
  FJOURNAL = {Annals of Mathematics. Second Series},
    VOLUME = {75},
      YEAR = {1962},
     PAGES = {485--535},
      ISSN = {0003-486X},
   MRCLASS = {20.65 (14.50)},
  MRNUMBER = {147566},
MRREVIEWER = {P.\ Cartier},
       DOI = {10.2307/1970210},
       URL = {https://doi.org/10.2307/1970210},
}

@book{Kapovich2001,
    AUTHOR = {Kapovich, Michael},
    TITLE = {Hyperbolic manifolds and discrete groups},
    SERIES = {Progress in Mathematics},
    VOLUME = {183},
    PUBLISHER = {Birkhäuser Boston Inc.},
    ADDRESS = {Boston, MA},
    YEAR = {2001},
    PAGES = {xxvi+467},
    ISBN = {0-8176-3904-7},
    MRCLASS = {57M50 (20F65 20H10 30F40 30F60 32G15)},
    MRNUMBER = {1792613}
}

@article{RoederHubbardDunbar2007,
    AUTHOR = {Roeder, Roland K. W. and Hubbard, John H. and Dunbar, William D.},
    TITLE = {Andreev's theorem on hyperbolic polyhedra},
    JOURNAL = {Ann. Inst. Fourier (Grenoble)},
    VOLUME = {57},
    YEAR = {2007},
    NUMBER = {3},
    PAGES = {825--882},
    ISSN = {0373-0956},
    MRCLASS = {52B10 (30F40 51M25 52C15)},
    MRNUMBER = {2336781},
    NOTE = {Corrects errors in Andreev's original proof}
}

@article{BobenkoSpringborn2004,
    AUTHOR = {Bobenko, Alexander I. and Springborn, Boris A.},
    TITLE = {Variational principles for circle patterns and {K}oebe's theorem},
    JOURNAL = {Trans. Amer. Math. Soc.},
    VOLUME = {356},
    YEAR = {2004},
    NUMBER = {2},
    PAGES = {659--689},
    ISSN = {0002-9947},
    MRCLASS = {52C15 (53A30)},
    MRNUMBER = {2022715}
}

@book{Stephenson2005,
    AUTHOR = {Stephenson, Kenneth},
    TITLE = {Introduction to circle packing},
    SUBTITLE = {The theory of discrete analytic functions},
    PUBLISHER = {Cambridge University Press},
    ADDRESS = {Cambridge},
    YEAR = {2005},
    PAGES = {xii+356},
    ISBN = {0-521-82356-0},
    MRCLASS = {52C15 (30-01 30C35 52-01)},
    MRNUMBER = {2131318}
}

@article{BrightwellScheinerman1993,
    AUTHOR = {Brightwell, Graham R. and Scheinerman, Edward R.},
    TITLE = {Representations of planar graphs},
    JOURNAL = {SIAM J. Discrete Math.},
    VOLUME = {6},
    YEAR = {1993},
    NUMBER = {2},
    PAGES = {214--229},
    ISSN = {0895-4801},
    MRCLASS = {05C10 (52C15)},
    MRNUMBER = {1215229}
}

@article{ChowLuo2003,
    AUTHOR = {Chow, Bennett and Luo, Feng},
    TITLE = {Combinatorial {R}icci flows on surfaces},
    JOURNAL = {J. Differential Geom.},
    VOLUME = {63},
    YEAR = {2003},
    NUMBER = {1},
    PAGES = {97--129},
    ISSN = {0022-040X},
    MRCLASS = {53C44 (58J35)},
    MRNUMBER = {2015261}
}

@article{ColindeVerdiere1991,
    AUTHOR = {Colin de Verdière, Yves},
    TITLE = {Un principe variationnel pour les empilements de cercles},
    JOURNAL = {Invent. Math.},
    VOLUME = {104},
    YEAR = {1991},
    NUMBER = {3},
    PAGES = {655--669},
    ISSN = {0020-9910},
    MRCLASS = {52C15 (30C35 58E30)},
    MRNUMBER = {1106755}
}

@article{Rivin1994,
    AUTHOR = {Rivin, Igor},
    TITLE = {Euclidean structures on simplicial surfaces and hyperbolic volume},
    JOURNAL = {Ann. of Math. (2)},
    VOLUME = {139},
    YEAR = {1994},
    NUMBER = {3},
    PAGES = {553--580},
    ISSN = {0003-486X},
    MRCLASS = {52B10 (30F40 52C15 53A35)},
    MRNUMBER = {1283870}
}

@article{Rivin1996,
    AUTHOR = {Rivin, Igor},
    TITLE = {A characterization of ideal polyhedra in hyperbolic {3}-space},
    JOURNAL = {Ann. of Math. (2)},
    VOLUME = {143},
    YEAR = {1996},
    NUMBER = {1},
    PAGES = {51--70},
    ISSN = {0003-486X},
    MRCLASS = {52B10 (52A55 53A35)},
    MRNUMBER = {1370757}
}

@article{Schramm1991,
    AUTHOR = {Schramm, Oded},
    TITLE = {Rigidity of infinite (circle) packings},
    JOURNAL = {J. Amer. Math. Soc.},
    VOLUME = {4},
    YEAR = {1991},
    NUMBER = {1},
    PAGES = {127--149},
    ISSN = {0894-0347},
    MRCLASS = {52C15 (30C35 52A45)},
    MRNUMBER = {1076089}
}

@article{Schramm1991b,
    AUTHOR = {Schramm, Oded},
    TITLE = {Existence and uniqueness of packings with specified combinatorics},
    JOURNAL = {Israel J. Math.},
    VOLUME = {73},
    YEAR = {1991},
    NUMBER = {3},
    PAGES = {321--341},
    ISSN = {0021-2172},
    MRCLASS = {52C15 (30C35)},
    MRNUMBER = {1135215}
}

@article{HeSchramm1993,
    AUTHOR = {He, Zheng-Xu and Schramm, Oded},
    TITLE = {Fixed points, {K}oebe uniformization and circle packings},
    JOURNAL = {Ann. of Math. (2)},
    VOLUME = {137},
    YEAR = {1993},
    NUMBER = {2},
    PAGES = {369--406},
    ISSN = {0003-486X},
    MRCLASS = {30C35 (52C15)},
    MRNUMBER = {1207210}
}

@article{Andreev1970a,
    AUTHOR = {Andreev, E. M.},
    TITLE = {On convex polyhedra in Lobachevskii space},
    JOURNAL = {Math. USSR Sb.},
    VOLUME = {10},
    YEAR = {1970},
    PAGES = {413--440},
    NOTE = {English translation}
}

@article{Andreev1970c,
    AUTHOR = {Andreev, E. M.},
    TITLE = {Convex polyhedra of finite volume in Lobačevskiĭ space},
    JOURNAL = {Mat. Sb. (N.S.)},
    VOLUME = {83},
    NUMBER = {125},
    YEAR = {1970},
    PAGES = {256--260}
}

@article {Koebe1936,
  title={Kontaktprobleme der konformen Abbildung},
  author={Koebe, Paul},
  journal={Berichte Verh. S\"achs. Akademie der Wissenschaften Leipzig, Math.-Phys. Klasse},
  volume={88},
  pages={141--164},
  year={1936}
}

@misc{Scharlau,
    author = {Scharlau, Rudolf},
    title = {On the classification of arithmetic reflection groups
             on hyperbolic 3-space},
    note = {Preprint, Bielefeld},
year = {1989}
}

@article {ScharlauWalhorn1992,
    AUTHOR = {Scharlau, Rudolf and Walhorn, Claudia},
     TITLE = {Integral lattices and hyperbolic reflection groups},
      NOTE = {Journ\'ees Arithm\'etiques, 1991 (Geneva)},
   JOURNAL = {Ast\'erisque},
  FJOURNAL = {Ast\'erisque},
    VOLUME = {209},
      YEAR = {1992},
     PAGES = {15--16, 279--291},
      ISSN = {0303-1179,2492-5926},
   MRCLASS = {11H06 (20H15)},
  MRNUMBER = {1211022},
MRREVIEWER = {Ren\ Ding},
}

@article {KontorovichNakamura2019,
    AUTHOR = {Kontorovich, Alex and Nakamura, Kei},
     TITLE = {Geometry and arithmetic of crystallographic sphere packings},
   JOURNAL = {Proc. Natl. Acad. Sci. USA},
  FJOURNAL = {Proceedings of the National Academy of Sciences of the United
              States of America},
    VOLUME = {116},
      YEAR = {2019},
    NUMBER = {2},
     PAGES = {436--441},
      ISSN = {1091-6490},
   MRCLASS = {52C26 (20F55 51F15)},
  MRNUMBER = {3904690},
       DOI = {10.1073/pnas.1721104116},
       URL = {https://doi.org/10.1073/pnas.1721104116},
}

@InProceedings{Ziegler2007,
  author =	 {G\"unter M. Ziegler},
  title =	 {Convex Polytopes: {E}xtremal constructions and
                  $f$-vector shapes},
  booktitle =	 {``Geometric Combinatorics'', Proc. Park City
                  Mathematical Institute (PCMI) 2004},
  year =	 2007,
  editor =	 {E. Miller and V. Reiner and B. Sturmfels},
  address =	 {Providence, RI},
  publisher =	 {Amer. Math. Society},
  series =       {IAS/Park City Math. Ser.},
  volume =       13,
  note =	 {With an appendix by Th.\ Schr\"oder and N. Witte},
  pages =	 {617-691}
}

@article {Vinberg1981,
    AUTHOR = {Vinberg, \`E. B.},
     TITLE = {The nonexistence of crystallographic reflection groups in
              {L}obachevski\u\i \ spaces of large dimension},
   JOURNAL = {Funktsional. Anal. i Prilozhen.},
  FJOURNAL = {Akademiya Nauk SSSR. Funktsional\cprime ny\u\i \ Analiz i ego
              Prilozheniya},
    VOLUME = {15},
      YEAR = {1981},
    NUMBER = {2},
     PAGES = {67--68},
      ISSN = {0374-1990},
   MRCLASS = {51F15 (20H15 22E40 52A25)},
  MRNUMBER = {617472},
MRREVIEWER = {H. Schwerdtfeger},
}

@article {Nikulin2007,
    AUTHOR = {Nikulin, V. V.},
     TITLE = {Finiteness of the number of arithmetic groups generated by
              reflections in {L}obachevski\u\i \ spaces},
   JOURNAL = {Izv. Ross. Akad. Nauk Ser. Mat.},
  FJOURNAL = {Rossi\u\i skaya Akademiya Nauk. Izvestiya. Seriya
              Matematicheskaya},
    VOLUME = {71},
      YEAR = {2007},
    NUMBER = {1},
     PAGES = {55--60},
      ISSN = {1607-0046},
   MRCLASS = {20F55 (11E12 22E40)},
  MRNUMBER = {2477273},
MRREVIEWER = {I. Dolgachev},
       DOI = {10.1070/IM2007v071n01ABEH002349},
       URL = {http://dx.doi.org/10.1070/IM2007v071n01ABEH002349},
}

@article {Vinberg1967,
    AUTHOR = {Vinberg, \`E. B.},
     TITLE = {Discrete groups generated by reflections in {L}oba\v cevski\u\i \
              spaces},
   JOURNAL = {Mat. Sb. (N.S.)},
    VOLUME = {72 (114)},
      YEAR = {1967},
     PAGES = {471--488; correction, ibid. 73 (115) (1967), 303},
   MRCLASS = {20.60 (10.25)},
  MRNUMBER = {0207853},
MRREVIEWER = {H. Schwerdtfeger},
}

@article {Belolipetsky2016,
    AUTHOR = {Belolipetsky, Mikhail},
     TITLE = {Arithmetic hyperbolic reflection groups},
   JOURNAL = {Bull. Amer. Math. Soc. (N.S.)},
  FJOURNAL = {American Mathematical Society. Bulletin. New Series},
    VOLUME = {53},
      YEAR = {2016},
    NUMBER = {3},
     PAGES = {437--475},
      ISSN = {0273-0979},
   MRCLASS = {20F55 (11F06 11H56 20H05)},
  MRNUMBER = {3501796},
MRREVIEWER = {O. V. Shvartsman},
       DOI = {10.1090/bull/1530},
       URL = {http://dx.doi.org/10.1090/bull/1530},
}

@article {ABSW2008,
    AUTHOR = {Agol, Ian and Belolipetsky, Mikhail and Storm, Peter and
              Whyte, Kevin},
     TITLE = {Finiteness of arithmetic hyperbolic reflection groups},
   JOURNAL = {Groups Geom. Dyn.},
  FJOURNAL = {Groups, Geometry, and Dynamics},
    VOLUME = {2},
      YEAR = {2008},
    NUMBER = {4},
     PAGES = {481--498},
      ISSN = {1661-7207},
   MRCLASS = {20F55 (11F06 11H56 22E40)},
  MRNUMBER = {2442945},
MRREVIEWER = {Ruth Kellerhals},
       DOI = {10.4171/GGD/47},
       URL = {http://dx.doi.org/10.4171/GGD/47},
}

@article {LongMaclachlanReid2006,
    AUTHOR = {Long, D. D. and Maclachlan, C. and Reid, A. W.},
     TITLE = {Arithmetic {F}uchsian groups of genus zero},
   JOURNAL = {Pure Appl. Math. Q.},
  FJOURNAL = {Pure and Applied Mathematics Quarterly},
    VOLUME = {2},
      YEAR = {2006},
    NUMBER = {2, part 2},
     PAGES = {569--599},
      ISSN = {1558-8599},
   MRCLASS = {20H10 (11F06)},
  MRNUMBER = {2251482},
MRREVIEWER = {Jack O. Button},
       DOI = {10.4310/PAMQ.2006.v2.n2.a9},
       URL = {http://dx.doi.org/10.4310/PAMQ.2006.v2.n2.a9},
}

@incollection {Agol2006,
    AUTHOR = {Agol, Ian},
     TITLE = {Finiteness of arithmetic {K}leinian reflection groups},
 BOOKTITLE = {International {C}ongress of {M}athematicians. {V}ol. {II}},
     PAGES = {951--960},
 PUBLISHER = {Eur. Math. Soc., Z\"urich},
      YEAR = {2006},
   MRCLASS = {30F40 (20H10 57M07)},
  MRNUMBER = {2275630},
MRREVIEWER = {Igor Belegradek},
}

@book{Thurstonbook,
    Author = {Thurston, William P.},
    Title = {The Geometry and Topology of Three-Manifolds},
    Note = {Princeton University lecture notes (1978-1981). Electronic version available at \url{https://library.msri.org/books/gt3m/}},
    Year = {2002},
    Publisher = {Princeton University}
}

\end{document}